\documentclass[absolute]{mymathart}
\usepackage{mymathmacros}
\usepackage{color,stmaryrd,caption,keyval}
\usepackage{subfig}
\usepackage[shortlabels]{enumitem}
\usepackage{microtype}

\usepackage[T1]{fontenc}
\usepackage{lmodern}

\newcommand{\Deriv}{{\rm D}}

\newcommand{\escapingcomposant}{escaping composant}
\newcommand{\Escapingcomposant}{Escaping composant}

\newcommand{\anguine}{anguine}
\newcommand{\Anguine}{Anguine}

\numberwithin{equation}{section}

\theoremstyle{changebreak}

\newtheorem{thm}{Theorem}[section]%
\newtheorem{lem}[thm]{Lemma}%
\newtheorem{cor}[thm]{Corollary}%
\newtheorem{prop}[thm]{Proposition}%
\newtheorem{obs}[thm]{Observation}%

\newcommand{\uarrow}[1]{\underleftarrow{#1}}

\newcommand{\HH}{\mathbb{H}}

\newcommand{\interior}{\operatorname{int}}

\theoremstyle{defnbreak}
\newtheorem{rmk}[thm]{Remark}%
\newtheorem{defn}[thm]{Definition}%

\newcommand{\invlim}{\varprojlim}
\newcommand{\T}{\mathcal{T}}

\newcommand{\B}{\mathcal{B}}

\newcommand{\real}{\operatorname{real}}

\usepackage{hyperref}

\usepackage{orcidlink}

\title[Arc-like continua in Julia sets]{Arc-like continua, \\ Julia sets of entire functions, \\ and Eremenko's Conjecture}

\author[Lasse Rempe]{Lasse Rempe\orcidlink{0000-0001-8032-8580}} 
\address{\noindent Dept. of Mathematics \\ The University of Manchester \\ Manchester \\ M13 9PL \\ UK }
\email{lasse.rempe@manchester.ac.uk}
\thanks{Partially supported by EPSRC Fellowship EP/E052851/1 and a Philip Leverhulme Prize.}

\newcounter{inductionstep}
\newcounter{inductionsubstep}
\newcounter{tempcounter}
\makeatletter

\newcommand{\begininductiveconstruction}{\setcounter{inductionstep}{1}\setcounter{inductionsubstep}{0}}
\newcommand{\nextinductivestep}{\stepcounter{inductionstep}\setcounter{inductionsubstep}{0}}

\newenvironment{enuminduction}{\stepcounter{inductionsubstep}\begin{enumerate}[label=(I\arabic{inductionstep}.\arabic*),
     start=\value{inductionsubstep},after=\setcounter{inductionsubstep}{\value{enumi}}
     ]}{\end{enumerate}}

\makeatletter
\newenvironment{equinduction}
 {%
  \setcounter{tempcounter}{\value{equation}}%
  \setcounter{equation}{\value{inductionsubstep}}%
  \let\saved@theequation\theequation
  \let\saved@theHequation\theHequation
   \renewcommand{\theequation}{I\arabic{inductionstep}.\arabic{equation}}%
  \renewcommand{\theHequation}{I\arabic{inductionstep}.\arabic{equation}}%
  \equation
 }
 {%
  \endequation
  \setcounter{inductionsubstep}{\value{equation}}%
  \setcounter{equation}{\value{tempcounter}}%
  \let\theequation\saved@theequation
  \let\theHequation\saved@theHequation
 }
 \makeatother

\makeatletter
\newenvironment{aligninduction}
 {%
  \setcounter{tempcounter}{\value{equation}}%
  \setcounter{equation}{\value{inductionsubstep}}%
  \let\saved@theequation\theequation
  \let\saved@theHequation\theHequation
   \renewcommand{\theequation}{I\arabic{inductionstep}.\arabic{equation}}%
  \renewcommand{\theHequation}{I\arabic{inductionstep}.\arabic{equation}}%
  \align
 }
 {%
  \endalign
  \setcounter{inductionsubstep}{\value{equation}}%
  \setcounter{equation}{\value{tempcounter}}%
  \let\theequation\saved@theequation
  \let\theHequation\saved@theHequation
 }
 \makeatother

\newenvironment{enumequ}{\stepcounter{equation}\begin{enumerate}[label={(\thesection.\arabic{enumi})},
     start=\value{equation},after=\setcounter{equation}{\value{enumi}}
     ]}{\end{enumerate}}

\renewcommand{\A}{\mathcal{A}}

\newcommand{\classS}{\mathcal{S}}

\newcommand{\F}{\mathcal{F}}
\newcommand{\G}{\mathcal{G}}

\newcommand{\Blog}{\mathcal{B}_{\log}}
\newcommand{\BlogP}{\Blog^{\operatorname{p}}}
\newcommand{\s}{\underline{s}}

\newcommand{\V}{\mathcal{V}}

\newcommand{\CH}{\hat{C}}

\newcommand{\Jsh}{\hat{J}_{\s}}
\newcommand{\Js}{J_{\s}}

\newcommand{\sep}{\operatorname{sep}}

\newcommand{\id}{\operatorname{id}}

\begin{document} 

\begin{abstract}
 A transcendental entire function that is hyperbolic with
  connected Fatou set is said to be 
   ``of disjoint type''. 
   A good understanding of these functions is known to have wider implications; 
    for example, a disjoint-type function provides
     a model for the dynamics near infinity of all maps in the same parameter space.

 The goal
     of this memoir is to study the topological properties of the Julia sets of entire functions of disjoint type.
    In particular, we give a detailed description of the possible topology of their connected components.
   More precisely, let $C$ be a connected component of such a Julia set,
   and consider the \emph{Julia continuum} $\hat{C} \defeq  C\cup\{\infty\}$. We show that $\infty$ is a terminal point of $\hat{C}$, and that
   $\hat{C}$ has span zero in the
   sense of Lelek; under a mild additional geometric assumption the continuum $\hat{C}$ is arc-like. (Whether every span zero continuum is
   also arc-like was a famous question in continuum theory, 
   posed by Lelek in 1961, and only recently resolved in the negative by work of Hoehn.) Conversely, 
  we construct a single disjoint-type entire function $f$ with the remarkable property that each arc-like continuum with at least one terminal point 
     is realised as a Julia continuum of $f$.
   We remark that the class of arc-like continua with terminal points is uncountable. It includes, in particular, the 
  $\sin(1/x)$-curve, the Knaster bucket-handle and the pseudo-arc, so these can all occur as Julia continua of a disjoint-type entire function. 

  We also give similar descriptions of the possible topology of Julia continua that 
   contain periodic points or points with bounded orbits, and answer a 
   question of Bara\'nski and Karpi\'nska by showing that Julia continua need not
   contain points that are accessible from the Fatou set.
   Furthermore, we construct a disjoint-type entire function whose Julia 
   set has connected components on which the iterates tend to infinity
   pointwise, but not uniformly. This property is related to a famous conjecture of Eremenko concerning escaping sets of entire functions.
\end{abstract}

\maketitle

\pagebreak

\tableofcontents

\pagebreak

\section{Introduction}\label{sec:intro}
  
  We consider the iteration of transcendental entire functions; i.e.\ of non-polynomial 
   holomorphic self-maps of the complex plane. 
   This topic was founded by Fatou in a seminal article of 1926 \cite{fatou}, and gives
   rise to many beautiful phenomena and interesting questions. In addition, the past decade
   has seen an increasing influence of transcendental phenomena on 
   the fields of rational and polynomial dynamics.
   For example, work of Inou and Shishikura as well as of Buff and Ch\'eritat implies that  
   certain well-known features of transcendental dynamics occur naturally near non-linearisable
   fixed points of
   quadratic polynomials 
     \cite{shishikuratalkcantorbouquets,shishikurarempemodels,cheraghicantorbouquets}.
     Hence results and arguments in 
   this context are now often motivated by properties first discovered in 
   transcendental dynamics. Similarly, Dudko and Lyubich 
   have applied recent techniques from
   transcendental dynamics directly to renormalisation fixed points,
   in order to obtain local connectivity of the Mandelbrot set at certain 
   parameters~\cite{dudkolyubichpacman}.
   Thus it is to be hoped that a better understanding of the 
   transcendental case will lead to further insights in the polynomial
   and rational setting as well. 

 In this memoir, we study a particular class of transcendental entire functions, namely those that are of \emph{disjoint type}; i.e.\ 
  hyperbolic with connected Fatou set.
  To provide the required definitions, 
   recall that the \emph{Fatou set} $F(f)$ of a transcendental entire function $f$ consists of those points $z$ that have a neighbourhood on 
   which the family of iterates 
    \[ f^n \defeq  \underbrace{f\circ\dots\circ f}_{n\text{ times}} \]
   is equicontinuous with respect to the spherical metric. (I.e., these are the points where 
   small perturbations of the starting point $z$ result only in small changes
   of $f^n(z)$, independently of $n$.) Its complement $J(f) \defeq  \C\setminus F(f)$ is called the \emph{Julia set}; it
   is the set on which $f$ exhibits ``chaotic'' behaviour. We also recall that the set $S(f)$ of (finite) \emph{singular values} is the closure of all
   critical and asymptotic values of $f$ in $\C$. Equivalently, it is the smallest closed set $S$ such that
   $f\colon \C\setminus f^{-1}(S)\to \C\setminus S$ is a covering map. 

 \begin{defn}[Hyperbolicity and disjoint type]
  An entire function $f$ is called \emph{hyperbolic} if the set $S(f)$ is bounded and every point in $S(f)$ tends to an attracting periodic 
   cycle of $f$ under iteration. If $f$ is 
   hyperbolic and furthermore $F(f)$ is connected, then we say that $f$ is of \emph{disjoint type}. 
 \end{defn}

As a simple example, consider the maps 
    \[ S_{\lambda}(z) \defeq  \lambda \sin(z), \qquad \lambda \in (0,1). \]
   It is elementary to verify that both singular values $\pm \lambda$, and indeed all
   real starting values, tend to the fixed point $0$ under iteration. In particular, $S_{\lambda}$ is 
   hyperbolic, and it is not difficult to deduce that the Fatou set consists only of the
   immediate basin of attraction of $0$. So each $S_{\lambda}$ is of disjoint type. 
   Fatou \cite[p.~369]{fatou} observed already in  1926 that $J(S_{\lambda})$ contains infinitely many curves on which the iterates
   tend to infinity (namely, iterated preimages of an infinite piece of the imaginary axis), and asked whether
   this is true for more general classes of functions. In fact, the entire set $J(S_{\lambda})$ can be written
   as an uncountable union of arcs, often called ``hairs'', each connecting a finite endpoint with $\infty$.\footnote{%
  To our knowledge, this fact (for $\lambda< 0.85$) was first proved explicitly by Aarts and Oversteegen \cite[Theorem 5.7]{aartsoversteegen}. Devaney and Tangerman \cite{devaneytangerman} had previously discussed at least the existence of
    ``Cantor bouquets'' of arcs in the Julia set, and the proof that the whole Julia set has this property is analogous to the
    proof for the disjoint-type exponential maps $z\mapsto \lambda e^z$ with $0<\lambda<1/e$, first established in
   \cite[p.~50]{devaneykrych}; see also 
   \cite{devgoldberg}.}
   Every point on one of these arcs, with the possible exception of the finite endpoint, tends to infinity under
   iteration. This led Eremenko
   \cite{alexescaping} to strengthen Fatou's question by asking whether, for an arbitrary entire function $f$, every point of the \emph{escaping set}
  \[ I(f) \defeq  \{z\in\C\colon f^n(z)\to\infty\} \]
   could be connected to infinity by an arc in $I(f)$.  

 It turns out that the situation is not as simple as suggested by these questions, even for a disjoint-type entire function $f$. Indeed, while 
  for such $f$ the answer to Eremenko's question is positive if the function
   additionally 
   has finite order of growth in the sense of classical function 
   theory (a result obtained independently in \cite{baranskihyperbolic} and \cite[Theorem 1.2]{strahlen}), there is a disjoint-type
    entire function 
   $f$ such that $J(f)\supset I(f)$ contains no arcs at all \cite[Theorem~8.4]{strahlen}. 
   This shows that there is a wide variety of possible topological types of components of $J(f)$,
   even for $f$ of disjoint type.

 In view of this fact, it may seem
   surprising that it is nonetheless possible to give an essentially
    complete solution to the question of
   which topological objects can arise in this manner. This solution, which will be achieved by combining 
   concepts and techniques from the study of topological continua (non-empty compact, connected metric spaces) 
  with the modern theory of transcendental dynamics,
   is the main result of this memoir. 

 More precisely, let 
  $f$ be of disjoint type. Then it is easy to see (Corollary~\ref{cor:uncountablymany}) that the Julia set $J(f)$ 
  is a union of uncountably many connected components,
  each of which is closed and unbounded. If $C$ is such a component, we call
  the continuum $\CH \defeq  C\cup\{\infty\}$ a \emph{Julia continuum} of $f$. 
  In the
  case of $z\mapsto \lambda \sin(z)$ with $\lambda\in (0,1)$, every Julia continuum is an arc, while in
  the example of \cite[Theorem~8.4]{strahlen}, every Julia continuum is a non-degenerate continuum that
  contains no arcs.

To obtain our most precise statement about the possible topology of Julia continua, we 
  impose a mild
  additional function-theoretic restriction on the entire function under consideration. 

\begin{defn}[{Bounded slope {\cite[Definition 5.1]{strahlen}}}]
  An entire function $f$ is said to have \emph{bounded slope} if there exists a curve 
    $\gamma\colon [0,\infty)\to\C$ such that $|f(\gamma(t))|\to\infty$ as $t\to\infty$ and 
    such that
      \[ \limsup_{t\to\infty} \frac{|\arg(\gamma(t))|}{\log|\gamma(t)|} < \infty. \]
\end{defn}
\begin{remark}[Remark 1]
 Intuitively, the condition says that the total spiralling of 
   the \emph{tracts} of $f$ (i.e., the connected components of the set of points where $f$ is large) 
  is no larger than that of a logarithmic spiral. Equivalently, under a logarithmic transformation, each such tract is contained in a fixed sector
  around the real axis (see Definition~\ref{defn:boundedslopelog}); this is the reason for the terminology ``bounded slope''. 
\end{remark}
\begin{remark}[Remark 2]
  Any disjoint-type entire function that is real on the real axis has bounded slope. So does the 
   counterexample to Eremenko's question constructed in \cite{strahlen}, and, as far as we 
   are aware, any 
   specific example or family of entire functions whose dynamics has been considered in the past. 
   Furthermore, suppose that $f$ and $g$ belong to the Eremenko--Lyubich class  $\B$ (see below) and that 
   additionally  $g$ has bounded slope~-- for example, $g\in\B$ has 
   finite order of growth, or is real on the real axis. Then $f\circ g$ also has bounded slope~-- regardless of the
   function-theoretic properties of $f$. 
   For these reasons, the 
   restriction is indeed rather mild. 
\end{remark}
\begin{remark}[Remark 3]
  We use ``bounded slope'' in this and the next section mainly for convenience, as it
    is an established property that is easy to define. However, in every one of these
    results, it can be replaced by a more technical, but far more general condition: 
    having ``{\anguine} tracts'' in the sense of Definition \ref{defn:arcliketracts}.  
\end{remark}

\begin{figure}
  \subfloat[The arc]{\includegraphics[width=.3\textwidth]{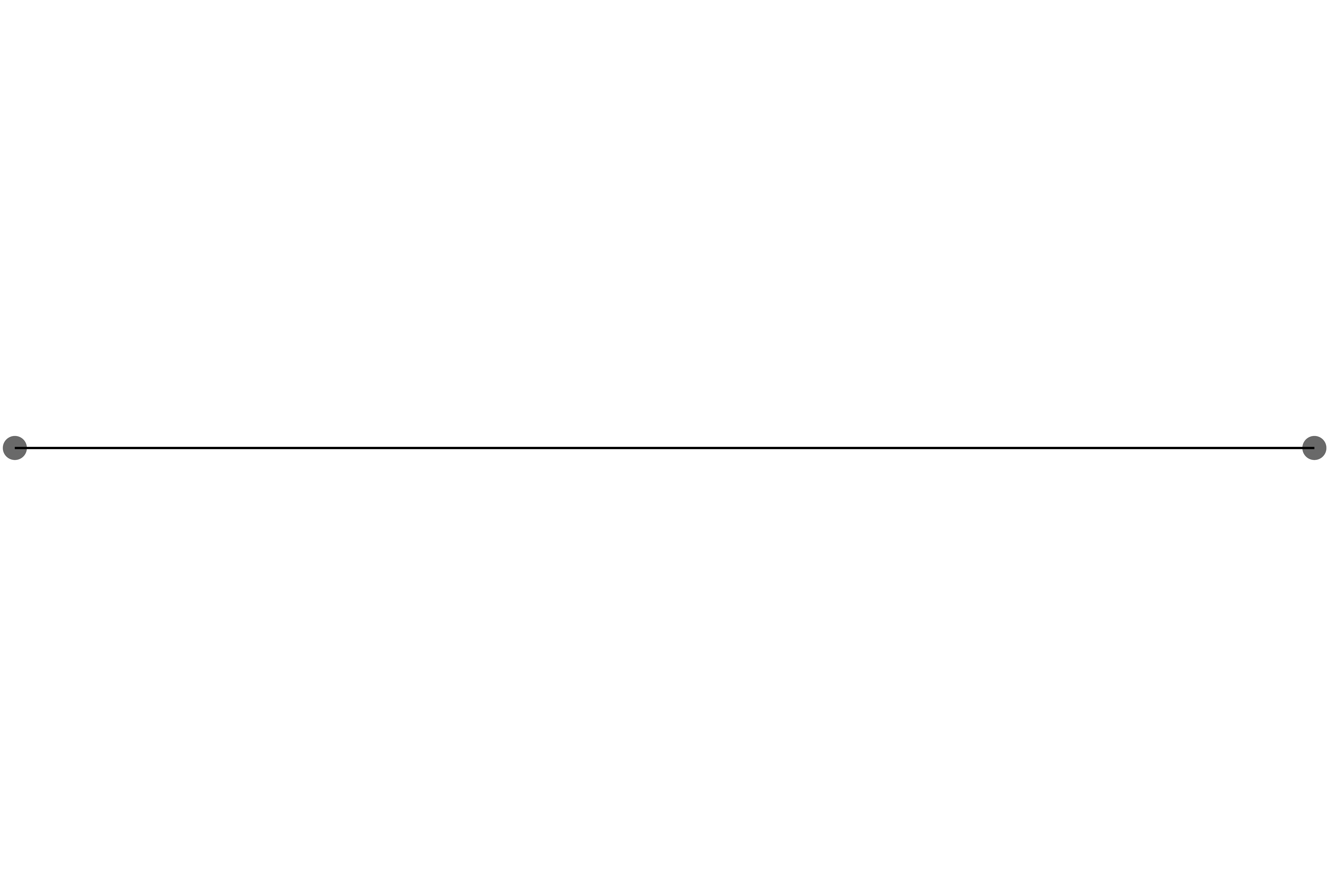}}\hfill
  \subfloat[%
   The $\sin(1/x)$-continuum]{\includegraphics[width=.6\textwidth]{sinecurve}\label{subfig:sinecurve}}\\
  \subfloat[Knaster bucket-handle]{%
    \includegraphics[width=.3\textwidth]{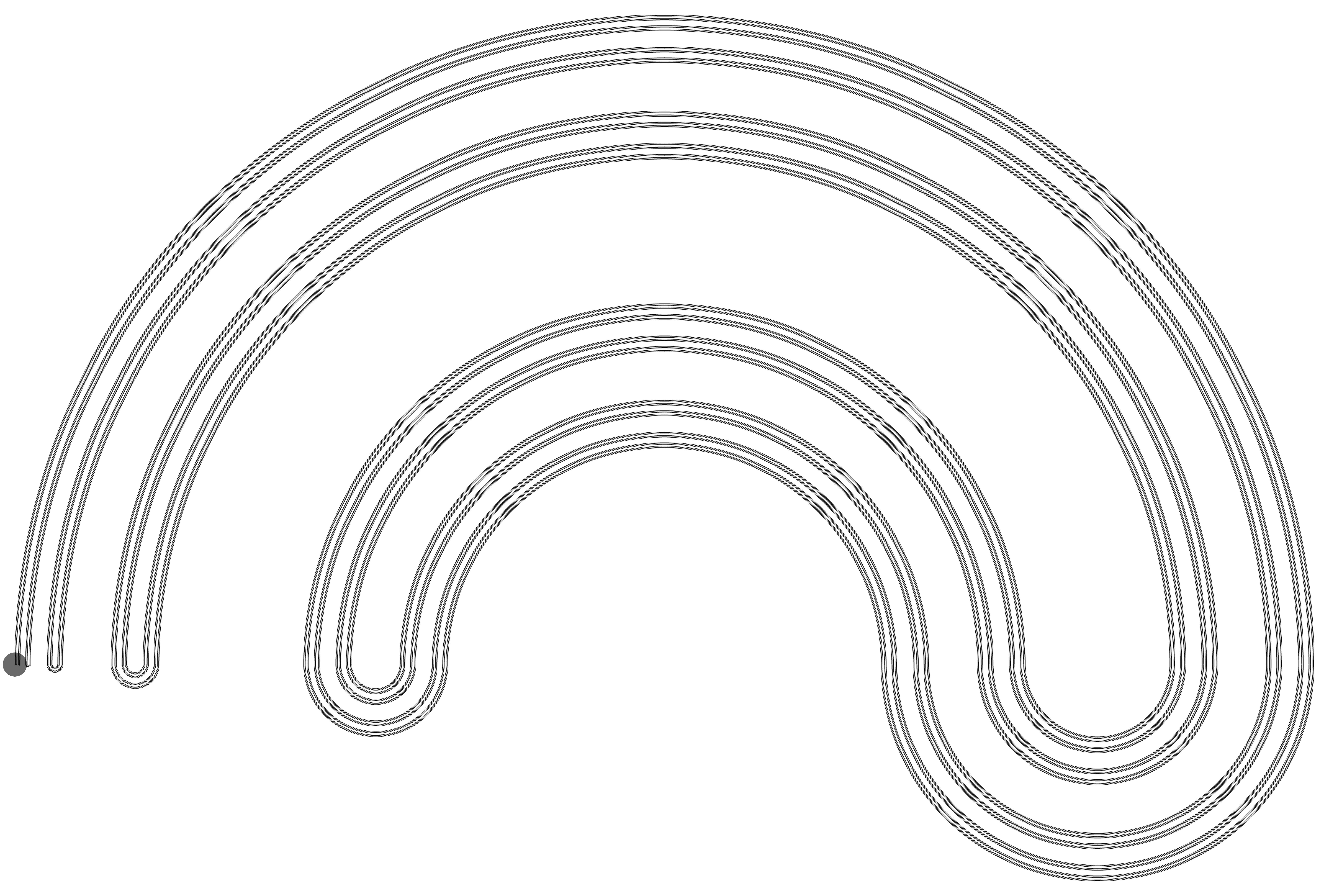}\label{subfig:knaster}}\hfill
  \subfloat[A double bucket-handle]{%
    \includegraphics[width=.6\textwidth]{doubleknaster}}
 \caption{Some examples of arc-like continua; terminal points are marked by grey circles.
   (The numbers of terminal points in these continua are two, three, one and zero, respectively.)%
    \label{fig:arclike}}
\end{figure}

 We can now state our main theorem: The possible Julia continua of disjoint-type entire functions of bounded 
    slope are precisely the \emph{arc-like continua having terminal points}. Essentially, a continuum is \emph{arc-like} if 
    can be turned into an arc by collapsing subsets of small diameter to points, while a terminal point is a natural
    generalisation of the end-point of an arc. (See Figure~\ref{fig:arclike} for examples, and
    Definition~\ref{defn:topologicalproperties} and Section~\ref{sec:topology} for formal
    definitions and further discussion.)  The class of arc-like continua is
    extremely rich (e.g., there are
    uncountably many pairwise non-homeomorphic arc-like continua) and has been studied 
    intensively by topologists since the work of  Bing \cite{bingsnakelike} in  the~1950s. 

\begin{thm}[Topology of bounded-slope Julia continua] \label{thm:mainarclike}
  Let $\CH$ be a Julia continuum of a disjoint-type entire function $f$ of bounded slope. Then
    $\CH$ is arc-like and $\infty$ is a terminal point of $\CH$. 

  Conversely, there exists a disjoint-type entire function $f$ having bounded slope with the following property. If $X$ is any arc-like continuum
    having a terminal point $x\in X$, 
     then there exists a Julia continuum $\CH$ of $f$ and a homeomorphism $\psi\colon X\to\CH$ such that
       $\psi(x)=\infty$.
\end{thm}

  Let us remark on a striking property of the function  $f$ whose existence is asserted in 
    Theorem~\ref{thm:mainarclike}.
    Suppose that $F_{\lambda}(x)=\lambda x(1-x)$ 
     is a member of the famous \emph{logistic family} of interval 
     maps, for  $\lambda\in [1,4]$. Then the \emph{inverse limit space} of  $F_{\lambda}$~-- that is, the space of
    all backward orbits~-- is an arc-like continuum having a terminal point (see Proposition~\ref{prop:arclikecharacterization}). It is known \cite[Theorem~1.2]{ingramconjecture} 
    that the topology of this continuum
    determines $F_{\lambda}$ up to a natural notion of conjugacy. 
    Hence the topology of $J(f)$, for the single disjoint-type
    function  $f$, contains all information about the topological 
    dynamics of all members of the logistic family.  

   In the case where $f$ does not have bounded slope (or even  
    ``{\anguine} tracts''), 
    we obtain the following result. 
\begin{thm}[Topology of Julia continua] \label{thm:mainspanzero}
  Let $\CH$ be a Julia continuum of a disjoint-type entire function $f$. Then
    $\CH$ has span zero and $\infty$ is a terminal point of $\CH$. 
\end{thm}

 Here ``span zero'' (see Definition~\ref{defn:topologicalproperties}) is another famous concept from continuum theory, which means intuitively that 
 two points that exchange their position by travelling within $X$ must meet each other. 
  It is well-known
   \cite{lelekspan}
   that every arc-like continuum has span zero.
   A long-standing question, posed by Lelek \cite[Problem~1]{lelekproblems} in 1971 and featured on many subsequent 
    problem lists in 
   topology, asked whether every continuum with span
   zero must be arc-like. This would imply that Theorem~\ref{thm:mainarclike} gives a 
   complete classification of Julia continua.
    However, Hoehn \cite{hoehn} constructed a counterexample to Lelek's 
    problem in 2011. 
    The conjecture does hold when the continuum in question is either 
    \emph{hereditarily decomposable}~ \cite{bingsnakelike} or 
    \emph{hereditarily indecomposable} \cite{hoehnoversteegenpseudoarc}.
    (A continuum is \emph{decomposable} if it can be written as the union of 
      two proper subcontinua; it is \emph{hereditarily} (in-)decomposable if
      all its non-degenerate subcontinua are also (in-)decomposable. 
     See Definition~\ref{defn:indecomposable}.)

Since any Julia continuum different from those constructed in Theorem~\ref{thm:mainarclike} 
    would be of span zero but not arc-like, it would 
    be of considerable topological interest. It is conceivable that one can construct a disjoint-type entire function having
   a Julia continuum of this type, thus yielding a new  proof of Hoehn's theorem mentioned above.
    We do not pursue this investigation here. 

 We prove a number of additional and more precise results, which are stated and discussed in  
    Section~\ref{sec:intro2}.     To conclude the introduction, we 
    discuss the wider significance 
    of the class of disjoint-type entire functions, and then mention two easily-stated 
    important applications of our results and methods.

\subsection*{Significance of disjoint-type functions}
   In the study of dynamical systems, hyperbolic systems 
    (also referred to as \emph{Axiom A}, using Smale's terminology) 
    are those that exhibit the simplest behaviour; understanding the hyperbolic case is usually the first step in 
    building a more general theory. Since entire functions of disjoint type are those hyperbolic functions that have 
    the simplest combinatorial structure, their properties are of intrinsic interest, and their detailed understanding
    provides an important step towards understanding larger classes. However, our results also
    have direct consequences far beyond the disjoint-type case. Indeed, 
    consider the \emph{Eremenko--Lyubich class}
    \[ \B \defeq  \{f\colon \C\to\C\text{ transcendental entire}\colon S(f) \text{ is bounded}\}. \]
    By definition, this class contains all hyperbolic functions, 
    as well as the particularly interesting \emph{Speiser class} $\classS$ of functions for which $S(f)$ is finite. 
    We remark that, for $f\in\B$, always $I(f)\subset J(f)$ \cite{alexmisha}. 

    If $f\in\B$, then the map
    \begin{equation}\label{eqn:rescalingdisjointtype} f_{\lambda}\colon\C\to\C;\quad z\mapsto \lambda f(z) \end{equation}
    is of disjoint type for sufficiently small $\lambda$ \cite[\S5, p.\ 261]{boettcher}. If additionally the
    original function $f$ is hyperbolic, then by \cite[Theorem~5.2]{boettcher}, the 
   dynamics of $f$ on its Julia set $J(f)$
    can be obtained, via a suitable semi-conjugacy, as 
     a quotient of that of $f_{\lambda}$ on $J(f_{\lambda})$. 
    This result has been generalised, first by Mihaljevi\'c-Brandt \cite{helenaconjugacy}
    and more recently by
     Alhabib, Sixsmith and the author \cite{geometricallyfinite} to important 
     larger classes of entire functions (more precisely: geometrically finite
     entire functions whose local backward branches on the Julia set have uniformly
     bounded degree). 
     
    In each of these cases, the semi-conjugacy in question restricts to 
     a homeomorphism on each connected component of $J(f_{\lambda})$. Hence 
     all of our  
     results concerning the topology of Julia components for the disjoint-type function  $f_{\lambda}$ apply 
     directly to the corresponding subsets of $J(f)$ also. 

  Furthermore, \cite{DHlanding} introduces the notion of
   periodic ``filaments''~-- a generalisation of the above-mentioned ``hairs''~-- for entire functions   
    with bounded postsingular sets. This includes all hyperbolic functions, and 
    more generally all  geometrically finite entire functions.
   These filaments promise to play an important role in the 
   combinatorial study of class $\B$ entire functions; indeed, they have already
   been used to prove the existence and uniqueness of (homotopy) Hubbard trees for
   general post-singularly finite entire functions \cite{pfrangthesis}. 
   Our results on disjoint-type Julia continua apply directly to (landing) filaments;
    indeed, any periodic filament together with its landing point
    is homeomorphic 
    to a corresponding component of $J(f_{\lambda})$, with $f_{\lambda}$ from~\eqref{eqn:rescalingdisjointtype} of disjoint type. 

     Yet more generally, for any $f\in\B$, it is shown in~\cite[Theorem~1.1]{boettcher}
    that $f$ and the disjoint-type function $f_{\lambda}$ have the
    same dynamics \emph{near infinity}. It follows that our results, which apply
    \emph{a priori} only to the
    disjoint-type function $f_{\lambda}$,
    also have interesting consequences for general  $f\in\B$. 

 \subsection*{Two applications}    
  A famous example of a continuum (i.e., a non-empty compact, connected metric space) that contains no arcs is 
   provided by the 
   \emph{pseudo-arc} (see Definition \ref{defn:topologicalproperties}). The pseudo-arc 
   is the unique hereditarily indecomposable arc-like continuum 
   and has the intriguing property of being
  homeomorphic to each  of its non-degenerate subcontinua. Theorem~\ref{thm:mainarclike} shows that 
    the pseudo-arc can arise in the Julia set of a transcendental entire function; as
   far as we are aware, this is the first time that a dynamically defined subset of the Julia set 
    of an entire or meromorphic function has been shown to be a pseudo-arc. 
     In fact, we show even more:

\begin{thm}[Pseudo-arcs in Julia sets]\label{thm:pseudoarcs}
 There exists a disjoint-type
   entire function $f\colon\C\to\C$ such that, for every connected component $C$ of $J(f)$,
   the set $C\cup\{\infty\}$ is a pseudo-arc. 
\end{thm}

   Observe that this result sharpens the aforementioned \cite[Theorem 8.4]{strahlen}. We prove 
    Theorem~\ref{thm:pseudoarcs} as a special case of much more general results; since this 
    manuscript was first circulated, a simpler and more direct proof of Theorem~\ref{thm:pseudoarcs}
    was developed by Benitez and the author~\cite{pseudoarcs}.

 A further motivation for studying the topological dynamics of disjoint-type functions comes from a second question asked by
  Eremenko in \cite{alexescaping}: \emph{Is every connected component of $I(f)$ unbounded?} This problem is now known as 
  \emph{Eremenko's Conjecture}, and has remained open despite considerable attention.\footnote{%
   Since this memoir was submitted, Eremenko's conjecture has been resolved in the negative for
   general transcendental entire functions~\cite{eremenkosconjecture}. It remains an open problem whether
   Eremenko's conjecture holds for all functions in the class $\B$.} For disjoint-type maps, and indeed for 
  any entire function with bounded postsingular set, it is known that the answer is positive \cite{eremenkoproperty}. However, the disjoint-type
  case nonetheless has a role to play in the study of Eremenko's Conjecture. 
   Indeed, as discussed in
   \cite[Section 7]{boettcher}, we may strengthen the question slightly by asking
   which entire functions have the following property:
 \begin{enumerate}
   \item[(UE)]
    For all  $z\in I(f)$, there is a connected and unbounded set $A\subset\C$ with $z\in A$ such that $f^n|_A\to\infty$ \emph{uniformly}.
  \end{enumerate}
  
  If there exists a counterexample $f$ to Eremenko's Conjecture in the class $\B$, then clearly $f$ cannot satisfy property (UE). 
    In this case, it
    follows from \cite[p.\ 265]{boettcher} that (UE) fails for \emph{every} map of the form $f_{\lambda} \defeq  \lambda f$. As noted above, $f_{\lambda}$ is
    of disjoint type for $\lambda$ sufficiently
    small, so we see that any counterexample $f\in\B$ would need to be closely related to a disjoint-type function for which (UE) fails. 
    The author suggested in \cite{eremenkoproperty,boettcher} 
   that it might be possible to construct such an example; 
   in this memoir we realise this construction for the first time.  
   Indeed, as we discuss in more detail in the next section, 
   there is
   a surprisingly close relationship between the topology of the Julia 
   set and the
   existence of points $z\in I(f)$ for which (UE) fails. Hence our results 
   allow us to give a good description of the circumstances in which 
   such points exist at all, which is likely to be important in any attempt
   to resolve Eremenko's Conjecture. In particular, we can prove
   the following, which considerably strengthens the examples alluded to in \cite{eremenkoproperty,boettcher}.
  
\begin{thm}[Non-uniform escape to infinity] \label{thm:nonuniform}
  There is a disjoint-type entire function $f$ and an escaping point
    $z\in I(f)$ with the following property. If $A\subset I(f)$ is connected and
    $\{z\}\subsetneq A$, then 
     \[ \liminf_{n\to\infty} \inf_{\zeta\in A} |f^n(\zeta)| < \infty. \]
\end{thm}

\section{Further discussion and results}\label{sec:intro2}
 In this section, we state and discuss further and more detailed results about the properties of Julia continua,
  going beyond those mentioned in the introduction. Let us begin by providing the formal definition of 
  the concepts from continuum theory that have already been mentioned. 
 
\begin{defn}[Terminal points; span zero; arc-like continua]\label{defn:topologicalproperties}
 Let $X$ be a continuum; i.e.\ $X$ is a non-empty compact, connected metric space. 
  \begin{enumerate}[(a)]
    \item A point $x\in X$ is called a \emph{terminal point} of
   $X$ if, for any two subcontinua $A,B\subset X$ with $x\in A\cap B$, either
   $A\subset B$ or $B\subset A$. 
  \item 
    $X$ is said to have \emph{span zero} if, for any subcontinuum
    $\mathcal{A}\subset X\times X$ whose first and second coordinates both
    project to the same subcontinuum 
    $A\subset X$, $\mathcal{A}$ must intersect the diagonal. (I.e., if $\pi_1(\A)=\pi_2(\A)$, then
    there is $x\in X$ such that $(x,x)\in \mathcal{A}$.)
   \item 
    A non-degenerate continuum 
      $X$ is said to be \emph{arc-like} if, for every $\eps>0$, there exists
     a continuous function $g\colon X\to [0,1]$ such that $\diam(g^{-1}(t))<\eps$ for all 
     $t\in [0,1]$.  
   \item $X$ is called a \emph{pseudo-arc} if $X$ is arc-like and also \emph{hereditarily indecomposable}.
 \end{enumerate}
\end{defn}
 For the benefit of those readers who have not encountered these concepts before, let us make a few comments regarding their meaning.
  \begin{enumerate}[(1)]
    \item As mentioned in the introduction, one may think of terminal points as analogous to the endpoints of an arc. However, keep in mind that there may be more than
      two terminal points~-- indeed, for any hereditarily indecomposable continuum, such as 
      a pseudo-arc, 
     \emph{all} points are terminal. 

    We remark that there are several different and inequivalent notions of ``terminal points'' in use in continuum theory. The above definition can
      be found e.g.\ in \cite[Definition~3]{fugatedecomposable}, and differs, in particular, from that of Miller
      \cite[p.~131]{millerunicoherent}. On the other hand, the distinction is not essential for our purposes, since both
      notions coincide for the types of continua studied in this paper.
    \item As we discuss in Section \ref{sec:topology}, 
     there are a number of equivalent 
     definitions of arc-like continua, the most important of which is that $X$ is arc-like if and only if it can be written as an inverse limit of arcs with surjective bonding maps. 
   \item Any two pseudo-arcs, as defined above, are homeomorphic \cite[Theorem~1]{binghereditarilyindecomposable}; for this reason, we also speak about \emph{the pseudo-arc}. 
      We refer to Exercise 1.23 in \cite{continuumtheory} for
      a construction that shows the existence of such an object, the introduction to Section 12 in the same book
      for a short history, and to  \cite{lewispseudoarc} for a survey of further results. 
  \end{enumerate}

\subsection*{Non-escaping points and accessible points}
  Recall that Theorems~\ref{thm:mainarclike} and~\ref{thm:mainspanzero} give an almost complete
   description of the possible topology of Julia continua. 
 Let us now turn to the behaviour of points in a Julia continuum $\CH = C\cup\{\infty\}$ under iteration.
  In the case of disjoint-type sine (or exponential) maps, and indeed for any disjoint-type entire function of finite order, 
  each component $C$ 
  of the Julia set is an arc and contains at most one point that does not tend to infinity under iteration,
  namely the finite endpoint of $C$. 
  Furthermore, this finite endpoint is always accessible from the Fatou set of $f$; no other 
  point can be accessible from $F(f)$. (Compare \cite{devgoldberg}.) This suggests the following questions:
\begin{enumerate}
 \item Can a Julia continuum contain more than one non-escaping point?
 \item Is every non-escaping point accessible from $F(f)$?
 \item Does every Julia continuum contain a point that is accessible from $F(f)$?
    This question is raised in
    \cite[p. 393]{baranskikarpinskatrees}, where the authors prove that the answer
     is positive when a certain growth condition is imposed on the external address (see Definition \ref{defn:externaladdress}) 
     of the component $C$.
\end{enumerate}

 To answer these questions, we require one more
  topological concept.
 \begin{defn}[Irreducibility]\label{defn:irreducible}
 Let $X$ be a continuum, and let $x_0, x_1\in X$. We say that $X$ is \emph{irreducible} between $x_0$ and $x_1$ if 
   no proper subcontinuum of $X$ contains both $x_0$ and $x_1$. 
 \end{defn}
  We shall apply this notion only in the case where $x_0$ and $x_1$ are terminal points of $X$. In this case, irreducibility of $X$ between $x_0$ and $x_1$
   means that, in some sense, the points $x_0$ and $x_1$ lie ``on opposite ends'' of $X$. For example, the $\sin(1/x)$-continuum of 
    Figure~\ref{fig:arclike}\subref{subfig:sinecurve} is irreducible between the terminal point on the right of the image and either of the two terminal points on the left, 
    but not between the latter two (since the limiting interval is a proper subcontinuum containing both).

 \begin{thm}[Non-escaping and accessible points]\label{thm:nonescapingaccessible}
  Let $\CH$ be a Julia continuum of a disjoint-type entire function
   $f$. Any non-escaping point $z_0$ in $\CH$ is a terminal
   point of $\CH$, and $\CH$ is irreducible between $z_0$ and $\infty$. The same is true
   for any point $z_0\in\CH$ that is accessible from $F(f)$. 

  Furthermore, the set of non-escaping points in $\CH$ has Hausdorff dimension zero.
   On the other hand, there exist a disjoint-type function having a Julia continuum for which
   the set of non-escaping points is a Cantor set and a disjoint-type function having a Julia continuum
   that contains a dense $G_{\delta}$ set of non-escaping points. 
 \end{thm}

  A Julia continuum contains at most one accessible point 
   (see Theorem~\ref{thm:accessibleterminal}), so
   the two functions whose existence is asserted in Theorem~\ref{thm:nonescapingaccessible} 
    must have non-escaping points that are not
   accessible from $F(f)$.   Furthermore, we can apply Theorem \ref{thm:mainarclike} to the bucket-handle continuum of Figure~\ref{fig:arclike}\subref{subfig:knaster}, 
   which has only a single terminal point. Hence the corresponding Julia continuum $\CH$ (which, as discussed below, can be chosen such that the iterates do not converge uniformly
   to infinity on $C$) contains neither non-escaping nor accessible points. 
   In particular,  this answers the question of
   Bara\'nski and Karpi\'nska. 

 We remark that it is also possible to construct an inaccessible Julia continuum that \emph{does} contain
  a finite terminal point $z_0$. Indeed, we shall see that the examples mentioned in the second half of the preceding theorem
   must have this property (compare the remark after Theorem~\ref{thm:accessibleterminal}). Alternatively, such an example could 
   be achieved by ensuring that the continuum is embedded 
  in the plane in such a way that $z_0$ is not accessible from the complement of $\CH$ (see Figure~\ref{fig:embedding}); 
   we do not discuss the details here. 

\subsection*{Bounded-address and periodic Julia continua}

 We now turn our attention to the different types of dynamics that $f$ can exhibit on
   a Julia continuum. In Theorem \ref{thm:mainarclike}, we saw that, given any arc-like continuum $X$ 
  having a terminal point,
   there are 
   a disjoint-type function $f$ and a Julia continuum
   $\CH$ of $f$ such that $\CH$ is homeomorphic to $X$. We shall see that it is possible
   to choose $\CH=C\cup\{\infty\}$ either such that 
   $f^n|_C\to \infty$ uniformly, or such that $\min_{z\in C} |f^n(z)|<R$ for some
   $R>0$ and infinitely many $n$. However, our construction will always lead to continua with 
   $\limsup_{n\to\infty}\min_{z\in C}|f^{n}(z)|=\infty$. In particular, the Julia continuum cannot be periodic. 

 Periodic points, and periodic continua consisting of escaping points, play a crucial role in complex dynamics. 
   Hence it is interesting to consider when we can improve on the preceding results, in the following sense.
\begin{defn}[Periodic and bounded-address Julia continua]
 Let $\CH=C\cup\{\infty\}$ be a Julia continuum of a disjoint-type function $f$. 
  We say that $\CH$ is \emph{periodic} if $f^n(C)=C$ for some $n\geq 1$. 

 We also say that $\CH$ has \emph{bounded address}  if 
   there is $R>0$ such that, for every $n\geq 0$, there exists a point $z\in C$ such that
   $|f^n(z)|\leq R$.
\end{defn}

 Upon reflection, it becomes 
  evident that not every arc-like continuum $X$ having a terminal point
   can arise as a Julia continuum with
   bounded address. Indeed, it is well-known that every Julia continuum $\CH$ at bounded 
  address contains a unique point with a 
  bounded orbit, and hence that every periodic Julia continuum contains a periodic point.
  (See Proposition~\ref{prop:boundedorbits}.) 
  In particular, by Theorem \ref{thm:nonescapingaccessible},
   $\CH$ contains some terminal point $z_0$ such that $\CH$ is irreducible between $z_0$ and
   $\infty$. So if $X$ is an arc-like continuum that
   does not contain two terminal points between which $X$ is irreducible (such as the 
   Knaster bucket-handle), then $X$ cannot be realised by a bounded-address Julia continuum.
  It turns out that this is the only restriction.

\begin{thm}[Classification of bounded-address Julia continua] \label{thm:boundedexistence}
  There is a disjoint-type entire function $f$ of bounded slope with the following property.

  Let $X$ be an arc-like continuum, and let $x_0,x_1 \in X$ be two terminal points between which $X$ is irreducible.
   Then there is a Julia continuum $\CH$ of $f$ with bounded address
      and a homeomorphism $\psi\colon X\to\CH$ such that
        $\psi(x_0)$ has bounded orbit under $f$
       and $\psi(x_1)=\infty$.
\end{thm}

  We also observe that not every continuum $X$ as in Theorem~\ref{thm:boundedexistence}
   can occur as a periodic Julia continuum. Indeed, if $\CH$ is a periodic Julia continuum, then
   $f^p\colon\CH\to\CH$ is a homeomorphism, where $p$ is the period of $\CH$, and $f^p$ is extended to $\CH$ by
    setting  $f^p(\infty)=\infty$. Furthermore, all points of $\CH$ but one 
   tend to $\infty$ under iteration by $f^p$. However, if $X$ is, say, the 
   $\sin(1/x)$-continuum from Figure \ref{fig:arclike}, then every self-homeomorphism of $X$ must map the limiting interval on the left to itself. Hence there
   cannot be any periodic Julia continuum
   $\CH$ that is homeomorphic to $X$. 
   The correct class of continua for this setting was discussed by Rogers \cite{rogerscontinua} in 1970.

\begin{defn}[Rogers  continua]\label{defn:rogerscontinua}
  Let  $X$ be homeomorphic to the inverse limit of a  surjective continuous self-map $h$
    of the interval $[0,1]$
    with $h(t)<t$ for  $0<t<1$, and let $x_0$ and $x_1$ denote the points of $X$ corresponding to 
     $0\mapsfrom  0 \mapsfrom  0\dots$ and $1\mapsfrom 1 \mapsfrom \dots$. 
     Then we shall say that $X$ is a \emph{Rogers continuum} (from  $x_0$ to  $x_1$). 
\end{defn}
\begin{remark}
  The \emph{inverse limit} generated by $h$ is the space of all backward orbits of $h$, 
      equipped with the
      product topology (Definition \ref{defn:inverselimit}). 
\end{remark}

\begin{thm}[Periodic Julia continua] \label{thm:periodicexistence}
 Let $X$ be a continuum and let $x_0,x_1\in X$. Then the following are equivalent:
 \begin{enumerate}[(a)]
   \item $X$ is a Rogers continuum from  $x_0$ to $x_1$.
  \item There exists a disjoint-type entire function $f$ of bounded slope, a periodic Julia continuum $\CH$ of $f$, say of period $p$, and 
    a homeomorphism $\psi\colon X\to\CH$ such that $f^p(\psi(x_0))=\psi(x_0)$ and $\psi(x_1)=\infty$.
 \end{enumerate}
\end{thm}
\begin{remark}[Remark 1]
  Both Theorem \ref{thm:boundedexistence} and the combination of Theorems~\ref{thm:mainarclike} and~\ref{thm:mainspanzero}
   can be stated in the following form: \emph{Any (resp.\ any bounded-address) Julia continuum (whether arc-like or not) has a certain intrinsic topological 
   property $\mathcal{S}$, and 
   any arc-like continuum with property $\mathcal{S}$ can be realised as a Julia continuum (resp.\ bounded-address Julia continuum) of a disjoint-type,
    bounded-slope entire function.} Rogers's result in \cite{rogerscontinua} gives such a description
    also for Theorem~\ref{thm:periodicexistence}, assuming additionally that $X$ is decomposable.  
    It is an interesting question whether this can be achieved also in the indecomposable case, although we note
    that there is no known intrinsic topological description of those arc-like continua that can be written as an inverse
    limit with a single bonding map.  
\end{remark}
\begin{remark}[Remark 2]
  Another difference between Theorem~\ref{thm:periodicexistence} and Theorems~\ref{thm:mainarclike} 
  and~\ref{thm:boundedexistence} is that not all Rogers continua can be realised  by the same  function.
   Indeed, it can be  shown that there are  uncountably many pairwise non-homeomorphic Rogers continua,
    while the set  of periodic Julia continua of any given function  is countable.
\end{remark}
\begin{remark}[Remark 3]
  By a classical result of Henderson \cite{hendersonpseudoarc}, the pseudo-arc is a Rogers continuum 
    (where $x_0$ and $x_1$ can be taken as any two points between which it is irreducible). 
   Hence we see from Theorem \ref{thm:periodicexistence}
   that it can arise as an invariant Julia continuum of a disjoint-type entire function.
  It follows from the nature of our construction in the proof of Theorem~\ref{thm:periodicexistence} that, in this case,
  \emph{all} Julia continua are pseudo-arcs (see Corollary \ref{cor:pseudoarcs}), 
   establishing Theorem \ref{thm:pseudoarcs} as stated in the introduction.  
\end{remark}

\subsection*{(Non-)uniform escape to infinity}

 We now return to the question of rates of escape to infinity, and the ``uniform Eremenko property''
   (UE). Recall that by Theorem \ref{thm:mainarclike} it is possible to construct a Julia continuum $\CH$ that contains no finite terminal points,
   and hence has the property that $C\subset I(f)$ by Theorem \ref{thm:nonescapingaccessible}.
   Also recall that we can choose $\CH$ in such a way that the iterates of $f$ do not tend to
   infinity uniformly on $C$. This easily implies that there is some point in $C$ for which 
   the property (UE) fails.

 To study this type of question in greater detail, we introduce the following natural definition. 

 \begin{defn}[{\Escapingcomposant}s]\label{defn:uniformescapeintro}
  Let $f$ be a transcendental entire function, and let $z\in I(f)$. The
    \emph{{\escapingcomposant} $\mu(z)$} is the union of
    all connected sets $A\ni z$ such that $f^n|_A\to\infty$ uniformly. 

  We also define $\mu(\infty)$ to be the union of all unbounded connected sets $A$ such that
    $f^n|_A\to\infty$ uniformly. 
 \end{defn}

 With this definition, property (UE) requires precisely that $I(f)=\mu(\infty)$. For a disjoint-type
    entire function, it makes sense to study this property separately for each  Julia  continuum $\CH$; i.e.\ 
    to ask whether all escaping points in $\CH$ belong to $\mu(\infty)$. Clearly this is the case by  definition
    when $f^n\to\infty$ uniformly on $C=\CH\setminus\{\infty\}$. Otherwise, it turns out that 
    there is a close connection between the above question and the topology of $\CH$.
   Recall that the \emph{(topological) composant} of a point $x_0$ in a continuum $X$ is the union of
   all proper subcontinua of $X$ containing $x_0$. The topological composant
    of  $\infty$ in a Julia continuum $\CH$ is always a proper subset of $\CH$, since $\infty$ is a terminal  point.

 \begin{thm}[Topological composants and {\escapingcomposant}s]\label{thm:composants}
  Let $\CH=C\cup\{\infty\}$ be a Julia continuum of a disjoint-type entire function, and suppose that
   $f^n|_{C}$ does not tend to infinity uniformly. Then the topological composant of
   $\infty$ in $\CH$ is given by $\{\infty\}\cup (\mu(\infty)\cap C)$. 

  If $\CH$ is periodic, then $\CH$ is indecomposable if and only if
   $C\cap I(f)\setminus \mu(\infty)\neq\emptyset$ (that is, if $C$ contains an escaping point for which property (UE) fails). 
 \end{thm}

 Any indecomposable continuum has uncountably many composants, all of which are pairwise disjoint. Hence we see that
   complicated topology of non-uniformly escaping Julia continua automatically leads to the existence of points 
   that cannot be connected to infinity by a set that escapes uniformly.
  However, our proof of Theorem \ref{thm:mainarclike} also allows us to construct 
  Julia continua that have very simple topology, but nonetheless contain
  points in $I(f)\setminus \mu(\infty)$. 

  \begin{thm}[A one-point {\escapingcomposant}]\label{thm:onepointuniform}
    There exists a disjoint-type entire function $f$ and a Julia continuum $\CH=C\cup\{\infty\}$ of $J(f)$ such that:
     \begin{enumerate}[(a)]
       \item $\CH$ is an arc, with one finite endpoint $z_0$ and one endpoint at $\infty$;
       \item $C\subset I(f)$, but $\liminf_{n\to\infty}\min_{z\in C} |f^n(z)|<\infty$. In fact, there is no non-degenerate connected set
         $A\ni z_0$ on which the iterates escape to infinity uniformly.\label{item:onepointuniform_escaping}
     \end{enumerate}
  \end{thm}

 Observe that this implies Theorem \ref{thm:nonuniform}.

\subsection*{Number of tracts and singular values}
 
 So far, we have not said much about the nature of the functions $f$ that occur in our examples,
   except that they are of disjoint type. Using recent results of Bishop \cite{bishopfolding,bishopclassBmodels,bishopclassSmodels}, we can say considerably more:

 \begin{thm}[{Class $\classS$ and number of tracts}]\label{thm:Stracts}
  All examples of disjoint-type entire functions $f$ mentioned in the introduction and in this section can be constructed 
    in such a way that $f$ has exactly two critical values and no finite asymptotic values, and such that all critical points of
    $f$ have degree at most $16$.

  Furthermore, with the exception of Theorem \ref{thm:boundedexistence}, 
   the function $f$ can be constructed such that
      \[ \mathcal{T}_R \defeq  f^{-1}(\{z\in \C\colon  |z|>R\}) \]
   is connected for all $R$. In Theorem \ref{thm:boundedexistence}, the function
   $f$ can be constructed so that $\mathcal{T}_R$ has exactly two connected components for sufficiently 
    large $R$.
 \end{thm}
\begin{remark}[Remark~1]
  On the other hand, if 
    $\mathcal{T}_R$ is connected for all $R$, then it turns out that every Julia continuum
    at a bounded address is homeomorphic to a periodic Julia continuum (Proposition \ref{prop:homeomorphic}). Hence
    it is indeed necessary to allow $\mathcal{T}_R$ to have two components in Theorem
    \ref{thm:boundedexistence}.
\end{remark}
\begin{remark}[Remark~2]
  The bound of $16$ on the degree of critical points can be reduced to $4$ by modifying Bishop's construction; compare Remark~\ref{rmk:bishopbound}.
\end{remark}

 As pointed out in \cite{hyperbolicboundedfatou}, this leads to an interesting observation. By Theorem \ref{thm:Stracts}, the function $f$ from Theorem
    \ref{thm:pseudoarcs} can be constructed such that $\# S(f)=2$, such that $f$ has no asymptotic values and such that all
    critical points have degree at most $16$. Let $v_1$ and $v_2$ be the critical values of $f$, and let $c_1$ and $c_2$ be critical points of 
    $f$ over $v_1$ resp.\ $v_2$. Let $A\colon\C\to\C$ be the affine map with $A(v_1)=c_1$ and $A(v_2)=c_2$. Then the function
    $g \defeq  f\circ A$ has super-attracting fixed points at $v_1$ and $v_2$. By the results from \cite{boettcher} discussed earlier,
    the Julia set $J(g)$ contains uncountably many invariant subsets for each of which the one-point compactification is 
    a pseudo-arc.
    On the other hand, $J(g)$ is locally connected by \cite[Corollary 1.9]{hyperbolicboundedfatou}. Hence we see that, in contrast to the 
    polynomial case, local connectivity of Julia sets does not imply simple topological dynamics, even for hyperbolic functions. Compare \cite{roeschnewton} for
    a similar phenomenon in the case of (albeit non-hyperbolic) rational maps. 

\subsection*{Embeddings}
  Given an arc-like continuum $X$, there are usually different ways to embed $X$ in the Riemann sphere $\Ch$.
   That is, there might be continua $C_1,C_2\subset \Ch$ such that $X$ is homeomorphic to
   $C_1$ and $C_2$, but such that no homeomorphism $\Ch\to\Ch$ can map $C_1$ to $C_2$. (That is,
   $C_1$ and $C_2$ are not \emph{ambiently} homeomorphic.) Our construction in the proof of 
   Theorem \ref{thm:mainarclike} 
   is rather flexible, and we could indeed use it to
   construct different Julia continua that are homeomorphic but not \emph{ambiently} homeomorphic.
   In particular, as briefly mentioned already in the discussion of results concerning accessibility,
   it would be possible to construct a Julia continuum $\hat{C}$ that is
   homeomorphic to the $\sin(1/x)$-continuum, and such that the limiting arc is not accessible
   from the complement of $\hat{C}$. (See Figure \ref{fig:embedding}).
 
\begin{figure}
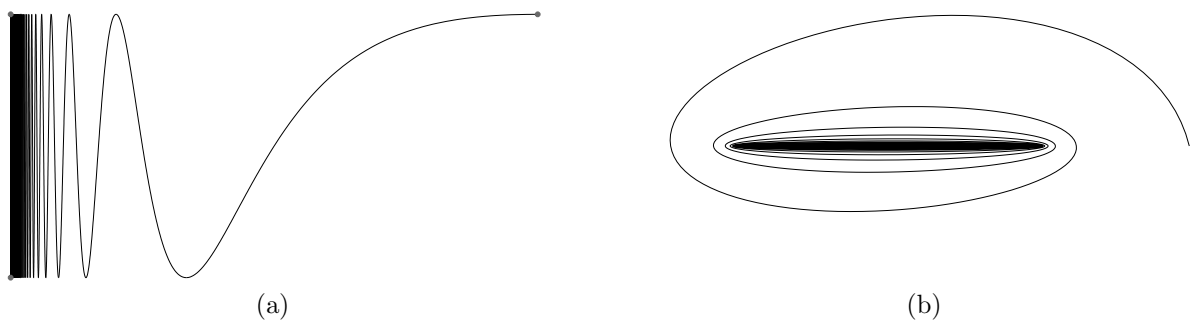
  
  \subfloat[]{\includegraphics[width=.45\textwidth]{sinecurve}}\hfill
  \subfloat[]{\includegraphics[width=.45\textwidth]{sinecurve_mod}}
 \caption{Two embeddings of the $\sin(1/x)$-continuum that are not
   ambiently homeomorphic%
    \label{fig:embedding}}
\end{figure}
  
  It is easy to see that, for a disjoint-type entire function which has bounded slope, 
   every Julia continuum can be 
   covered by a \emph{chain} with arbitrarily small links such that every link is a connected subset of
   the Riemann sphere. (For the definition of a chain, compare the remark after Proposition
   \ref{prop:arclikecharacterization}.) It is well-known 
   \cite[Example~3]{bingsnakelike} that
   there are embeddings of arc-like continua without this property.
   It is natural to ask whether  this is the only restriction on the continua that can arise by our
   construction, but we shall defer this and similar problems to future investigations.  

\subsection*{Structure of the memoir}
 In Section
  \ref{sec:preliminaries}, we collect background on the dynamics of 
   disjoint-type entire functions. In particular, we review the \emph{logarithmic change of 
   variable}, which is used throughout the remainder of the paper.  We also recall some basic facts
   from the theory of continua. Section~\ref{sec:conjugacy} discusses a general conjugacy principle between
   expanding inverse systems, which will be useful throughout. 
   Following these preliminaries, the memoir essentially splits into two parts, which can largely
   be read independently of each other:
   \begin{itemize}
    \item \emph{General topology of Julia continua.}
     We first study general properties of Julia continua of disjoint-type entire functions. More precisely,
      in Section \ref{sec:spanzero} we show that each such continuum has span zero, and prove the various
      results concerning terminal
      points stated earlier; in particular, we establish Theorem~\ref{thm:mainspanzero}. In Section \ref{sec:uniform}, we investigate the structure of {\escapingcomposant}s. 
      Section~\ref{sec:arcliketracts} studies conditions under which all Julia continua are arc-like, and establishes one half of 
      Theorem~\ref{thm:periodicexistence}. Finally, Section \ref{sec:homeomorphicsubsets} shows that, in certain circumstances,
      different Julia continua are homeomorphic to each other. These results show that our constructions in the
      second part of the paper are optimal in a certain sense, and are also used in the proof of
      Theorem~\ref{thm:pseudoarcs}.
    \item \emph{Constructing prescribed Julia continua.} 
    The second part of the paper is concerned with constructions that allow us to find entire functions having prescribed arc-like Julia continua, 
     as outlined in the theorems stated in the introduction and the present section. We review topological background
     on arc-like continua in Section~\ref{sec:topology} and, in Section~\ref{sec:arclikeexistence},
      give a detailed proof of a slightly weaker version of Theorem~\ref{thm:mainarclike} 
     (where the function $f$ is allowed to 
      depend on the
     arc-like continuum in question). Section~\ref{sec:onepointuniform} applies this general construction to obtain the
      examples from Theorems~\ref{thm:nonuniform}, \ref{thm:onepointuniform} and~\ref{thm:nonescapingaccessible},
      and the proof  of  Theorem~\ref{thm:mainarclike} is completed in Section~\ref{sec:allinone}. 
      Section~\ref{sec:boundedexistence} discusses the modification  
      of the construction from Section~\ref{sec:arclikeexistence} necessary to obtain
      bounded-address continua for Theorem~\ref{thm:boundedexistence}. 
      Section~\ref{sec:periodicexistence} contains the proofs of 
     Theorems~\ref{thm:periodicexistence} and
      \ref{thm:pseudoarcs}. Finally,   we briefly discuss how to modify the constructions to  obtain 
      Theorem~\ref{thm:Stracts} (Section~\ref{sec:classS}).
\end{itemize}

\subsection*{Basic notation}
 As usual, we denote the complex plane by $\C$ and the Riemann sphere by $\Ch$. We also denote
   the unit disc by $\D$ and the right half-plane by $\HH$. All boundaries and closures of plane
   sets will be
   understood to be taken in $\C$, unless explicitly noted otherwise.

  We shall also continue to use the notations introduced throughout Sections~\ref{sec:intro} and~\ref{sec:intro2}. 
   In particular, 
   the Fatou, Julia and escaping sets of an entire function are denoted by
   $F(f)$, $J(f)$ and $I(f)$, respectively. Euclidean distance is denoted by $\dist$.
   
   The symbols $\N$ and $\N_0$ denote the positive and non-negative integers, respectively. 

In order to keep the paper accessible to readers with
   a background in either continuum theory or transcendental dynamics, but not necessarily
   both,
   we aim to introduce all notions and results required from either area.
  For further background on transcendental iteration theory, we refer to \cite{waltermero}.  
   For a wealth of information on continuum theory,
   including the material treated here, we refer to \cite{continuumtheory}. In particular, a detailed
   treatment of arc-like continua can be found in
   \cite[Chapter~12]{continuumtheory}.

 We shall    assume that the reader is familiar with plane hyperbolic geometry; see e.g.\ \cite{beardonminda}. 
  If $U\subset\C$ is simply-connected, then we denote the density of the hyperbolic metric by 
   $\rho_U\colon U\to (0,\infty)$. In particular, we shall frequently use the \emph{standard estimate} 
   on the hyperbolic metric
   in a simply-connected domain \cite[Theorems~8.2~and~8.6]{beardonminda}:
\begin{equation}\label{eqn:standardestimate}
   \frac{1}{2\dist(z,\partial U)}\leq \rho_U(z) \leq \frac{2}{\dist(z,\partial U)}.
\end{equation}
  If $U$ is disjoint from its translates by integer multiples of $2\pi i$, then
    $\dist(z,\partial U)\leq \pi$ for all $z\in U$. Thus~\eqref{eqn:standardestimate} yields 
      what we shall call
     the \emph{standard bound for logarithmic tracts}:
  \begin{equation}\label{eqn:standardestimate2pi}
     \rho_U(z) \geq \frac{1}{2\pi}.
  \end{equation}

We also denote hyperbolic diameter in $U$ by $\diam_U$, and hyperbolic
   distance by $\dist_U$. Furthermore, the derivative of a holomorphic function $f$ with respect to the hyperbolic
   metric is denoted by $\|\Deriv f(z)\|_U$. (Note that this is defined whenever $z,f(z)\in U$.) 

\subsection*{Acknowledgements} I am  extremely  grateful  to  Mashael  Alhamd, Patrick Comd\"uhr,  Alexandre DeZotti, 
    Leticia Pardo Sim{\'o}n, Dave Sixsmith and Stephen Worsley for the careful reading of an early version of the manuscript and many 
    thoughtful comments, questions and suggestions for improvement. I would also 
    like to thank Anna Benini, Chris Bishop, Clinton Curry, 
   Toby Hall, Daniel Meyer, Phil Rippon and Gwyneth Stallard for interesting and stimulating
  discussions regarding the reported research.    
  This memoir is the culmination of a long-running research programme, which was begun and partially conducted while I 
   was working at the University of Liverpool, and gratefully acknowledge support by the University for this research
   during this time. I am indebted to the referee for many careful corrections that have led to
  considerable improvements in the exposition, particularly in the second half of the paper.



\section{Preliminaries}
  \label{sec:preliminaries}

\subsection*{Disjoint-type entire functions}
  Let $f\colon\C\to\C$ be a transcendental entire function. Recall that $f$ is said 
  to be of \emph{disjoint type} if
    it is hyperbolic with connected Fatou set. As far as we are aware,
    this class of functions was first studied by 
    Bara{\'n}ski and Karpi{\'n}ska~\cite{baranskikarpinskatrees}. 
    The material in this section is not new, 
    but is taken from different papers with
    sometimes conflicting notation. For this reason 
    we give a consistent account tailored to our purposes; we include (sketches of) proofs
    where these are relevant for the applications that follow. 
    
    The following (see
   \cite[Lemma~3.1]{baranskikarpinskatrees} or  \cite[Proposition~2.8]{helenaconjugacy})
   provides an alternative definition, which is
    the one that we shall work with.

 \begin{prop}[Characterisation of disjoint-type functions]\label{prop:disjointtype}
   A transcendental entire function $f\colon\C\to\C$ is of disjoint type if and only if there exists a  
    bounded Jordan domain $D$ with $S(f)\subset D$ and $f(\overline{D})\subset D$. 
 \end{prop}
  
  Let $f$ be of disjoint type, and consider the domain $W \defeq  \C\setminus \overline{D}$, with $D$ as in 
    Proposition~\ref{prop:disjointtype}. 
   Since $S(f)\cap W=\emptyset$, if $V$ is any connected component of $\V \defeq  f^{-1}(W)$, then
    $f\colon V\to W$ is a covering map. These components are called the \emph{tracts} of $f$ (over $\infty$).  Note that the disjoint-type condition implies that the boundaries $\partial W$ and
    $\partial\V$ are disjoint; this is the reason
   for our choice of terminology.

    Since $f$ is transcendental, it follows from the classification of covering maps of the punctured
    disc
   \cite[Theorem~5.10]{forsterriemannsurfaces} that every tract $V$ is simply-connected and unbounded, 
   and that $f\colon V\to W$ is 
   a universal covering map. Observe that a slightly smaller domain 
    $\widetilde{D}$ (with closure contained in $D$) still satisfies the conclusions of Proposition~\ref{prop:disjointtype}.
   Applying the above observations to $\widetilde{D}$ and the resulting collection of tracts  $\widetilde{\V}$,
   we can deduce that the boundary of any connected component $V$ of $\V$  
    is the preimage of the simple closed  curve 
    $\partial D$ under a universal covering map.  Hence $V$ is an
    \emph{unbounded Jordan domain}, 
    i.e.\ a simply-connected domain whose boundary in $\Ch$ is a simple closed curve passing through $\infty$.
   Furthermore, different
    tracts have disjoint closures, as every component of the open  set $\widetilde{\V}$ contains the closure of 
     exactly one  tract $V$. 
    Furthermore, any compact set $K$ intersects at most finitely many tracts
     (since $\widetilde{\V}$ is an open cover of  $K\cap \overline{\V}$). 

 The following characterises the Julia set  of a disjoint-type entire function;
   for readers who are unfamiliar with transcendental dynamics, it may
    alternatively  serve as  the \emph{definition} of the Julia set in  this setting.
\begin{prop}[Julia sets]
    If $f$ is of disjoint type and $D$ is as in Proposition \ref{prop:disjointtype}, then
     \[J(f) = \{z\in\C \colon f^n(z)\notin \overline{D}\text{ for all $n\geq 0$}\}.\]
 \end{prop}
\begin{proof}
  Since $f$ is hyperbolic and of disjoint type, $F(f)$ 
   is the immediate basin of attraction of a fixed point $p$.
   Since $f(\overline{D})\subset D$, we have $\overline{D}\subset F(f)$ 
   and $p\in D$. Thus a point $z$ belongs to $F(f)$ 
   if and only if its orbit eventually enters 
    $\overline{D}$, as claimed.
\end{proof}

\subsection*{The logarithmic change of variable}  
  Following Eremenko and Lyubich, we study $f$ using the
   \emph{logarithmic change of variable} \cite[Section~2]{alexmisha}. To this end, let us assume for simplicity that $0\in D$, which
   can always be achieved by
   conjugating $f$ with a translation.  Set
   $H \defeq  \exp^{-1}(W)$ and $\T \defeq  \exp^{-1}(\V)$. Since $0\notin W\supset \V$, and 
   since $f$ is a universal covering on each tract, there is a holomorphic function
   $F\colon \T\to H$ such that $f\circ\exp = \exp\circ F$. We may choose this map $F$ to be $2\pi i$-periodic,
   in which case we refer to it as a \emph{logarithmic transform} of $f$. 

  This representation is extremely convenient: for every component $T$ of $\T$, the map $F\colon T\to H$ is a conformal isomorphism,
   rather than a universal covering map as in the original coordinates. This 
   makes it much easier to consider inverse branches. From now on, we shall always study the 
  logarithmic transform of $f$. In fact,
   it is rather
   irrelevant for further considerations 
   that the map $F$ has arisen from a globally defined entire function. This 
   leads us to work in the following greater generality from now on.
  \begin{defn}[The classes $\Blog$ and $\BlogP$ \cite{boettcher,strahlen}]\label{defn:Blog}
   The class $\Blog$ consists of all holomorphic functions 
    \[ F\colon \T\to H, \]
   where $F$, $\T$ and $H$ have the following properties:
   \begin{enumerate}[(a)]
    \item $H$ is a $2\pi i$-periodic unbounded Jordan domain that contains
     a right half-plane.   \label{item:H}
    \item $\T\neq\emptyset$ 
     is $2\pi i$-periodic and $\re z$ is bounded from below in
     $\T$, but unbounded from above.
    \item Every connected component $T$ of $\T$ is an 
      unbounded Jordan domain that is disjoint
      from all its $2\pi i\Z$-translates. 
      For each such $T$, the restriction
      $F\colon T\to H$ is a conformal 
      isomorphism whose continuous extension to the closure
      of $T$ in $\Ch$ satisfies $F(\infty)=\infty$, and whose
      inverse we denote by $F_T^{-1} \defeq  (F|_T)^{-1}$. 
      Each such component $T$ is called a \emph{tract of $F$}.\label{item:tracts}%
    \item The tracts of $F$ have pairwise disjoint closures and
      accumulate only at $\infty$; i.e.,
      if $z_n\in\T$ is a sequence of points all belonging to different
      tracts of $F$, then $z_n\to\infty$.\label{item:accumulatingatinfty}%
   \end{enumerate}

    If furthermore $F$ is $2\pi i$-periodic, then we say that $F$ belongs to the class
    $\BlogP$. If $\cl{\T}\subset H$, then we say that $F$ is of \emph{disjoint type}. 
   If $F\in\Blog$ is of disjoint type,
   then the \emph{Julia set} and \emph{escaping set} of $F$ are defined by
     \begin{align*}
       J(F) \defeq \ &\{z\in H\colon F^n(z)  \text{ is defined for all $n\geq 0$} \} \\
             =\ &\bigcap_{n\geq 0} F^{-n}(\overline{\T}) \qquad\text{and}\\
       I(F) \defeq\  &\{z\in J(F)\colon \re F^n(z)\to\infty \text{ as $n\to\infty$}\}. \end{align*} 
  \end{defn}
\begin{remark}[Remark 1]
  If $F\in\Blog$ has disjoint type, then, by conjugation with an isomorphism $H\to\HH$ that commutes with translation by $2\pi i$, 
   we obtain a disjoint-type function $G\in\Blog$ that is conformally conjugate to $F$ and whose range is the right half-plane $\HH$. 
   It is not difficult to see that all geometric properties discussed in this paper, such as bounded 
  slope, are invariant under this transformation. Hence
   we could always assume that $H=\HH$ in the following; this is the approach taken
   by Bara\'nski and Karpi\'nska \cite[Section~3]{baranskikarpinskatrees}. 
   However, we prefer to work directly with the above more general class, which retains a more direct
   connection to the original entire functions.
\end{remark}
\begin{remark}[Remark 2]
 We are mainly interested in the class $\BlogP$, since this contains the logarithmic transforms of class $\B$ entire functions. However, several of our results
  hold more generally without the periodicity assumption, and we shall state them as such. If $F\in\BlogP$, then $J(F)$ is $2\pi i$-periodic, which need not be
  the case for general $F\in\Blog$. 
\end{remark}
\begin{remark}[Remark 3]
  Observe that the
    final condition \ref{item:accumulatingatinfty} can be equivalently rephrased as follows:
    \begin{enumerate}
      \item[(d')] The tracts of $F$ have pairwise disjoint closures. Furthermore, for
        every $R>0$, there are only finitely
        many tracts of $F$, up to translations by integer multiples of $2\pi i$, that intersect the 
       vertical line at real part $R$. 
    \end{enumerate}
\end{remark}
  In view of the preceding remark, we record the following observation, which will be used repeatedly.
 \begin{obs}[Points to the left of a given line]\label{obs:pointstotheleft}
   Let $F\colon \T\to H$ in $\Blog$ be of disjoint type. Then for every
    $R>0$, there exists a constant $\Delta>0$ 
     with the following property.

   Let $T$ be a tract of $F$, let $T_{> R}$ be the unbounded connected component of
      the set $\{z\in T\colon \re z > R\}$, and write $T_{\leq R} \defeq  T\setminus T_{> R}\supset 
                   \{ z\in T\colon \re z \leq R\}$. Then $T_{\leq R}$ has 
    diameter at most $\Delta$, both in the Euclidean metric and
     in the hyperbolic metric of $H$. 
    
  Furthermore, for every $R>0$ there is $M>0$ such that, for every $z\in H$:
    \[ \diam_H(\{w\in\T\colon |z-w|\leq R\}) \leq M. \]
 \end{obs}
\begin{proof} 
   Let us fix $R>0$.   
    For a fixed tract $T$, the set $T_{\leq R}$ is compactly contained in $H$, and hence has  
    finite Euclidean and hyperbolic diameter. Clearly $\tilde{T}_{\leq R}$ has the same diameter 
    as $T_R$ if $\tilde{T}$ is a $2\pi i\Z$-translations of $T$. By property (d') above, up to translations from 
    $2\pi i\Z$, there are only finitely many tracts
    for which $T_{\leq R}\neq \emptyset$. This proves the first claim. 

  Furthermore,  since $\overline{\T}\subset H$,  and both  $\T$  and $H$ are $2\pi  i$-periodic,
   $\dist(z,\partial H)$ is uniformly bounded on $\T$. By the standard  estimate~\eqref{eqn:standardestimate}, 
    the density 
   $\rho_H(z)$ is bounded from above, and the  second claim follows. 
\end{proof}

  Any logarithmic transform $F$ of a disjoint-type entire function,
   as described above,  
   belongs to the class $\BlogP$ and has disjoint type. The following theorem, which 
   follows from recent ground-breaking work of 
   Bishop \cite{bishopclassBmodels,bishopclassSmodels} together with the
   results of \cite{boettcher}, 
   shows essentially that the converse also holds. 

  \begin{thm}[Realisation of disjoint-type models] \label{thm:realization}
    Let $G\in\BlogP$ be of disjoint type
     and let $g$ be defined by $g(\exp(z)) = \exp(G(z))$. Then there exists 
     a disjoint-type function $f\in \B$ such that $f|_{J(f)}$ is topologically (and, in fact,
     quasiconformally) conjugate
     to $g|_{\exp(J(G))}$. 

    Furthermore, there is a disjoint-type function $\tilde{f}\in\classS$ such that every connected
     component of
     $J(G)$ is homeomorphic to a connected component of $J(\tilde{f})$ 
     (but not necessarily vice-versa). The function $\tilde{f}$ may be chosen
     to have exactly two critical values, no asymptotic values, and no critical points of degree 
    greater than $12$.
  \end{thm}
  \begin{proof}
   The first statement is 
     \cite[Theorem 1.2]{bishopclassBmodels}, which is a consequence of 
    Theorem 1.1 in the same paper and \cite[Theorem~3.1]{boettcher}. 
    The second statement is \cite[Theorem~1.5]{bishopclassSmodels}, which similarly follows from \cite[Theorem 1.2]{bishopclassSmodels} and 
    \cite{boettcher}.
  \end{proof}

 Hence, in order to construct the examples of disjoint-type entire functions described in Sections~\ref{sec:intro} 
    and~\ref{sec:intro2}, apart from the more precise statement in Theorem \ref{thm:Stracts}, 
   it will be 
   sufficient to construct functions in the class $\BlogP$ with the analogous properties. 
   (With some extra care, the construction of these examples 
   in the class $\B$ could also be carried out using the earlier approximation result in 
    \cite[Theorem~1.9]{approximationhypdim}, rather than the 
    results of \cite{bishopclassBmodels}.)  By nature of the construction, the approximation will also preserve certain geometric properties,
   such as the bounded-slope condition; compare Remark~\ref{rmk:preserved}.

The same examples can be constructed in the 
   class $\classS$ using the second half of the above theorem, but in this case the entire functions
   will generally acquire additional tracts when compared with the original $\BlogP$ models. (More precisely,
   they will have twice as many tracts as the original model.) 
   In order to construct functions in the class $\classS$ with
   the number of tracts stated in Theorem \ref{thm:Stracts}, we shall instead use Bishop's 
   more precise 
   methods in \cite{bishopfolding} in Section \ref{sec:classS}.

\subsection*{The combinatorics of Julia continua}
  Let $F\in\Blog$ be of disjoint type. The tracts of $F$ and their  iterated preimages provide successively finer partitions of 
    $J(F)$, leading to a natural notion of symbolic dynamics as follows.

\begin{defn}[External addresses and Julia continua]\label{defn:externaladdress}
  Let $F\in\Blog$ have disjoint type. An \emph{external address} of $F$ is
    a sequence $\s= T_0 T_1 T_2 \dots$ of tracts of $F$. 

  If $\s$ is such an external address, then we define
   \begin{align*} J_{\s}(F) \defeq\  &\{ z \in J(F)\colon F^n(z)\in T_n\text{ for all $n\geq 0$}\} \\
                   =\ &\{ z \in J(F)\colon F^n(z)\in \overline{T_n} \text{ for all $n\geq 0$}\} , \\
            \hat{J}_{\s}(F) \defeq\  &J_{\s}(F)\cup\{\infty\} \qquad\text{and}\\
         I_{\s}(F) \defeq\  &I(F)\cap J_{\s}(F). \end{align*}
  
  When $J_{\s}(F)$ is non-empty, we say that $\s$ is \emph{admissible} (for $F$). In this case, $\hat{J}_{\s}(F)$ is called a  
   \emph{Julia continuum} of $F$. An address $\s$ is called \emph{bounded} if it contains only finitely many different tracts, and 
   \emph{periodic} if there is $p\geq 1$ such that $T_j = T_{j+p}$ for all $j\geq 0$.

  For $n\geq m\geq 0$, we also define 
    \[ F^{m  \dots n}_{\s} \defeq  F|_{T_n} \circ F|_{T_{n-1}}\circ \dots \circ F|_{T_m}, \qquad F^n_{\s} \defeq  F^{0 \dots n-1}_{\s}\qquad\text{and}\qquad  F^{-n}_{\s}\defeq (F^n_{\s})^{-1}. \]
\end{defn}
\begin{remark}[Remark 1]
   We can write $\hat{J}_{\s}(F)$ as a nested intersection of continua: 
      \begin{equation}\label{eqn:Jsintersection}
         \hat{J}_{\s}(F) = \bigcap_{n=0}^{\infty} \bigl(F_{\s}^{-n}(\overline{T_n})\cup  \{\infty\} \bigr). \end{equation}
      Hence $\hat{J}_{\s}$ is indeed a continuum. Furthermore, if $\s^1\neq \s^2$, then for some $n\geq 0$, 
     the sets $J_{\s^1}(F)$ and $J_{\s^2}(F)$ belong to different connected components of
     the open set  $F^{-n}(\T)$ (where $\T$ again denotes the domain of $F$). Hence every connected component of $J(F)$ is contained in a single Julia continuum. 
\end{remark}
\begin{remark}[Remark 2]
   In \cite[Remark after Theorem~C']{baranskikarpinskatrees}, the question is raised whether each $J_{\s}(F)$ is connected.
  It follows from \cite[Corollary~3.6]{eremenkoproperty} that this is indeed the case; compare the remark after Definition
   \ref{defn:uniformescape}. 
    We reprove this fact in Theorem \ref{thm:infty}, by showing that $\infty$ is a terminal point of
    $\hat{J}_{\s}(F)$. Indeed,  
     a terminal point of a continuum $X$ cannot be a cut point of $X$ (this follows e.g.\ from the 
     \emph{Boundary Bumping Theorem}~\ref{thm:boundarybumping} below), so
     $J_{\s}(F) = \hat{J}_{\s}(F)\setminus\{\infty\}$ is connected.

  In particular, if $f$ is an entire function of disjoint type and $F\in\BlogP$ is a logarithmic transform of $f$, then every (arbitrary/bounded-address/periodic)
    Julia continuum of $f$,
    as defined in the introduction, is homeomorphic (via a branch of the logarithm) to a Julia continuum of $F$ at an (admissible/bounded/periodic) address, and vice versa. 
\end{remark}
\begin{remark}[Remark 3]
 Recall that
    $F|_T$ extends continuously to $\infty$ with $F(\infty)=\infty$, for all tracts $T$ of $F$. Hence, throughout the memoir, we shall use the 
    convention that $F(\infty)=\infty$. We emphasise that, while this extension is continuous on a given tract, and hence on a given Julia continuum,
    it is not a continuous extension of the global function $F$. When considering subsets of 
    $\overline{T}\cup\{\infty\}$, it will  also be convenient to use the convention that $\re \infty = \infty$.  
\end{remark}

\subsection*{Hyperbolic expansion}
  In order to study disjoint-type functions, we shall use the fact that they are
   \emph{expanding} on the Julia set, with respect to the hyperbolic metric on $H$. Recall that    
    $\|\Deriv F(z)\|_H$ denotes the hyperbolic derivative of $F$, measured in the
   hyperbolic metric of $H$, and that $\diam_H$ denotes hyperbolic diameter in
   $H$.

\begin{prop}[Expanding properties of $F$]\label{prop:expansion}
  Let $F\colon\T\to H$ be a disjoint-type function in $\Blog$. Then there is a constant 
     $\Lambda>1$ such that
     \[ \|\Deriv F(z)\|_{H} \geq \Lambda, \]
    for all $z\in\T$;
    furthermore $\|\Deriv F(z)\|_{H}\to\infty$ as $\re z\to\infty$. 
\end{prop}
\begin{proof}
  This fact is well-known and follows from the standard estimate~\eqref{eqn:standardestimate} on the hyperbolic metric; see e.g.\ 
   \cite[Lemma~3.3]{baranskikarpinskatrees} or \cite[Lemma~2.1]{strahlen}.
\end{proof}

 This expanding property implies that two different 
   points within the same Julia continuum must eventually
    separate under iteration. (Compare \cite[Lemma~3.1]{strahlen} for a quantitative
    statement.) We state the following slightly more general fact, 
    which will be used frequently in Section \ref{sec:geometry}. 
   (Recall that $\dist_H$ denotes hyperbolic distance in $H$.) 

\begin{lem}[Separation of orbits]\label{lem:separationoforbits}
 Let $F\colon \T\to H$ be a disjoint-type function in $\Blog$, and let  $\s=T_0 T_1 T_2\dots $ be an admissible external address for $F$. Suppose that
    $A\subset H$ 
   is compact, and that 
    $B\subset H$ is closed with 
    $F^n(A\cup B)\cap\T\subset T_n$ and $F^n(A), F^n(B)\neq\emptyset$ for all $n\geq 0$.
 If furthermore 
      \begin{equation} \label{eqn:hypdistancebounded} \liminf_{n\to\infty}
          \dist_H(F^{n}(A) , F^{n}(B)) < \infty,
      \end{equation}
 then $A\cap B\cap J_{\s}(F) \neq \emptyset$. 

  In particular, if $z,w\in J_{\s}(F)$ are distinct points, then 
     $\lim_{n\to\infty}|F^n(z)-F^n(w)|=\infty$. 
\end{lem}
\begin{proof}
  By assumption, there is $\delta>0$, 
    a strictly increasing sequence $n_k$, as well as points $a_k\in F^{n_k}(A)$ and $b_k\in F^{n_k}(B)$ such that 
\[ \dist_H( a_k , b_k) \leq \delta   \]
   for all $k$. Set $\alpha_k \defeq  F^{-n_k}_{\s}(a_k)\in A$ and $\beta_k \defeq  F^{-n_k}_{\s}(b)\in B$. 
    By Proposition \ref{prop:expansion}, we have 
    $\dist_H(\alpha_k,\beta_k) \leq  \Lambda^{-n_k}\delta$ for all $k$, where $\Lambda>1$ is the expansion
    constant. Since $A$ is compact, there is a limit point $\alpha\in A\cap B$ of $(\alpha_k)$,  and furthermore
    $\alpha\in  J_{\s}(F)$ by~\eqref{eqn:Jsintersection}.
 
  In particular, if $A=\{z\}$ and $B=\{w\}$, with $z\neq w$, then
     $\dist_H(F^n(z),F^n(w))\to \infty$. Hence $|F^n(z)-F^n(w)|\to\infty$ by the final statement in
     Observation~\ref{obs:pointstotheleft}.  
 \end{proof}

It follows that there cannot be infinitely many
   times at which two different orbits both return to the same left half plane. 
 \begin{cor}[Growth of real parts {\cite[Lemma~3.2]{strahlen}}]\label{cor:growthofrealparts}
   Let $F\in\Blog$ be of disjoint type, and let $\s$ be an admissible external address of $F$. If
    $z$ and $w$ are distinct points of $J_{\s}(F)$, then 
      $\max(\re F^n(z), \re F^n(w)) \to\infty$
   as $n\to\infty$. 
  \end{cor}
  \begin{proof}
   Let $R>0$, and write $\s=T_0 T_1 T_2 \dots$. Recall from Observation~\ref{obs:pointstotheleft} that 
    the diameter of $\{z\in T_n\colon \re z\leq R\}$ is bounded independently of $n$. Hence by Lemma~\ref{lem:separationoforbits},
     $\max(\re F^n(z), \re F^n(w)) > R$
    for all but finitely many $n$, as claimed.
  \end{proof} 

 Another simple consequence of hyperbolic expansion is the fact, mentioned in Section~\ref{sec:intro2}, that each Julia continuum at a bounded address
   contains a unique point with bounded orbit.

 \begin{prop}[Points with bounded orbits]\label{prop:boundedorbits}
   Let $F\in\Blog$ be of disjoint type, 
    and let $\s$ be a bounded external address. Then there is a unique non-escaping point
    $z_0\in J_{\s}(F)$. This point has bounded orbit, i.e.\ 
    $\sup_{j\geq 0} \re F^j(z_0)<\infty$, and is accessible from $\C\setminus J(F)$. If $\s$ is periodic of period $p$, so is~$z_0$.
    If $A\subset J_{\s}(F)\setminus\{z_0\}$ is closed, then $\re F^n|_A\to\infty$ uniformly. 
 \end{prop}
 \begin{proof}
   The proposition follows from \cite[Theorem~B]{baranskikarpinskatrees} and its proof, which
     however deals also with certain unbounded addresses and is therefore somewhat technical. For the reader's convenience, 
     we provide the simple argument in our case.
 
   Let us begin by proving the existence of a point $z_0\in J_{\s}(F)$ with bounded orbit, which is accessible from $\C\setminus J(F)$. 
    Let $\zeta_0$ be an arbitrary base point in $H\setminus \T$. For each of the tracts $T_i$ in the address
    $\s=T_0 T_1 T_2\dots$, let $\Gamma_i\subset H\setminus J(F)$ be a smooth curve 
    connecting $\zeta_0$ 
    to $F_{T_i}^{-1}(\zeta_0)$. (Observe that $H\setminus J(F)$ is connected.)

    Since $\s$ contains only finitely many different tracts, the curves $\Gamma_i$ can be chosen to have uniformly bounded hyperbolic length, say
    $\ell_H(\Gamma_i)\leq \Delta$ for all $i$. 

   For $j\geq 0$, let $\gamma_j$ be defined as the preimage of the curve $\Gamma_j$ under the appropriate branches of $F^{-1}$:
    \[ \gamma_j \defeq  F_{T_0}^{-1}( F_{T_1}^{-1}(\dots ( F_{T_{j-1}}^{-1}(\Gamma_j) ) \dots ) ). \]
   Then $\gamma \defeq  \bigcup_{j\geq 0}\gamma_j$ is a curve in $H\setminus J(F)$, beginning at
     $\zeta_0$ and having hyperbolic length
     \[ \ell_H(\gamma) = \sum_{j=0}^{\infty} \ell_H(\gamma_j) \leq
                       \Delta\cdot \sum_{j=0}^{\infty} \Lambda^{-j} = \frac{\Delta\Lambda}{\Lambda-1}. \]

   Hence $\gamma$ has a finite endpoint $z_0$, which  belongs to $J_{\s}(F)$ by~\eqref{eqn:Jsintersection}
     and is accessible from $\C\setminus J(F)$ by definition. Furthermore, for all $n\geq 0$, the curve
     $\bigcup_{j\geq n} F^n(\gamma_j)$
    connects $\zeta_0$ to $F^n(z_0)$, and its length is also bounded by $\Delta\Lambda/(\Lambda-1)$. Thus $z_0$ has bounded orbit. 

  By Corollary \ref{cor:growthofrealparts}, $z_0$ is the unique non-escaping point in $J_{\s}(F)$, and 
   by Lemma~\ref{lem:separationoforbits}, $\re F^n|_A\to\infty$ uniformly if $z_0\notin \overline{A}$.  In particular, if $\s$ is periodic
    of period $p$, so is $z_0$.  
 \end{proof}

\begin{cor}[Uncountably many Julia continua]\label{cor:uncountablymany}
 The Julia set $J(f)$ of a disjoint-type entire function $f$ has uncountably many connected components. Each of these components is closed and unbounded.
\end{cor}
\begin{proof}
  Connected components of a closed set are always closed. A function $f\in\B$ has no multiply-connected Fatou components \cite[Proposition~2]{alexmisha}, hence
   $J(f)\cup\infty$ is connected, and by the boundary bumping theorem (see below) every connected component of $J(f)$ is unbounded.

 Now let $F$ be a disjoint-type logarithmic transform of $f$, and let $T$ be a tract of $F$. There are uncountably many bounded external addresses of $F$ whose initial entry is
   $T$. By Proposition~\ref{prop:boundedorbits}, these addresses are all admissible, so there are associated
  non-trivial Julia continua. As noted in Remark 2, each Julia continuum of 
   $F$ corresponds to a Julia continuum of $f$ under the exponential map, and since $\exp$ is injective on $T$, these are pairwise disjoint. This completes the proof.
\end{proof}

\subsection*{Results from continuum theory}
  We shall frequently require the following fact. 

 \begin{thm}[Boundary bumping theorem {\cite[Theorem~5.6]{continuumtheory}}]\label{thm:boundarybumping}
   Let $X$ be a continuum, and let $E\subsetneq X$ be non-empty. If $K$ is a connected component of $X\setminus E$, then
    $\overline{K}\cap \partial E \neq\emptyset$. 
 \end{thm}

Recall the definition of (in-)decomposability from the introduction. 
  \begin{defn}[(In-)decomposable continua]\label{defn:indecomposable}
    A continuum $X$ is called \emph{decomposable} if it can be written as the union of two proper subcontinua of $X$. Otherwise, $X$ is
     \emph{indecomposable}. 

   Furthermore $X$ is called \emph{hereditarily} (in-)decomposable if every non-degenerate subcontinuum of $X$ is (in-)decomposable. 
  \end{defn}

  We now collect some well-known results concerning irreducibility and indecomposability, 
    mainly to be used in Section \ref{sec:uniform}. 
   The \emph{(topological) composant} of a point $x\in X$ is the union of all proper subcontinua containing $x$. We say that
    $x\in X$ is a \emph{point of irreducibility} of $X$ if there is some $y\in X$ such that $X$ is irreducible between $x$ and $y$
    in the sense of Definition \ref{defn:irreducible}.

\begin{prop}[Properties of composants]\label{prop:composants}
  Let $X$ be a continuum, and let $x\in X$. 
  \begin{enumerate}[(a)]
    \item The point $x$ is a point of irreducibility of $X$ if and only if its composant is a proper subset of $X$.
    \item The point $x$ is terminal if and only if $x$ is a point of irreducibility of every subcontinuum that contains $x$.\label{item:terminal_irreducible}
    \item A continuum is hereditarily indecomposable if and only if every point of $X$ is a terminal point.\label{item:hereditarilyindecomposable}
    \item If $C$ is a composant of $X$, then $X\setminus C$ is connected.\label{item:composantcomplement}
    \item If $X$ is decomposable, then there are either one or three different composants. 
    \item If $X$ is indecomposable, then there are uncountably many different composants, every two of which are disjoint, and 
             each of which is dense in $X$.\label{item:composantindecomposable}
  \end{enumerate}   
\end{prop}
\begin{proof}
 The first claim is immediate from the definition; for~\ref{item:terminal_irreducible} see \cite[Theorem~12]{bingsnakelike}. By definition, $X$ is hereditarily
   indecomposable if and only if any two subcontinua of $X$ are either nested or disjoint. Clearly this is the case if and only if all points of $X$
   are terminal. 
The remaining
   statements can be found in Theorems 11.4, 11.13 and 11.17, as well as Exercise 5.20, of \cite{continuumtheory}.
\end{proof}

Finally, we recall the definition of \emph{inverse limits}, which can be considered~-- from a dynamical systems point of view~-- as the
  space of inverse orbits of a (possibly non-autonomous) dynamical system. We refer to 
   \cite[Chapter~2]{continuumtheory} for more background on inverse limits of continua.
 \begin{defn}[Inverse limits]\label{defn:inverselimit}
   Let $(X_j)_{j\geq 0}$ be a sequence of metric spaces, and let $f_j\colon X_{j}\to X_{j-1}$ be a continuous function for every
     $j\geq 1$. Then the collection $(X_j,f_{j+1})_{j\geq 0}$ is called an \emph{inverse system}. 

   Let  $X$ be the set of all ``inverse orbits'', $(x_0 \mapsfrom x_1 \mapsfrom x_2 \mapsfrom \dots)$,
   with $x_j\in X_j$ for all $j\geq 0$ and $f_j(x_{j})=x_{j-1}$ for all $j\geq 1$. 
  Then $X$, with the product topology, is called the \emph{inverse limit} of the functions $(f_j)$, and
   denoted $\invlim (X_j,f_{j+1})_{j=0}^{\infty}$ or $\invlim (f_j)_{j=1}^{\infty}$. 
   The maps $f_j$ are called the \emph{bonding maps} of the
   inverse limit $X$. 

  If $(X_j,f_{j+1})$ is an inverse system and $n> j \geq 0$, then we shall abbreviate
    \[ f_{n,\dots,j} \defeq  f_{j+1}\circ\dots\circ f_{n} \colon X_{n}\to X_j. \]
\end{defn}
\begin{remark}\leavevmode
 \begin{enumerate}[1.]
     \item
   If all spaces $X_j$ are continua, then the inverse limit is again a continuum.
  \item
   The simplest example of an inverse system is the case where $X_j=X$ is a fixed space and $f_j=g^{-1}$ for all $j$, where
    $g\colon X\to X$ is a homeomorphism. In this case, the inverse limit is homeomorphic to $X$ itself. (A
    homeomorphism
   is given by projection  to the $j$-th coordinate, for any $j\geq 0$.)

   Similarly, with the same notation, if $g\colon X\to \tilde{X}\supset X$ is a homeomorphism, then 
    the inverse limit $\invlim (\tilde{X},g^{-1})=\invlim( X, g^{-1}|_X)$ is homeomorphic to the set of points that stay in $X$ forever under iteration under $g$.
  \item
   To connect the concept to the study of Julia continua, suppose that $F\in\Blog$ and that $\s=T_0 T_1 T_2\dots$ is an admissible external address of $F$.
    Define $X_j \defeq \overline{T_j}$ and $f_j \defeq (F_{T_{j-1}}^{-1})|_{\overline{T_{j}}}\colon X_{j} \to X_{j-1}$. Then $\invlim (f_j,X_j)$ is homeomorphic to 
    $J_{\s}(F)$, under projection to $X_0$. Note that, as in this example, we usually think  of the bonding maps
    in inverse systems as being connected to \emph{backward} branches of the (expanding) dynamical systems under     
    consideration.  
  \item When defining inverse  limits, one must choose whether $f_j$ should denote the bonding map from the $j$-th to the $(j-1)$-th
      coordinate, or from the $(j+1)$-th to the  $j$-th  coordinate. In  the literature, the latter is often used;
      we have chosen the former as it leads to more natural notation in  our main construction.
 \end{enumerate}
\end{remark}

    The introduction of some further topological background concerning arc-like continua will be delayed until 
      Section \ref{sec:topology}, as it is only required in the second part of this memoir.

\section{Shadowing and a conjugacy principle for inverse limits}\label{sec:conjugacy}
  It is a fundamental principle in dynamics that expanding systems that are close to each other are topologically conjugate~-- this is 
    a uniform version of the \emph{shadowing lemma} (compare \cite[Theorem~18.1.3]{katokhasselblatt}). The same ideas appear frequently also in the study of
    inverse limits to show that two such spaces are homeomorphic; compare e.g.~\cite[Proposition 2.8]{continuumtheory}. 

  We shall use this principle in a variety of different settings and to different ends: 
     to construct an inverse limit representation for a periodic Julia continuum (Theorem \ref{thm:invariantarclike}), 
       to 
     construct homeomorphisms between different subsets in the Julia set of a disjoint-type function $F\in\Blog$ 
     (Section \ref{sec:homeomorphicsubsets}), and to 
     construct Julia continua homeomorphic to given inverse limit spaces (Sections \ref{sec:arclikeexistence}, 
     \ref{sec:boundedexistence} and \ref{sec:periodicexistence}).
  Our goal in this section is to provide a unified account. 
   We shall
   use a slightly non-standard notion of expanding dynamics. 
 \begin{defn}[Expanding inverse system]\label{defn:expandingsystem}
   Let $(X_j,f_{j+1})_{j\geq 0}$ be an inverse system, where all $X_j$ are complete metric spaces, and let $d_{X_j}$ denote the metric on $X_j$\footnote{%
      Observe that this is a slight abuse of notation; even in the case where all the spaces in the inverse system are the same as sets, we may be using
      different metrics for different choices of $j$. We use this notation to simplify the discussion in which there are several inverse systems, and it should not
      lead to any ambiguities.}.
    We say that the system is \emph{expanding} if
    there are constants $\lambda>1$ and $K\geq 0$ with the following properties. 
   \begin{enumerate}[(1)]
     \item For all $j\geq 1$, and all $x,y\in X_{j}$,\label{item:expandinginversesystem}
          \begin{equation}\label{eqn:expandinginversesystem}
                 \max(K, d_{X_{j}}(x,y)) \geq \lambda d_{X_{j-1}}(f_j(x) , f_j(y)) .
           \end{equation}
     \item For all $\Delta>0$ and all $k\geq 0$,\label{item:expandingbackwardsshrinking}
          \[ \lim_{j\to\infty} 
              \sup\{ d_{X_k}( f_{j,\dots,k}(x),f_{j,\dots,k}(y))\colon x,y\in X_{j}, d_{X_{j}}(x,y)\leq \Delta \} =0.  \]
    \end{enumerate}

   A sequence $(\tilde{x}_j)_{j\geq 0}$ is called a \emph{pseudo-orbit} of the system if 
     $\tilde{x}_j\in X_j$ for all $j$ and $m\defeq \sup_{j\geq 1} d_{X_{j-1}}(\tilde{x}_{j-1},f_j(\tilde{x}_{j})) < \infty$.
     We also call the sequence an $M$-pseudo-orbit ($M>0$) if $m\leq M$.
 \end{defn}
\begin{remark}[Remark]
   Clearly, the backwards shrinking property~\ref{item:expandingbackwardsshrinking} is automatically satisfied if
    the system is expanding at all scales, i.e.\ if~\eqref{eqn:expandinginversesystem} holds with $K=0$. 
    However, we frequently apply the results of this section in cases where we can only ensure 
    expansion above a certain scale, i.e.\ where $K>0$. Indeed, this is crucial in 
   Observation~\ref{obs:obtainingexpandingsystems} below, which in turn plays an important role in our constructions.
 \end{remark}

  Note that we call a system as in Definition~\ref{defn:expandingsystem} expanding, although the maps $f_j$ themselves are (weakly) contracting. 
    The reason is that
      we think of the dynamics of the \emph{system} as going in the opposite direction of the functions $f_j$ 
      (recall Remark 3 after Definition~\ref{defn:inverselimit}). One simple consequence of expansion  is the following.

\begin{obs}[Convergence of orbits]\label{obs:inverselimitseparating}
  Let $(X_j,f_{j+1})_{j\geq 0}$ be an  expanding system,  and let 
      $x=(x_j)_{j=0}^{\infty}\in \invlim(f_{j})$. Suppose that 
     $(x^k)_{k\geq 0}$ is a sequence in  $\invlim(f_j)$, and that there is a sequence $j_k\to\infty$ and a number
     $M>0$ such that $d_{X_{j_k}}(x_{j_k},x^k_{j_k}) \leq M$. Then $x^k\to x$. 
\end{obs}
 \begin{remark}
   In particular, if the inverse limit of an expanding inverse system contains more than one point, then the diameter of the spaces $X_j$ with respect to
    the metric $d_{X_j}$ must be either infinite, or tend to infinity as $j\to\infty$. 
 \end{remark}
\begin{proof}
 This is an immediate consequence of property~\ref{item:expandingbackwardsshrinking} in the definition. 
\end{proof}

   We also remark that expansion is a property of the inverse system, rather than of the underlying inverse limit. 
     Indeed, when given any inverse system, we can 
     often artificially blow up the metrics involved to obtain a modified system that is expanding, but 
    has the same inverse limit. This fact will be useful in the second part of the paper. 

\begin{obs}[Obtaining expanding systems]\label{obs:obtainingexpandingsystems}
   Let $(X_j,f_{j+1})$ be an inverse system, with metrics $d_{X_j}$ on  $X_j$. Suppose that all $X_j$ are compact.
    Then there are constants $\gamma_j\geq 1$, 
    for $j\geq 0$, such that the system is expanding when $X_j$ is equipped 
    with the metric $\tilde{d}_{X_j}\defeq \gamma_j\cdot d_{X_j}$. 

   More precisely, for given $K>0$ and  $\lambda>1$ 
     there are functions $\Gamma_j\colon [1,\infty)\to [1,\infty)$ (for $j\geq 1$) with the following property:
     If $(\gamma_j)_{j\geq  0}$ satisfies $\gamma_{j}\geq \Gamma_j(\gamma_{j-1})$ for all $j\geq 1$,
     then the system $(X_j,f_{j+1})$ equipped 
     with the metrics $\tilde{d}_{X_j}$ satisfies Definition~\ref{defn:expandingsystem}
     for $K$ and  $\lambda$.  Moreover,  $\Gamma_j$ depends on $f_1,\dots,f_{j}$, but not on $f_k$ for $k > j$. 
\end{obs}
\begin{proof}
  We first claim that there is a sequence $(\gamma_j^0)_{j\geq 1}$ 
     such that condition~\ref{item:expandingbackwardsshrinking}
     is satisfied for the blown-up system whenever  $\gamma_j\geq \gamma_j^0$ for all  $j$.
     Indeed, set
      \begin{align*} 
          M_j\defeq &\min\{d_{X_j}(x,y)\colon x,y\in X_j, d_{X_k}(f_{j,\dots,k}(x),f_{j,\dots,k}(y))\geq 1/j\text{ for some $k\leq j$}\}\qquad\text{and} \\
       \gamma_j^0  \defeq &\max(1,2j/M_j)  \end{align*}
    and the claim follows easily. 

   Now, for $j\geq 1$ and $\gamma\geq 1$, define
    \[ \tilde{\Gamma}_{j}(\gamma) \defeq
   \max\Bigl\{\frac{\lambda \cdot \gamma \cdot d_{X_{j-1}}(f_{j}(x),f_{j}(y))}{%
                           d_{X_{j}}(x,y)}\colon x,y\in X_{j}, d_{X_{j-1}}(f_j(x),f_j(y)) \geq K/(\lambda\cdot\gamma)\Bigr\} \Bigr). \]
    Suppose that $(\gamma_j)$ satisfies $\gamma_{j}\geq \tilde{\Gamma}_j(\gamma_{j-1})$ for all $j$. 
     Then clearly~\ref{item:expandinginversesystem}
      holds for the metrics $\tilde{d}_{X_j}$.      Setting $\Gamma_{j}(\gamma) \defeq \max(\gamma_{j}^0, \tilde{\Gamma}_{j}(\gamma))$, we are done.
\end{proof}

\begin{lem}[Shadowing lemma]\label{lem:shadowing}
  Let $(X_j,f_{j+1})$ 
   be an expanding inverse system, and suppose that $(\tilde{x}_j)_{j\geq 0}$ is an $M$-pseudo-orbit. Then there is a unique orbit 
     $(x_j)_{j\geq 0}\in \invlim (f_j)$  such that 
         \[ \limsup_{j\to\infty} d_{X_j}(x_j,\tilde{x}_j) < \infty. \]
    More precisely, $d_{X_j}(x_j,\tilde{x}_j) \leq \lambda\cdot \max(M,K)/(\lambda-1)$ for all $j$, where $K$ is the constant from 
      Definition~\ref{defn:expandingsystem}.
\end{lem}
\begin{proof}
   Set $\widetilde{M}\defeq \max(M,K)$. For $n\geq j \geq 0$, define $x_j^n\in X_j$ recursively by
      $x_j^j \defeq \tilde{x}_j$ and 
    \[ x_j^{n+1} \defeq f_{j+1}(x_{j+1}^{n+1}). \] 
    We claim that $d_{X_j}(x_j^n,\tilde{x}_j)\leq \lambda \widetilde{M}/(\lambda-1)$ for all $j$ and $n$. Indeed, this follows by induction
      on $n-j$, since 
       \begin{align*} d_{X_j}(x_j^{n+1} , \tilde{x}_j) &=
            d_{X_j}(f_{j+1}(x_{j+1}^{n+1}) , \tilde{x}_j) \leq d_{X_j}(f_{j+1}(x^{n+1}_{j+1}) , f_{j+1}(\tilde{x}_{j+1})) + M \\ 
        &\leq 
             \frac{\max(K , d_{X_{j+1}}(x^{n+1}_{j+1} , \tilde{x}_{j+1}))}{\lambda}  + M \leq 
             \widetilde{M}\cdot \left(1 + \frac{1}{\lambda-1}\right) = \frac{\lambda\widetilde{M}}{\lambda-1}. \end{align*}

  By property~\ref{item:expandingbackwardsshrinking} of Definition~\ref{defn:expandingsystem}, it follows that $(x_j^n)_{n\geq j}$ is a Cauchy sequence for every $j\geq 0$, and hence
    has a limit $x_j$ (recall that the space $X_j$ is assumed to be complete). We have $f_j(x_{j})=x_{j-1}$ by construction, which completes the proof of existence. Uniqueness follows immediately from
  Observation~\ref{obs:inverselimitseparating}. 
\end{proof}

 \begin{defn}[Pseudo-conjugacy between inverse systems]\label{defn:pseudoconjugacy}
  Two inverse systems $(X_j,f_{j+1})$ and $(Y_j,g_{j+1})$ will be called \emph{pseudo-conjugate} if 
   there exists a sequence of (not necessarily continuous) functions
     $\psi_j\colon X_j \to Y_j$ and a constant $M>0$ with the following properties.
   \begin{enumerate}[(a)]
     \item $d_{Y_{j-1}}(g_j(\psi_{j}(x)) , \psi_{j-1}( f_j(x))) \leq M$ for all $j\geq 1$ and all $x\in X_{j}$.%
               \label{item:pseudoconjugacy}
     \item The image $\psi_j(X_j)$ is $M$-dense in $Y_j$ for all $j$; i.e., for all $y\in Y_j$, there is
                 $x\in X_j$ such that $d_{Y_j}(\psi_j(x),y) \leq M$.\label{item:pseudosurjectivity}
     \item For all $\Delta\geq 0$, there is $R>0$ such that
       $d_{Y_{j-1}}(\psi_{j-1}(f_{j}(x)),\psi_{j-1}(f_{j}(\tilde{x})))\leq R$
       whenever $j\geq 1$ and $x,\tilde{x}\in X_{j}$ with 
       $d_{X_{j}}(x,\tilde{x})\leq \Delta$.\label{item:pseudocontinuity}
     \item For all $\Delta\geq 0$,  there is $R>0$ such that
       $d_{X_j}(x,\tilde{x})\leq R$ whenever $j\geq 0$ and $x,\tilde{x}\in X_j$ with 
       $d_{Y_j}(\psi_j(x),\psi_j(\tilde{x}))\leq \Delta$.\label{item:pseudoinjectivity}
   \end{enumerate}
 \end{defn}
\begin{remark}[Remark 1]
  Observe that a pseudo-conjugacy can be considered an approximate conjugacy. Indeed,
    the four conditions given are approximate versions of,  respectively, the conjugacy relation, 
    surjectivity of $\psi_j$, continuity of $\psi_j$ and a combination of injectivity of $\psi_j$ and continuity of
      its inverse. 
\end{remark}
\begin{remark}[Remark 2]
  There is an asymmetry between the conditions~\ref{item:pseudocontinuity} 
   and~\ref{item:pseudoinjectivity}, in that one concerns the map  
    $\psi_{j-1}\circ f_j$ while  the other  deals with  $\psi_j$ itself. One could weaken the requirements on
     a pseudo-conjugacy 
    to obtain a perhaps more natural, but also more technical, definition, which  still yields an analog of
    Proposition~\ref{prop:conjugacyprinciple} below. However, the  above version  is sufficient for all  our purposes.
\end{remark}
\begin{remark}[Remark 3]
   Observe that the definition depends strongly on  the choice of metrics. In particular, suppose two pseudo-conjugate 
     systems are transformed into expanding systems 
    using Observation~\ref{obs:obtainingexpandingsystems}.
     Then  the resulting blown-up systems will not,  in general, again  by  pseudo-conjugate  to each other.   
\end{remark}

 \begin{prop}[Conjugacy principle for expanding inverse systems]\label{prop:conjugacyprinciple}
  Two expanding inverse systems $(X_j,f_{j+1})$ and $(Y_j,g_{j+1})$ that are pseudo-conjugate have homeomorphic
    inverse limits. 

  More precisely, if $(\psi_n)$ is the sequence of maps guaranteed by the definition of pseudo-conjugacy, then there is a 
    homeomorphism $\theta\colon\invlim (f_j)\to \invlim (g_j)$ such that
      \begin{equation}\label{eqn:conjugacycloseness}
            d_{Y_j}( \theta_j(x) , \psi_j(x_j)) \leq \max(M,K)\cdot \frac{\lambda}{\lambda-1}  \end{equation}
    for all $j\geq 0$ and all $x = (x_0 \mapsfrom x_1 \mapsfrom\dots ) \in \invlim (f_j)$. (Here $\lambda$ and $K$ are
    the expansion constants for the system $(Y_j,g_{j+1})$, as in Definition~\ref{defn:expandingsystem}, and $M$ is the number from 
    Definition~\ref{defn:pseudoconjugacy}.)
 \end{prop}
 \begin{proof}
  Let $x\in \invlim (f_j)$, and define $\tilde{y}_j \defeq \psi_j(x_j)$ for $j\geq 0$. Then $(\tilde{y}_j)$ is a pseudo-orbit for 
     the inverse system $(Y_j,g_{j+1})$, since
     \[ d_{Y_{j-1}}( g_j(\tilde{y}_{j}), \tilde{y}_{j-1} ) = 
          d_{Y_{j-1}}( g_j(\psi_{j}(x_{j}) , \psi_{j-1}( f_j(x_{j}) )  ) \leq M \]
     for all $j\geq 1$. 

  By Lemma~\ref{lem:shadowing}, there is a unique element $y\in \invlim(g_j)$ such that
    $\tilde{y}_j$ and $y_j$ are at most
     distance $\delta\defeq \max(M,K)\cdot \lambda/(\lambda-1)$ apart for all $j$. So if we define $\theta(x)\defeq y$,
  then
     $\theta$ satisfies~\eqref{eqn:conjugacycloseness}. 

  To show that  $\theta$ is continuous, let $x^i \to x$ in $\invlim(f_j)$ as  $i\to\infty$,  and set
    $y^i\defeq \theta(x^i)$, $y\defeq \theta(x)$. Let $R$  be as in~\ref{item:pseudocontinuity}, say for $\Delta=1$,
    and let $j\geq  1$ be arbitrary. If  $i$ is large enough, then 
    $d_{X_j}(x^i_{j} , x_{j})\leq 1$ and hence 
    \[ d_{Y_{j-1}}( \psi_{j-1}(x^i_{j-1}) , \psi_{j-1}(x_{j-1}) ) = 
         d_{Y_{j-1}}( \psi_{j-1}(f_j(x^i_{j})) , \psi_{j-1}(f_j(x_{j})))
         \leq R. \]
    By~\eqref{eqn:conjugacycloseness},
    we see  that
    $d_{Y_{j-1}}(y^i_{j-1},y_{j-1})\leq R + 2\delta$ for sufficiently large $i$. 
    It follows from Observation~\ref{obs:inverselimitseparating} that $y^i\to y$, as  required. 

   To show that $\theta$ is injective, suppose that 
     $x^1,x^2\in  \invlim(f_j)$ satisfy  $\theta(x^1)=\theta(x^2)$, and let
     $R$ be as in~\ref{item:pseudoinjectivity}, for $\Delta=2\delta$. Then 
       $\dist_{X_j}(x^1_j,x^2_j)\leq  R$ for all $j$, and hence $x^1=x^2$ by the uniqueness statement in
    Lemma~\ref{lem:shadowing}. 
    Continuity of the inverse (where defined) similarly follows from~\ref{item:pseudoinjectivity} and
      Observation~\ref{obs:inverselimitseparating}.

   It remains to show that $\theta$ is surjective. Let $y\in \invlim(g_j)$. By~\ref{item:pseudosurjectivity}, for
    all $j\geq 0$ there is $x_j\in X_j$ such  that $\dist_{Y_j}(y_j,\psi_j(x_j))\leq M$ for all $j$. 
    Then  
     \begin{align*} \dist_{Y_j}(\psi_j(x_j), \psi_j(f_{j+1}(x_{j+1}))) &\leq
         M + \dist_{Y_j}(y_j,\psi_j(f_{j+1}(x_{j+1})) \\ &\leq
         2M + \dist_{Y_j}(y_j, g_{j+1}(\psi_{j+1}(x_{j+1}))) \\ &= 
         2M + \dist_{Y_j}(g_{j+1}(y_{j+1}), g_{j+1}(\psi_{j+1}(x_{j+1}))) \\
       &\leq 3\max(M,K). \end{align*}
   Here we used~\ref{item:pseudoconjugacy} and  the contracting property  of $g_{j+1}$.
   Let $R_1$ be as in~\ref{item:pseudoinjectivity} with $\Delta = 3\max(M,K)$; then
    $\dist_{X_j}(x_j,f_{j+1}(x_{j+1}))\leq R_1$.
    By Lemma~\ref{lem:shadowing}
    there  is $\hat{x}\in  \invlim(f_j)$ such that $\dist(x_j,\hat{x}_j)\leq \delta_1$
    for some $\delta_1 >0$ and all $j$. 
  Set $\hat{y}\defeq \theta(\hat{x})$. Then 
    \begin{align*}
      \dist_{Y_j}(y_j , \hat{y_j}) &\leq  
                   \dist_{Y_j}(y_j, \psi_j(f_{j+1}(x_{j+1}))) + \dist_{Y_j}( \psi_j(f_{j+1}(x_{j+1})) , \hat{y}_j ) \\ &\leq
                2\max(M,K) + \delta + \dist_{Y_j} ( \psi_j(f_{j+1}(x_{j+1})) , \psi_j(f_{j+1}(\hat{x}_{j+1} ))) \\ &\leq 2\max(M,K) + \delta + R_2,
     \end{align*}
   where $R_2$ is chosen according to~\ref{item:pseudocontinuity} for $\Delta=\delta_1$. By the uniqueness part of Lemma~\ref{lem:shadowing}, we have
    $y=\hat{y}=\theta(\hat{x})$, as  desired. 
 \end{proof} 

 In order to deduce additional properties of the conjugacy $\theta$, it is useful to note the following.
\begin{obs}[Converging to the conjugacy]\label{obs:convergingtotheconjugacy}
  With the notation of Proposition~\ref{prop:conjugacyprinciple}, the $n$-th coordinate $\theta_n$ of the 
    homeomorphism $\theta$ is the uniform limit, as $k\to\infty$, of the functions 
     \[ \invlim(f_j)\to Y_n; \qquad (x_j)_{j\geq 0} \mapsto g_{n,\dots,k}(\psi_k(x_k)). \] 
\end{obs}
\begin{proof}
 This is immediate from the definition of $\theta$ and the proof of the shadowing lemma. 
\end{proof}
\begin{remark}
  In particular, if both systems are autonomous, say generated by $(f,X)$ and $(g,Y)$, and the map $\psi_j\colon X\to Y$ is independent of $j$, then
    $\theta$ is a conjugacy between the functions $\tilde{f}$ and $\tilde{g}$ induced by $f$ and $g$ on the respective inverse limits. 
\end{remark}


\section{Topology of Julia continua}\label{sec:geometry}\label{sec:spanzero}

We now study the general topological properties of Julia continua for a function in the class 
   $\Blog$. In particular, we prove that every such Julia continuum has span zero. The idea of the proof can be outlined in rather simple terms: 
   Since each tract $T$ cannot intersect its own $2\pi i$-translates, two points cannot exchange position by moving inside
    $T$ without coming within distance $2\pi$ of each other. Now let $\s$ be an admissible external address. By applying the preceding
    observation to the tract $T_j$, for $j$ large, and using the expanding property of $F$, we see that two points cannot
    cross each other within $J_{\s}(F)$ without passing within distance $\eps$ of each other, for every $\eps>0$. 
    This establishes that $J_{\s}(F)$ has span zero. (This idea is similar in spirit to the proof of Lemma~3.3 and Corollary~3.4 in 
   \cite{eremenkoproperty},
    which we in fact recover below.)

  However, some care is required, since the tract $T$ can very well contain points whose imaginary parts differ by a large amount 
   (see Figure \ref{fig:windingtract}). Hence we shall have to take some care in justifying the informal argument above, by studying the
   possible structure of \emph{logarithmic tracts} somewhat more closely.
 
  \begin{defn}[Logarithmic tracts]
    Let $T\subset\C$ be an unbounded simply connected domain bounded by a Jordan curve (passing through infinity).
    If $T$ does not intersect its $2\pi i\Z$-translates and is  unbounded to the right (i.e., $\re z\to +\infty$ as $z\to\infty$ in $T$), then $T$
     is called a \emph{logarithmic tract}. 
  \end{defn}
\begin{remark}
In particular, every tract of a function $F\in\Blog$ is a logarithmic tract by definition. 
\end{remark}


  Within such a tract, we wish to understand when points can move around without having to come close to each other. To study this question,
   we introduce the following terminology. (See Figure \ref{subfig:sep}.)

\begin{figure}
  \subfloat[]{\def\svgwidth{.5\textwidth}
  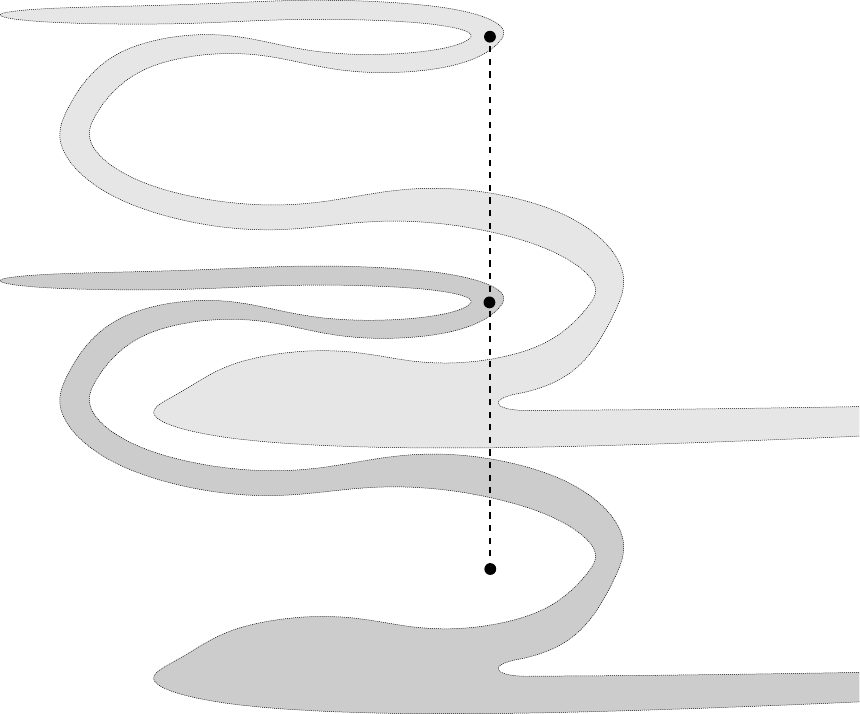\protect\label{subfig:Iz}}
  \subfloat[]{\def\svgwidth{.5\textwidth}
  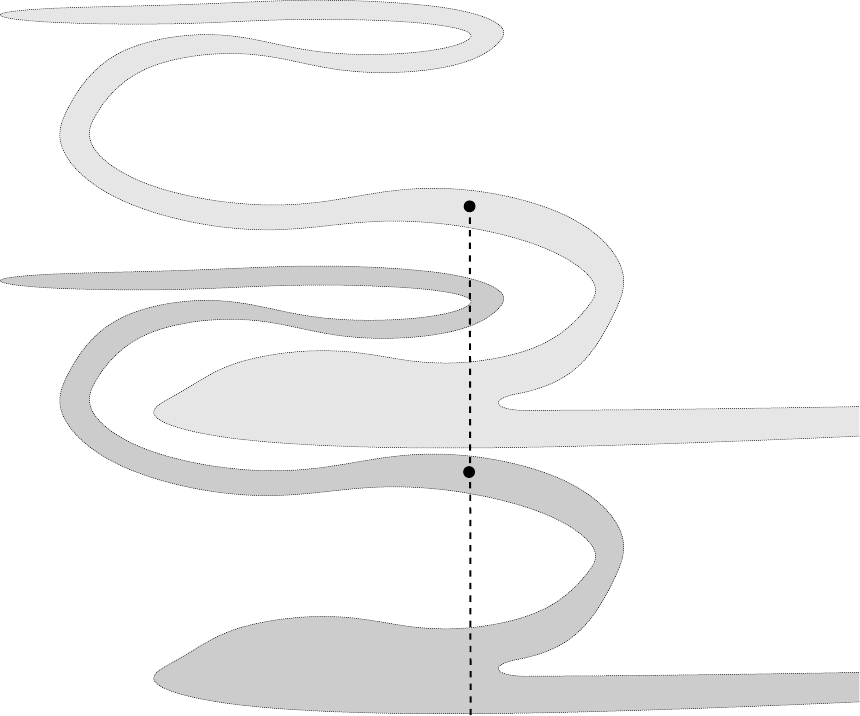\protect\label{subfig:sep}}
 \caption{\label{fig:windingtract}
   A logarithmic tract containing points whose imaginary parts are further apart than $2\pi$. (The translate $T+2\pi i$ is shown in
    light grey to demonstrate that $T$ is indeed disjoint from its $2\pi i\Z$-translates.)
     Subfigure~\protect\subref{subfig:Iz} illustrates
     the definition of the segment $I_z$, and shows that we may have~$\sep_T(a,b,z)=0$ even though
     $\re a < \re z < \re b$. The configuration in~\protect\subref{subfig:sep} demonstrates that 
    the number $\sep_T(a,b;z)$ can decrease under perturbation of $z$ (it will change from $2$ to $0$ if we move
    the point $z$ slightly to the right).}
\end{figure}

  \begin{defn}[Separation number]\label{defn:separationnumber}
    For any $z\in \C$, we denote by 
     $I_z$ the line segment
     \[ I_z \defeq  \{ z + i\cdot t \colon  t\in [-2\pi ,2\pi] \}. \]

   Let $T$ be a logarithmic tract, and let $z\in T$. 
    If $a,b\in (T\cup \{\infty\})\setminus I_z$, then we define
      $\sep_T(a,b;z)\geq 0$ to denote the smallest number of points that 
      a curve $\gamma$ connecting $a$ and $b$ in $T$ may have in common with the segment $I_z$. 

  By convention, $I_{\infty}=\{\infty\}$; hence we also define $\sep_T(a,b;\infty)=0$ for $a,b\in T$. 
  \end{defn}
  (The tract $T$ will usually be fixed in the following, and we shall then suppress the subscript $T$ in this notation.)

  \begin{prop}[Continuous parity of separation numbers]\label{prop:sep}
   Let $T$ be a logarithmic tract, let $a,b,z\in T\cup\{\infty\}$, and suppose that $a,b\notin I_z$. Then 
     the parity of $\sep_T(a,b;z)$ remains constant under
     small perturbations of $a$, $b$ and $z$.

  That is, if $X$ denotes the set of points $(a,b;z)\in (T\cup \{\infty\})^3$ with $a,b\notin I_z$, then the function
      \[ \sep_T\colon X\to \Z_2; \quad (a,b;z)\mapsto \sep_T(a,b;z)\ (\mod 2) \]
     is continuous.

   On the other hand, if $z\in\C$, then $\sep_T(a,b;z)$ increases or decreases by $1$ (and hence 
    changes parity) as $a$ or $b$ passes through the segment $I_z$ transversally. Similarly, this number 
    changes parity as $z$ passes through (exactly) one  of  $I_a$ and $I_b$ transversally. 
  \end{prop}
  \begin{remark}
     The function $\sep_T(a,b;z)$ itself need not be continuous in $z$, but 
      is always upper semi-continuous. That is, 
     under a small perturbation of $z$, the separation number might decrease, but will never increase; see Figure~\ref{fig:windingtract}. 
  \end{remark}
 \begin{proof}
   Continuity of $\sep_T(a,b;z)$ at infinity (in each of the three variables) is clear from the definition. Hence in the following we may assume that
    $a,b,z\in T$.

  Observe that $I_z\cap T$ is a union of vertical cross-cuts of the tract $T$. Clearly $\sep(a,b;z)$ is precisely the number of
   such cross-cuts that separate $a$ from $b$ in $T$. Recall that each cross-cut $C$ separates $T$ into precisely two components, one on each
   side of $C$. In particular, as the point $a$ or $b$ crosses $I_z$, keeping the other points
    fixed, the number $\sep(a,b;z)$ increases or decreases by $1$.

   Also note that, if $\gamma\subset T$ is a smooth curve connecting
   $a$ and $b$, and $\gamma$ intersects $I_z$ only transversally, then 
   the number of intersections between $\gamma$ and $I_z$ has the same
   parity as $\sep(a,b;z)$. Indeed, 
   the curve $\gamma$ must
   intersect every cross-cut that separates $a$ from $b$ in an odd number
   of points, and every cross-cut that does not separate $a$ from $b$ in
   an even number of points.

   So let $a,b,z\in T$ with $a,b\notin I_z$. Clearly a sufficiently small perturbation of $a$ or of $b$ does not change the value (and hence the 
   parity) of $\sep(a,b;z)$, so we only need to focus on what happens when we perturb
   $z$ to a nearby point $\tilde{z}$. 

  Let $\gamma$ be a smooth curve, as above, connecting $a$ and $b$ and intersecting $I_z$ only transversally. If $\tilde{z}$ is close enough to $z$, then
     $\gamma$ also intersects $I_{\tilde{z}}$ only transversally, and in the same number of points. Hence we see that
     $\sep(a,b;z)$ and $\sep(a,b;\tilde{z})$ have the same parity, as claimed. 

 The final claim follows from what has already been proved:
    the effect of $z$ crossing $I_a$ (say) transversally can be  obtained by
    first moving $a$ without intersecting $I_z$ (leaving the separation number unchanged), then
    moving $z$ without intersecting $I_a$ or $I_b$ (preserving parity),  and letting $a$ cross
    $I_z$ transversally  to  its original position (changing parity). 
 \end{proof}

 \begin{cor}[Moving along a connected set]\label{cor:movingpoints}
   Let $T$ be a logarithmic tract, and let $A\subset T\cup\{\infty\}$ be compact and connected. 
    Fix a left-most point $a\in A$ and a right-most point  $b\in A$; i.e. 
   \[ \re a = \min_{z\in A} \re z\quad\text{and}\quad
       \re b = \max_{z\in A} \re z. \]
   (Recall from Section~\ref{sec:preliminaries} that $\re \infty = +\infty$ by convention.)

  \begin{enumerate}[(a)]
    \item \label{item:closetoleftorright}
   Let $w\in A$, and let $\zeta_0\in (A\cap I_a)\setminus I_w$. If $\zeta\in A\setminus I_a$ is sufficiently close to
    $\zeta_0$,   then  $\sep(a,w;\zeta)$ is  odd. 
    The same statement holds with $a$ replaced by  $b$. 

   \item \label{item:movingpoints}
   In particular, let $z\in A$ such that $a,b\notin I_z$. Then
    $\sep(a,b;z)$ is odd, and hence $I_z$ separates $a$ from $b$ in
    $T$. 
 \end{enumerate}
 \end{cor}
 \begin{remark}[Remark 1]
   The final statement implies that, in order to move from a left-most point of $A$ to a right-most point, we must pass 
    within distance at most $2\pi$ of every point of $A$. This is a key idea to keep in mind.
 \end{remark}
\begin{proof}
  Note first that, if $\tilde{\zeta}\in T\setminus A$ is to  the left of  $a$, then 
    $\sep(a,w;\tilde{\zeta})=0$, since $A$ is connected, contains $a$ and $w$ and does not intersect
      $I_{\tilde{\zeta}}$. As  $\tilde{\zeta}$  crosses $I_a$ at $\zeta_0$, the parity of  the separation number
     changes by Proposition~\ref{prop:sep}, proving the first claim. (A slightly more careful argument would show that
     the separation number is exactly equal to $1$.) 
     If $b\neq \infty$, then the second part of~\ref{item:closetoleftorright} follows analogously. 
    We defer the case
     $b=\infty$ until the end of  the proof, and show first how part~\ref{item:movingpoints} follows 
      from~\ref{item:closetoleftorright}. 

  So suppose that~\ref{item:closetoleftorright} holds, and let
      $z\in  A$ such  that $a,b\notin I_z$. Let $K$ be the component of $A\setminus (I_a\cup I_b)$ 
      containing  $z$.  
     By Proposition~\ref{prop:sep}, the set of $\zeta$ such that
     $\sep(a,b;\zeta)$ is odd is relatively open and closed in 
      $K$; so it remains only to show that this set is non-empty. 

   By the boundary bumping theorem (Theorem \ref{thm:boundarybumping}), the closure of $K$ intersects
    $I_a$ or $I_b$; let us suppose for example that there is $\zeta_0 \in I_a\cap \cl{K}$. 
    If $\zeta$ is  sufficiently close  to $\zeta_0$, then $\sep(a,b;\zeta)$ is odd by~\ref{item:closetoleftorright},  as 
      desired.

   The same argument establishes the second half of~\ref{item:closetoleftorright} when $b=\infty$. 
     Indeed, fix $w\in  T$; we shall show that $\sep(w,\infty;\zeta)$ is odd whenever  $\zeta\in T$ has sufficiently
      large real part. Let $\gamma$ be an arc in $T$ connecting $w$ to infinity and let $a$ be a left-most point  of
      $\gamma$. Then the closure of every connected component of $\gamma\setminus I_a$ intersects 
      $I_a$. As we saw, this implies that $\sep(a,\infty; z)$ is odd for all $z\in \gamma\setminus I_a$.  
      But if $\zeta\in  T$ has sufficiently large real part, then  $\zeta$ can be  connected  to $\gamma$ without
      intersecting  $I_a$, and the  claim follows by the continuous parity of separation numbers (Proposition~\ref{prop:sep}).
\end{proof}

We are now ready to prove the statement alluded to at the beginning of the section, which then allows us to
   deduce that every Julia continuum has span zero.

 \begin{prop}[Bounded span of tracts]\label{prop:boundedspan}
  Let $T$ be a logarithmic tract, and let $A\subset T\cup \{\infty\}$ be a continuum. 
   Suppose  that
   $X\subset (T\cup\{\infty\})^2$ is a continuum whose first and second 
   components
   both project to $A$. 

   Then there is a point $(z,w)\in X$ 
    such that $w\in I_z$. In particular, either $z=w=\infty$ or 
    $z,w\in\C$ and $|z-w|<2\pi$. 
 \end{prop}
 \begin{proof}
  We shall prove the contrapositive: suppose that 
   $X\subset (T\cup \{\infty\})^2$ is any compact set whose first and second
   components both project to $A$, and such that $w\notin I_z$ for all
   $(z,w)\in X$. Observe that this implies $(\infty,\infty)\notin X$, and hence 
    $A\cap T\neq\emptyset$. We shall show that $X$ is disconnected.

  Let $a$ be a left-most point of $A$, and let $U$ consist of the
   set of all points $(z,w)\in X$ such that 
   $\sep(a,w;z)$ is defined and even. By Proposition~\ref{prop:sep}, this set is open in $X$.

  We claim that $V\defeq  X\setminus U$ is also open in $X$. 
   Let $(z,w)\in V$; then $w\notin I_z$ by assumption on $X$. If also $a\notin I_z$, then $V$ contains a neighborhood of
    $(z,w)$ in $X$ by Proposition~\ref{prop:sep}.

  Now suppose that $a\in I_z$. Let $B_z$ and $B_w$ be connected neighbourhoods of $z$ and $w$ in $T\cup\{\infty\}$, 
    chosen sufficiently small to ensure that
    $\tilde{w}\notin I_{\tilde{z}}$ for $(\tilde{z},\tilde{w}) \in B_z\times B_w$. 
    By Corollary~\ref{cor:movingpoints}, $\sep(a,w;\tilde{z})$ is odd for $\tilde{z}\in A$ sufficiently close to $z$. 
    By Proposition~\ref{prop:sep},
    the value of $\sep(a,\tilde{w};\tilde{z})$ is constant (and hence odd) for all
    $(\tilde{z},\tilde{w})\in ((A\cap B_z)\setminus I_a)\times B_w$. 
  Hence, if $(\tilde{z},\tilde{w})\in X\cap (B_z\times B_w)$, then either $\tilde{z}\in I_a$ or 
     $\sep(a,\tilde{w};\tilde{z})$ is odd. We have $(\tilde{z},\tilde{w})\in V$ in either case, as required.

  Furthermore, both $U$ and $V$ are non-empty. Indeed, by assumption there are 
   $z,w\in A$ such that $(z,a),(a,w)\in X$. We have $(a,w)\in V$ by definition (since $a\in I_a$). Similarly, we have
    $z\notin I_a$, and hence $a\notin I_z$, by assumption on $X$, and $\sep(a,a;z) = 0$ by definition. So $(a,z)\in U$. 
   We have shown that $X$ is disconnected, as desired. 
 \end{proof} 

 \begin{thm}[Julia continua have span zero]\label{thm:spanzero}
  Let $F\in\Blog$ be of disjoint type, and let $\CH$ be a Julia continuum of $F$.
   Then $\CH$ has span zero. 
 \end{thm}
 \begin{proof}
  Suppose that $X\subset \CH^2$ is a continuum whose projections to the
   first and second coordinates are the same subcontinuum $A\subset \CH$. We must show that
   $(\zeta,\zeta)\in X$, for some point $\zeta\in A$. We assume that $(\infty,\infty)\notin X$,
   as otherwise there is nothing to prove.

  Let $\s= T_0 T_1 T_2 \dots$ be the external address of $\CH$, and consider 
   $A_n \defeq  F^n(A)$ and 
    $X_n \defeq  \{(F^n(z), F^n(w)) \colon  (z,w)\in X\}$ for $n\geq 0$. (Recall that $F(\infty)=\infty$ by convention.)

   By Proposition~\ref{prop:boundedspan}, $X_n$ contains
   a point $(z_n , w_n)\in T_n\times T_n$ such that $|z_n - w_n| < 2\pi$.
   Let $\zeta_n,\omega_n\in A$ such that $z_n = F^n(\zeta_n)$ and $w_n = F^n(\omega_n)$. 
   The hyperbolic distance (in the range $H$ of $F$) between $z_n$ and $w_n$ is
   uniformly bounded by Observation~\ref{obs:pointstotheleft}, and $F$ uniformly expands the hyperbolic metric
    by Proposition~\ref{prop:expansion}.
   It follows that the hyperbolic distance in $T_0$ between $\zeta_n$ and
   $\omega_n$ tends to zero, and thus the spherical distance also tends to zero. If $\zeta$ is any limit point of the sequence
   $(\zeta_n)$, then 
   $(\zeta,\zeta)\in X$, as required. 
 \end{proof}

 We shall next prove that infinity, as well as any non-escaping or accessible point, is terminal in each Julia continuum. 
 \begin{thm}[The role of $\infty$] \label{thm:infty}
   Let $\CH$ be a Julia continuum of a disjoint-type function $F\in\Blog$. Then $\infty$ is a terminal point
    of $\CH$.
\end{thm}
\begin{remark}
  Theorems \ref{thm:spanzero} and \ref{thm:infty} together establish Theorem~\ref{thm:mainspanzero}. 
\end{remark}
\begin{proof}
  Let $\s = T_0 T_1 T_2 \dots$ be the address of $\CH$. 
  Suppose that $\hat{A}^1,\hat{A}^2\subset\CH$ are subcontinua both containing $\infty$; we must show that
    one of these is contained in the other. We may assume that both continua are non-trivial, as otherwise there is nothing to prove.
    Let us set $\hat{A} \defeq  \hat{A}^1\cup\hat{A}^2$, and define 
    $\hat{A}_n \defeq  F^n(\hat{A})$ and $\hat{A}_n^j \defeq  F^n(\hat{A}^j)$ for $n\geq 0$ and $j\in\{1,2\}$.

  For each $n$, let $a_n$ be a left-most point of $\hat{A}_n$ as in Corollary \ref{cor:movingpoints}. There is $j\in\{1,2\}$ such that
    $a_{n_k}\in \hat{A}_{n_k}^j$ for an infinite sequence $(n_k)$. Without loss of generality, we may suppose that $j=1$; we shall show that
    $\hat{A}^2\subset \hat{A}^1$.

 Let $z\in \hat{A}^2\setminus\{\infty\}$, and set $z_k \defeq  F^{n_k}(z)$. We claim that, for all $k$, the point 
   $z_k$ satisfies $I_{z_k}\cap \hat{A}^1_{n_k}\neq\emptyset$, and hence $\dist(z_k,\hat{A}_{n_k}^1\setminus\{\infty\})\leq 2\pi$. If
   $z_k\in I_{a_{n_k}}$, this is trivial. Otherwise,
    $I_{z_k}$ separates $a_{n_k}$ from $\infty$ in the tract
    $T_{n_k}$ by Corollary \ref{cor:movingpoints}. Hence $I_{z_k}$ 
    does indeed intersect $\hat{A}^1_{n_k}$. 

  By the expanding property of $F$~-- more precisely, by Lemma \ref{lem:separationoforbits}, applied to the sets $\{z\}$ and $\hat{A}^1\setminus\{\infty\}$~-- 
    we have 
    $z\in \hat{A}^1$, as claimed.
\end{proof}

We next use a similar argument to
  show that non-escaping points are also terminal points of Julia continua, as claimed in Theorem \ref{thm:nonescapingaccessible}.
\begin{thm}[Non-escaping points are terminal]\label{thm:nonescapingterminal}
  Let $F\in\Blog$ be of disjoint type, and let $\CH$ be a Julia continuum of $F$. If
   $\zeta_0\in \CH\setminus\{\infty\}$ is non-escaping, then
     $\zeta_0$ is a terminal point of $\CH$, and $\CH$ is irreducible between
    $\zeta_0$ and $\infty$. 
 \end{thm}
\begin{proof}
  Let $\s=T_0 T_1 \dots$ be the address of $\CH$. Since 
  $\zeta_0$ is a non-escaping point, there
   is a number $R>0$ and a sequence $(n_k)$ such that
   $\re \zeta_{k} < R$ for all $k$, where  $\zeta_k\defeq F^{n_k}(\zeta_0)$. 
  By Observation~\ref{obs:pointstotheleft}, we can find a constant
   $Q>0$, independent of $k$, with the following property:
   Any two points in $T_{n_k}$ having real part at most $R$ can be connected by a curve $\gamma\subset T_{n_k}$ that
   consists entirely of points at real parts less than $Q$. 

  Let $\hat{A}^1,\hat{A}^2\subset\CH$ be subcontinua both containing $z_0$. Similarly as in the preceding proof,
   let us set $\hat{A} \defeq  \hat{A}^1\cup \hat{A}^2$, and let $b_k$ be a right-most point of 
    $\hat{A}_{k} \defeq  F^{n_k}(\hat{A})$. By relabelling, and by passing to a further subsequence if necessary, 
    we may assume that $b_k\in \hat{A}_k^1 \defeq  F^{n_k}(\hat{A}^1)$. We shall show that
    $\hat{A}^2\subset \hat{A}^1$.

 Let $z\in \hat{A}^2$, and consider the point $z_k \defeq  F^{n_k}(z)$. 
   By Corollary~\ref{cor:growthofrealparts}, we have $\max(\re z_k, \re \zeta_k)\to\infty$, 
    and hence $\re z_k \geq Q$ when $k$ is chosen sufficiently large.
   In particular, if $a_k$ is a left-most point of $\hat{A}_k$, then $a_k\notin I_{z_k}$.  
   By Corollary~\ref{cor:movingpoints}, either $b_k\in I_{z_k}$, or 
   the segment $I_{z_k}$ separates $a_k$ from $b_k$. Recall that $I_{z_k}$ does not separate $a_k$ from
   $\zeta_k$ by choice of $Q$; so in the latter case, 
    $I_{z_k}$  also separates $\zeta_k$ from $b_k$. Hence $I_{z_k}\cap \hat{A}_k^1\neq\emptyset$ in either case. 
    By Lemma~\ref{lem:separationoforbits}, it follows that 
    $z\in \hat{A}^1$. 

  This proves that $\zeta_0$ is a terminal point. 
    Furthermore, if $\hat{A}^1\subset\CH$ is a continuum containing both $\zeta_0$ and $\infty$, then we can
    choose $\hat{A}^2 = \CH$ in the above argument, and conclude that $\CH\subset\hat{A}^1$. Thus $\CH$ is indeed irreducible between
    $\zeta_0$ and $\infty$.
\end{proof}

 We next prove the claim in Theorem \ref{thm:nonescapingaccessible} that concerns
    the Hausdorff dimension of the set of non-escaping points in a given Julia continuum, 
    showing that this set is geometrically rather small. (We refer to \cite{falconerfractalgeometry}
    for the definition of Hausdorff dimension.) 

\begin{prop}[Hausdorff dimension of non-escaping points with a given address]\label{prop:nonescapinghausdorff}
  Let $F\in\Blog$ be of disjoint type, and let $\CH$ be a Julia continuum of $F$. Then the 
   Hausdorff dimension of the set of non-escaping points in $\CH$ is zero.
\end{prop}
\begin{proof}
  If $z$ is a non-escaping point, then there is $K>0$ such that $\re F^n(z)\leq K$ infinitely often. So the set of non-escaping points in 
    $\CH=J_{\s}(F)$
    can be written as
  \[ J_{\s}(F)\setminus I(f) = \bigcup_{K=0}^{\infty} \bigcap_{n_0=0}^{\infty} \bigcup_{n= n_0}^{\infty} F_{\s}^{-n}(\{z\in T_n\colon  \re z \leq K\}), \]
   where $\s=T_0 T_1 \dots$ is the address of $\CH$.

 Since a countable union of sets of Hausdorff dimension zero has Hausdorff dimension zero, it is sufficient to fix $K$ and show that the set
    \[ S(K) \defeq  \bigcap_{n_0\geq 0} \bigcup_{n\geq n_0} F_{\s}^{-n}(\{z\in T_n\colon  \re z \leq K\}) \]
   has Hausdorff dimension zero. 

  Let $K>0$. By Observation~\ref{obs:pointstotheleft},    
    there is a number $\Delta=\Delta(K)$ such that, for all $n$, the set of points in $T_n$ with real part $\leq K$ has   diameter 
   at most 
   $\Delta$ in the hyperbolic metric of $H$. 

   Keeping in mind that the map $F\colon T_n\to H$ is a hyperbolic isometry, it follows that 
     \[ \diam_{T_0}\bigl(  F_{\s}^{-n}(\{z\in T_n\colon  \re z \leq K\})  \bigr) \leq \Delta \cdot \Lambda^{-(n-1)} \]
    for $n\geq 1$, where $\Lambda>1$ is the expansion constant from Proposition~\ref{prop:expansion}. 
    By the standard bound~\eqref{eqn:standardestimate2pi}, the Euclidean diameter of this set is hence bounded by 
   $2\pi\cdot \Delta \cdot \Lambda^{-(n-1)}$.

  Let $t>0$. Then for every $n_0\geq 1 $, the $t$-dimensional Hausdorff measure of $S(K)$ is bounded from above by
     \begin{align*} \liminf_{n_0\to\infty}  \sum_{n\geq n_0} \diam(F_{\s}^{-n}(\{z\in T_n\colon  \re z \leq K\}))^t & \leq
         \liminf_{n_0\to\infty} \sum_{n\geq n_0} \bigl(2\pi \cdot \Delta \cdot\Lambda^{(-(n-1))}\bigr)^t  \\ &=
           (2\pi \Delta)^t \cdot \lim_{n_0\to\infty} \sum_{n\geq n_0-1} (\Lambda^{-t})^n = 0 .\end{align*}
    Thus $\dim(S(K))\leq t$. Since $t>0$ was arbitrary, we have 
    $\dim(S(K))=0$, as claimed. 
\end{proof}

 To complete the proof of the first half of Theorem \ref{thm:nonescapingaccessible}, our final topic in this section is the study of points in $J(F)$ that 
   are accessible from $H\setminus J(F)$. 

\begin{thm}[Accessible points]\label{thm:accessibleterminal}
  Let $F\in\BlogP$ be of disjoint type, and let $\CH=C\cup\{\infty\}$ be a Julia continuum of $F$. Suppose that $z_0\in C$ is accessible from
    $\C\setminus J(F)$.

  Then $z_0$ is a terminal point of $\CH$, and $\CH$ is irreducible between $z_0$ and $\infty$. Furthermore, $z_0$ is the unique point of $\CH$ that is 
    accessible from $\C\setminus J(F)$, and $\CH\setminus\{z_0\}\subset I(F)$. 
\end{thm}
\begin{remark}
  In particular, any Julia continuum containing
   more than one non-escaping point (such as those that will be constructed in Section \ref{sec:onepointuniform} to prove
    the final part of Theorem~\ref{thm:nonescapingaccessible}) cannot
   contain any accessible points. 
\end{remark}
\begin{proof}
  Let $\gamma$ be an arc that connects $\partial H$ to $z_0$ without intersecting $J(F)$ in any other points. Then, for every 
   $n\geq 0$, the image $F^n(\gamma)$ contains a piece that connects $F^n(z_0)$ to $\partial H$. 

  Let $a_n$ be a left-most point of $C_n \defeq  F^n(C)$, and let $\gamma_n$ be a piece of $F^n(\gamma)$ that connects 
   $F^n(z_0)$ with a point of real part $\re a_n$, containing no point of real part less than $\re a_n$. 
   Since $\gamma_n$ does not intersect the $2\pi i\Z$-translates of $C_n$, it follows that the closed set
    \[ \tilde{C}_n \defeq  C_n\cup \gamma_n \]
   is disjoint from its own $2\pi i\Z$-translates. We can find a logarithmic tract
   $\tilde{T}_n$ with $\tilde{T}_n \supset \tilde{C}_n$. (First shrink $T_n$ around $C_n$ to avoid the $2\pi i\Z$-translates of $\gamma_n$, 
    then add a small neighbourhood of $\gamma_n$.) The set $\tilde{T}_n$ is not a tract of $F$, but will act as an auxiliary object
    to which we can apply Definition~\ref{defn:separationnumber} and Corollary~\ref{cor:movingpoints}.

 To show that $z_0$ is terminal, suppose that 
   $\hat{A}^1,\hat{A}^2$ are subcontinua of $\CH$ both containing $z_0$, and consider  the sets
   \[ \hat{A}_n^1 \defeq  F^n(\hat{A}^1), \qquad \hat{A}_n^2 \defeq  F^n(\hat{A}^2), \qquad\text{and}\qquad
         \hat{A}_n \defeq  \hat{A}_n^1 \cup \hat{A}_n^2 \cup \gamma_n. \] 
   Let $b_n$ be a right-most point of $\hat{A_n}$; we assume that the sets are labelled such that $b_n\in \hat{A}_n^1\cup\gamma_n$ for infinitely many $n$. 
   As before, Corollary \ref{cor:movingpoints} implies that $I_z\cap (\hat{A}_n^1\cup \gamma_n)\neq \emptyset$ for every $z\in \hat{A}_n^2$. 
   By Lemma~\ref{lem:separationoforbits}, we thus have 
   $z\in  (\hat{A}^1\cup\gamma) \cap \hat{C}=\hat{A}^1$, as required. As in the proof of Theorem~\ref{thm:nonescapingterminal}, if 
   $\hat{A}^1\subset\hat{C}$ is a subcontinuum containing both $z_0$ and $\infty$, then we may take $\hat{A}^2=\hat{C}$, and see that
   $\hat{C}=\hat{A}^1$.

 Now, suppose that $\zeta_0\in C$ is a non-escaping point, 
    say $\re F^{n_k}(\zeta_0)\leq R$ for a strictly increasing sequence $(n_k)$, and assume by contradiction that $\zeta_0\neq z_0$. 
   Set $\zeta_k \defeq  F^{n_k}(\zeta_0)$ and $z_k \defeq  F^{n_k}(z_0)$; then $\re z_k \to \infty$ by Corollary~\ref{cor:growthofrealparts}. 
    In particular~-- for sufficiently large $k$, and hence without loss of generality for all $k$~-- the segment
    $I_{\zeta_k}$ does not separate $z_k$ from $\infty$ in $T_{n_k}$ (recall Observation~\ref{obs:pointstotheleft}). It follows that
    the auxiliary tract $\tilde{T}_{n_k}$ can be chosen such that 
    $I_{\zeta_k}$ also does not separate $z_k$ from $\infty$ in $\tilde{T}_{n_k}$. (We leave the details to the reader.)

    Let $\omega_k$ be a left-most point of $\gamma_{n_k}$ (so $\re \omega_k = \re a_{n_k}$). 
    By Corollary \ref{cor:movingpoints}, either $\omega_k\in I_{\zeta_k}$, or the segment
    $I_{\zeta_k}$ separates $\omega_k$ from infinity, and hence also from $z_k$,  in $\tilde{T}_{n_k}$. 
    Hence 
    $I_{\zeta_k}\cap\gamma_{n_k}\neq\emptyset$ in either case. By Lemma~\ref{lem:separationoforbits}, we have
   $\zeta_0\in\gamma\cap \hat{C}=\{z_0\}$, which is a contradiction. 

  A similar argument shows that   $\CH$ cannot contain two different accessible points. We omit the details since this fact is well-known
    (see e.g.\ \cite[Corollary C' on p.~412]{baranskikarpinskatrees}). Indeed,
    the set $\CH$ is precisely the impression of a unique prime end of $\C\setminus J(F)$, 
    and hence
    contains at most one accessible point. 
\end{proof}

\section{Uniform escape} \label{sec:uniform}

We next discuss the connection between 
  topological properties of Julia continua and  {\escapingcomposant}s, proving Theorem \ref{thm:composants}.
\begin{defn}[{\Escapingcomposant}s for $F\in\Blog$]\label{defn:uniformescape}
  Let $F\in\Blog$ be of disjoint type, and let $\s$ be an admissible external address. 
   If $z\in I_{\s}(F)$, then the \emph{\escapingcomposant} of $z$, denoted 
    $\mu(z) \defeq  \mu_{\s}(z)\defeq \mu_{\s}(F,z)$, is the union of all connected sets $A\subset J(F)$ with $z\in A$ for which $\re F^n|_A$ converges to infinity 
    uniformly. 

  We also define
     \begin{align*}
       \mu_{\s}(\infty) \defeq  \mu_{\s}(F,\infty)\defeq
          \{& z\in J_{\s}(F)\colon \text{there is an unbounded, closed, connected set} \\ 
    &\text{$A\subset J_{\s}(F)$ such that
            $z\in A$ and $\re F^n|_A\to\infty$ uniformly}\}. \end{align*}
\end{defn}
\begin{remark}[Remark 1]
  The set $\mu_{\s}(\infty)$ 
    appears implicitly in \cite[Corollary~3.4]{eremenkoproperty}, which implies that it is always connected as a subset of
    the complex plane. In particular, if $z\in \mu_{\s}(\infty)$, then $\mu_{\s}(z)=\mu_{\s}(\infty)$. 
\end{remark}
\begin{remark}[Remark 2]
  Suppose that $F$ is the logarithmic transform of an entire function $f$. Then, for $z\in\C$, $\exp(\mu_{\s}(F,z))$ agrees precisely with 
   the {\escapingcomposant} $\mu(\exp(z))$ of $f$ as given in Definition~\ref{defn:uniformescapeintro}. 
   The {\escapingcomposant} 
   $\mu(\infty)$ of $f$ is given by $\mu(\infty)=\bigcup_{\s} \exp(\mu_{\s}(F,\infty))$.
\end{remark}

 In \cite[Proposition~3.2]{eremenkoproperty}, an unbounded and connected subset of $I_{\s}(F)$ is constructed whose
  points escape ``as fast as possible'' in a certain sense. This shows that $\mu_{\s}(\infty)$ is non-empty, and suggests the following definition.

\begin{defn}[$\s$-fast escaping points]\label{defn:AsF}
  Let $F\in\Blog$ be of disjoint type, and let 
    $\s = T_0 T_1 \dots$ be an admissible external address. We say that a point $z\in J_{\s}(F)$ belongs to
    the \emph{$\s$-fast escaping set} $A_{\s}(F)$ if there exists an open set $D_0$ intersecting $J_{\s}(F)$ with the following property:
    If we inductively define $D_{n+1} \defeq  F(T_n\cap D_n)$, then $F^n(z)$ belongs to the unbounded connected component of
     $T_n\setminus D_n$ for all $n$. 
\end{defn}
\begin{remark}
  The definition is reminiscent of, and motivated by, the description of the \emph{fast escaping set} $A(f)$ of an entire function that was 
   given by Rippon and Stallard \cite{ripponstallardfatoueremenko}. However, we note that there is no simple relation between the two sets. Indeed,
   it is not only possible that the $\s$-fast escaping set contains points that are not ``fast'' for the global function, but also that some 
   points that are ``fast'' for the global function may not belong to $A_{\s}(F)$. We shall not discuss this relation further here.
\end{remark}

The following proposition is essentially a restatement of \cite[Proposition~3.2]{eremenkoproperty}. However, the terminology there is slightly 
  different, so for completeness and for the reader's convenience, we shall give a self-contained proof here. 
 
\begin{prop}[Existence of $\s$-fast escaping points]\label{prop:fast}
  Let $F\in\Blog$ be of disjoint type, and let $\s$ be an admissible external address. 
     For every $z\in A_{\s}(F)$, there is an unbounded closed connected set $X\subsetneq A_{\s}(F)$ that contains $z$, and on which the iterates
    escape to infinity uniformly. In particular, 
    $A_{\s}(F)\subset \mu_{\s}(F,\infty)$. Furthermore,  $A_{\s}(F)$ is dense in $J_{\s}(F)$.
\end{prop}
\begin{proof}
  Let $\s= T_0 T_1 T_2 \dots$, and let $D_0\subset T_0$ be a disc that intersects $J_{\s}(F)$. We define
    $D_{n+1} \defeq  F(T_n\cap D_n)$ as in Definition \ref{defn:AsF}, and let 
   $X_n=X_n(D_0)$ be the unbounded connected component of $T_n\setminus D_n$. 
      Then $F_{T_n}^{-1}(\cl{X_{n+1}})\subset X_n$ by definition. It follows that
    \[ \hat{X} \defeq  \hat{X}(D_0) \defeq  \bigcap_{n\geq 0} F_{\s}^{-n}(\overline{X_n})  \cup\{\infty\}\]
    is a compact and connected set containing both $z$ and $\infty$. Furthermore, the subset $X=X(D_0)\defeq \hat{X}(D_0)\setminus\{\infty\}$ is contained in
    $A_{\s}(F)$ by definition, and it is connected since $\infty$ is a terminal point of $\CH$. 

  Observe that 
     $\overline{X_n} \cap \partial D_n\neq \emptyset$ for all $n$, and hence $X\cap \partial D_0\neq \emptyset$.
    Since we can apply the construction to any disc $D_0$ around any point $z_0\in J_{\s}(F)$, this implies density of 
    $A_{\s}(F)$ in $J_{\s}(F)$. 
   Since $X$ is a closed proper subset of $J_{\s}(F)$, this also
    proves $X\subsetneq A_{\s}(F)$.

  To show that points in $X$ escape to infinity uniformly, we shall use the following general observation (see \cite[Lemma~3.1]{eremenkoproperty}). 
   \begin{claim}
    For every $R>0$, there is a number $Q>0$ with the following property. If $T$ is a tract of $F$, and $z,w\in T$ with
     $\re w \leq R$ and $\re z \geq Q$, then the vertical line segment $I_z$ of length $4\pi$ centred on $z$ separates $w$ from $\infty$ in $T$. 
   \end{claim}
  \begin{subproof}
   This follows immediately from Observation~\ref{obs:pointstotheleft} and Corollary~\ref{cor:movingpoints}.
  \end{subproof}

   Let $z_0\in D_0\cap J_{\s}(F)$, and set $z_n \defeq  F^n(z)$. Let $\eps>0$ be sufficiently small that
    $D_0$ contains a hyperbolic ball of  radius $\eps$ around $z_0$ in the hyperbolic metric of the range $H$ of $F$. Then
    $\dist_H( X_n , z_n) \geq \eps\cdot \Lambda^n$, where $\Lambda>1$ is the expansion factor from
    Proposition~\ref{prop:expansion}. Now let $R>0$ be arbitrary and let $Q>0$ be as in the claim. For sufficiently large 
    $n$, the hyperbolic ball of radius $\eps\cdot \Lambda^n$ around $z_n$ contains a Euclidean ball of radius $2\pi$ around a point of real part at least $Q$. 
     Hence $\re F^n(z) >R$ for all $z\in X$.
      So we have indeed shown that $X\subset \mu_{\s}(\infty)$,  as claimed. 
\end{proof}

Interestingly, it turns out that we can define $A_{\s}(F)$ purely using the topology of $\Jsh(F)$: 

\begin{prop}[Topological composants and uniform escape]\label{prop:composantsanduniformescape}
 Let $F\in\Blog$ be of disjoint type, and let $\s$ be an admissible external address of $F$. Then $\hat{A}_{\s}(F) \defeq  A_{\s}(F)\cup \{\infty\}$ is the 
   topological composant of $\infty$ in $\Jsh(F)$.
 
 In other words, $\Jsh(F)$ is irreducible between $z\in J_{\s}(F)$ and $\infty$ if and only if $z\notin A_{\s}(F)$. 
\end{prop}
\begin{proof}
  By Proposition \ref{prop:fast}, $\hat{A}_{\s}(F)$ is contained in the topological composant of $\infty$.

  On the other hand, let $K\subsetneq\Jsh(F)$ be a proper subcontinuum containing $\infty$; we must show that $K\subset\hat{A}_{\s}(F)$. 
   Since $\Jsh(F)\setminus K\neq\emptyset$, we see that
   $\hat{A}_{\s}(F)\setminus K\neq\emptyset$ by density of $\hat{A}_{\s}(F)$ (the final statement of Proposition~\ref{prop:fast}).
  Hence, by the first part of Proposition~\ref{prop:fast},
   there is a continuum $A\subset \hat{A}_{\s}(F)$ with $\infty\in A$ and
   $A\not\subset K$. Since $\infty$ is a terminal point of $\Jsh(F)$, we have $K\subset A\subset \hat{A}_{\s}(F)$, as desired. 
\end{proof}

\begin{cor}[Characterisation of decomposability]\label{cor:decomposablejuliacontinua}
  For every admissible address $\s$ of a disjoint-type function $F\in\Blog$, 
   the set $J_{\s}(F)\setminus A_{\s}(F)$ is non\-empty and connected. Moreover, the following are equivalent:
\begin{enumerate}[(a)]
  \item $\Jsh(F)$ is a decomposable continuum;
  \item $J_{\s}(F)\setminus A_{\s}(F)$ is bounded.
\end{enumerate}
\end{cor}
\begin{proof}
  Let us set $\CH \defeq  \Jsh(F)$ and $\hat{A} \defeq  A_{\s}(F)\cup\{\infty\}$; so
   $\hat{A}$ is the topological composant of $\infty$ in $\CH$.

  The set $B \defeq  J_{\s}(F)\setminus A_{\s}(F) = \CH\setminus \hat{A}$
   is non-empty because $\infty$ is a terminal point of 
   $\CH$, and hence a point of irreducibility by Proposition~\ref{prop:composants}~\ref{item:terminal_irreducible}. 
   By Proposition~\ref{prop:composants}~\ref{item:composantcomplement}, $B$ is connected. If $\CH$ is indecomposable, then 
   $B$ is unbounded by Proposition~\ref{prop:composants}~\ref{item:composantindecomposable}.

 On the other hand, suppose that $\CH$ is decomposable, say $\CH=X\cup Y$, where $X$ and $Y$ are proper subcontinua, say with
   $\infty\in X$. Since $\infty$ is a terminal point of $\CH$, we have $\infty\notin Y$, and hence $Y\subset J_{\s}(F)$ is bounded.
   Furthermore, $X\subset \hat{A}$ by the definition of topological composants, and hence $B\subset Y$ is bounded as claimed. 
\end{proof}

In many instances, the following statement will allow us to infer that there exist points in $I_{\s}(F)\setminus \mu_{\s}(F,\infty)$; i.e., escaping points
   that do not satisfy the ``uniform Eremenko property'' (UE) mentioned in the introduction.

\begin{cor}[{\Escapingcomposant}s different from $\mu_{\s}$]\label{cor:uniformeremenko}
   Let $\s$ be an admissible address of a disjoint-type function $F\in\Blog$. 
   \begin{enumerate}[(a)]
     \item Either $\min_{z\in J_{\s}(F)} \re F^n(z) \to\infty$ as $n\to\infty$, and hence $\mu_{\s}(F,\infty)=J_{\s}(F)$, or $\mu_{\s}(F,\infty) = A_{\s}(F)$.
      \label{item:nonuniformescapeandcomposants}
    \item  If $J_{\s}(F)\setminus A_{\s}(F)$ contains more than one point, then $I_{\s}(F)\setminus A_{\s}(F)\neq\emptyset$.
        \label{item:escapingpointsoutsidecomposant}
     \item In particular, if $\#(J_{\s}(F)\setminus A_{\s}(F)) >1$ and $J_{\s}(F)$ does not escape uniformly, then 
         $I_{\s}(F) \setminus \mu_{\s}(F,\infty)\neq\emptyset$.
       \label{item:nonuniformexistence}
   \end{enumerate}
\end{cor}
\begin{proof}
 Recall that always $A_{\s}(F)\subset \mu_{\s}(\infty)$. 
    Suppose that $\re F^n(z)$ does not tend to infinity uniformly on $J_{\s}(F)$, and let $z\in \mu_{\s}(\infty)$. Then there is a closed, unbounded and
    connected set $K\subset J_{\s}(F)$ with $K\ni z$ such that $\re F^n|_K\to\infty$ uniformly. By assumption, we have $K\neq J_{\s}(F)$, so 
    $K\cup\{\infty\}$ is a proper subcontinuum of $\Jsh(F)$.  Thus $z$ belongs to the topological
   composant of $\infty$ in $\Jsh(F)$, and hence to
    $A_{\s}(F)$ by Proposition \ref{prop:composantsanduniformescape}. This proves~\ref{item:nonuniformescapeandcomposants}

  To prove~\ref{item:escapingpointsoutsidecomposant}, recall that the set
    $X \defeq  J_{\s}(F)\setminus A_{\s}(F)$ is connected
      by Corollary \ref{cor:decomposablejuliacontinua}, and that the set of non-escaping points in $J_{\s}(F)$ has Hausdorff dimension zero by Proposition
    \ref{prop:nonescapinghausdorff}. Hence, if $X$ contains more than one point, it must intersect $I(F)$. (Indeed, this intersection has Hausdorff dimension
    at least $1$.) 

   Claim~\ref{item:nonuniformexistence} is an immediate consequence of~\ref{item:nonuniformescapeandcomposants} and~\ref{item:escapingpointsoutsidecomposant}. 
\end{proof}

\begin{proof}[Proof of Theorem \ref{thm:composants}]
  Let $F\in\Blog$ be a disjoint-type logarithmic transform of the disjoint-type entire function $f$, and let $\s$ be an external address such that 
    $\exp(J_{\s}(F))=C$ (where $\CH$ is our original Julia continuum of $f$). Clearly it suffices to prove the claims for $F$ and $\Jsh(F)$. 

 The first statement is a direct consequence of Proposition~\ref{prop:composantsanduniformescape} and
    Corollary~\ref{cor:uniformeremenko}~\ref{item:nonuniformescapeandcomposants}. 

    If $\Jsh(F)$ is an indecomposable continuum, then by Corollary~\ref{cor:decomposablejuliacontinua}, $X\defeq J_{\s}(F)\setminus A_{\s}(F)$ is unbounded,
     and hence $I_{\s}(F)\setminus \mu_{\s}(\infty)\neq\emptyset$ by Corollary~\ref{cor:uniformeremenko}. 

   On the other hand, suppose that $\s$ is a periodic address of period $p$, and that $\Jsh(F)$ is decomposable. 
    Then $X$ is bounded by Corollary~\ref{cor:decomposablejuliacontinua}, and furthermore $F^p(X)=X$. (Indeed, $F(A_{\s}(F))= A_{\sigma(\s)}(F)$ for
     all addresses $\s$, where $\sigma$ is the shift map. In particular, $F^p(A_{\s}(F))=A_{\s}(F)$.)
    Hence all points in $X$ have bounded orbits, and by 
      Proposition~\ref{prop:boundedorbits} and  Corollary~\ref{cor:uniformeremenko}~\ref{item:nonuniformescapeandcomposants}, 
     $X= J_{\s}(F)\setminus \mu_{\s}(\infty)$ consists of a single periodic point. 
\end{proof}

\section{{\Anguine} tracts}\label{sec:arcliketracts}

We now turn to a class of functions for which we can say more about the topology of Julia 
  continua.

\begin{defn}[{\Anguine} tracts]\label{defn:arcliketracts}
   Let $F\colon \T\to H$ be a disjoint-type function in the class $\Blog$. 
   We say that a tract $T$ is \emph{\anguine} if there exists a continuous function $\phi\colon T\to [0,\infty)$ with $\phi(z)\to \infty$ as $z\to \infty$
    and a constant $K>0$ such that 
      \[ \diam_H(\phi^{-1}(t)) \leq K \]
  for all $t$. If all tracts of $F$ are {\anguine} with the same constant $K$, then we say that $F$ \emph{has {\anguine} tracts}.
\end{defn}
\begin{remark}
  Each tract of any iterate of $F$ is contained in a tract of $F$. Hence any iterate of a function with {\anguine} tracts also has {\anguine} tracts.
\end{remark}

 The following two cases of {\anguine} tracts are particularly important.
\begin{defn}[Bounded slope and bounded decorations]\label{defn:boundedslopelog}
  Let $F\colon \T\to H$ be a disjoint-type function in the class $\Blog$.
   We say that $F$ has \emph{bounded slope} if there exists a curve
    $\gamma\colon [0,\infty)\to \T$ and a constant $K$ such that $\re \gamma(t)\to +\infty$ as $t\to\infty$ and
    \[ |\im \gamma(t) | \leq K\cdot \re \gamma(t) \]
   for all $t$. 

  If $H=\HH$, then we say that $F$ has \emph{bounded decorations} if there is a constant
    $K$ such that 
   \begin{equation}\label{eqn:boundeddecorations} 
     \diam_{\HH}( F_T^{-1}(\{\zeta\in \HH\colon  |\zeta|=\rho\}) ) \leq K 
   \end{equation}
   for all sufficiently large $\rho\geq 0$ and all tracts $T$ of $F$. 

  If $H\neq \HH$, then we say that $F$ has bounded decorations if 
       $\mu \circ F$ has bounded decorations, where $\mu\colon H\to \HH$ is a conformal isomorphism that 
     commutes
     with translation by $2\pi i$.
\end{defn}
\begin{remark}[Remark 1]
  Let $f\in\B$, and let $F$ be a logarithmic transform of $f$. Then $f$ 
     has bounded slope, as defined in the Introduction, if and only if $F$ 
    has bounded slope in the sense defined here (and vice versa). Indeed, both conditions are
     equivalent to saying that some tract of $F$~-- and hence all tracts of $F$~-- eventually
    lie within some fixed sector around the positive real axis, of opening angle less than $\pi$.
\end{remark}
\begin{remark}[Remark 2]
  For a fixed tract $T$, it is easy to see that the condition~\eqref{eqn:boundeddecorations} is 
    preserved under post-composition of $F$ with an affine automorphism of $\HH$, allowing for a 
    change of the constant $K$; see also Proposition \ref{prop:decorations} below. Hence it
    makes sense to say that a logarithmic tract $T$ itself has \emph{bounded decorations} 
    if~\eqref{eqn:boundeddecorations} is satisfied for some conformal isomorphism
    $F\colon T\to\HH$ with $F(\infty)=\infty$. If $F\in\BlogP$ has only finitely many tracts, up to
    translations from $2\pi i\Z$, then $F$ has bounded decorations if and only if all tracts of
    $F$ do. 
\end{remark}
\begin{rmk}[Bounded slope and decorations preserved under approximation]\label{rmk:preserved}
  Suppose that the entire function $f$ is obtained from $G\in\BlogP$ via Theorem~\ref{thm:realization}. We claim that this can be done in 
    such a way that, if $G$ has bounded slope,
    so does $f$,  and if $G$ has bounded decorations, then so does any logarithmic transform $F$ of $f$. 

  Indeed, for the first part of Theorem~\ref{thm:realization} (where the function $f$ belongs to the Eremenko--Lyubich class), by 
    \cite{bishopclassBmodels}, the maps $F$ and $G$ are \emph{quasiconformally equivalent} in the 
    sense of \cite{boettcher}, and it follows from \cite[Proof~of~Lemma~2.6]{boettcher} that the  
    bounded-slope and the bounded-decorations conditions are each preserved under quasiconformal equivalence.

  For the second part of Theorem~\ref{thm:realization}, the logarithmic transform $F$ of the function $f$
    as constructed in~\cite{bishopclassSmodels} does not satisfy the bounded-decorations condition in general. Indeed, $F$ has a number of
    ``additional'' tracts (not arising from tracts of $G$), and these  need not have bounded decorations (see \cite[Figure~10]{bishopclassSmodels}). 

 However, $G$ is nonetheless quasiconformally equivalent to a restriction $\tilde{F}$ of 
     $F$ (to the complement of the union of these additional tracts). As above, this shows that $f$ has bounded slope, and that $\tilde{F}$ has 
    bounded decorations. This is in fact sufficient for our purposes. However, it is not difficult to modify the construction from~\cite{bishopclassSmodels} to 
    ensure that the additional tracts also have (uniformly) bounded decorations; we omit the details.
\end{rmk}

 The following observation makes it easy to verify the bounded decorations condition. (It will be used in 
   the second part of the paper.)
 \begin{prop}[Characterisation of bounded decorations]\label{prop:decorations}
  Let $T$ be a logarithmic tract, let $\zeta_0\in\partial T$, and let 
    $\gamma^+$ and $\gamma^-$ be the two connected components of
   $(\partial T) \setminus \{\zeta_0\}$. (Recall that boundaries are always understood as
   subsets of $\C$.)
  The following are equivalent:
    \begin{enumerate}[(a)]
       \item $T$ has bounded decorations.\label{item:boundeddecorations}%
      \item Every sufficiently large point $z\in \overline{T}$ can be connected to both $\gamma^+$ and $\gamma^-$ by 
        an arc in $T$ whose hyperbolic diameter (in $\HH$) is uniformly bounded independently of $z$.
      \label{item:boundedconnections}%
    \end{enumerate}
 \end{prop}
\begin{proof}
   First suppose that $T$ has bounded decorations; then by definition every sufficiently 
      large point of $T$ can be connected to both $F^{-1}\bigl(i[0,\infty)\bigr)$ and $F^{-1}\bigl((-\infty,0]\bigr)$ by 
      a curve of uniformly bounded hyperbolic diameter. (Here $F$ is the conformal isomorphism from the definition of bounded
   decorations.) The difference between the latter two curves and $\gamma^+$ and $\gamma^-$ is bounded,
      so~\ref{item:boundedconnections} holds.  
        
 To prove the converse, we use the following fact of geometric function theory. 
  \begin{claim}
    There exists a universal constant $C>0$ with the following property. Suppose that $T$ is any simply-connected domain and
     $\phi\colon \HH\to T$ is a conformal isomorphism. Let $r>0$ and let $I$ be a connected component of 
      $\R\setminus \{ -r , 0 , r \}$. Then there exists $t \in I$ such that the Euclidean length of
      $\phi(\gamma_t)$ is at most $C\cdot \dist(F^{-1}(r), \partial T)$, where $\gamma_t$ is 
      the hyperbolic geodesic $\gamma$ of $\HH$ connecting $r$ and $it$. 
  \end{claim}
  \begin{subproof}
    This is an immediate consequence of~\cite[Corollary~4.18]{pommerenke} and the fact that $I$ has harmonic measure $1/4$ viewed from $r$. 
  \end{subproof}
  
  Now suppose that~\ref{item:boundedconnections} holds for $T$, and let $F\colon T\to\HH$ be a conformal isomorphism with $F(\infty)=\infty$ and $F(\zeta_0) = 0$. 
    Let $r>0$ and let $\beta$ be the hyperbolic geodesic of $T$ connecting $F^{-1}(r)$ and $F^{-1}(ir)$. 
    We apply the claim to $\phi = F^{-1}$ and the intervals $I_1 = (0,ir)$ and $I_2 = (ir,\infty)$. Thus we obtain 
    geodesics $\alpha_1$ and $\alpha_2$ of $T$, each of Euclidean length at most $C\cdot \pi$, such that 
    $\alpha_1 \cup \alpha_2$ separates all points of $\beta$ from $F^{-1}\bigl(i(-\infty,0)\bigr)$. If $r$ is sufficiently large, then $\alpha\defeq \alpha_1 \cup \alpha_2\subset\HH$ and has
    hyperbolic diameter at most $1$. By assumption, every point of $\beta$ can be connected to $F^{-1}\bigl(i(-\infty,0)\bigr)$, and hence to $\alpha$, by a curve of uniformly bounded hyperbolic
    diameter. Hence the hyperbolic diameter of $\beta$ is bounded independently
     of $r$. 
   The same argument applies to the geodesic connecting $F^{-1}(r)$ and $F^{-1}(-ir)$. We have established~\ref{item:boundeddecorations}, as desired.  
\end{proof}

\begin{obs}[Examples of {\anguine} tracts]
  If the disjoint-type function $F\in\Blog$ has bounded slope or bounded decorations, then $F$ has {\anguine} 
   tracts.
\end{obs}
\begin{proof}
  First suppose that $F\colon  \T\to H$ has bounded decorations. 
   Set $\phi(z) \defeq |\mu(F(z))|$, where $\mu$ is as in the
    definition of bounded decorations (so $\mu=\id$ if $H=\HH$). Then $\phi^{-1}(t)$ has bounded  hyperbolic 
    diameter for all sufficiently large $t$ by  definition,  and for the remaining  $t$ by continuity (recall that $F$
    is of disjoint type). 
 
  In the case of bounded slope, 
   the desired function $\phi$ is given by $\phi(z) \defeq \re z$. Indeed,
    suppose that $\gamma$ is a curve as in the definition of bounded slope. If 
    $T$ is a tract of $F$, then $T$ must tend to infinity between two $2\pi i\Z$-translates of $\gamma$ that lie $4\pi i$ 
    apart. By the bounded slope assumption, the intersection of $\gamma$ (and hence of any of its translates) with the line $L_R\defeq \{\re z = R \}$ is contained in
    an interval of length at most $C\cdot R$ for some $C>0$. Thus~-- again assuming $R$ is sufficiently large, say $R\geq R_0$~--
    the set $\phi^{-1}(R)$, which is the intersection of $T$ with $L_R$, is contained in an interval of length
    $C\cdot R + 4\pi$, and hence has bounded hyperbolic diameter. For $t\leq R_0$, the hyperbolic diameter of $\phi^{-1}(t)$ is bounded
    by~\ref{obs:pointstotheleft}.
\end{proof}

The key reason for the above definitions is given by the following observation, which (together with Theorem~\ref{thm:infty}) completes the proof of the first half of Theorem~\ref{thm:mainarclike}. 

\begin{prop}[{\Anguine} tracts imply arc-like continua]\label{prop:arcliketracts}
  Suppose that $F$ has {\anguine} tracts. Then every Julia continuum of $F$ is arc-like.
\end{prop}
\begin{proof}
 Let $\s=T_0 T_1 T_2\dots$ be the external address of a Julia continuum $\CH$. For each $T_j$,  let $\phi_j\colon T_j\to [0,\infty)$ be a
   function from the definition of {\anguine} tracts. Recall that there is $K>0$
     such that $\diam_H(\phi_j^{-1}(t))\leq K$ for all $j$ and $t$. We define
     \[ g_j \colon  \CH \to [0,\infty];\quad g_j(z) \defeq  \begin{cases} \phi_j(F^j(z)) &\text{if } z\in C \\ \infty &\text{if }z=\infty. \end{cases} \]
  By Proposition \ref{prop:expansion} (hyperbolic expansion), there is $\Lambda>1$ such that 
     \[ \diam_H(g_j^{-1}(t)) = \diam_H(F_{\s}^{-j}(\phi_j^{-1}(t)) \leq \Lambda^{-j} \cdot \diam_H(\phi_j^{-1}(t))\leq \frac{K}{\Lambda^j}\] 
    for all $t\in [0,\infty)$. It follows that $\CH$ is arc-like. 
\end{proof}
\begin{remark}
  Note that, for functions with anguine tracts, many of the results from Section~\ref{sec:spanzero} regarding terminal points
   can be proved much more directly. For example, the maps $g_j$ constructed in the above proof have $\eps(\infty)=\infty$, and it follows
   immediately that $\infty$ is terminal for $\CH$ (compare Proposition~\ref{prop:arclikecharacterization} below). 
\end{remark}

Our final result in this section proves one direction of Theorem \ref{thm:periodicexistence}, concerning the topology of periodic Julia continua. 
\begin{thm}[Invariant continua in {\anguine} tracts]\label{thm:invariantarclike}
  Let $F\in\Blog$ be of disjoint type such that each tract of $F$ is {\anguine}, and suppose that $\CH$ is a periodic Julia continuum of $F$.
   Let $z_0$ be the unique periodic point of $F$ in  $\CH$. Then $\CH$ is a Rogers continuum from $z_0$ to $\infty$.
%
%
\end{thm}
\begin{remark}
 In the  case where $\CH$ is decomposable, this follows from the  main theorem of \cite{rogerscontinua}. We do not require this observation.
\end{remark}
\begin{proof}
  By passing to an iterate, we may assume that $\CH$ is invariant. That is, we are in the situation where 
   $T$ is an {\anguine} logarithmic tract, $F\colon T\to H$ is a conformal isomorphism and
    $\CH$ consists of all points that stay in $T$ under iteration, together with $\infty$.

 Let $\phi$ and $K$ be as in the definition of {\anguine} tracts.  Recall that the hyperbolic diameter of 
   $\phi^{-1}(t)$ is bounded by $K$, independently of $t$. 
    By restricting the function $F$ to a slightly smaller domain, we may assume 
   that $\phi$ extends continuously to the boundary
    of $T$. 
     Increasing $K$ if necessary, we additionally suppose that $\phi(z_0)=0$, and, in particular, that $\phi$ is surjective.

  Let us define a sequence $\zeta_j\in T$ inductively as follows. Let $\zeta_0 \defeq z_0$.
    For $j\geq 0$, let $\zeta_{j+1}\in \cl{T}$ be a point with $\dist_H(\zeta_j,\zeta_{j+1}) = 3K$
    such that $\phi(\zeta_{j+1})> \phi(\zeta_j)$ and such that $\phi(\zeta_{j+1})$  is minimal with this property. To see that such a point exists, note that
   $\phi^{-1}(\phi(\zeta_j))$ is contained in the hyperbolic disc of radius $3K$ around $\zeta_j$. Hence the boundary of the disc must
   contain some points of $\phi^{-1}( (\zeta_j,\infty))$ by continuity and surjectivity of $\phi$, as well as connectedness of $T$. 

  We claim that $x_j \defeq  \phi(\zeta_j)\to\infty$. Indeed, the sequence $x_j$ is increasing by construction. Suppose its limit was a finite
     value $\hat{x}<\infty$. Since $\diam_H(\phi^{-1}(\hat{x})) \leq K$, by continuity of $\phi$ there is a small interval $J$ around $\hat{x}$ such that
      $\diam(\phi^{-1}(J)) < 2K$. But this contradicts the fact that $\phi(\zeta_j)\in J$ for all sufficiently large $j$, and that
     $\dist_H(\zeta_j,\zeta_{j+1}) = 3K$. 

  Postcomposing $\phi$ with a homeomorphism $[0,\infty]\to[0,\infty]$, we may assume for simplicity that
    $x_j = j$ for all $j$.  Observe that, by construction, any point in $\phi^{-1}([j,j+1])$ has hyperbolic
    distance at most $4K$ from $\zeta_j$, and hence  
    \begin{equation}\label{eqn:phiboundeddiameter}
       \diam_H(\phi^{-1}([j,j+1])) \leq 6K \qquad\text{for all $j\geq 0$}. \end{equation}

  For $n\geq 0$, we now define a function $h_n\colon [0,\infty)\to [0,\infty)$ by setting
    \[ h_n(4j) \defeq  \phi(F^{-n}(\zeta_{4j})), \]
   for $j\geq 0$
   and interpolating linearly between these points. Observe that $h_n(0)=0$. 

 \begin{claim}[Claim 1]
  If $n$ is sufficiently large, then 
     $\diam(\phi(F^{-n}(A)))\leq 2$ whenever $A\subset H$ with $\diam_H(A)\leq 32K$. In particular, 
   \begin{enumerate}[(1)]
      \item $\diam(\phi( F^{-n}(\phi^{-1}([4j,4(j+1)]))))\leq 2$ for all $j$;\label{item:diameter} 
     \item $|h_n(x)-h_n(y)|\leq |x-y|/2$ for all $x,y\in [0,\infty)$, and
        \label{item:smalldiameter}
     \item $h_n(x) < x$ for all $x>0$. \label{item:gettingsmaller}
   \end{enumerate}
 \end{claim}
 \begin{subproof}
  Let $n$ be sufficiently large
   to ensure that $\Lambda^{n-1} > 64$, where $\Lambda$ is once more the expansion factor from Proposition 
    \ref{prop:expansion}. 
   Then, in the hyperbolic metric of $T$, the diameter of $F^{-n}(A)$ is less than $K/2$. Let $B$ be an open hyperbolic disc of $T$, of radius $K$
    and centred at a point of $F^{-n}(A)$. Then $\phi(B)\subset [0,\infty)$ is connected. Furthermore, $\phi(B)$ can contain
    at most one integer, since the hyperbolic distance between any point of $\phi^{-1}(m)$ and any point of $\phi^{-1}(m+1)$ is at least $K$, by
     construction. Hence $\phi(F^{-n}(A))\subset \phi(B)$ has diameter at most $2$, as claimed.

  By~\eqref{eqn:phiboundeddiameter}, the hyperbolic diameter (in $H$) of $A_j \defeq  \phi^{-1}([4j,4(j+1)])$ is bounded by
    $32K$ (independently of $j$).   Hence~\ref{item:diameter} follows.
  In particular, we have $|h_n(x) - h_n(y)|\leq 2$ for $x=4j$ and $y=4(j+1)$. This implies 
   that the slope of $h_n$ on each interval of linearity is at most $1/2$, establishing~\ref{item:smalldiameter}. Claim~\ref{item:gettingsmaller} 
   follows from~\ref{item:smalldiameter}, using $y=0$. 
\end{subproof}

  Let us set $h \defeq  h_n$, where $n$ is as in the claim. 
    The construction is carried out so that $\phi$ is a pseudo-conjugacy between the 
     (autonomous) inverse systems generated by 
$F_T^{-n}\colon \overline{T}\to \overline{T}$ 
    (with the  hyperbolic metric on  $H$)  and 
    $h\colon [0,\infty)\to [0,\infty)$ (with the usual metric), in the sense of Section~\ref{sec:conjugacy}. 
  \begin{claim}[Claim 2]
   For all $z\in \overline{T}$, 
    $|h(\phi(z))- \phi(F^{-n}(z))|\leq 4$. 
  \end{claim}
 \begin{subproof}
   Choose $j\geq 0$ such that $\phi(z)\in [4j,4(j+1)]$. Recall from Claim 1 that both 
    $h([4j,4(j+1)])$ and $\phi(F^{-n}(\phi^{-1}([4j,4(j+1)])))$ have diameter at most $2$. Hence both $h(\phi(z))$ and
    $\phi(F^{-n}(z))$ have distance at most $2$ from the point $h(4j)$, and the claim follows. 
 \end{subproof}

 It follows that $\phi$ is indeed a pseudo-conjugacy between the two inverse systems mentioned above. 
   Indeed, Claim 2 is precisely 
   property~\ref{item:pseudoconjugacy} of the 
   definition, while  $\phi$ is  surjective and hence satisfies~\ref{item:pseudosurjectivity}.
   Property~\ref{item:pseudocontinuity} follows from Claim 1, observing that 
    a hyperbolic disc of some fixed radius can be covered by a fixed number  of hyperbolic discs of
    radius $32K$. In the same manner,~\eqref{eqn:phiboundeddiameter} implies~\ref{item:pseudoinjectivity}. 

  By Claim 1, the system
     $(h,[0,\infty))$ is (uniformly) expanding, while the system $(F^{-n},\overline{T})$ is uniformly expanding by
     Proposition~\ref{prop:expansion}. Proposition~\ref{prop:conjugacyprinciple} yields a homeomorphism
      $\theta\colon C\to X\defeq \invlim([0,\infty),h)$ which conjugates $F^{-n}$ and the natural map
      $\tilde{h}\colon X\to X$  induced on the inverse limit. 
      Since $h$ extends continuously to $\infty$ with $h(\infty)=\infty)$, it follows from~\eqref{eqn:conjugacycloseness}
      that $\theta$ extends continuously to a map between the one-point compactifications $\hat{C}$ and
     $\hat{X} \defeq  \invlim([0,\infty],h) = X\cup \{\infty\mapsfrom \infty \mapsfrom\dots\}$.

    The conjugacy  $\theta$ maps $z_0$ to the unique fixed point $0\mapsfrom 0 \mapsfrom \dots$ of $\tilde{h}$ in $X$. (This also follows
     directly from the fact that $\phi(z_0)=0$ and Observation~\ref{obs:convergingtotheconjugacy}.)
 \end{proof}

\section{Homeomorphic subsets of Julia continua}\label{sec:homeomorphicsubsets}

  To conclude this part of the paper, we shall 
   establish that any two bounded-address Julia continua
   of an entire function with a single tract are (ambiently) homeomorphic. More generally, we establish a technical statement 
   about homeomorphisms between subsets of Julia continua that has a number of interesting
   consequences. In particular, Corollary~\ref{cor:pseudoarcs} below will 
   allow us to establish 
   Theorem~\ref{thm:pseudoarcs} in  Section~\ref{sec:periodicexistence}.

\begin{prop}[Julia continua with similar addresses]\label{prop:homeomorphic}
  Let $F\in\BlogP$ be of disjoint type. 
  Let $\s = T_0 T_1 \dots$ be an external address, let $(m_j)_{j\geq 0}$ be a sequence of integers, and 
  consider the address $\s^1 = T_0^1 T_1^1 \dots$ where $T_j^1 = T_j + 2\pi i m_j$.

 Suppose that $\hat{K}\subset \hat{J}_{\s}$ is compact, and that there is 
   $\delta>0$ such that
   \begin{equation}\label{eqn:Kboundeddistance} \delta\cdot |m_j| \leq \max(1,\re F^j(z)) \end{equation}
  for all $z\in \hat{K}\cap\C$ and all $j\geq 0$. 

 Then there exists a compact subset $\hat{A}\subset \hat{J}_{\s^1}$ and a homeomorphism
   $\theta\colon \hat{K}\to \hat{A}$, with the property that
   \begin{equation}\label{eqn:thetadistance}
     d_{H}(F^n(z), F^n(\theta(z))) \leq C \end{equation}
  for all $z\in \hat{K}\cap\C$ and $n\geq 0$. (Here the constant $C$ depends on
  $F$ and $\delta$, but not otherwise on $\hat{K}$, $\s$ and $\s^1$.) In particular, if $\infty\in \hat{K}$, then $\infty\in \hat{A}$ and $\theta(\infty)=\infty$. 
\end{prop}
\begin{remark}
 The reader may wish to keep in mind the
  simplest case, where $F$ has a single tract, $\s$ is a fixed address and
  $\s^1$ is a bounded address, so that the sequence $|m_j|$ is uniformly
  bounded. In this case we can take $\hat{K}= \hat{J}_{\s}$, and it follows
  easily that $A=\hat{J}_{\s^1}$ (see Corollary \ref{cor:homeomorphic} below); so the two Julia continua are
  homeomorphic.
\end{remark}
\begin{proof}
  This is essentially a non-autonomous version of the argument used in \cite{boettcher}
   to construct conjugacies between subsets of Julia sets of different
   functions in the class $\BlogP$. Here we can interpret it as a special case of the conjugacy principle in Proposition~\ref{prop:conjugacyprinciple}.

  Indeed, let us consider two inverse systems $(X_j,f_{j+1})_{j\geq 0}$ and $(Y_j,g_{j+1})_{j\geq 0}$, 
    defined as follows.
    The first system consists of $X_j \defeq  F^j(\hat{K}\cap\C)$, with the hyperbolic metric in the range $H$ of $F$, and $f_{j+1} \defeq  F_{T_{j}}^{-1}$. 
   To define the second system, fix a constant 
    $M>0$, and define $Y_j$ to consist of all points having hyperbolic distance (in $H$) at most $M$ from
    $X_j$. We also define $g_{j+1} \defeq  F^{-1}_{T_{j}^1}\colon H\to T_{j}^1$. As we shall see below, if $M$ is sufficiently large, then
    $g_{j+1}(Y_{j+1})\subset Y_{j}$, so $(Y_j,g_{j+1})$ defines an inverse system (again endowed with the hyperbolic metric on $H$). 
   
 Define maps $\psi_j \colon X_j\to T_j^1$ by
   $\psi_j(z) \defeq  z + 2\pi i m_j$; then $g_{j} = \psi_{j-1}\circ f_{j}$ for $j\geq1$. 
   By~\eqref{eqn:standardestimate} and~\eqref{eqn:Kboundeddistance}, there is a constant $\rho>0$ 
(depending only on $F$ and $\delta$) such that
    \[ \dist_H(z, \psi_j(z)) \leq \rho \]
   for all $j\geq 0$ and $z\in X_n$. Hence
     \begin{equation}\label{eqn:homeomorphic_pseudoconjugacy} 
        \dist_H( \psi_{j-1} ( f_j(z)) , g_j(\psi_{j}(z)) ) = 
       \dist_H( g_j(z) , g_j(\psi_{j}(z))) \leq \rho /\Lambda \end{equation}
    for $j\geq 1$ and $z\in X_{j}$, 
    where $\Lambda$ is the expansion constant of $F$ from Proposition~\ref{prop:expansion}. 

   Suppose now that $M$ was chosen such that  $M \geq \Lambda\rho / (\Lambda-1)$, 
    and let $z\in Y_{j}$, $j\geq 1$. 
    By definition of $Y_{j}$, there is $\zeta\in X_{j}$ with $\dist_H(z,\zeta)\leq M$, and 
     \begin{align*}
      \dist_H(g_j(z) , f_j(\zeta)) &\leq \dist_H( g_j(z), g_j(\zeta)) + \dist_H( g_j(\zeta) , f_j(\zeta)) \\
          &=
                 \dist_H(g_j(z), g_j(\zeta)) + \dist_H( \psi_{j-1}(f_j(\zeta)), f_j(\zeta)) \leq M/\Lambda + \rho \leq M. \end{align*}
     Thus indeed $g_j(z)\in Y_{j-1}$, and $(Y_j,g_{j+1})$ is an inverse system.

   Both systems are expanding by Proposition~\ref{prop:expansion}. The sequence $(\psi_j)$ is a pseudo-conjugacy
    with constant  $(M+\rho)$
 by~\eqref{eqn:homeomorphic_pseudoconjugacy},
    the definition of $Y_j$, and the fact that $\psi_n$ is a hyperbolic isometry. Hence by Proposition~\ref{prop:conjugacyprinciple}, the two
    inverse limits are homeomorphic. Now $\invlim(X_j, f_{j+1})$ is homeomorphic to $X_0=\hat{K}\cap \C$, and 
   $\invlim(Y_j,g_{j+1})$ is homeomorphic to a 
    closed subset $A\subset J_{\s^1}$, in both cases via projection to the first coordinate. Let $\theta$ be the homeomorphism between the two sets; 
    then $\theta$ satisfies~\eqref{eqn:thetadistance} by Proposition~\ref{prop:conjugacyprinciple}. If $\infty\notin \hat{K}$, then
     $\hat{A} \defeq  A$ has the desired properties. Otherwise, 
    $A$ is unbounded by~\eqref{eqn:thetadistance}, and $\theta$ extends to a homeomorphism between the one-point compactifications 
    $\hat{K}$ and $\hat{A}$.  
\end{proof}

\begin{cor}[Homeomorphic Julia continua]\label{cor:homeomorphic}
  With the notation of the preceding Proposition, suppose that
  the sequence $(m_j)$ is bounded. Then $J_{\s}(F)$ and $J_{\s^1}(F)$
  are homeomorphic.
\end{cor}
\begin{proof}
  In this case, we may take $\hat{K}=J_{\s}(F)$ and $Y_j = T_j^1$ in the proof of
  Proposition~\ref{prop:homeomorphic}, and we see that $A=J_{\s^1}(F)$. 
\end{proof}

\begin{cor}[Bounded-address Julia continua are homeomorphic]\label{cor:boundedhomeomorphic}
  Suppose that $F\in\BlogP$ has only a single tract up to translation by
   integer multiples of $2\pi i$. Then any two Julia continua of $F$ at bounded external addresses
   are homeomorphic. 
\end{cor}

For Julia continua at unbounded addresses, the situation is more complicated, as Theorem~\ref{thm:mainarclike} shows. 
  However, the following
  result shows that the topology of bounded-address Julia continua also influences the topology of
   (some) Julia continua at unbounded addresses. 

\begin{thm}[Subsets of bounded-address continua yield Julia continua]\label{thm:subsetsasJuliacontinua}
  Let $F\in\BlogP$ be of disjoint type and has bounded slope. 
     Let $\Jsh$ be a Julia continuum at a bounded external address 
     $\s=T_0 T_1 T_2 \dots$, and let
     $\hat{K}\subset\Jsh$ be a subcontinuum with  $\{\infty\}\subsetneq \hat{K}$. 

    Then there is a Julia continuum $\hat{J}_{\s^1}$ of $F$ homeomorphic to $\hat{K}$. The homeomorphism
      fixes $\infty$, and $\s^1$ is of the form 
        $\s^1 = (T_0 + 2\pi  i m_0) (T_1 +  2\pi i m_1) \dots$
       for some sequence $(m_n)_{n\geq 0}$ of non-negative integers.
\end{thm}

 We shall  use the following fact about preimages of ``initial'' pieces of tracts. 

\begin{prop}[Short preimages]\label{prop:shortpreimages}
  Let $F\in\Blog$ be of disjoint type, say $F\colon \T\to H$. Also let $R_0 > 0$ be such that $[R_0,\infty)\subset H$, and let $\theta>0$.  
    Then there is a constant $C>0$ with the following
    property. If $T$ is a tract of $F$ and $R\geq R_0$, then there is an integer 
      $m\in [R/(2\pi), R/\pi+1 ]$ such that 
      \[ \diam( F_{T}^{-1}( \{ z + 2\pi i m \colon \dist_H(z, [R_0,R]) \leq \theta \})) \leq C. \]
   If $F$ also has bounded decorations,  then there is $C'>0$ such that 
      \[ \diam_H( F_{T}^{-1}( 
      \{ z + 2\pi i m \colon \dist_H(z, [R_0, R_0 + 2\pi |m|] ) \leq \theta \} \leq C'\]
    for all tracts $T$ and all $m\in\Z$, where $H$ is the range of  $F$. 
\end{prop}
\begin{proof}
     Assume for simplicity that $H=\HH$ and $R_0=1$. (Otherwise, replace 
     $F$ by its postcomposition with a conformal homeomorphism $H\to\HH$ in what follows; observe that such a conformal map will not move $R$ by more than
     a bounded hyperbolic distance.)
 
    The harmonic measure of the interval $[Ri,2Ri]$ in $\HH$, as seen from  $R$, is independent of 
     $R$. As $\zeta\defeq F_{T}^{-1}(R)$ has distance at most  $\pi$ from the boundary of $T$, it follows 
       \cite[Corollary~4.18]{pommerenke} that there is $t\in [R,2R]$ such that the geodesic $\gamma$ of $T$
      connecting
      $\zeta$ and $F_{T}^{-1}(ti)$ has Euclidean diameter bounded by a universal  constant  $C_1$.  

      Set $m\defeq \lceil t/2\pi \rceil \in [R/(2\pi),R/\pi+1]\cap \N$. Then the hyperbolic 
       distance of any point in $[1,R]+2\pi i m$ to the circular arc $F(\gamma)$ connecting $ti$ and $R$
       is uniformly bounded (by $3$). So the distance between any point of 
       $S\defeq  \{ z + 2\pi i m \colon \dist_{\HH}(z, [1,R]) \leq \theta \}$ and $F(\gamma)$ is bounded above
      (by $3+\theta$). 
       As $F\colon T\to \HH$ is a conformal isomorphism, the distance between  $F_{T}^{-1}(z)$ and $\gamma$ in the hyperbolic  metric of 
      $T$ is uniformly bounded, and  the  claim  follows from the  standard bound~\eqref{eqn:standardestimate2pi}.  

    If $F$ has bounded decorations, then in the above we  can  take any  $t=2\pi m$, and let $\gamma$ be
     the geodesic of $T$ connecting $F_{T}^{-1}(ti)$ to $F_{T}^{-1}(t)$. This geodesic has bounded  hyperbolic
     length in  $H$ by assumption, and  the remainder of  the proof proceeds as above.
\end{proof}

\begin{proof}[Proof of Theorem~\ref{thm:subsetsasJuliacontinua}]
%
    Let $R_0$ be as in Proposition~\ref{prop:shortpreimages}.
     By the bounded-slope condition, there is 
     $\theta>0$ such that $T_n\subset S\defeq \{z\colon \dist_H(z,[R_0,\infty)\leq \theta\}$ for
     each of the (finitely many) tracts $T_n$. Set $m_0\defeq 0$. 
       For $n\geq 0$, define $\hat{K}_n\defeq  F^n(\hat{K})$ and
      \[ R_n \defeq  \min_{z\in \hat{K}_n} \re  z \]
    for all $n\geq 0$.  Also choose an integer $m_{n+1}$ according to
    Proposition~\ref{prop:shortpreimages}, using $T=T_{n}$  and  $R=\max(1,R_{n+1})$. 
   This choice determines the  address $\s^1$ in the statement of the theorem;
     note that there is $\delta>0$ such that 
    $\delta\cdot m_n \leq \max(1,R_n)$ for all $n$.  

   By Proposition~\ref{prop:homeomorphic}, there is a subset $A\subset \hat{J}_{\s^1}$ homeomorphic to $\hat{K}$, and
     the corresponding homeomorphism moves  points by  at most a finite hyperbolic distance.  
    Moreover, for $n\geq 0$ and $z\in A_{n+1}\defeq  \{z+ 2\pi i m_{n+1}\colon z\in T_{n+1}, \re z \geq R_{n+1}\}$,
     the hyperbolic distance from $z$ to $\hat{K}\setminus \{\infty\}$ is uniformly bounded
      by $2\theta + 2\pi/\delta$.  
     On the other hand,
     \[ \diam( F_{T_n+2\pi i m_n}^{-1}\bigl((T_{n+1}+2\pi i m_{n+1})\setminus A_{n+1}\bigr)) \]
     is uniformly bounded by choice of $m_n$. It follows that 
     \[ F^n(J_{\s^1}) \subset F_{T_n + 2\pi i m_n}^{-1}(T_{n+1}+2\pi i m_{n+1}) \subset  Y_{n}, \] 
     where $Y_n$ is the set constructed in the proof of Proposition~\ref{prop:homeomorphic} (provided $M$ was 
     chosen suffciently large there). Therefore 
    $A=\hat{J}_{\s^1}$, as  required.  
\end{proof}

One may ask whether, for functions having bounded slope and bounded decorations, and  with finitely many tracts,
   the bounded-address Julia continua completely determine the topological types of \emph{all}  Julia continua. 
   We shall prove a weaker result in this direction. 

\begin{prop}[$\eps$-dense subcontinua of Julia continua]\label{prop:epsilondensesubcontinua}
  Let $F\in\BlogP$ be  of disjoint type, with bounded slope and bounded decorations. Let 
     $\s=T_0 T_1 T_2 \dots$ be a bounded  external  address of $F$, and  let 
        $\s^1 = (T_0 + 2\pi  i m_0) (T_1 +  2\pi i m_1) \dots$
       for some sequence $(m_n)_{n\geq 0}$ of integers. 

     Then every proper subcontinuum of $\hat{J}_{\s^1}$ containing $\infty$ is homeomorphic
      to some subcontinuum of $\Jsh$ containing $\infty$.

    More precisely, let $K\subset \hat{J}_{\s^1}$ be a non-degenerate continuum. Then for every  $\eps>0$ 
      there is a subcontinuum $K'\subset \Jsh$ 
      and a continuous injective map 
         \[ \phi\colon K' \to K\] 
      such  that every point of $K$ has Euclidean distance at most $\eps$ from $\phi(K')$.
      (That is, $\phi(K')$ is $\eps$-dense in  $K$.) 
       If $\infty\in K$, then  $\infty\in K'$ and $\phi(\infty)=\infty$. 
\end{prop}
\begin{proof}
   To see that the first part of the proposition follows from the second, let $\tilde{K}\subset \hat{J}_{\s^1}$ be
     a proper subcontinuum containing infinity, and let $z\in \hat{J}_{\s^1}\setminus \tilde{K}$.
     Apply the second part of the Proposition with  $K=\hat{J}_{\s^1}$ and $\eps\defeq \dist(z,\tilde{K})/2$. 
     Then $\phi(K')$ is a subcontinuum of $K$ with $\infty\in  \phi(K')$ and  $\phi(K')\not\subset \tilde{K}$. 
     Since $\infty$ is a terminal point  of $K$, we have $\tilde{K}\subset \phi(K')$, and hence 
     $\phi^{-1}(\tilde{K})$  is indeed a subcontinuum of $\Jsh$ homeomorphic to $\tilde{K}$.  

   Now let $K$ be a non-degenerate subcontinuum of  $\hat{J}_{\s^1}$. The idea of  the proof is as follows. For each $n$, we construct a 
    subcontinuum $K_n$ of $F^n(K)$, containing $\infty$ if $K$ does, whose iterates $F^{j-n}(K_n)$ stay to the right of the line
    $L_j = \re z = 2\pi |m_{j}|$. By Proposition~\ref{prop:homeomorphic}, there is 
    a subcontinuum of $F^n(\Jsh)$ homeomorphic to $K_n$. The bounded decorations condition ensures (via Proposition~\ref{prop:shortpreimages}) that the part of $F^j(K)$ to the left of
   the line $L_j$ shrinks to a set of bounded diameter when pulled back. This and the 
   construction of $K_n$~-- which is similar to the proof of Proposition~\ref{prop:fast}~-- ensures that $K_n$ is $\Delta$-dense in $F^n(K)$, for some $\Delta$ independent of 
   $n$. The claim now follows by pulling back under $F^n$; if $n$ is sufficiently large then the resulting continuum will be $\eps$-dense in $K$.

  To provide the details, let $H$ be the range of $F$. Since $F$ has bounded slope, there  is $C>0$ such that 
   \begin{equation}\label{eqn:boundedatrealparts} \diam_H (\{z\in T\colon  \re z = x \}) \leq C \end{equation}
   for all tracts $T$ and all $x\geq 0$. By Proposition~\ref{prop:shortpreimages}, if $C$ is sufficiently 
    large, then also 
   \begin{equation}\label{eqn:boundedalongpieces} 
          \diam_H(F_T^{-1}(\{z+2\pi i m\colon  z\in T_n\text{ and }\re z \leq 2\pi |m| \}))\leq C \end{equation}
    for all $m\in\Z$, all tracts $T$ and all  $n\geq 0$. 

   Let $\Lambda>1$ be 
    the hyperbolic expansion factor of $F$, and set 
      \[ \Delta \defeq  C\cdot\frac{\Lambda}{\Lambda-1} . \]
    We can
    assume without loss of generality that the set $K$ has hyperbolic diameter
    greater than $C+\Delta$. (Otherwise, replace $K$ by a suitable forward iterate.) 
    By expansion, $F^n(K)$ has hyperbolic diameter greater
    than $C+\Delta$ for all $n\geq 0$. 

  For every choice of $n,j\geq 0$, we inductively define a continuum
     $K_n^j \subset F^n(K)$ as follows. Set 
     $K_{n}^0 \defeq  F^n(K)$. If $K_{n+1}^j$ has been defined, choose a point
     $\zeta_{n+1}^j\in K_{n+1}^j$ with maximal real part, and let 
     $\tilde{K}_{n+1}^j$ be the connected component of
      \[ \{ z\in K_{n+1}^j\colon  \re z \geq 2\pi |m_{n+1}| \} \]
     that contains $\zeta_{n+1}^j$. The proof of the claim below will show that 
     $\re \zeta_{n+1}^j\geq 2\pi |m_{n+1}|)$, so that such a component exists. We set
      \[ K_n^{j+1} \defeq  F_{T_n + 2\pi i m_n}^{-1}(\tilde{K}_{n+1}^j). \]

   \begin{claim}
     Let $n\geq 0$. Then the above construction defines a continuum $K_n^j$ for every $j$, and furthermore
      every point of $F^n(K)$ has
      hyperbolic distance at most $\Delta$ from $K_n^j$.
   \end{claim}
   \begin{subproof}
     We prove the claim by induction on $j$. It is trivial for $j=0$.

     Suppose that the claim is true for $j$ (and all $n$). Since
      $F^{n}(K)$ has diameter greater than $\Delta + C$, for all $n$, it follows from~\eqref{eqn:boundedalongpieces} 
      and the inductive hypothesis that
      $K_{n+1}^j$ contains a point at real part greater than $2\pi |m_{n+1}|$. 
      Hence $\tilde{K}_{n+1}^{j+1}$ is indeed defined for all $n$.

    Let $z\in F^n(K)$. By the  inductive hypothesis, there is $\zeta_1\in K_{n+1}^j$ at distance
     at most $\Delta$ from $F(z)$. By construction and the boundary bumping theorem 
      (Theorem \ref{thm:boundarybumping}), there is a point $\zeta \in \tilde{K}_{n+1}^{j}$ with
      real part $\max(\re \zeta_1, 2\pi |m_{n+1}|)$. If  $\re \zeta_1 \geq 2\pi |m_{n+1}|$, then
          \[ \dist_H(F_{T_n+2\pi i m_n}^{-1}(\zeta), F_{T_n+2\pi i m_n}^{-1}(\zeta_1)) <
               \dist_H(\zeta,\zeta_1) \leq C \]
       by~\eqref{eqn:boundedatrealparts}. Otherwise, 
         \[  \dist_H(F_{T_n+2\pi i m_n}^{-1}(\zeta), F_{T_n+2\pi i m_n}^{-1}(\zeta_1))\leq C \]
    by~\eqref{eqn:boundedalongpieces}. So 
     $\dist_H(z,K_n^{j+1}) \leq C + \Delta/\Lambda = \Delta$, as required.
   \end{subproof}

  The inductive construction ensures that
   $K_n^{j+1}\subset K_n^j$ and $F(K_n^{j+1})\subset K_{n+1}^j$ 
     for all $n$ and $j$. We define
    \[ K_n \defeq  \bigcap_{j\geq 0} K_n^j. \]
   Then, for all $n$, 
    \begin{itemize}
      \item $K_n$ is a subcontinuum of $F^n(K)$;
      \item $\dist_H(F^n(z), K_n)\leq \Delta$ for all $z\in K$; in particular, $K_n$ is non-degenerate; 
      \item $F(K_n)\subset K_{n+1}$;
      \item $\re F(z)\geq 2\pi |m_{n+1}|$ for all $z\in K_n$.
    \end{itemize}

   By Proposition~\ref{prop:homeomorphic} (applied with the roles of $\s$ and $\s^1$ reversed), there is a subcontinuum
   $K_{n+1}'$ of $F^{n+1}(\Jsh)$ and a homeomorphism $\theta_{n+1}\colon F(K_n)\to K_{n+1}'$ with
    $\theta_{n+1}(\infty)=\infty$ if $\infty\in K$. Letting $n$ be sufficiently large, set 
     $K'\defeq F_{\s}^{-(n+1)}(K_{n+1}')$ and 
     \[ \phi\colon K'\to K; \qquad \phi(z) = F_{\s^1}^{-(n+1)}(\theta_{n+1}^{-1}(F^{n+1}(z))). \]
    Then $\phi$ is continuous and injective. By expansion, every point of $F(K)$ has distance at most
       $\Delta/\Lambda^{n-1}$ from $F(\phi(K'))$ in the hyperbolic distance of $H$; so 
      every point of $K$ has distance at most $\Delta/\Lambda^{n-1}$ from $\phi(K')$ in the hyperbolic metric of
        $T_0+2\pi  i  m_0$. The claim follows by the standard bound~\eqref{eqn:standardestimate2pi},
       if $n$ was chosen sufficiently large to ensure that
        $\Delta/\Lambda^{n-1} \leq \eps/(2\pi)$.  
\end{proof}

The following will be used to prove Theorem~\ref{thm:pseudoarcs}. 

\begin{cor}[Pseudo-arcs]\label{cor:pseudoarcs}
   Suppose that $F\in\BlogP$ is a disjoint-type function having a unique tract $T$ 
    up to translation by $2\pi i$, and such that
    $T$ has bounded slope and bounded decorations.   

    If one (and hence each) bounded-address Julia continuum of $F$ is a pseudo-arc, then 
    every Julia continuum of $F$ is a pseudo-arc.
\end{cor}
\begin{proof}
  Let $\CH$ be the 
     invariant Julia continuum of $F$ contained in $T$. Then, by assumption and 
     Corollary~\ref{cor:boundedhomeomorphic}, $\CH$ is a pseudo-arc. 
 Since all Julia continua of $F$ are  arc-like by Proposition \ref{prop:arcliketracts},  it suffices to  show that they
   are hereditarily indecomposable. 

    So let $\Jsh(F)$ be a Julia continuum, and let $K\subset \Jsh(F)$ be a subcontinuum. 
    Let $(m_n)_{n\geq 0}$ be the sequence of integers
    such that the $n$-th entry of $\s$ is given by $T+2\pi  i  m_n$. 
    By Proposition~\ref{prop:epsilondensesubcontinua}, for every $\eps>0$, there is a subcontinuum $K'$ of $\hat{C}$ and 
    a continuous injection $\phi\colon K'\to K$ such that the subcontinuum $\phi(K')$ is $\eps$-dense in $K$. 
   As $\CH$ is hereditarily indecomposable,
     $K'$, and hence $\phi(K')$, is indecomposable.  

   \begin{claim} Let  $K$ be a continuum, and suppose that $K$ is hereditarily unicoherent. (That is, if $A,B\subset K$ are subcontinua, then $A\cap B$ is connected.)
     If $K$ has $\eps$-dense
     indecomposable subcontinua for all $\eps>0$, then  $K$ is indecomposable.
   \end{claim}
   \begin{subproof} 
     We prove the contrapositive, so let 
      $K=A\cup  B$ with $A$, $B$ proper subcontinua. If $\eps$ is sufficiently small, then any
      $\eps$-dense subcontinuum $K_{\eps}$ satisfies   
       $K_{\eps}\not\subset A$ and  $K_{\eps}\not\subset B$. By assumption, 
       $K_{\eps}\cap  A$ and  $K_{\eps}\cap B$ are continua; hence $K_{\eps}$ is decomposable,  as required.
   \end{subproof} 

    No subcontinuum of $K$ separates the plane, so $K$ is hereditarily unicoherent. 
      (This also follows from Theorem~\ref{thm:spanzero}, as every span zero continuum is hereditarily unicoherent.) 
    So $K$ is indecomposable by the claim,  and the proof is complete. 
\end{proof}

   Theorem~\ref{thm:mainarclike} asserts the existence of a disjoint-type entire function $f$ having bounded slope
    such that every arc-like continuum having a terminal point is realised as a Julia continuum of $f$.
    We shall  see 
    in Remark~\ref{rmk:alternativeproof}  that
     $f$ can be  chosen  either to  have a single tract, or to have  two  tracts with
     bounded decorations.  
   The following consequence of Proposition~\ref{prop:epsilondensesubcontinua} shows that this is best possible,
     in that 
    $f$ cannot have
    only a single tract which also has bounded decorations.

\begin{cor}[Julia continua requiring unbounded decorations]\label{cor:unboundedrequired}
  There exists an arc-like continuum $X$ having 
    a terminal point with the following property. Suppose $F\in\BlogP$ has disjoint type, with a single tract up to 
     translations
    from $2\pi i \Z$, and having bounded slope and bounded decorations. Then no Julia  continuum of $F$ is 
    homeomorphic to $X$.
\end{cor}
\begin{proof}
  Let $F$ be a function with the stated properties, let $\CH$ be an invariant Julia continuum of $F$, and
     let $\Jsh$ be any Julia  continuum of $F$. Then by the final part of Proposition~\ref{prop:epsilondensesubcontinua},
     every proper subcontinuum $K\subset\Jsh$ with  $\infty\in  K$ is homeomorphic to  a corresponding proper
     subcontinuum
      $K'$ of $\CH$ with $\infty\in K'$. 

    However, there are significant restrictions on the possible topology of  $K'$. E.g., 
     $F^n|_{K'}\to\infty$ uniformly on $K'$ by Proposition~\ref{prop:boundedorbits}. Since $\infty$ is a terminal point,  we have
     $F(K')\subsetneq K'$. But
     there exists an arc-like continuum $X$ such that no two non-degenerate subcontinua of $X$ are 
     homeomorphic to each other \cite{andrewschainable}, and one can see from the proof that this continuum
     has at least one terminal point. So $K'$ cannot be homeomorphic to $X$. On the other hand,  
    $X$ can easily be realised as a proper subset
     of another arc-like continuum having no terminal points except in $X$; this provides the desired example.
 
   For a more elementary construction, observe that $K'$ cannot be homeomorphic to the 
     $\sin(1/x)$-continuum $S$, as any continuous injection $S\to S$ maps the limiting interval bijectively to
     itself. So let $X$ be the union of $S$ with its reflection in 
    its limiting interval, and assume by contradiction  that $\Jsh$ is homeomorphic to $X$. The continuum $X$
    has exactly two terminal points; one of these corresponds to $\infty$ in $\Jsh$. Then there is a 
    proper subcontinuum $K\ni \infty$ homeomorphic to  $S$, and a corresponding subcontinuum $K'\subset \CH$.
    This is a contradiction.
\end{proof}

For completeness, we also make an observation concerning the embedding of the Julia continua considered
    in  Proposition~\ref{prop:homeomorphic}. (This result will not be used
   in the following.) 

\begin{prop}[Ambient homeomorphism]
 The sets $K$ and $A$ in Proposition \ref{prop:homeomorphic} are
  \emph{ambiently} homeomorphic.
   More precisely, the map $\theta$ extends to
  a quasiconformal homeomorphism of $\C$.

 Moreover, as $\s^1\to \s$ (for fixed $\delta$), the maps
  $\theta=\theta_{\s^1}$ converge uniformly to the identity. 
\end{prop}
\begin{remark}
 Here convergence of addresses should be understood in the product topology, where we use the discrete metric on the alphabet. In other words, 
   $\s^1$ is close to $\s$ if the two addresses agree in a long initial sequence of entries. 
\end{remark}
\begin{proof}
   This follows from the $\lambda$-lemma of Bers and Royden \cite[Theorem~1]{bersroyden}. Indeed, we can embed the two inverse systems
    from the proof of Proposition~\ref{prop:homeomorphic} into a holomorphic family 
    $(Y_j, f_{j+1}^{\mu})$, where $Y_j$ is defined as before, $\mu$ belongs to a suitable simply-connected neighbourhood $U$ of both $0$ and $1$, and 
     \[ f_{j+1}^{\mu}(z) \defeq F_{T_j^1}^{-1}(z) + 2\pi i m_n\cdot \mu. \]

    If we assume for simplicity that the range $H$ of $F$ is the right half-plane $\HH$, then we can let $U$ be a sufficiently small neighbourhood of the
     segment $[0,1]$. For each $\mu\in U$, we find a homeomorphism $\theta^{\mu}\colon  \hat{K}\to \hat{A}^{\mu}$, where
     $\theta^{0}=\id$ and $\theta^1=\theta$. Furthermore, the maps $\theta^{\mu}$ are the locally uniform limit of 
     maps that depend holomorphically on $\mu$ (Observation \ref{obs:convergingtotheconjugacy}). Hence these maps are themselves
     holomorphic in $\mu$, and thus define a \emph{holomorphic motion} of the set $\hat{K}$. By the Bers-Royden $\lambda$-lemma, 
     each $\theta^{\mu}$ extends to a quasiconformal homeomorphism. 

 The final claim follows by applying Proposition \ref{prop:homeomorphic} to the addresses $\sigma^n(\s^1)$ and $\sigma^n(\s)$, 
    using~\ref{eqn:thetadistance} and
    the fact that $F$ is expanding. 
\end{proof}

\section{Background on arc-like continua}
\label{sec:topology}
  In this second part of the memoir, we are going to discuss the construction of entire function with prescribed arc-like continua
   in the Julia set. We shall need to collect some further background on arc-like continua (also 
   referred  to as \emph{snake-like} continua, following Bing~\cite{bingsnakelike}). 
   Let us begin by recalling their definition and introducing some additional terminology. 

 \begin{defn}[$\eps$-maps]
 An \emph{$\eps$-map} from a metric space $A$ to a topological space $B$ is a continuous function $g\colon A\to B$ such that $g^{-1}(x)$ has diameter
   less than $\eps$ for every $x\in B$.

 Recall that a continuum $X$ is \emph{arc-like} if, for every $\eps>0$, there exists an $\eps$-map $g$ from $X$ 
   onto an arc.
\end{defn}

 There are a number of equivalent definitions of arc-like continua. The key one for our construction is the following,
   in terms of inverse limits. 

\begin{prop}[Characterisation of arc-like continua with terminal points] \label{prop:arclikecharacterization}
 Let $X$ be a continuum, and let $p\in X$. The following are equivalent.
   \begin{enumerate}[(a)]
      \item $X$ is arc-like and $p$ is terminal; \label{item:arclike}
      \item for every $\eps>0$, there is a surjective $\eps$-map $g\colon X\to [0,1]$ with $g(p)=1$; \label{item:epsmap}
      \item there is a sequence 
        $g_j\colon [0,1]\to [0,1]$ of continuous and surjective functions    
         with $g_j(1)=1$ for all $j$ such that there is a homeomorphism from
        $X$ to $\invlim\bigl((g_j)_{j=1}^{\infty})\bigr)$ that
        maps $p$ to the point $(1\mapsfrom 1 \mapsfrom 1 \mapsfrom \dots)$. \label{item:inverselimit}
   \end{enumerate}
  If any (and hence all) of these properties hold, and $q$ is a second terminal point such that $X$ is irreducible between $p$ and $q$, then
   the maps $g_j$ can be additionally chosen to fix $0$, with the point $q$ corresponding to the point $0\mapsto 0 \mapsto \dots$. Similarly,
   the $\eps$-map $g$ in~\ref{item:epsmap} can be chosen such that $g(q)=0$. 
\end{prop}
\begin{remark}
  An additional equivalent
   formulation, which  is very intuitive,  is in terms of \emph{chainability}:
    for every $\eps>0$, there is an $\eps$-chain in $X$ that covers $X$ and 
    such that $p$ belongs to the final link of this chain.
    That is, there is a finite sequence
     $U_1,\dots , U_n$ of non-empty open subsets (``links'') of $X$ whose union equals $X$, where
     two links $U_j$ and $U_k$ intersect if and only if $|j-k|=1$, and such that $p\in U_n$ and 
     $\diam(U_j)<\eps$ for all $j$.
     We shall not use chainability; see \cite[Definition 12.8]{continuumtheory} for details. 
\end{remark}
 \begin{proof}
   This result is well-known, but we are  not aware of a reference to  it that includes the 
   statements concerning the terminal points $p$ and   $q$. Without these, 
   the equivalence is proved in \cite[Theorem 12.19]{continuumtheory}. For completeness,
   let us sketch how one may modify that proof to obtain the above version, referring to \cite{continuumtheory} and
   \cite{bingsnakelike} where necessary. 

  First observe that \ref{item:inverselimit} clearly implies
    \ref{item:epsmap}, as we can let $g$ be the projection to the $j$-th coordinate,
    for $j$ sufficiently large.

  Conversely, it follows from the proof of 
    Theorem 12.19 in \cite{continuumtheory} that 
    \ref{item:epsmap} implies \ref{item:inverselimit}. 
    Indeed, that proof 
   constructs a suitable inverse limit, and an inspection of the proof
   of Lemma 12.17 immediately shows that the map $\phi$ defined 
   there, which is used in the construction of the inverse limit,
   satisfies $\phi(1)=1$. (Indeed, with the notation of that proof, we have 
   $t_n = 1$ and $\phi(1)=\phi(t_n)=s_{i(n)}=i(n)/m$.
   Here the definition
   of $i(n)$ ensures that $i(n)=m$, provided that $g_1(p)=g_2(p)=1$.)

  To see that \ref{item:epsmap} implies \ref{item:arclike},
    let us prove the contrapositive. 
    So suppose that $p$ is not terminal; then there are continua $A,B\subset X$ with $p\in A\cap B$, but
    $A\not\subset B$ and $B\not\subset A$. Let $\eps>0$ be so 
    small that some point of $A$ has distance greater than 
    $\eps$ from $B$, and vice versa.
    If $g$ is an $\eps$-map, 
    $g(A)\not\subset g(B)$ and $g(B)\not\subset g(A)$.
    Since $g(A)$ and $g(B)$ are closed subintervals of $[0,1]$, it follows
    that $g(p)\neq 1$, as required. 

  Finally, Bing \cite[Theorem 13]{bingsnakelike} showed that 
    \ref{item:arclike} is equivalent to the statement on chainable continua
    mentioned in the remark after the statement of the theorem. This,
    in turn, is easily seen to imply \ref{item:epsmap}. 

  The final part of the proposition follows analogously.
 \end{proof}

Proposition~\ref{prop:arclikecharacterization} allows us to express an arbitrary arc-like  continuum $Y$
   (having a terminal point) as an inverse limit as in~\ref{item:inverselimit} above.  We then construct
   a function $F\in\Blog$ whose behaviour (more precisely, the behaviour of suitable inverse branches of   
    iterates of  $F$) is designed to mimic that of the  functions  $g_j$. We then use 
    the conjugacy principle from  Proposition~\ref{prop:conjugacyprinciple} to show that $F$ does indeed  have a 
     Julia continuum  
    homeomorphic to $Y$. 
  Moreover, the following observation will allow us to realise \emph{all} arc-like continua with terminal points at once as Julia continua of a \emph{single}
  entire function, as claimed  in  Theorem~\ref{thm:mainarclike}. 

\begin{prop}[Countable generating set]\label{prop:countable}
  There exists a universal 
    countable set $\G$ of surjective continuous functions $g\colon [0,1]\to[0,1]$ with $g(1)=1$ such that
    all maps in Proposition \ref{prop:arclikecharacterization}~\ref{item:inverselimit} can be chosen to belong to $\G$.
    
 Similarly, if $X$ is an arc-like continuum having two terminal points between which $X$ is irreducible, then all maps $g_j$ in the 
    final part of Proposition~\ref{prop:arclikecharacterization} (which are required to fix $0$) can be chosen to belong to $\G$. 
\end{prop}
\begin{proof}
 For general arc-like continua, this is stated in \cite{continuumtheory} (without the assumption that $g(1)=1$).
  The proof in our case is entirely analogous; the set $\G$ consists of all  piecewise linear functions 
   $g\colon[0,1]\to[0,1]$ with $g(1)=1$, where all points of non-linearity are assumed  to be rational,
    and the function takes rational values at rational points. Clearly this set is countable.

  Let $(g_j)_{j\geq 1}$ be any sequence as in
    \ref{prop:arclikecharacterization}~\ref{item:inverselimit}, and consider constants $\gamma_j\geq 1$ ($j\geq 0$)
   as in
   Observation~\ref{obs:obtainingexpandingsystems}, so that the system $(g_j)$ is expanding when
   the $j$-th copy of the interval is equipped with the metric $d_j$ obtained by scaling the Euclidean metric
   by $\gamma_j$. We may assume that all $\gamma_j$ are integers.

   Clearly if we  approximate each $g_j$ closely  enough by a function $\tilde{g}_j\in \G$, then 
     the system $(\tilde{g}_j)$ is also expanding (using the same  metrics $d_j$, although possibly with slightly
     different expansion constants), and the two
    systems will  be pseudo-conjugate in the sense of Definition~\ref{defn:pseudoconjugacy}, with all
     maps  $\psi_j$ given by the identity.  The claim hence follows by Proposition~\ref{prop:conjugacyprinciple}.

  If all maps $g_j$ fix $0$  (as in  the final statement of Proposition~\ref{prop:arclikecharacterization}, then 
    the functions $\tilde{g}_j$ can also be chosen to do so. 
\end{proof}

\section{All arc-like continua with terminal points are Julia continua}
 \label{sec:arclikeexistence}

 The goal of this section is to prove a slightly weakened version of the main part of Theorem \ref{thm:mainarclike}. 

\begin{thm}[Realisation of arc-like continua] \label{thm:arclikeexistenceBlog}
   Let $Y$ be an arc-like continuum containing a terminal point $y_1$.
   Then there exist a disjoint-type function $F\in\BlogP$ with bounded slope and a Julia continuum
    $\Jsh(F)$  that is homeomorphic to $Y$, with $\infty$ corresponding to $y_1$. 

  Moreover, $F$ and $\s$ can be chosen either such that $\re F^k|_{\Js(F)}\to\infty$ uniformly as $k\to\infty$,
    or such that $\liminf_{k\to\infty}\min_{z\in\Js(F)} \re F^k(z) <\infty$. 
\end{thm}
\begin{remark}
  By Theorem \ref{thm:realization}, 
   the existence of such a function in $\BlogP$ automatically yields a function $f\in\B$, and
   even a function $f\in\classS$, having a Julia
    continuum homeomorphic to $Y$. 
\end{remark}

 The idea of the proof of Theorem~\ref{thm:arclikeexistenceBlog}
 is to express $Y$ as the inverse limit of a sequence $(g_k)_{k=1}^{\infty}$ of self-maps of the unit interval
   according to Proposition~\ref{prop:arclikecharacterization}, and to
   construct $F$ in such a way that suitable branches of iterates of $F$ mimic the behaviour of $g_k$. We obtain a homeomorphism between the resulting
   Julia continuum and $Y$ by appealing to the pseudo-conjugacy principle from Proposition~\ref{prop:conjugacyprinciple}. This construction
   provides some additional information concerning the relation between the two systems, which we record for further applications.

\begin{prop}[Properties of  the construction] \label{prop:arclikeexistenceprecise}
    Let $(g_k)_{k=1}^{\infty}$ be a sequence of continuous and surjective functions $g_k\colon[0,1]\to[0,1]$ with $g_k(1)=1$. Set
     $\uarrow{Y}\defeq \invlim(g_k)_{k=1}^{\infty}$ and $y_1\defeq 1\mapsfrom 1\mapsfrom \dots$. Also fix 
       any sequence $(M_k)_{k\geq 0}$  with  $M_k\geq 10$ for  all  $k$.

   Then the function $F$, the address $\s$, and the homeomorphism 
     $h\colon \Jsh(F)\to Y \defeq \uarrow{Y}$ whose existence is asserted in Theorem~\ref{thm:arclikeexistenceBlog} 
    can be  constructed so that there is a strictly increasing sequence
     $(n_k)_{k\geq 0}$ with the following properties for all $z\in \Js(F)$. 
    \begin{enumerate}[(a)]
      \item $\re F^n(z) \geq \re F^{n_k}(z) \geq M_k-1$ for $n_k\leq n < n_{k+1}$.
        (In particular, $z\in I(F)$ if and only if 
                 $\lim_{k\to\infty} \re F^{n_k}(z) = \infty$.)\label{item:largerthanMk}
   \item\label{item:controlonbehaviour}%
     Let $h_k\colon \Jsh\to [0,1]$ denote the $k$-th component of $h$. Then
     \[ h_k(z) = 0 \quad \Longrightarrow\quad \re F^{n_k}(z)\leq M_k + 1\] 
    and
      \[ \liminf_{k\to\infty} h_k(z) >0 \quad\Longrightarrow\quad \re F^{n_k}(z) - M_k \to \infty. \]
   \end{enumerate}
\end{prop}
\begin{remark} 
   This implies the final statement of Theorem~\ref{thm:arclikeexistenceBlog}:
      the iterates of $F$ will tend to infinity uniformly  on $J_{\s}(F)$ if and  only  if 
      $M_k\to\infty$. 
      The proposition also allows us to construct examples where $M_k=10$ for all $k$, yet 
      the set of non-escaping points is empty (leading to a proof of Theorem~\ref{thm:nonuniform}), and examples where the set of non-escaping points is uncountable,
      simply by specifying a suitable inverse limit construction. 
\end{remark}

   To prove the theorem, we construct a suitable simply-connected domain $T$ with $\overline{T}\subset\HH$ which does not
     intersect its $2\pi i\Z$-translates and a 
     conformal isomorphism $F\colon T\to \HH$ such that $F(z)\to \infty$ in $\HH$ only if $z\to\infty$ in $T$. 
     For simplicity, we do not require here that $T$ is a Jordan domain.  
    We observe that $T$ and $F$ naturally give rise to a disjoint-type function as follows. 

 \begin{rmk}[Obtaining a function in $\BlogP$]\label{rmk:obtainingBlog}
   With $T$ and $F$ as above, set $H\defeq \{\re z > \eps\}$, where 
    $\eps>0$ is small enough that $\overline{T}\subset H$,  and consider $\tilde{T}\defeq F^{-1}(H)$. 
    Then the $2\pi i$-periodic extension $\tilde{F}$ of $F|_{\overline{T}}$ is a disjoint-type function  in 
    $\BlogP$  (with range $H$), having a single tract up to translation by $2\pi i\Z$.

   Observe that the tracts of $\tilde{F}$ are all of the form
     $\tilde{T}+2\pi i  m$, where $m\in\Z$. We can hence simplify notation by 
     identifying an  external address 
     \[ \s = (\tilde{T}+ 2\pi i s_0) (\tilde{T} + 2\pi is_1) (\tilde{T}+2\pi i s_2) \dots\] with the sequence
     $s_0 s_1 s_2 \dots$ of integers.
 \end{rmk}

Let us fix the maps $(g_k)_{k\geq 1}$ and the sequence $(M_k)_{k\geq 0}$ from Proposition~\ref{prop:arclikeexistenceprecise} from 
 now on.
We devote the remainder of the section to constructing $T$ and $F$ such that 
   $\tilde{F}$ as in Remark~\ref{rmk:obtainingBlog} satisfies the conclusions of the proposition.

\subsection*{Definition of the tract $T$}
   Our tract $T$ is a subset of the horizontal half-strip  $\{x+iy \colon  x>4, |y|<\pi\}$; 
  in particular, the function $\tilde{F}$ has bounded slope. 
  The tract consists of a central straight half-strip, to which a number of ``side channels'', 
   domains $U_k$, 
   are attached that mimic the behaviour of the maps
   $g_k$; see Figure \ref{fig:tract_definition}. For technical reasons, it will also be convenient to narrow the 
    central strip to a small window, of size $\chi_k>0$, 
   just before the place where $U_k$ is attached.

\begin{figure}
 \begin{center}
   \def\svgwidth{\textwidth}
   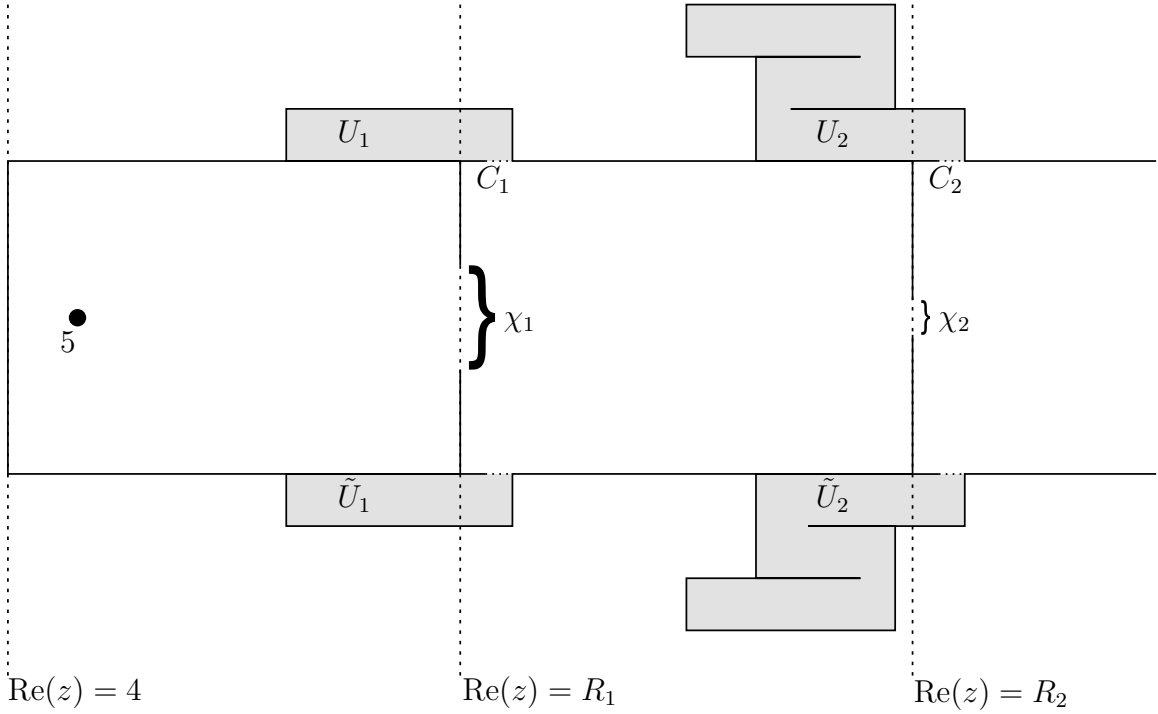
 \end{center}
 \caption{The tract $T$.\label{fig:tract_definition}}
\end{figure}

  More precisely, define
    \[ S \defeq  \{ x + iy\colon  x>4 \text{ and } |y|<\pi/2 \}. \] 
   The tract $T$ depends on sequences
     $(R_j)_{j\geq 1}$, $(\chi_j)_{j\geq 1}$, $(U_j)_{j\geq 1}$ and  $(C_j)_{j\geq 1}$. Here
      $R_{j+1} - 1 > R_j > 5$ and $0<\chi_j\leq\pi$. The $U_j$ are simply-connected
     domains such that 
     \begin{equation}\label{eqn:Ujposition} U_j \subset \{ x + iy \colon  R_{j-1} + 1 < x < R_j+1 \text{ and } \pi > y > \pi/2 \} \end{equation}
    (we use the convention that $R_{0}=4$) and each $C_j$ is an associated  arc 
     \[ C_j \subset \{  x + i\pi/2 \colon  x\in [R_j,R_j+1] \} \cap \partial U_j. \]
  The tract $T$ is defined as
     \begin{align*}
      T \defeq  \left(S \setminus \bigcup_{j\geq 1} \{ R_j + iy\colon  |y|\geq \chi_j/2\}\right) 
        \cup \bigcup_{j\geq 1} 
        \left(U_j \cup \interior(C_j) \cup  \tilde{U}_j \cup \interior(\tilde{C}_j) \right). \end{align*}
   Here $\tilde{U}_j = \{x+iy\colon  x-iy\in U_j\}$ consists of all complex conjugates of points in $U_j$, 
    (likewise for $C_j$), and $\interior(C_j)$ refers to the arc without its endpoints.

  The conformal isomorphism
    $F\colon T\to \HH$
  is determined uniquely by requiring that $F(5)=5$ and $F'(5)>0$. Since
  $T$ is symmetric with respect to the real axis, this implies that 
   $F([5,\infty))=[5,\infty)$. By the expanding property of $F$,
  we have $F(t)>t$ for $t>5$. 

  This completes the description of $T$ and $F$, depending on  the sequences
      $(R_j)$, $(\chi_j)$,  $(U_j)$ and $(C_j)$, which are chosen as part of a recursive
   construction below.
     To carry out and analyse this construction, we need to estimate $F$ on parts of the tract  
     while only having made the first $k$ such choices,  for $k\geq 0$. To this end,  we define 
     the ``partial tract''
     \begin{equation}\label{eqn:partialtract}
       T_k \defeq  \left(S \setminus \bigcup_{j=1}^{k} \{ R_j + iy\colon  |y|\geq \chi_j/2\}\right) 
             \cup \bigcup_{j= 1}^k
          \left(U_j \cup \interior(C_j) \cup  \tilde{U}_j \cup \interior(\tilde{C}_j) \right), \end{equation}
  and let $F_k\colon T_k\to\HH$ be the conformal isomorphism with $F_k(5)=5$ and $F_k'(5)>0$. Note that $T_0=S$. 
  The following observation shows that $F_k$ closely approximates the
  map $F$, provided that
   $R_{k+1}$ is chosen  sufficiently large  (independently of any further choices). 
\begin{obs}[Continuity of the construction]\label{obs:caratheodoryconvergence}
  Let $k\geq 0$, and let $T_k$ be a partial tract as in~\eqref{eqn:partialtract} above. (In other  words,
    fix $R_j$,  $\chi_j$,  $U_j$ and $C_j$ for $1\leq j \leq  k$.) 
  Then for all $N,S\in\N_0$, all $\eps>0$ and  all $M>1$, there is a number $R(T_k,N,\eps,S,M)>0$ with the 
   following property. 

  Suppose that $R_{k+1}$ is chosen such that $R_{k+1}\geq R(T_k,N,\eps,S,M)$, and that 
   the remainder of the sequences $(R_j)$, $(\chi_j)$, $(U_j)$ and $(C_j)$ are chosen subject only to the 
    restrictions mentioned above. 
   Then the resulting tract $T$ and function $F$ satisfy 
    \[ |F^{-n}(z) - F_k^{-n}(z)|\leq \eps
      \]
   for all $n\leq N$ and all $z$ with $1\leq \re z \leq M$ and  $|\im z|\leq (2S+1)\pi$. 
\end{obs}
\begin{proof}
   Having fixed $T_k$, we obtain a tract $T$ for any allowed choice of $(R_j)_{j\geq k+1}$,
    $(\chi_j)_{j>k}$, $(U_j)_{j>k}$ and $(C_j)_{j>k}$.
    As $R_{k+1}\to\infty$, these tracts $T$ converge to $T_k$ with respect to 
     Carath\'eodory kernel convergence, regardless of the remaining choices. 
     (See \cite[Theorem~1.8]{pommerenke} for the definition of kernel convergence.) 
     Indeed, let $A\ni 5$ be compact and connected. Suppose 
     $R_{k+1}\geq \max_{z\in A}\re z$. Then  clearly $A\subset T_k$ if and only if $A\subset  T$, as
    $A$ cannot cross the arcs $C_j$ or  $\tilde{C}_j$ for  $j\geq k+1$.

   Recall that $F$ is a conformal isomorphism from $T$ to $\HH$ with
     $F(5)=5$ and $F'(5)>0$; the same holds for $F_k$ and $T_k$. 
    By the kernel convergence theorem, as $R_{k+1}\to\infty$, we have 
     $F^{-1}\to F_k^{-1}$, uniformly on compact subsets of $\HH$.
     This proves the claim for $N=1$; 
     the general case follows by induction.
\end{proof}

\begin{obs}[Expansion of the maps $F$ and $F_k$] \label{obs:expansion}
 Let $T\ni 5$ be a logarithmic tract with $\re z > 4$ and $\lvert \im z\rvert < \pi$ 
  for all $z\in T$, and let
  $G\colon \T\to \HH$ be a conformal isomorphism with 
  $G(5)=5$ and $G(\infty)=\infty$. Then $|G'(z)|\geq \re G(z)/2$ for all $z\in T$.

  In particular, $|G'(z)|\geq 2$ whenever $\re G(z)\geq 4$. 
    In addition, if $z\in T$ with $G(z)\in \overline{T}$ and $\re z \geq 9$, then $\re G(z) > \re z$. 
\end{obs}
\begin{proof}
  The first claim (and hence the second claim) 
   holds for any $G\in\Blog$ with range $H=\HH$; 
   see~\cite[Lemma~2.1]{EremenkoLyubichconstant}. In our case, where 
    $T \subset \tilde{S}=\{a+ib\colon  |b|<\pi\}$, we may establish it directly
    by estimating the density of the hyperbolic metric of $T(G)$ from below
    by that of $\tilde{S}$, which is bounded from below by $1/2$; 
    compare~\cite[Lemma~3.3]{pseudoarcs}. 

  To prove the final claim, suppose that $\re z\geq 9$ and $G(z) \in \overline{T}$. 
    Let  $\gamma$ be a straight line segment connecting 
    $G(z)$ and $5$, and let $\tilde{\gamma}\defeq F^{-1}\circ\gamma$. Then,
    by the first claim and the assumptions on $T$ and $F$, 
   \begin{align*} |\re G(z) - 5| &\geq |G(z)-5| - \pi = \ell(\gamma) - \pi \geq
      2\ell(\tilde{\gamma}) -  \pi \\ &\geq 2|z-5| - \pi \geq 2(\re z - 5) - \pi > (\re z - 5) + 4 - \pi >\re z -  5. \end{align*}
    Since $\re G(z)\geq 4$, we  hence have $\re G(z) >\re z$, as  claimed.
\end{proof}

\subsection*{Overview of the recursive construction}
From now on, we fix sequences of continuous and surjective 
  functions $g_{k}\colon [0,1]\to[0,1]$ ($k\geq 1$) and numbers 
  $M_k\geq 10$ ($k\geq 0$), as in Proposition~\ref{prop:arclikeexistenceprecise}. 
    We also fix a number $k_* \in\{0,\dots, k-1\}$ for each $k\geq 1$. For the 
     purpose of
      the current section,
     we set
     \begin{equation}\label{eqn:kstar} k_*\defeq k-1. \end{equation}
  \begin{remark}
   In Section~\ref{sec:allinone}, we use the same construction, but
   with a different choice of ``predecessor'' $k_*$. With this
   modification, we are able to realise not just the single Julia
   continuum $\uarrow{Y}=\invlim ([0,1],g_{k+1})_{k=0}^{\infty}$, but the
   inverse limit of any system $([0,1],g_{k_j})_{j=0}^{\infty}$, where
   $(k_{j+1})_* = k_j$ for all $j\geq 0$, and $(k_0)_* = 0$. 
   (There are uncountably many 
    such sequences for general choices of $k\mapsto k_*$, although there
     is only one for~\eqref{eqn:kstar}.) 
   We shall use this fact to 
   realise all continua $Y$ as in Theorem~\ref{thm:arclikeexistenceBlog} as Julia continua of a
   \emph{single} function, as claimed in Theorem~\ref{thm:mainarclike}. In this section,    
  where we realise a \emph{single} Julia continuum, 
   setting~\eqref{eqn:kstar} suffices. To facilitate the discussion in 
  Section~\ref{sec:allinone}, we nonetheless describe the construction of $F$ and $T$,
  and the description of their properties, for a more general choice of $k\mapsto k_*$. 
  \end{remark}

  The address $\s$ in Proposition~\ref{prop:arclikeexistenceprecise} will be of the form 
       \begin{equation}\label{eqn:formofs} \s = s(0) 0^{N_1}\, s(1)\, 0^{N_2}\, s(2)\, 0^{N_3}\, s(3) \dots, \end{equation}
   where $(N_k)_{k=1}^{\infty}$ is a (rapidly increasing) sequence of non-negative integers, and $(s(k))_{k=0}^{\infty}$ is another sequence of non-negative
   integers. The purpose of the 
   block of zeros of length $N_k$ is mainly to move the Julia continuum sufficiently far to the right to reach the
   side channel $U_k$; for $k\geq 1$,  
   the integer $s(k)$ is chosen to ensure that the continuum does indeed enter $U_k$
   at this point. The numbers $n_k$ that appear in Proposition~\ref{prop:arclikeexistenceprecise} are then given by
    \begin{equation}\label{eqn:nk}
      n_k \defeq k + \sum_{j=1}^k N_j.
    \end{equation}  
(That is, $\sigma^{n_k}(\s)$ is the part of $\s$ from the entry $s(k)$ onwards.)

  Our goal is to
    apply the conjugacy principle from Proposition~\ref{prop:conjugacyprinciple} to
    two inverse systems.
   The first system $(X_k,f_{k+1})_{k=0}^{\infty}$ is given by $X_k=\hat{T}\defeq \overline{T}\cup\{\infty\}$ for all $k$ and
       \begin{equation}\label{eqn:defnfk}
     f_k\colon \hat{T}\to \hat{T};\qquad z\mapsto F^{-(N_k+1)}(z+2\pi i s(k)) \end{equation}
     for $k\geq 1$. 
     Note that $\uarrow{X}\defeq \invlim(X_k,f_{k+1})$ is homeomorphic to  the  Julia continuum $\Jsh(\tilde{F})$.  
     We shall endow each $X_k$ with  a metric that
      agrees with the Euclidean metric, except near infinity and on the left side  of  the
      tract; see~\eqref{eqn:dXk}. This modification is done  to ensure that the 
      system is expanding, so  that we can indeed apply 
     Proposition~\ref{prop:conjugacyprinciple}.

The second system $(Y_k,g_k)_{k=0}^{\infty}$ 
    has $Y_k = [0,1]$ for all $k$ and consists of the functions 
    $g_k\colon [0,1]\to [0,1]$ fixed 
     above.  Each copy 
     $Y_k$ of the interval $[0,1]$ 
     is endowed with a ``blown-up'' version $d_{Y_k}$ of the Euclidean metric 
      in the sense of
      Observation~\ref{obs:obtainingexpandingsystems}, again to ensure expansion.

 The construction of the tract proceeds recursively, along with the construction of a number of additional objects:
  \begin{enumerate}[(a)]
    \item \label{item:sandNk}
         The sequences $(s(k))_{k\geq 0}$ and $(N_k)_{k\geq 1}$ of non-negative integers
           determine the address $\s$ as in~\eqref{eqn:formofs}. Recall 
      from~\eqref{eqn:nk} that this also determines $n_k$.
    \item A sequence of natural numbers $\gamma_k\in\N$, $k\geq 0$ is used to define the metric $d_{Y_k}$, used in the system $(g_k)$.
    More precisely, $d_{Y_k}$ is defined by scaling up the Euclidean metric by
        a factor of $\gamma_k$. These numbers are chosen so large that they satisfy the hypotheses of 
        Observation~\ref{obs:obtainingexpandingsystems}. 
    \item\label{item:Xi}%
     We also define $\Xi_k \defeq \{0 = \xi_k^0 < \xi_k^1<\dots<\xi_k^{\gamma_k}=1\}$, where 
       $\xi^j_k \defeq  j/\gamma_k$ for $0\leq j \leq \gamma_k$. The purpose of $\Xi_k$
       is to provide  a sufficiently fine partition of the range of $g_k$. 
    \item For each  $k\geq 1$, a finite subset  $\Omega_k = \{0=\omega^0_k < \omega^1_k < \dots < \omega^{m_k}_k = 1 \}\subset [0,1]$  will  be defined, 
      to be used as a partition of
       the domain of $g_k$. 
    \item A surjective, continuous and non-decreasing map $\phi_k\colon [0,\infty]\to [0,1]$, 
     for every $k\geq 0$, relates the dynamics of the two inverse systems $(f_k)$ and $(g_k)$. This function has the property 
          that 
    \begin{equation}\label{eqn:phiconditions}
       \phi_k|_{[0,M_k]}=0, \qquad \phi_k|_{[\widetilde{M}_k,\infty]} = 1,
           \qquad \text{and}\qquad 
          \text{$\phi_k|_{[M_k,\widetilde{M}_k]}$ is a 
         homeomorphism.} 
         \end{equation}
    Here  $M_k\geq 10$ is the number from the hypothesis of 
          Proposition~\ref{prop:arclikeexistenceprecise}, 
            and $\widetilde{M}_k > M_k$ is a number chosen during the construction.
    In a slight abuse of notation, we use $\phi_k^{-1}$ to denote 
     the inverse of the (bijective) restriction $\phi_k|_{[M_k,\widetilde{M}_k]}$.

    \item\label{item:psidefinition}   We also define  $\psi_k\colon \hat{T}\to [0,1]$ by $\psi_k(z)\defeq \phi_k(\re z)$, using the convention
        that $\re \infty = \infty$. These  maps form our pseudo-conjugacy for use with 
        Proposition~\ref{prop:conjugacyprinciple}.  
  \end{enumerate}

The idea of the construction is summarised in Figure \ref{fig:tract_construction}: the map $\psi_{k_*}$ provides an  
   identification between the tract $T$
   and the domain of the map $g_{k_*}$, or equivalently the range of the map $g_{k}$. The notation within the figure
   will become clear in the detailed definitions; the reader may find it useful to keep the figure to hand while
   reading about the construction. 
   
     As the technical details of the inductive step of the construction are somewhat involved, let us first 
     set out its overall structure, and explain the motivation  for the choices made.
   We begin by defining $s(0)$, $\widetilde{M}_0$, $\phi_0$ and $\gamma_0$. Then suppose that $k\geq 1$, and that
    all relevant objects have been defined up to stage $k-1$. 

 \begin{enumerate}
  \item[I1.] We choose the set $\Omega_k$ such that the 
    complementary intervals of $\Omega_k$ are small enough; see ~\ref{I1item:Omegaimage}. This is done  so that
    the information of how these 
    intervals are mapped over the set $\Xi_{k_*}$ encodes the essential information about 
    $g_k$. 
  \item[I2.] We define $N_k$ and $R_k$ sufficiently large.
   Recall  from  Proposition~\ref{prop:arclikeexistenceprecise} that the $n_{k_*}$-th iterate of the
     Julia continuum will begin near real part $M_{k_*}$. Moreover, from the definition of
      the map $\psi_{k_*}$, the essential information about the image of the map $f_k$
       concerns 
      points with real parts in the interval $[M_{k_*},\widetilde{M}_{k_*}]$ (as the other parts of the tract
      are collapsed by $\psi_{k_*}$). Recall that $f_k$ was defined in~\ref{eqn:defnfk}.

   Our  choice of  $N_k$ is made to ensure that $F^{N_k}$ maps this interval to real  parts
     larger than $R_{k-1}$, i.e. to the right of  all 
     decorations already constructed at this stage; see~\eqref{eqn:Nkright}. 
     (In the actual construction, the final map $F$ is not yet defined, so we instead use 
      $F_{k-1}$ in~\eqref{eqn:Nkright}; the same applies for the remainder of this informal description.) 

 Choosing $N_k$ large enough also
     ensures (using Observation \ref{obs:expansion}) that 
     $F^{-N_k}$ is strongly contracting~\eqref{eqn:Fkexpansion}.

    Once $N_k$ is fixed, the number $R_k$ is chosen larger than 
      $F^{N_k}(\widetilde{M}_{k_*})$ (see~\ref{I2item:Rksize}), 
     which is necessary to be able to ensure that the interval 
     $F^{N_k}([M_{k_*},\widetilde{M}_{k_*}])$ lies below  the domain $U_k$ (see Figure~\ref{fig:tract_construction}).

    Furthermore, if $R_k$ is large enough, then the final function  
      $F$ will be close enough to
      $F_{k-1}$ as to not disturb any of the previous steps of the 
      construction; see~\ref{I2item:closeness}.
  \item[I3.] This is the key step in the construction, where we define the domain $U_k$.
     This domain is chosen to follow the same structure as the map $g_k$: the real parts of $U_k$ run
     across the interval $F^{N_k}([M_{k_*},\widetilde{M}_{k_*}])$ in the same way 
     as the graph of $g_k$ runs over the interval $[0,1]$. (Here
    the identification between the
     two intervals given by $\phi_{k_*}\circ F^{-N_k}$). 

 More precisely, $U_k$ is a chain of $m_k$ Jordan quadrilaterals, one for each complementary interval of
   $\Omega_k$, placed at the appropriate real parts (see~\ref{I3item:sizeofUk}). In~\ref{I3item:Cm} we also choose 
   the arc $C_k$, along with additional arcs connecting the quadrilaterals \ref{I3item:Cj}. These 
   arcs are taken sufficiently short to ensure good control  over geodesics of $T_k$ running through the
    domain $U_k$; see~\ref{I3item:Ukhyperbolicdistance} and~\ref{I3item:Ukrealparts}. 
 We remark that the additional arcs are introduced 
    primarily for convenience in the analysis. 
  \item[I4.] We fix the opening size $\chi_k$, which completes the choice of the partial
    tract $T_k$ and the function $F_k$. This means 
    that in the remaining
    steps, we may use $F_k$  as an approximation of the limiting function $F$, 
    instead of $F_{k-1}$. 

  Fixing $\chi_k$ small enough ensures that we can choose the entry
    $s(k)$ 
     in such a way that the curve $F^{-1}(x+2\pi i s(k))$, for $x\geq M_k$,
      runs through the entire decoration $U_k$; see~\ref{I4item:endpoint} and~\ref{I4item:endpiece}. 

      If $\chi_k$ is sufficiently small, then furthermore 
        the curve does not enter the part of the central strip  $S$
        that is to the left of real part $R_{k}$~\ref{I4item:wholegeodesic}. 
      (In other words, the picture does indeed look 
      as sketched at the bottom of Figure~\ref{fig:tract_construction}.) 
  \item[I5.] We pick a sufficiently large number  $\widetilde{M}_k$,       
     see~\eqref{eqn:Mtilde}. The key here is that 
     the pull-back of $[\widetilde{M}_k,\infty)+2\pi i s(k)$ under $F^{n_k}$,
    along the address
     $\s$, should still have large real parts. In the limit, this means that 
      no orbit of $F$ with address $\s$ has real part 
      to the right of $\widetilde{M}_k$ at every time $n_k$ and 
       ensures that
       the inverse system $(f_k)$ is expanding. 

   We can then define 
    $\phi_k$ according to Figure \ref{fig:tract_construction}.
    That is, if $z$ belongs to one of the constituent quadrilaterals of $U_k$, and $F(z)\in T+2\pi i s(k)$, then 
     $\phi(\re z)$ should belong to the corresponding complementary interval of $\Omega_k$, or possibly to one of the adjacent intervals. 
    This allows us to establish
      the key pseudo-conjugacy condition (see Lemma~\ref{lem:pullingback}). 

  Finally, we choose $\gamma_k$ sufficiently large according to Observation~\ref{obs:obtainingexpandingsystems},
    so that the system $(g_k)$ is expanding with respect to the
    metrics $d_{Y_k}$~\ref{I5item:expansion}, and so large that
    preimages of intervals of size $1/\gamma_k$ under $\phi_k$ have bounded
    Euclidean diameter~\ref{I5item:smallgaps}. This is needed to verify
    condition~\ref{item:pseudoinjectivity} of Definition~\ref{defn:pseudoconjugacy}. 

  With this, the recursive definition will be complete. (Recall that
   the choice of $\gamma_k$ also determines the partition $\Xi_k$ of the range of
  $g_k$ into intervals of size $1/\gamma_k$.) 
 \end{enumerate}

\begin{figure}
 \begin{center}
   \def\svgwidth{.9\textwidth}
   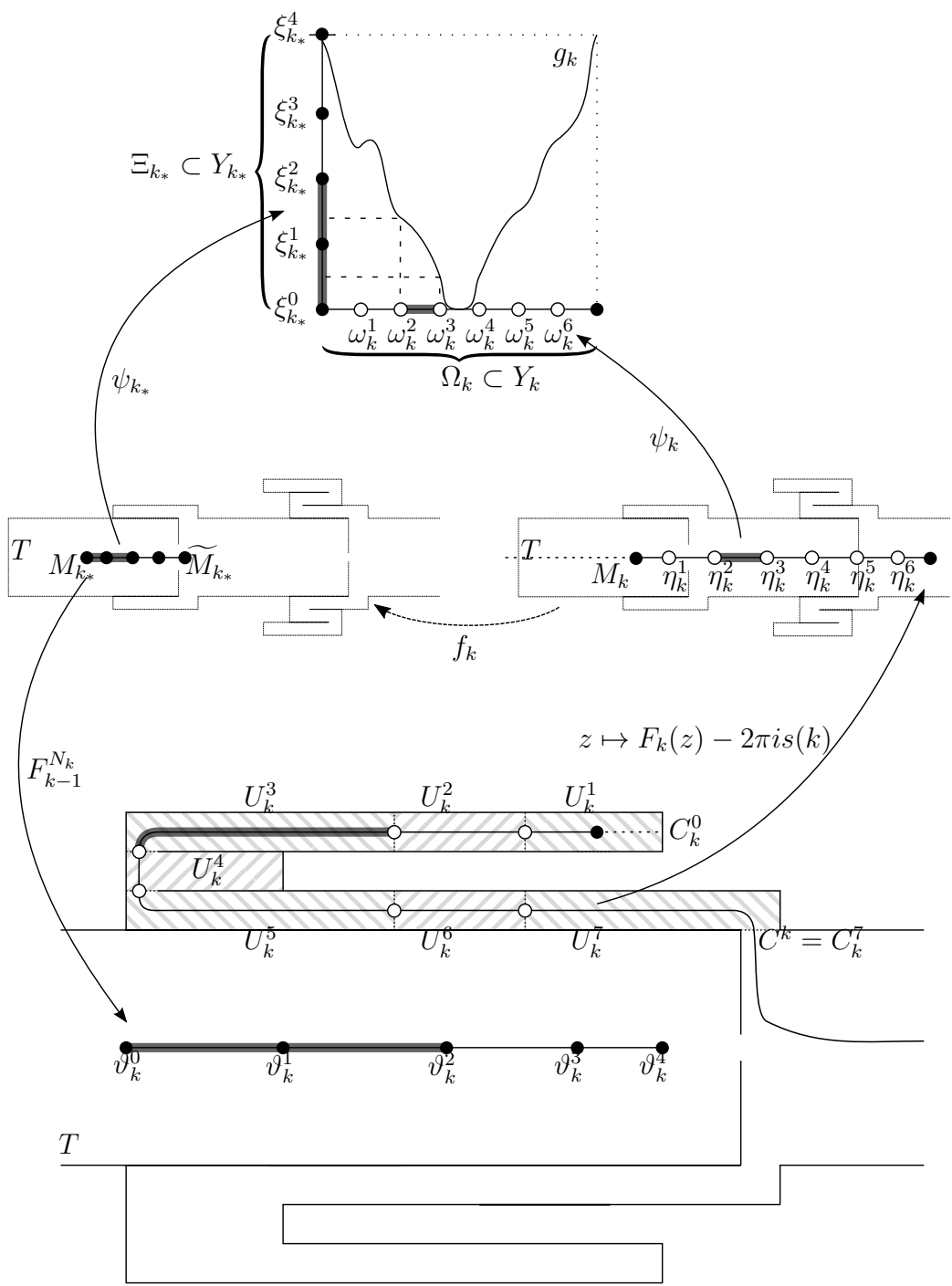
 \end{center}
 \caption{Construction of the domain $U_k$. The shaded intervals at the top of the figure are $[\omega_k^2,\omega_k^3]$ (in the domain
    of $g_k$) and $I_k^3$ (in the range).\label{fig:tract_construction}}
\end{figure}

\subsection*{Details of the recursive construction} 
 To anchor the recursion, define $s(0)\defeq 0$, 
   $\widetilde{M}_0 \defeq M_0 +1$, $\phi_0(x)\defeq x -  M_0$ for $x\in  [M_0,\widetilde{M}_0]$ 
     and according to~\eqref{eqn:phiconditions} elsewhere, 
    and $\gamma_0 \defeq 5$, which also fixes $\Xi_0$ according to~\ref{item:Xi} above. 
    For the  recursive step of the construction,  fix $k\geq 1$ such that all  relevant objects
    have been defined up to stage $k-1$. Recall that $k_*\leq k-1$ (and in fact $k_*=k-1$ for the purpose of this section).

\begininductiveconstruction
\textbf{I1.}
 We choose 
   $\Omega_k = \{0=\omega^0_k < \omega^1_k < \dots < \omega^{m_k}_k = 1 \}\subset [0,1]$ 
   such that, for each $j$, 
 \begin{enuminduction}
    \item the image $g_k([\omega_k^{j-1},\omega_k^j])$ contains at most one point of $\Xi_{k_*}\setminus\{0,1\}$.\label{I1item:Omegaimage}
 \end{enuminduction}
 
\nextinductivestep\textbf{I2.} 
 The next step is to choose the numbers $N_k$ and $R_k$. Define 
     \begin{equinduction}\label{eqn:alpha}
        \alpha_{k_*} \defeq \min\bigl (1, \min\bigl \{|x-y|\colon  |\phi_{k_*}(x)-\phi_{k_*}(y)|\geq \frac{1}{2\gamma_{k_*}}\bigr\}\bigr) > 0.
     \end{equinduction}

  Now we choose $N_k$
   sufficiently large that 
     \begin{aligninduction}
       F_{k-1}^{N_k}(M_{k_*})&> R_{k-1} + 1 \qquad\text{and}\label{eqn:Nkright}\\
       N_k &> \frac{\log(12)-\log(\alpha_{k_*})}{\log 2} > 3.\label{eqn:Nk}
     \end{aligninduction} 
  By Observation~\ref{obs:expansion}, condition~\eqref{eqn:Nk} implies 
\begin{equinduction}\label{eqn:Fkexpansion}
  \bigl\lvert \bigl( F^{-N_k}\bigr)'(z)\rvert \leq 2^{-N_k} \leq
       \frac{\alpha_{k_*}}{12}
\end{equinduction}
   whenever $\re z\geq 4$; in particular, 
 \begin{equinduction}\label{eqn:fkexpansion}
  \lvert f_k'(z)\rvert \leq \frac{1}{12}
 \end{equinduction}
 for all $z\in \overline{T}$. 
  We also choose $R_k > R_{k-1}+1$ so large  that
  \begin{enuminduction}
   \item $R_k \geq F_{k-1}^{N_k}(\widetilde{M}_{k_*}+1)$ and \label{I2item:Rksize}
    \item $R_k \geq R(T_{k-1} , N ,  \eps , S , M)$ according to Observation~\ref{obs:caratheodoryconvergence}.  Here we  take
                  $N=n_k$, $\eps = \alpha_{k_*}/2$ , $S=\max_{j< k} s(j)$ and 
                  $M = \max(\widetilde{M}_{k-1},F_{k-1}^{N_k}(\widetilde{M}_{k_*})+10)$.\label{I2item:closeness}
 \end{enuminduction}

\nextinductivestep\textbf{I3.} Now we define the domain $U_k$ and the arc $C_k$. 
 For $0\leq \ell \leq \gamma_{k_*}$, let
   \begin{equinduction}\label{eqn:thetadefn} 
   \theta^{\ell}_{k} \defeq  F_{k-1}^{N_k}(\phi_{k_*}^{-1}(\xi_{k_*}^{\ell})). \end{equinduction}
   Note that $\theta_k^{\ell} \geq \theta_k^0 = F_{k-1}^{N_k}(M_{k_*}) \geq R_{k-1} + 1$ for all $\ell$, by~\eqref{eqn:Nkright}. 
   As mentioned above, $U_k$ is the union of a sequence of Jordan quadrilaterals $(U_k^j)_{j=1}^{m_k}$, 
    with one quadrilateral (namely $U_k^j$) associated  
     to each complementary interval $[\omega_k^{j-1},\omega_k^{j}]$ of $\Omega_k$. 
     We  also construct pairwise disjoint arcs  
     $(C_k^j)_{j=0}^{m_k}$, 
    with $C_k^j$ joining $U_k^{j}$ and  $U_k^{j+1}$ for $1\leq j\leq  m_k-1$. 
    More precisely,
 \begin{enuminduction}
   \item $C_k^{j-1}\cup C_k^{j}\subset \partial U_k^j$ for all $j$, and
     $\interior(C_k^i)\cap \partial U_k^j =\emptyset$ for $i\neq j-1,j$; and \label{I3item:Cj}
   \item $C_k\defeq  C_k^{m_k}\subset \{x + i\pi/2\colon  R_k<x<R_k+1\}$.  \label{I3item:Cm}
 \end{enuminduction}
 Once these quadrilaterals and arcs are chosen, the domain $U_k$ is given by 
  \begin{equinduction}\label{eqn:Ukdefn}
     U_k \defeq  \bigcup_{j=1}^{m_k} U_k^j \cup \bigcup_{j=1}^{m_k-1} \interior(C_k^j). 
  \end{equinduction}

 For $j=1,\dots,m_k$, let $I_k^j\subset[0,1]$ be the smallest closed interval bounded by two points of
   $\Xi_{k_*}$ whose relative interior in $[0,1]$ contains $g_k([\omega_k^{j-1},\omega_k^{j}])$. (See the top of Figure~\ref{fig:tract_construction}.) Recall 
     from~\ref{I1item:Omegaimage} that $I_k^j$ has length at most $2/\gamma_{k_*}$, and note that $I_k^j$ and $I_k^{j+1}$ 
     overlap by an interval of length at least $1/\gamma_{k_*}$. Also observe  that
     $I_k^{m_k} \supset [\xi_{k_*}^{\gamma_{k_*}-1} ,  1]$. 
     Define $\tilde{I}_k^j \defeq  F_{k-1}^{N_k}(\phi_{k_*}^{-1}( I_k^j)\cap [M_{k_*},\widetilde{M}_{k_*}])$
      for $j=1,\dots , m_k-1$ (so $\tilde{I}_k^j$ is bounded by two of the points $\theta^{\ell}_k$, corresponding to the endpoints
     of the interval $I_k^j$). For $j=m_k$, we modify this definition by extending $\tilde{I}_k^{m_k}$ to $R_k+1$ on the right. 
    That is, $\tilde{I}_k^{m_k} = [\theta_k^{\gamma_{k_*}-1} , R_k+1]$ or $\tilde{I}_k^{m_k} = [\theta_k^{\gamma_{k_*}-2} , R_k+1]$, depending on whether
    $\omega_k^{j-1}>\theta_k^{\gamma_{k_*}-1}$ or not. 
        
  Now the domains $U^j_k$ and arcs $C^j_k$ are chosen such that, for $1\leq j\leq m_k$:
 \begin{enuminduction}
   \item $\re z\in \tilde{I}_k^j$ for all  $z\in U_k^j$.\label{I3item:sizeofUk}
  \item The hyperbolic distance, in $T_k$, between $C_k^{j-1}$ and  $C_k^{j}$ is at least $1$.\label{I3item:Ukhyperbolicdistance}
  \item Any geodesic segment of  $T_k$ that connects
     $C_k^{j-1}$ and $C_k^{j}$ has real parts in
      $\tilde{I}_k^j$.\label{I3item:Ukrealparts}
 \end{enuminduction}
 Since $\tilde{I}_k^j \subset [\theta_k^0, R_k + 1 ]\subset [R_{k-1}+1 ,R_k + 1]$,  we see 
  that~\ref{I3item:sizeofUk} ensures~\eqref{eqn:Ujposition}. We emphasise that
  the domain $T_k$ 
  depends not only on $U_k$ and $C_k$, which we are fixing in this step,
  but also on the opening size $\chi_k$, which will only be chosen in 
  step~I4. Our claim is that, by appropriate choice of
  the $U^j_k$ and $C^j_k$, we can ensure 
  that properties~\ref{I3item:Ukhyperbolicdistance} 
  and~\ref{I3item:Ukrealparts} hold for \emph{any} choice of $\chi_k$. 
 
 For example, we may take 
      \[ U_k^j\defeq \left\{x+iy\colon  x\in \interior(\tilde{I}_k^j), 
        y \in \pi\cdot \left( 1 - \frac{j}{2m_k},1-\frac{j-1}{2m_k}\right) \right\}. \] 
    Then~\ref{I3item:sizeofUk} is trivially satisfied. The arcs $C_k^j$ are chosen to be short horizontal line segments 
    at imaginary part $\pi\cdot(1 - j/(2m_k))$, with 
     \begin{equinduction} \label{eqn:Ckintervals}\re C_k^j \subset \begin{cases} \tilde{I}_k^1 & \text{if $j=0$}, \\
                                                        \tilde{I}_k^j \cap \tilde{I}_k^{j+1} & \text{if $0<j<m_k$},\\
                                                        [R_k,R_{k}+1] &\text{if $j=m_k$}.\end{cases} \end{equinduction}
    (Recall that the intervals $\tilde{I}_k^j$ and  $\tilde{I}_k^{j+1}$ overlap for $0<j<m_k$.) Then~\ref{I3item:Cj} and~\ref{I3item:Cm}
     are satisfied. We claim that, by choosing each arc $C_k^j$ sufficiently short, say with real parts centred around the midpoint of the 
      corresponding interval in~\eqref{eqn:Ckintervals}, 
       we 
       can also ensure that 
      Properties~\ref{I3item:Ukhyperbolicdistance} and~\ref{I3item:Ukrealparts} hold,
       regardless of the choice of $\chi_k$. 
   For~\ref{I3item:Ukhyperbolicdistance},
   this follows immediately from the standard estimate~\eqref{eqn:standardestimate} on the hyperbolic metric. 

For~\ref{I3item:Ukrealparts}, we use the following fact. (Compare~\cite[Lemma~A.3]{strahlen}.) 
  Suppose that $Q\subset\HH$ is a conformal quadrilateral; that is, a Jordan domain whose boundary is subdivided into four arcs, where two opposing arcs are designated the ``horizontal sides'' and the others the ``vertical sides''. Recall that the \emph{modulus} of $Q$ is the modulus
 of the family of curves in $Q$ connecting the horizontal sides, or equivalently the extremal length of the curves in $Q$ connecting the vertical sides.  
   If the horizontal sides of $Q$ are contained in $\partial\HH$, and $\mod(Q)>1$, 
 then $Q$ contains a hyperbolic geodesic of $\HH$ that connects the two horizontal sides. To see this, 
we may suppose (applying a M\"obius transformation of $\HH$) that one of the two vertical
sides contains $0$ and the other contains $\infty$. By reflection in the imaginary
axis, and connecting the resulting quadrilateral to $Q$ along the horizontal sides, 
 we obtain an annulus $A$ of modulus greater than $1/2$ that separates $0$ from
$\infty$. By the Teichm\"uller modulus 
 theorem, this annulus contains a round circle, centred at~0;
 see~\cite[Section~4.11]{ahlforsconformal}. The intersection 
  of this circle with the right half-plane is a hyperbolic geodesic of $\HH$ contained in $Q$, 
 as desired. 
 
 Now let $1 < j<m_k$. 
  If the arc $C_k^{j-1}$ is chosen sufficiently short, then
   there is a quadrilateral $Q^-\subset U_k^{j-1}$ of modulus greater than
   $1$ with horizontal sides in $\partial T_k$, having $C_k^{j-1}$ as one of the vertical sides, and such that furthermore 
    any $z\in Q$ has $\re z \in\tilde{I}_k^j$. If  $C_k^{j+1}$ is 
   sufficiently small, then there is a similar quadrilateral $Q^+\subset U_k^{j+1}$. 
   By the above, for any choice of $\chi_k$, the quadrilateral $Q^-$ contains a geodesic $\beta^-$ of $T_k$ that
   connects the two horizontal sides, and $Q^+$ contains a similar geodesic
   $\beta^+$. So the connected component of $T_k\setminus (\beta^-\cup \beta^+)$
   that contains $U_k^j$ has real parts in $\tilde{I}_k^j$. Two different 
   geodesics of $T_k$ intersect in at most one point. So a geodesic segment of $T_k$ connecting $C_k^{j-1}$ and
    $C_k^j$ cannot intersect $\beta^-$ or $\beta^+$, and thus has 
    real parts in $\tilde{I}_k^j$. This establishes~\ref{I3item:Ukrealparts} for $1< j<m_k$.
    For $j=1$, the same argument applies, except that there is no need for 
    the quadrilateral $Q^-$ or the geodesic $\beta^-$. Similarly, for $j=m_k$, 
     apply the same argument but with $Q^+$ a quadrilateral contained 
     in the central half-strip $S$, and having real parts in 
     $[R_k,R_k+1]\subset \tilde{I}_k^{m_k}$. 

 \begin{remark}
  There are many other ways of choosing the domains $U_k^j$ with the desired properties---e.g., 
  Figure~\ref{fig:tract_construction} uses an arrangement that is more economical with vertical space.
  By using different choices of how the domains are arranged, one may obtain inequivalent embeddings of the
  same continuum in the Julia set, as mentioned at the end of Section~\ref{sec:intro2}. 
\end{remark}

\nextinductivestep\textbf{I4.}
  We now choose $\chi_k$. Moreover, depending on $\chi_k$, $s(k)$ will be defined 
   as follows. Let $A$ be the interval
   on $\partial\HH$ corresponding to the arc $C_k^0$ under $F_k$ (note that $F_k$
   itself depends on the choice of $\chi_k$). 
  We choose $s(k)$ such that
  $2\pi i s(k)$ is closest to the midpoint of $A$. We claim that, if
   $\chi_k$ is chosen sufficiently small, then 
 \begin{enuminduction}
     \item $2\pi i s(k)\in A$, and hence
        $F_k^{-1}( 2\pi i s(k) )\in C_k^0$;\label{I4item:endpoint} 
     \item $F_k^{-1}( x + 2\pi i s(k) ) \in U_k^1$ for 
          $x\leq M_k+1$, and \label{I4item:endpiece}
    \item For all $x\geq 0$, the point $F_k^{-1}(x + 2\pi i s(k))$ 
        belongs to the unbounded connected component $Q$
        of
$T_k\setminus \bigl(R_k + \frac{i}{2}\cdot [-\chi_k,\chi_k]\bigr)$.
     In other words,
        the point either belongs to $U_k$ or has real part greater than $R_k$.\label{I4item:wholegeodesic}
  \end{enuminduction}
   Indeed, consider what happens to the tract $T_k$ and the map $F_k$ as we
    let $\chi_k$ tend to zero. Clearly the harmonic measure of 
    $C_k^0\subset \partial U_k$ in $T_k$, viewed from
    $R_k+1$, remains bounded from below. On the other hand, the length
    of the geodesic segment $[5,R_k+1]$ tends to infinity as $\chi_k\to 0$, so 
    $F_k( R_k+1)\to\infty$. 
    This implies that, as $\chi_k\to 0$, the length of $A$ tends to infinity.
    In particular, $2\pi i s(k)$ lies well inside $A$ as long as
    $\chi_k$ is sufficiently small, establishing~\ref{I4item:endpoint}. 

  Furthermore, $A$ and $F_k(C_k^1)$ are separated by a quadrilateral (namely $F_k(U_k^1)$) whose modulus is independent of $\gamma_k$. 
   Again applying the Teichm\"uller modulus theorem, 
   it follows that the distance from $2\pi i s(k)$ to $F_k(C_k^1)$ grows at least
  propertionally with the length of $A$. 
   So~\ref{I4item:endpiece} also holds when $\chi_k$ is small enough. 

 Finally, consider $Q$ as a quadrilateral 
  whose ``horizontal'' sides are given by the two vertical segments
 $R_k \pm i [\chi_k,\pi]/2$.
  If $\chi_k$ is small enough, $Q$ has modulus greater than $1$, so again there
  is a geodesic of $T$ contained in $Q$ connecting the two  line segments.
  This geodesic is disjoint from 
  $F_k^{-1}([0,\infty) + 2\pi i s(k))$, and separates it from $T\setminus Q$. 
   This establishes~\ref{I4item:wholegeodesic}. Observe that 
$F_k^{-1}([0,\infty) + 2\pi i s(k))$ is disjoint from the geodesic $[4,\infty)$, 
   and hence does not
  enter $\tilde{U}_k$. So the second claim of~\ref{I4item:wholegeodesic}
   does indeed follow from the first.

\nextinductivestep\textbf{I5.}
For $j=1,\dots,m_k$, define 
     \begin{equinduction}\label{eqn:etaj}
         \eta_k^j \defeq \max\{ x\colon F_{k}^{-1}(x + 2\pi i s(k)) \in C_k^{j} \} \geq M_k+1. \end{equinduction}
   Furthermore, set 
      \begin{equinduction}\label{eqn:Mtilde}
         \widetilde{M}_k\defeq\max\bigl((2s(k)+1)\pi , \eta_k^{m_k} , 
             F_k(R_k+3) , F_k^{N_k+1}(\widetilde{M}_{k_*}+1) ,
               F_k^{n_k}(k+5)\bigr).\end{equinduction}

  By construction, we have $M_k < \eta_k^1 < \dots < \eta_k^{m_k} \leq \widetilde{M}_k$. 
    Hence we can define an order-preserving homeomorphism
   $\phi_k\colon [M_k,\widetilde{M}_k]\to [0,1]$ such that 
   \begin{equinduction}\label{eqn:phikdefn}
      \phi_k(\eta_k^j)=\omega_j, \qquad j=1,\dots,m_{k}-1.\end{equinduction}
     We may extend $\phi_k$ to $[0,\infty]$ by setting
   $\phi_k(x) \defeq  0$ for $x<M_k$ and $\phi_k(x) \defeq 1$ for $x\geq \widetilde{M}_k$. 

 It remains to choose $\gamma_k$ sufficiently large such that the following hold.
    \begin{enuminduction}
      \item\label{I5item:expansion} $\gamma_k \geq \Gamma_k(\gamma_{k_*})$, where $\Gamma_k$ is  the 
        function from Observation~\ref{obs:obtainingexpandingsystems}, for expansion  constants
          $K=3$ and  $\lambda =  4$. This ensures that the 
        system $(g_j)_{j=1}^{\infty}$ is indeed expanding with respect to the sequence of rescaled metrics $d_{Y_k}$. 
      \item $\max\{ |x_1 - x_2|\colon x_1,x_2\in [M_k,\widetilde{M}_k]\text{ and } |\phi_k(x_1)-\phi_k(x_2)|\leq 1/\gamma_k\}<1/4$.
               \label{I5item:smallgaps}
     (Observe that this holds also for $k=0$ by choice of  $\gamma_0$.)
     \end{enuminduction}
%

 This completes the recursive construction.

\subsection*{Analysis of the construction}
   We must now show that our construction does indeed yield two expanding inverse systems that are
     pseudo-conjugate, and hence have homeomorphic inverse limits.  The key property is
     as follows  (see also Figure \ref{fig:tract_construction}). 

\begin{lem}[Pseudo-conjugacy relation] \label{lem:pullingback}
  Let $k\geq 1$ and $w\in \overline{T}$. Then       $\re f_k(w) \geq M_{k_*}-1$ and
   \begin{equation}\label{eqn:toproveforpseudoconjugacy} |\psi_{k_*}( f_k(w)) - g_{k}(\psi_{k}(w))| < 3/\gamma_{k_*}. \end{equation}
\end{lem}
\begin{proof}
  Set $z\defeq f_k(w) = F^{-(N_k+1)}(w+2\pi i s(k))$. 
  Let $j\in\{1,\dots,m_k\}$ be such that $\psi_k(w)=\phi_k(\re w)\in[\omega_k^{j-1}, \omega_k^j]$
   and recall the definition of
   the interval $\tilde{I}^j_k$ from step I3 of the construction. We set
   $\tilde{I}\defeq \tilde{I}^j_k$ if $j<m_k$. For $j=m_k$, we set
   $\tilde{I}\defeq \tilde{I}^{m_k}_k \cup [R_k+1,\infty)$, so that 
     $[R_k,\infty)\subset \tilde{I}$ (recall that $[\theta_k^{\gamma_{k_*}-1},R_k+1]\subset \tilde{I}^{m_k}_k$,
     and that $\theta_k^{\gamma_{k_*}-1} \leq R_k$ by~\ref{I2item:Rksize}). In either case, 
    \begin{equation}\label{eqn:Itilde} \phi_{k_*}(F_{k-1}^{-N_k}(\tilde{I})) = I^j_k 
       \supset g_k([\omega_k^{j-1},\omega_k^j])\ni g_k(\psi_{k}( w)).\end{equation}


  \begin{claim}
        Set $\zeta\defeq F^{-1}(w+2\pi i s(k))$. 
       Then the Euclidean distance between $\zeta$ and $\tilde{I}$ is less than 6. 
  \end{claim}
 \begin{subproof}
   Set $\tilde{w} \defeq \re w + 2\pi i s(k)$, and consider the points 
   $\zeta_1\defeq F^{-1}(\tilde{w})$ and $\zeta_2\defeq F_k^{-1}(\tilde{w})$. 
   Then $|\zeta - \zeta_1|\leq \pi/2$ by Observation~\ref{obs:expansion} 
   and~\ref{I2item:closeness}. 
   Furthermore,  
     \begin{equation}\label{eqn:zeta2} \re \zeta_2 \in \tilde{I}.\end{equation}
    Indeed, if $\re w \leq \eta_k^{m_k}$, then by the definition of $\phi_k$ in~\eqref{eqn:phikdefn}, 
     $\eta_k^{j-1} \leq \re w \leq \eta_k^j$  (with the convention that $\eta_k^0=0$). The 
     curve 
      \[ \{ F_k^{-1}(x + 2\pi i s(k))\colon \eta_k^{j-1} \leq x \leq \eta_k^j \} \]
      is a geodesic of $T_k$ 
      connecting $C_k^{j-1}$ and $C_k^j$ by choice of $\eta_k^j$ in~\eqref{eqn:etaj}
      (using~\ref{I4item:endpoint} in the case where $j=1$). Hence by~\ref{I3item:Ukrealparts}, its real parts belong to
      $\tilde{I}_k^j$, establishing~\eqref{eqn:zeta2}. On  the other hand, if
      $\re w >  \eta_k^{m_k}$, then  $\zeta_2\notin U_k$ by~\eqref{eqn:etaj}. Hence
       $\re \zeta_2 > R_k$ by~\ref{I4item:wholegeodesic}, and~\eqref{eqn:zeta2} holds in this case  also.
    
   To complete the proof of the claim, first suppose that $\re w \leq \widetilde{M}_k$.  Then 
     $|\zeta_1 -  \zeta_2|\leq 1/2$ by the choice of $R_{k+1}$ in~\ref{I2item:closeness}, and hence, 
     by~\eqref{eqn:zeta2}, 
    \[ \dist(\zeta, \tilde{I}) \leq |\zeta - \zeta_1| + |\zeta_1 -  \zeta_2|  + \dist(\zeta_2,\tilde{I}) \leq
            \frac{\pi}{2} + \frac{1}{2} + \pi  < 6. \] 

    On the other hand, if $\re w > \widetilde{M}_k$, then by~\eqref{eqn:Mtilde} 
          \[ |\im(w+2\pi i  s(k))|\leq (2s(k)+1)\pi \leq \widetilde{M}_{k} < \re w. \]
      By Observation~\ref{obs:expansion}, $\lvert (F^{-1})'\rvert \leq 2/\re w$ on the vertical segment connecting
       $w$ and $w + 2 \pi i s(k)$. So $|\zeta - F^{-1}(\re w)| < 2$ and 
        \[ \re \zeta \geq \re F^{-1}(\re w) - 2 \geq F^{-1}(\widetilde{M}_k) - 2  \geq
               F_k^{-1}(\widetilde{M}_k) - 5/2 > R_k \]
      by~\ref{I2item:closeness} and~\eqref{eqn:Mtilde}. 
     Since $[R_k, \infty] \subset \tilde{I}$,   this completes the proof of the claim.
 \end{subproof}

  By~\eqref{eqn:Fkexpansion} $z = F^{-N_k}(\zeta)$ has distance less than 
   \[ 6\cdot 2^{- N_k  } \leq 6\cdot \frac{\alpha_{k_*}}{12} = \frac{\alpha_{k_*}}{2} \] 
   from $F^{-N_k}( \tilde{I})$.  Furthermore, by~\ref{I2item:closeness}, 
   every point of  $F^{-N_k}(\tilde{I})$ has distance at most $\alpha_{k_*}/2$ from  
     $F_{k-1}^{-N_k}(\tilde{I})$. In particular, $\re z \geq M_{k_*} - \alpha_{k_*} \geq M_{k_*} -1$. 
  
  Finally, recall that 
      $I_k^j$ has length  at  most $2/\gamma_{k_*}$ by~\ref{I1item:Omegaimage}. 
    By~\eqref{eqn:Itilde} and the definition of $\alpha_{k_*}$ in~\eqref{eqn:alpha},
     we have indeed established~\eqref{eqn:toproveforpseudoconjugacy}.
\end{proof}

  Recall that 
     the metric $d_{Y_k}$ on $Y_k$ is defined by scaling the Euclidean metric by a factor of
      $\gamma_k$. We claim that 
      the functions 
       $\psi_k\colon \hat{T}\to [0,1]; \psi_k(z) =\phi_k(\re z)$ 
     satisfy  Property~\ref{item:pseudoconjugacy} of Definition~\ref{defn:pseudoconjugacy}. Indeed,
     for 
       $x\in \overline{T} = \C\cap \hat{T}$, this follows from Lemma~\ref{lem:pullingback}, while it holds trivially for $x=\infty$, since 
      $\psi_{k-1}(f_k(\infty))=\psi_{k-1}(\infty)=1=g_k(1) = g_k(\psi_k(\infty))$. 

   We next verify the remaining properties from  Definition~\ref{defn:pseudoconjugacy}. To  do so,  
       we must define the metric $d_{X_k}$ on $X_k=\hat{T}$. We set
      \begin{align}\notag &L_k\colon [0,\infty] \to [M_k-1, \widetilde{M}_k+1];  \qquad
           x\mapsto \begin{cases}
             M_k-1+\frac{x}{M_k} & \text{if }x<M_k \\
             \widetilde{M}_k+1 - \frac{\widetilde{M}_k}{x} &   \text{if }x>\widetilde{M}_k \\
             x &\text{otherwise},\end{cases} \\ \notag
       &\tilde{L}_k\colon \hat{T}\to \C; \quad x+iy\mapsto L(x) + iy, \quad\text{and}\\
       &d_{X_k}(z,w) \defeq |\tilde{L}_k(z) -  \tilde{L}_k(w)|.\label{eqn:dXk} \end{align}

   \begin{prop}[Further pseudo-conjugacy conditions]\label{prop:pseudoconjugacyconditions}
     Let $k\geq 0$. Then $\psi_k\colon \hat{T}\to [0,1]$ is surjective. Furthermore, the following holds
      for all $\Delta\in \N$. 
    \begin{enumerate}[(a)]
    \item  If $z,\tilde{z}\in \hat{T}$ with $d_{Y_k}(\psi_k(z),\psi_k(\tilde{z})) \leq \Delta$, then 
       $d_{X_k}(z,\tilde{z})\leq \Delta/4 + 9$.\label{item:pseudosurjectivityofconstruction}
     \item If    $k\geq 1$ and $z,\tilde{z}\in \hat{T}$ with $d_{X_k}(z,\tilde{z})\leq \Delta$, then  
     $|\psi_{k_*}(f_k(z))- \psi_{k_*}(f_k(\tilde{z}))| \leq (\Delta+8)/\gamma_{k_*}$.%
        \label{item:pseudocontinuityofconstruction}
   \end{enumerate}
   \end{prop}
  \begin{proof}
       By the definition of $\phi_k$ in~\eqref{eqn:phikdefn}, $\psi_k$ is surjective. 
         
     To prove~\ref{item:pseudosurjectivityofconstruction}, let us assume without loss of  generality that
    $\psi_k(z)\leq \psi_k(\tilde{z})$. 
  Recall that the complementary intervals of $\Xi_k$ have Euclidean
 length $1/\gamma_k$, and hence unit length with respect to $d_{Y_k}$. 
  So the assumption implies that the interval $[\psi_k(z),\psi_k(\tilde{z})]$ 
      intersects at most  $\Delta+1$ complementary intervals of $\Xi_k$.  
      By~\ref{I5item:smallgaps} and the definition  of $d_{X_k}$, we have
       $d_{X_k}(\re z , \re \tilde{z}) \leq (\Delta+1)/4 + 2$. The claim 
      follows since $2\pi + 1/4 + 2 <9$. 

    For~\ref{item:pseudocontinuityofconstruction}, 
     observe from~\ref{I3item:Ukhyperbolicdistance} and~\eqref{eqn:etaj} 
     that, for $1\leq j \leq m_k-1$, 
    \[  \dist_{\HH}\bigl( \eta_k^j + 2\pi i s(k), \eta_k^{j+1} + 2\pi i s(k) \bigr)
      \geq 1, \]
    and in particular $\eta_k^{j+1} > \eta_k^j + 1$. 
  So the interval between $\re z$ and $\re \tilde{z}$ contains at most $\Delta$ 
  of the $\eta_k^j$. By the definition of
    $\phi_k$ in step I5, 
     the interval between $\phi_k(\re z)$ and  $\phi_k(\re \tilde{z})$ intersects at most
      $\Delta+1$ of the complementary intervals of $\Omega_k$. The image of each of these intervals by $g_k$ contains
      at most one element of $\Xi_{k_*}\setminus\{0,1\}$ by~\ref{I1item:Omegaimage}, so 
          \[ |g_k(\phi_k(\re z)) - g_k(\phi_k(\re \tilde{z}))| \leq (\Delta + 2)/\gamma_{k_*}.\]
   By Lemma~\ref{lem:pullingback}, it  follows that 
      \[ |\psi_{k_*}(f_k(z))- \psi_{k_*}(f_k(\tilde{z}))| \leq
              \frac{\Delta + 8}{\gamma_{k_*}}. \qedhere \]
\end{proof}

  The final piece of the puzzle is the expanding property of the maps $f_k$.  

\begin{prop}[Expanding system]\label{prop:fkexpanding}
  The system $(X_k,f_{k+1})_{k=0}^{\infty}$, equipped with the metrics $d_{X_k}$ as described above, is expanding in the sense of
   Definition~\ref{defn:expandingsystem}. 
\end{prop}
 \begin{proof}
  Let $k\geq 1$ and let $z,\tilde{z}\in \hat{T}$. Set $\Delta\defeq d_{X_k}(z,\tilde{z})$. 
  By Proposition~\ref{prop:pseudoconjugacyconditions}~\ref{item:pseudocontinuityofconstruction}, we have 
    \[ \lvert \psi_{k_*}(f_k(z))-\psi_{k_*}(f_k(\tilde{z}))\rvert \leq \frac{\lceil\Delta\rceil +8}{\gamma_{k_*}}. \]
    By~\ref{I5item:smallgaps} and the definition of $L_k$, we conclude that 
       \[ \lvert L_k(\re f_k(z)) - L_k(\re f_k(\tilde{z}))\rvert \leq \frac{\lceil \Delta\rceil + 8}{4} + 2, \]
   and thus 
     \[ d_{X_{k_*}}(f_k(z),f_k(\tilde{z})) < \frac{\lceil \Delta\rceil +8}{4} + 2 + 2\pi  \leq
                 \frac{\Delta}{4} + 11 \leq \frac{\max(\Delta,44)}{2}. \]
    This establishes part~\ref{item:expandinginversesystem} of  Definition~\ref{defn:expandingsystem}.

  By~\eqref{eqn:fkexpansion}, if $A\subset \overline{T}$ has Euclidean diameter at most $\Delta$,
    then \[ \diam_{X_k}(f_{k,\dots,j}(A))  \leq \diam(f_{k,\dots,j}(A)) \leq 2^{k-j}\cdot \Delta \] for all $k\geq j$. Recall that the metric  $d_{X_k}$ agrees with 
   the  Euclidean metric at real parts between $M_k$ and $\widetilde{M}_k$. So
   consider the sets
$A_k\defeq \{z\in\overline{T}\colon \re  z \leq M_k\}$   and $B_k\defeq \{z\in\overline{T}\colon \re  z \geq \widetilde{M}_k\}$.
   To establish part~\ref{item:expandingbackwardsshrinking} of Definition~\ref{defn:expandingsystem}, it remains only to show that, for each  $j$, the diameter of 
   $f_{k,\dots,j}(A_k)$ and  $f_{k,\dots,j}(B_k)$ tends to zero as $k\to\infty$.

   First let $k\geq 1$ and consider the  set $A_k$. Then
     $\phi_k(\re z)=0$ for all $z\in A_k$. By 
    Lemma~\ref{lem:pullingback}, $\re f_k(z)\geq M_{k_*}-1$ and 
    $\lvert \psi_{k_*}(f_k(z)) - \psi_{k_*}(f_k(z'))\rvert < 6/\gamma_k$ when $z,z'\in A_k$. By~\ref{I5item:smallgaps}, $\re f_k(z)$ and $\re f_k(z')$ differ by 
    at most $5/2$, and hence the Euclidean diameter of $f_k(A_k)$ is bounded
     by $5/2+2\pi$. The desired property follows from the previous observation. 

     Now  consider 
    $B_k$.  By continuity of $F$ at $\infty$, 
    it is sufficient to show that
     $f_{k,\dots,0}(B_k)\to\infty$ as $k\to\infty$. 
  Let $z\in B_k$, and set $x\defeq \re z \geq \widetilde{M}_k$. Define
    $z_k\defeq z$, $x_k\defeq x$, and, inductively,
 \begin{equation}\label{eqn:expandingproofzj}
     z_{j-1} \defeq f_j(z_j) \qquad\text{and}\qquad x_{j-1}\defeq F^{-(N_j+1)}(x_j) \qquad (1\leq j \leq k). 
 \end{equation}
  \begin{claim}
      $|\re  z_j - x_j|\leq 1$ and $x_j\geq \widetilde{M}_j$ for all $j=0,\dots,k$.
  \end{claim}
   \begin{subproof}
  This is trivial for $j=k$. Now assume that the claim holds for $j\in\{1,\dots,k\}$. First note that
     \begin{equation}\label{eqn:expandingproofinequality}  
           x_{j-1} \geq F^{-(N_j+1)}(\widetilde{M}_j) \geq F_j^{-(N_j+1)}(\widetilde{M}_j)- \frac{1}{2} \geq 
                   \widetilde{M}_{j_*} + 1/2 = 
                    \widetilde{M}_{j-1}+1/2 
     \end{equation}
   by~\ref{I2item:closeness} and~\eqref{eqn:Mtilde}. Furthermore,   
     $2\pi  s(j) \leq \widetilde{M}_{j}$ also by~\eqref{eqn:Mtilde}. So       
   \[
      |F^{-1}(x_j) - F^{-1}(z_j + 2\pi i s(j))| \leq
      |F^{-1}(x_j) - F^{-1}(x_j  + 2\pi  i s(j))| + \frac{\pi + 1}{2} \leq 
        5 \] 
    by Observation~\ref{obs:expansion}. Recalling that $(F^{-N_k})' \leq 1/12$ on $\overline{T}$ by~\eqref{eqn:Fkexpansion},
      we  are done. 
    \end{subproof} 

   In particular, using~\ref{I2item:closeness} and~\eqref{eqn:Mtilde} again, and recalling the definition of $n_k$ in~\eqref{eqn:nk}, 
    \begin{align*} \re f_{k,\dots,0}(z) &= \re z_0 \geq x_0 - 1 = F^{-n_k}(x_k) -1  \\ &\geq F^{-n_k}(\widetilde{M}_k)-1 
               \geq F_k^{-n_k}(\widetilde{M}_k) - 3/2  \geq k + 3/2 \to \infty. \qedhere \end{align*}
 \end{proof} 

\begin{proof}[Proof of Theorem~\ref{thm:arclikeexistenceBlog} and Proposition~\ref{prop:arclikeexistenceprecise}] 
   
   The system $(g_k)$ is expanding by con\-struc\-tion and Observation~\ref{obs:obtainingexpandingsystems},
   while $(f_k)$ is expanding by Proposition~\ref{prop:fkexpanding}. By Lemma~\ref{lem:pullingback} and
    Proposition~\ref{prop:pseudoconjugacyconditions}, the 
   inverse systems $(f_k)$ and $(g_k)$ are pseudo-conjugate via 
    the maps $\psi_k$.  

  Applying Proposition~\ref{prop:conjugacyprinciple}, we see that the 
   inverse limits $X=\invlim(X_k,f_{k+1})_{k=0}^{\infty}$ and $\uarrow{Y}=\invlim(Y_k,g_k)_{k=0}^{\infty}$ are
   homeomorphic. Moreover, since $\psi_k(\infty)=1$ for all $k$, the homeomorphism 
   between  the two  systems sends $\infty\mapsfrom\infty\mapsfrom\dots$ to $1\mapsfrom 1 \mapsfrom  \dots$
   by Observation~\ref{obs:convergingtotheconjugacy}. Also recall that $X$ is homeomorphic to the Julia continuum
   $\Jsh(\tilde{F})$, where $\tilde{F}$ is as in Remark~\ref{rmk:obtainingBlog} and 
    $\s$ is given by~\eqref{eqn:formofs}, via projection  to  the initial coordinate. 
   This proves the  first part of  
   Theorem~\ref{thm:arclikeexistenceBlog}; recall that the  second part will follow once we have established
   Proposition~\ref{prop:arclikeexistenceprecise}.  

  To do so, let $z\in J_{\s}$. Applying Lemma~\ref{lem:pullingback} to 
   $w=F^{n_{k+1}}(z)-2\pi i s(k+1)$, we see that $\re F^{n_k}(z) \geq M_k-1 \geq 9$. By Observation~\ref{obs:expansion}, 
    it follows that $\re F^{n}(z)\geq \re F^{n_k}(z)$ for $n_k\leq n < n_{k+1}$. This establishes
    part~\ref{item:largerthanMk} of Proposition~\ref{prop:arclikeexistenceprecise}. 

   To prove~\ref{item:controlonbehaviour}, let $h\colon \Jsh \to \uarrow{Y}$ denote the  homeomorphism whose existence follows from
     Proposition~\ref{prop:conjugacyprinciple}, and $h_k$ its $k$-th component. 
     Recall that the expansion constants $K=3$ and $\lambda=4$ for the system $(g_j)_{j=1}^{\infty}$ were chosen in~\ref{I5item:expansion}, and
     that the constant $M=3$ for the pseudo-conjugacy property~\ref{item:pseudoconjugacy} of Definition~\ref{defn:pseudoconjugacy} was obtained in 
     Lemma~\ref{lem:pullingback}. Hence we can estimate $h_k$ using~\eqref{eqn:conjugacycloseness}:
     \[ |h_k(z) - \psi_k(F^{n_k}(z))| \leq  \frac{4}{\gamma_{k_*}} \]
     for all $z\in \Jsh$ and all $k\geq 0$. 
     If $h_k(z)=0$, this implies by~\ref{I5item:smallgaps} that $\re F^{n_k}(z) < M_k+1$, as desired. 

  Finally, suppose that $\liminf \re F^{n_k}(z) - M_k  < \infty$.  Since $\psi_k(M_k)=0$, it follows from 
     Proposition~\ref{prop:pseudoconjugacyconditions}~\ref{item:pseudocontinuityofconstruction} that
    $\liminf \gamma_k \cdot \psi_k(F^{n_k}(z)-2\pi i s(k)) < \infty$. Note that  $\gamma_k\to\infty$: Indeed, by definition
     the diameter
    of the spaces in an expanding inverse system with non-degenerate inverse limit cannot remain bounded. It follows that 
    $\liminf h_k(z) = 0$, as desired. 
\end{proof}

\begin{rmk}[Number of  tracts]\label{rmk:infinitelymanytracts}
  We have chosen to carry out our construction with a function having a single tract (up  to  translations
    by multiples of  $2\pi  i$). Some of the technicalities of the construction could be simplified 
    by allowing an infinite 
    number of  tracts. Indeed, we would then introduce a new tract at each stage, having a similar shape as our
    domains $U_k$, removing the need to modify these tracts and their functions at later  stages. This
    means those arguments that ensure that the iterates of $F$ are sufficiently
    closely approximated by those of $F_{k-1}$ (such  as in the choice of  $R_k$) can be omitted. 
    At each stage of the construction, we would still use $N_k$ iterates in an auxiliary tract, 
   which can be chosen just as  the strip  $S$,  to 
    transport the interval  $[M_k,\widetilde{M}_k]$ sufficiently far to  the right.

  The latter can be avoided if we are content with constructing Julia continua
    that escape to infinity uniformly, allowing a further simplification in which we are  left only  with a 
    sequence of tracts  $T_k$, one for each  function  $g_k$, and the adddress  $T_0 T_1 T_2 \dots$.   
    However, this 
    would prevent us from realising all Julia continua in a single function as in Theorem~\ref{thm:mainarclike}.
    Moreover, proving
    Theorem~\ref{thm:nonuniform} and its strengthening, Theorem~\ref{thm:onepointuniform}, which we shall do in the
    next section, 
    necessarily requires the construction of non-uniformly escaping Julia continua. 
\end{rmk} 

\section{Applications of the construction: point {\escapingcomposant}s and non-escaping points} \label{sec:onepointuniform}

 We now proceed to some applications of Theorem \ref{thm:arclikeexistenceBlog} and Proposition~\ref{prop:arclikeexistenceprecise}.

\begin{proof}[Proof of Theorems~\ref{thm:nonuniform} and~\ref{thm:onepointuniform}]
  Recall  that Theorem~\ref{thm:onepointuniform} implies Theorem~\ref{thm:nonuniform}. Furthermore,
    the second sentence in part~\ref{item:onepointuniform_escaping} of Theorem~\ref{thm:onepointuniform}
    follows from the remainder
    of the theorem, using 
   Theorem~\ref{thm:composants}. 
     (Alternatively, this property can  also be  easily deduced directly from the construction below.) 

  Consider the function
    \[ g\colon [0,1]\to[0,1]; \qquad x\mapsto \begin{cases}
                                                                  2x-1 &\text{if } x\geq \frac{1}{2} \\
                                                                  \frac{1-2x}{4} &\text{if } x < \frac{1}{2}. \end{cases} \] 
   (Compare Figure~\ref{fig:onepointuniform}.) By~\cite[Section~1.2, Example~12]{ingraminverselimits}, $\uarrow{Y}\defeq\invlim g$ is 
    homeomorphic to an arc. Furthermore, suppose that $x_0\mapsfrom x_1 \mapsfrom x_2 \mapsfrom \dots$ is a point
    in this inverse limit with $x_0>1/4$; then $x_n\to 1$ as $n\to\infty$. As $g^{-1}(0)=\{1/2\}$, this implies that
     \begin{equation}\label{eqn:liminfnonuniform} \liminf_{n\to\infty} x_n > 0 \end{equation}
    for all inverse orbits $(x_n)_{n=0}^{\infty}\in \uarrow{Y}$.

   Let $F$, $\s$ and $n_k$ be as in Proposition~\ref{prop:arclikeexistenceprecise}, taking $M_k=10$ for all $k$. 
     Consider the Julia continuum $\CH\defeq \Jsh(F)$. 
     Then $C\subset I(F)$ by~\eqref{eqn:liminfnonuniform} and Proposition~\ref{prop:arclikeexistenceprecise}~\ref{item:controlonbehaviour}. 
    On the other hand, 
     \[ \liminf_{n\to\infty} \min_{z\in C} \re F^n(z) = \liminf_{k\to\infty} \min_{z\in C} \re F^{n_k}(z) \leq 11 \]
    by choice of $M_k$ and Proposition~\ref{prop:arclikeexistenceprecise}. 

    We now apply Theorem~\ref{thm:realization} to obtain a transcendental entire function $f\in \B$ topologically conjugate
      to the function $\exp(z)\mapsto \exp(F(z))$. As desired, the image of $\exp(\CH)$ under this conjugacy is a Julia continuum of $f$ that
     is homeomorphic to an arc, contains only escaping points, but on which the iterates do not escape to infinity uniformly. 
   \end{proof}

\begin{figure}
 \hfill\subfloat[Theorem \ref{thm:onepointuniform}]{\includegraphics[height=.25\textheight]{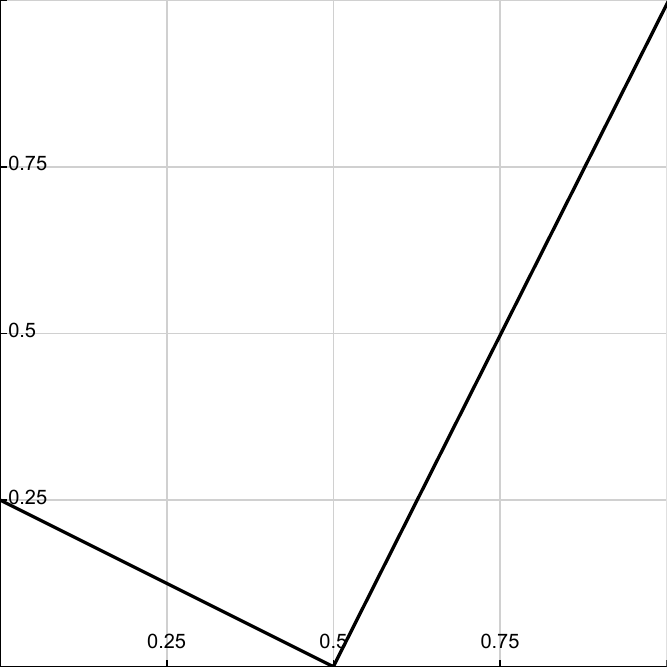}}\hfill%
 \subfloat[Proposition \ref{prop:cantornonescaping}]{\includegraphics[height=.25\textheight]{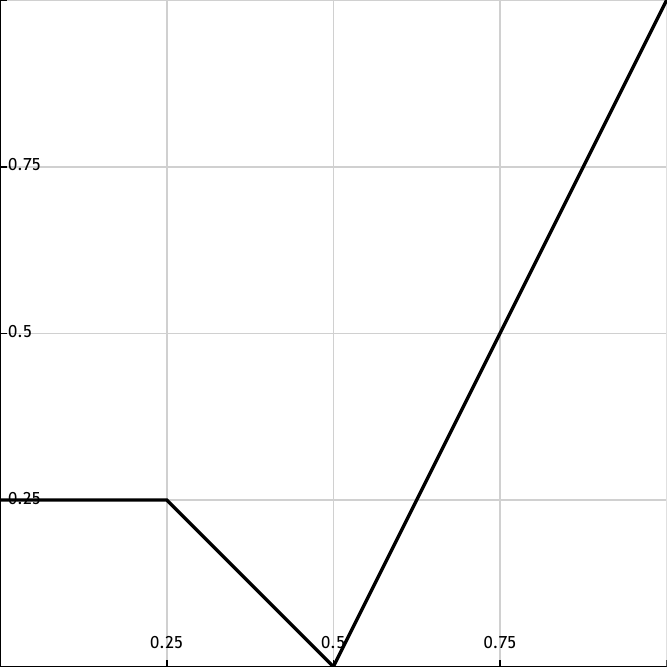}}\hfill\mbox{}
 \caption{The graph of the function $g$ from the proof of Theorem \ref{thm:onepointuniform} (left), and the function $g_a$
    from the proof of Proposition \ref{prop:cantornonescaping}, with $a=1/4$.\label{fig:onepointuniform}}
\end{figure}

We now complete the proof of Theorem \ref{thm:nonescapingaccessible}. Recall that the 
  first half of the theorem, concerning the terminal nature of non-escaping and accessible points, 
   was established in Theorems \ref{thm:nonescapingterminal} and \ref{thm:accessibleterminal}.
   Also recall that the set of non-escaping points in a given Julia continuum has Hausdorff dimension zero by 
   Proposition \ref{prop:nonescapinghausdorff}. Hence it remains to prove that 
   there is a disjoint-type entire function $f$ having a Julia continuum $\CH$ such that
   the set of non-escaping points in $\CH$ is a Cantor set, and another function
   having a Julia continuum containing a dense $G_{\delta}$ 
   set of non-escaping points. Both results 
   will be proved using Proposition \ref{prop:arclikeexistenceprecise}. 

To prove the second statement, we shall use the following general topological fact.
\begin{prop} \label{prop:countablymanyterminalpoints}
  Let $X$ be an arc-like continuum. Suppose that $x$ is a terminal point of $X$, and that
    $E$ is a finite or countable set of terminal points
     such that $X$ is irreducible between $x$ and $e$ for every $e\in E$.  

  Then there is a sequence $g_j\colon [0,1]\to [0,1]$ of continuous and surjective functions such that
    $\uarrow{Y} \defeq  \invlim (g_j)$ is homeomorphic to $X$, in such a way that $x$ corresponds to the
    point $1\mapsfrom 1 \mapsfrom \dots$ and such that every point of $E$ corresponds to
    a point $x_0\mapsfrom x_1 \mapsfrom \dots$ with $x_j=0$ for infinitely many $j$. 
\end{prop}
\begin{proof}
  This follows from the proof of 
   \cite[Theorem 12.19]{continuumtheory}, similarly as in the proof of Proposition     
   \ref{prop:arclikecharacterization}. We omit the details.
\end{proof}

To construct a Cantor set of non-escaping points, we make an explicit inverse limit construction.

\begin{prop}[Cantor sets in inverse limit spaces]\label{prop:cantornonescaping}
  There exists a sequence $(g_n)_{n\geq 1}$ of surjective continuous maps $g_n\colon [0,1]\to[0,1]$,
    each fixing $1$,
    such that the inverse limit $\uarrow{Y} \defeq  \invlim (g_n)$ has the following properties.
  \begin{enumerate}[(a)]
    \item The set $A \defeq  \{(x_n)_{n=0}^{\infty}\in \uarrow{Y}\colon  x_n=0 \text{ for infinitely many $n$}\}$ is a 
        Cantor set;\label{item:cantorA}
    \item every $(x_n)_{n=0}^{\infty}\in \uarrow{Y}\setminus A$ satisfies $\liminf_{n\to\infty} x_n > 0$.\label{item:liminfA}
  \end{enumerate}
\end{prop}
\begin{proof}
  For $a\in [0,1)$, let us define
       \[ g_a\colon [0,1]\to[0,1]; \qquad x\mapsto
             \begin{cases} 
                 a                  &\text{if } x\leq \frac{1}{4} \\
                 a\cdot(2 - 4x)          &\text{if } \frac{1}{4} < x \leq \frac{1}{2} \\
                 2x - 1          &\text{if } \frac{1}{2} < x \leq 1. \end{cases}. \]
   Note that $g_a(1-2^{-j})=g_a(1-2^{-(j-1)})$ for $j\geq 1$. The desired sequence of maps is given by 
    $g_n \defeq  g_{a_{n}}$, where $a_1=0$ and 
      \begin{equation}\label{eqn:an} a_n \defeq 1 - \frac{1}{2^{\left\lfloor \frac{n-1}{2}\right\rfloor}} \quad (n\geq 2). \end{equation}

  For $n\geq 2$, set 
      \[ k(n)\defeq n-1 -  \lfloor (n-1)/2\rfloor = \lceil (n-1)/2 \rceil = \lfloor n/2\rfloor. \]
     By~\eqref{eqn:an}, if   $k(n) \leq \tilde{k}< n$, then 
      \[ g_{n ,\dots, \tilde{k}}(0)= g_{(n-1),\dots, \tilde{k}}(a_n) = 1 - 2^{k(n)-\tilde{k}}. \] 
In particular, 
     $g_{n,\dots,k(n)}(0)=0$ and $g_{n,\dots,\tilde{k}}(0)\neq 0$ for 
      $k < \tilde{k}< n$.

 Now consider $\uarrow{Y} =  \invlim(g_n)$ and the set  
   $A$ from the statement of the proposition. 
   Let $(x_n)_{n=0}^{\infty}\in A$ and $k\geq 1$, and let $j$ be minimal with $x_j=0$ and $j\geq 2^k$. 
    Then, by construction,  $x_{k(j)}=x_{\lfloor j/2 \rfloor} = 0$, and hence 
    $j < 2^{k+1}$. So
        \[ A = \bigcap_{k\geq 1} 
                 \bigcup_{j= 2^k}^{2^{k+1}-1} \{ (x_n)_{n=0}^{\infty}\in \uarrow{Y} \colon x_j=0\}, \]
     and  $A$ is compact as an intersection  of compact sets. 
     Furthermore, $A$ is totally disconnected, as its projection to the $k$-th coordinate is countable for all $k$. 

  Fix $(x_n)_{n=0}^{\infty}\in A$, and let $n_0\geq 1$ such that  $x_{n_0}=0$. 
   There are exactly two values of $n$ with
   $k(n)=n_0$, and hence $A$ contains at least one other point $(y_n)_{n=0}^{\infty}$ with $y_{n_0}=0$. The distance between
    $(x_n)_{n=0}^{\infty}$ and $(y_n)_{n=0}^{\infty}$ tends to zero as  $n_0\to\infty$, so $(x_n)_{n=0}^{\infty}$ is not an isolated point. 
    Thus we have proved that $A$ is a Cantor set, 
    establishing~\ref{item:cantorA}. 

   Finally, suppose that $(x_n)_{n=0}^{\infty}\in \uarrow{Y}$ and
   $\liminf_{n\to\infty} x_n = 0$. Then there are infinitely many values of $n$ for which
    $x_n < 1/4$, and hence $x_{n-1} = g_n(x_n)=a_n=g_n(0)$. So
      \[ x_{k(n)} = g_{n,\dots,k(n)}(x_n) = g_{n,\dots,k(n)}(0) = 0. \]
     Since $k(n)\to\infty$ as $n\to\infty$, we see that $(x_n)_{n=0}^{\infty}\in A$, establishing claim~\ref{item:liminfA}.
\end{proof}

 \begin{cor} \label{cor:densenonescaping}
   There exists a disjoint-type entire function $f$ having a Julia continuum $\CH$ such that
    the set of non-escaping points in $\CH$ contains a dense $G_{\delta}$ subset of $\CH$.

  There also exists a disjoint-type entire function $f$ having a Julia continuum $\CH$ such that
    the set of non-escaping points in $\CH$ is a Cantor set. 
 \end{cor}
\begin{remark}
 This completes the proof of Theorem \ref{thm:nonescapingaccessible}.
\end{remark}
\begin{proof}
  Let $Y$ be an arc-like continuum containing a terminal point $x$ and a dense countable set $E$ of terminal 
    points such that $Y$ is irreducible between $x$ and $e$ for all $e\in E$. Apply 
    Proposition~\ref{prop:countablymanyterminalpoints} to obtain an inverse limit representation of $Y$ where
    the set $E$ consists of points having infinitely many coordinates equal to $0$. 
    By Proposition~\ref{prop:arclikeexistenceprecise}, there is a disjoint-type 
     $F\in\BlogP$ for which there exist a Julia continuum $\Jsh(F)$ and
     a homeomorphism $\theta\colon Y\to \Jsh(F)$ such that the orbit of
     any element of $\theta(E)$ contains infinitely many points at real part less than $12$. The set of points having the latter property is a $G_{\delta}$ subset of $\Jsh(F)$, which is dense in 
     $\Jsh(F)$ since it contains $\theta(E)$. An application of Theorem~\ref{thm:realization} now yields
     the  first claim.  

  (An example of a continuum $Y$ with the desired property is given by the pseudo-arc. Indeed, since 
    every point is terminal, we can simply select a countably dense subset $E$ of one composant, and choose $x$ in some other composant.
   We remark that it is also straightforward to construct an inverse limit $Y$ with the desired properties directly, without using 
   Proposition \ref{prop:countablymanyterminalpoints}.)

  The second claim follows in analogous manner, using the inverse limit  
    from Proposition \ref{prop:cantornonescaping}.
\end{proof}
\begin{remark}
  A continuum containing a dense set of terminal points
   must necessarily be either indecomposable or the union of two
   indecomposable continua \cite{densesetofendpoints}. 
\end{remark}

\section{Realising all arc-like continua by a single function}
 \label{sec:allinone}

 In this section, we complete the proof of Theorem \ref{thm:mainarclike}, by showing that all the arc-like continua
   in question can be realised by a single function.

 \begin{thm}[A universal Julia set for arc-like continua]\label{thm:allinone}
    There exists a disjoint-type function $F\in\BlogP$ with bounded slope and the following property.

   Let $Y$ be an arc-like continuum having a terminal point $y_1$, Then there is a Julia continum
      $\Jsh(F)$ such  that $\Jsh(F)$ is homeomorphic to $Y$, with $\infty$ corresponding to $y_1$. 
 \end{thm}
 \begin{proof}
  By Proposition \ref{prop:countable}, there is a countable set $\G$ of functions $g\colon [0,1]\to [0,1]$ with 
    $g(1)=1$ such that every arc-like
    continuum with a terminal point can be written as an inverse limit with bonding maps in $\G$. 

  Let $(S^k)_{k\geq 1}$ be an enumeration of all finite 
       sequences of maps in $\G$;
        that is, $k\mapsto S^k$ gives a bijection from $\N$ to the set 
        $\bigcup_{\ell \geq 1} \G^{\ell}$. 
    Let us write 
     $S^k = (g^k_0 , g^k_1 , \dots , g^k_{\ell_k})$, with $g^k_j\in \G$,  for $k\geq 1$,
      and let 
     us also write $\kappa$ for the inverse of this bijection; that is,  $\kappa(S^k)=k$.
     We may assume that no sequence $S$ appears before any of its prefixes in the enumeration. In other words, 
     denote by $\sigma(S)$ the sequence obtained from $S$ by forgetting the final entry. Then 
     $\kappa(\sigma(S)) < \kappa(S)$. Here we use the convention that $\kappa(\sigma(S))=0$ if $S$ has length $1$.   
  
   We now set $g_k \defeq g^k_{\ell_k}$,
      \begin{equation}\label{eqn:kstarallinone} k_*\defeq \kappa(\sigma(S^k)) \end{equation}
    and  $M_k\defeq 10$ 
   for all $k$. We  then carry out exactly the same construction as in the 
    proof of Theorem~\ref{thm:arclikeexistenceBlog}, but using~\eqref{eqn:kstarallinone} instead of
    $k_*=k-1$ as in~\eqref{eqn:kstar}, and defining $n_k$ inductively by $n_0\defeq  0$ and
       \[ n_k  \defeq n_{k_*} + N_k + 1,  \]
     instead of using~\eqref{eqn:nk}. This yields a tract $T$, a  function $F$ (with 
     associated $\tilde{F}\in\BlogP$), and sequences   $N_k$ and $s(k)$ of integers. 

  If $S$ is an infinite sequence in  $\G$, let $S(j)$ be the finite sequence consisting of the first $j$ entries of
    $S$,  and  set $k_j\defeq \kappa(S(j))$. Then  $S =  (g_{k_j})_{j\geq 1}$, and $k_j=(k_{j+1})_*$ for all $j$.  
    Define an external address of $\tilde{F}$ by  
      \[  \s \defeq \s(S) \defeq 0^{N_{k_1}} s(k_1) 0^{N_{k_2}} s(k_2) \dots. \]
    The proof of Theorem~\ref{thm:arclikeexistenceBlog} now goes through as before, with a small number
    of obvious adjustments to notation  in the proofs of Proposition~\ref{prop:fkexpanding} and Theorem~\ref{thm:arclikeexistenceBlog}.  For example, we should   
     replace~\eqref{eqn:expandingproofzj} by 
     \begin{equation}\label{eqn:expandingproofzjstar}
        z_{j_*} \defeq f_j(z_j) \qquad\text{and}\qquad x_{j_*}\defeq F^{-(N_j+1)}(x_j).
     \end{equation}

   We conclude that
     $\Jsh(\tilde{F})$ is homeomorphic to $\invlim (g_{k_j})_{j=1}^{\infty}$. 
      As every Julia continuum
     with a terminal point arises in this way, the proof is complete.
 \end{proof}

\begin{proof}[Proof of Theorem \ref{thm:mainarclike}]
  The first part of the theorem was established in Theorem~\ref{thm:infty} and Proposition~\ref{prop:arcliketracts}. 
  The second part of the theorem follows from Theorem \ref{thm:allinone}, again combined with Theorem \ref{thm:realization} to realise the
   example by a disjoint-type entire function. Recall from Remark~\ref{rmk:preserved} that the bounded-slope property is preserved by the approximation
   in Theorem~\ref{thm:realization}. 
\end{proof}

\section{Bounded-address Julia continua}\label{sec:boundedexistence}
  
  We now turn our attention to constructing Julia continua at bounded external addresses with prescribed
    topology. Recall from Proposition~\ref{prop:boundedorbits} that every such Julia continuum contains a unique point with bounded orbit. 

\begin{thm}[Prescribed bounded-address Julia continua] \label{thm:boundedexistenceBlog}
   Let $Y$ be an arc-like continuum containing two terminal points $y_0$ and $y_1$ such that $Y$ is
     irreducible between $y_0$ and $y_1$.
   Then there exists a disjoint-type function $F\in\BlogP$, having bounded slope and bounded decorations, 
    and a bounded external address
    of $F$ such that the Julia continuum
    $\Jsh(F)$ is homeomorphic to $Y$. Under this homeomorphism, the unique point in $\Jsh(F)$  that has
     a bounded orbit corresponds to $y_0$, while $\infty$ corresponds to $y_1$.  
\end{thm}

 The construction is very similar to the proof of Theorem \ref{thm:arclikeexistenceBlog}.  
   In order to construct a bounded address, we cannot, however, use 
   ``side channels'' of a single tract as in Section~\ref{sec:arclikeexistence}. 
  Hence we instead construct a function having exactly 
   two tracts $S$  and $T$ (modulo $2\pi i\Z$). Corollary~\ref{cor:boundedhomeomorphic} shows 
   that this is indeed necessary.  

  Recall from Proposition~\ref{prop:arclikecharacterization} that $Y$ is homeomorphic
   to an inverse limit $\uarrow{Y} = \invlim (g_k)_{k=1}^{\infty}$ of functions
     $g_k\colon [0,1]\to [0,1]$ with  $g_k(0)=0$ and  $g_k(1)=1$. Under the
     homeomorphism, 
     $y_0$ corresponds to the point $(0\mapsfrom 0 \mapsfrom \dots)$ and
     $y_1$ to the point $(1\mapsfrom 1\mapsfrom \dots)$. 
     We fix this sequence for the remainder of the construction.

 \subsection*{Description of the tracts $S$ and $T$}
   Both tracts are contained in the half-strip $\{x+iy\colon x>2\text{ and } -\pi/2<y<3\pi/2\}$. Similarly as in
   Section~\ref{sec:arclikeexistence}, we shall not require that the tracts are Jordan domains, or that
   their closures are disjoint. Once the conformal isomorphisms $F_S\colon S\to \HH$ and $F_T\colon T\to\HH$
    are defined, we obtain a disjoint-type function $\tilde{F}$ by restriction and periodic extension as in 
   Remark~\ref{rmk:obtainingBlog}.

\begin{figure}
 \begin{center}
   \def\svgwidth{\textwidth}
   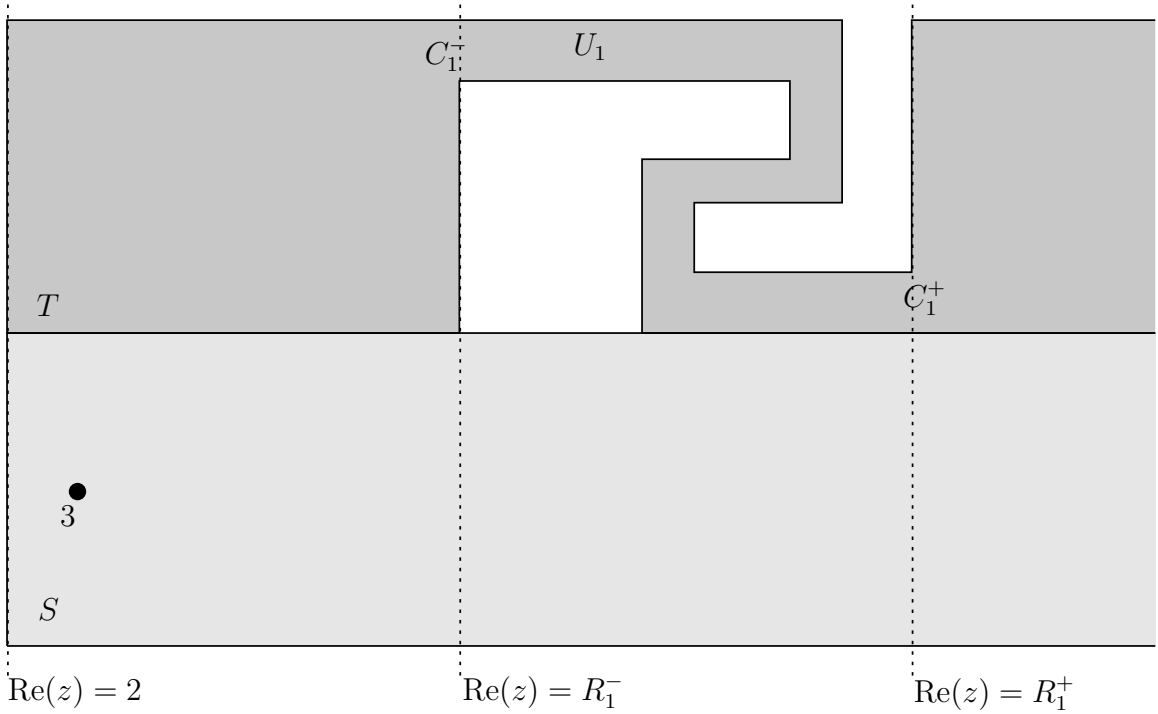
 \end{center}
 \caption{The tracts $S$ and $T$.\label{fig:tract_definition-bounded}}
\end{figure}

   The first of the two tracts is the half-strip
    \[ S \defeq  \{ x + iy\colon  x>4\text{ and } |y|<\pi/2\}, \]
   together with the conformal isomorphism
   $F_S\colon S\to \HH$ with $F_S(5)=5$ and $F_S'(5)>0$. This tract plays the
    same role
    as the central strip of the tract from Section \ref{sec:arclikeexistence}. 

    The second
    tract,
      \[ T\subset \{x + iy \colon  x>4\text{ and } \pi/2<y<3\pi/2\} \]
    is determined by numbers $(R_k^-)_{k=1}^{\infty}$ and $(R_k^+)_{k=1}^{\infty}$
    such that 
     $6\leq R_k^- < R_k^+ < R_{k+1}^-$ for all $k\geq 1$, simply-connected domains 
     \[ U_k \subset \{ x + iy \colon  R_k^- < x < R_k^+\text{ and } \pi/2<y<3\pi/2\} \]
    and arcs 
    $C_k^-$ and $C_k^+$, with 
   \[ C_k^{\pm} \subset \partial U_k \cap \{x+ iy \colon  x = R_k^{\pm}\}. \]
   (See Figure \ref{fig:tract_definition-bounded}.) The tract $T$ is defined as 
 \[   T \defeq  \bigcup_{k\geq 1} \{x + iy\colon   R_{k-1}^+ < x < R_k^- \text{ and }\pi/2<y<3\pi/2\} 
            \cup \interior(C_k^-) \cup U_k \cup \interior(C_k^+).  \]
   (Here we use the convention that $R_0^+ = 4$.) The function $F_T\colon  T\to\HH$ is chosen such that
   $F_T(5+ \pi i)=5$ and $F_T(\infty)=\infty$. 

 The choices of $R_k^{\pm}$, $U_k$ and $C_k^{\pm}$ will again be specified in a recursive construction.  
  For this purpose we define partial tracts 
 \begin{align}  \label{eqn:Tkbounded} T_k \defeq  &\left(\bigcup_{j= 1}^k \{x + iy\colon   R_{j-1}^+ < x < R_j^-\text{ and } \pi/2<y<3\pi/2\} 
            \cup \interior(C_j^-) \cup U_j \cup \interior(C_j^+)\right) \\ &\cup \{x + iy\colon  x > R_k^+\text{ and } \pi/2<y<3\pi/2\},\notag  \end{align}
   and associated conformal isomorphisms $F_k\colon T_k \to \HH$, again
   normalised such that $F_k(5+\pi i)=5$ and $F_k(\infty) = \infty$. 

The expansion properties from Observation~\ref{obs:expansion} hold for each of the maps
  $F_S$, $F_T$ and $F_k$.

\begin{obs}[Expansion of the maps $F$ and $F_k$] \label{obs:expansion2}
 Let $G$ be one of the maps $F_S$, $F_T$ and $F_k$, where $F_T$ and $F_k$ are of 
 the form discussed above, and denote its domain by $T(G)$. 
 Then 
  Then $|G'(z)|\geq \re G(z)/2$ for all $z\in T(G)$.

  In particular, $|G'(z)|\geq 2$ whenever $\re G(z)\geq 4$. 
    In addition, let $z\in T(G)$ with $G(z)\in \overline{S}$ and $\re z \geq 9$; 
    then $\re G(z) > \re z$. Similarly, if $z\in S$, $\lvert \im F_S(z)\rvert \leq 2\pi$ and
    $\re z\geq 9$, then $\re F_S(z)\geq \re z$.
\end{obs}
\begin{proof}
 Everything except the final statement follows from
  Observation~\ref{obs:expansion} (note that this observation applies 
   to the 
   pre-composition of $F_T$ resp.\ $F_k$
  with translation by $\pi i$). 
  
   The final statement follows from the explicit formula 
     $F_S(z) = \frac{5}{\sinh(1)} \cdot \sinh(z-4)$:
      \begin{align*} \re F_S(z)  &\geq  \lvert F_S(z)\rvert - 2\pi \geq 
            4 \lvert \sinh(z-4) \rvert - 2\pi \\ & \geq 
             4\frac{e^{\re z-4}-e^{-(\re z-4)}}{2} - 2\pi 
             > 2e^{\re z-4}  - 7 > 2\re z - 7 > \re z, \qedhere \end{align*} 
    where we used the fact that $e^{x-4} > x$ for $x>6$. 
\end{proof} 
    
An analogue of Observation~\ref{obs:caratheodoryconvergence} also holds
  for the convergence of $F_k$ to $F_T$: 

\begin{prop}\label{prop:caratheodoryconvergence2}
  Let $k\geq 0$, and fix a partial tract $T_k$ as in~\eqref{eqn:Tkbounded},
   along with the associated function $F_k$. 
     For every $\eps>0$ and all $M>1$, there is $R^-(T_k,\eps,M)$ with the
following property.

   Suppose that the tract $T$ is chosen as above, subject only
    to the constraints that $T_k$ is the $k$-th partial tract of $T$, and that
     $R_{k+1}^- \geq R^-(T_k,\eps,M)$. Then
   \[ \lvert F_T^{-1}(z) - F_k^{-1}(z) \rvert \leq \eps \]
   for all $z$ with $1\leq \re z \leq M$ and $\lvert \im z\rvert \leq 2\pi$. 
\end{prop} 
\begin{proof}
 Denote by $\F(F_k)$ the collection of all functions
  $F\colon T\to \HH$ as above such that $T_k$ 
   is the 
   $k$-th partial tract of $T=T(F)$. 
  As in Observation~\ref{obs:caratheodoryconvergence}, as $F\in\mathcal{F}(F_k)$ 
    varies in such way that 
    $R_{k+1}^-=R_{k+1}^-(F)\to\infty$, we have $T(F)\to T_k$ 
in the sense of Carath\'eodory kernel convergence, with
respect to the base point $z_0\defeq 5 + \pi i$. By the 
    Carath\'eodory kernel convergence theorem, we conclude that the corresponding  
     functions $F$ satisfy
    \begin{equation}\label{eqn:caratheodory-convergence}
      F^{-1} \circ \alpha_F \to F_k^{-1},\end{equation} where 
    $\alpha_F$ is the unique M\"obius automorphism of $\HH$ that satisfies
    $M(5)=5$ and 
     $\arg\bigl(F^{-1}\circ\alpha_F)'(5)\bigr)=\arg\bigl((F_K^{-1})'(5)\bigr)$. 
     (The tracts $T$ and $T_k$ are not symmetric with
     respect to the real axis, so we cannot conclude that $\alpha_f=\id$,
      in contrast to Observation~\ref{obs:caratheodoryconvergence}.) 
   
   We must show that $\alpha_F\to \id$ as $R_{k+1}^-(F)\to\infty$; 
   compare~\cite[Proposition~8.2]{pseudoarcs} for a similar argument. 
   Fix some $R$ much greater than $R_k^+(F_k)$. If $R_{k+1}^-(F) > R$, then 
      \[ Q\defeq \{x+iy\colon R_k^+(F_k) < x < R-1 \text{ and } \pi/2 < y < 3\pi/2\} \]
    is a quadrilateral of large modulus (with horizontal sides on $\partial T$) that 
     separates $z_0$ from 
      $\infty$ and from any point of $T$ of real part greater than $R-1$. 
    Applying the Teichm\"uller modulus theorem, we see that 
       \[ \inf \{ \lvert F(z)\rvert \colon F\in\F(F_k), R^{-}_{k+1}(F)\geq R, 
               z\in T(F) \text{ and }\re z \geq R-1 \}\rvert \to \infty \]
       as $R\to \infty$. By~\eqref{eqn:caratheodory-convergence}, 
       if $R_k^-(F)$ is sufficiently large (depending on $R$), then
       \[ \re F^{-1}(\alpha_F(F_k(R))) > R-1. \]
        Hence the M\"obius transformation $\alpha_F$ maps the point $F_k(R)$, which has
       large modulus, to 
       another point of large modulus. It follows that $\alpha_F$ is close to the identity.
        Letting $R\to\infty$ yields
       the desired conclusion. 
\end{proof} 

  Our goal (in order to prove Theorem~\ref{thm:boundedexistenceBlog}) is to construct $T$ (and hence $\tilde{F}$) in such away that 
   $J_{\s}(\tilde{F})$ is homeomorphic to $\uarrow{Y}$. This address will be of the form
     \begin{equation}\label{eqn:formofsbounded} \s = S^{N_1} T S^{N_2} T S^{N_3} T \dots \end{equation}
    for some sequence of positive integers $(N_k)_{k=1}^{\infty}$. We define $n_k$ as in~\eqref{eqn:nk}; i.e.,
     $n_k \defeq k + \sum_{j=1}^k N_j$. The inverse system $(f_k)_{k=1}^{\infty}$, 
     which 
     we will construct to be pseudo-conjugate to the system $(g_k)_{k=1}^{\infty}$,
     is thus given by the maps 
     \[ f_k\colon \hat{S}\to \hat{S}; \qquad z\mapsto F_S^{-N_k}(F_T^{-1}(z)), \]
    where $\hat{S}=\overline{S}\cup\{\infty\}$.

   We construct 
   the tract $T$ recursively, by almost the same 
    process as in the proof of Theorem~\ref{thm:arclikeexistenceBlog}.
    In particular, the proof also constructs the same objects as summarised
    in~\ref{item:sandNk} to~\ref{item:psidefinition} at the beginning of 
    the recursive construction in Section~\ref{sec:arclikeexistence}, with the same properties
    as stated there,
    with only the following modifications:
  \begin{enumerate}
   \item[(a')] There is no need to construct an address entry 
      $s(k)$ in our construction. Indeed,~\ref{eqn:formofsbounded}
      means that $\s$ is uniquely determined by the sequence $(N_k)_{k\geq 1}$.
    \item[(e')] 
      The function $\phi_k$ still satisfes~\eqref{eqn:phiconditions}, 
        but the number $M_k\geq 10$ can no longer be freely chosen; instead, we use 
        $M_k\defeq 10$ for all $k$. Although this number no longer depends on $k$, we retain the 
        notation $M_k$ in analogy with Section~\ref{sec:arclikeexistence}. 
   \end{enumerate}

     We now describe the recursive construction, noting places where adjustments
      to the conditions from 
     Section~\ref{sec:arclikeexistence} are necessary; see also 
     Figure~\ref{fig:bounded_tract_construction}. 
     Begin by setting $\widetilde{M}_0\defeq 11$, 
     $\phi_0(x) = x-10$ for $M_0 =10 \leq x \leq  \widetilde{M}_0$ and satisfying~\eqref{eqn:phiconditions} elsewhere, and $\gamma_0\defeq 5$.  The number  $k_*\leq k-1$, for $k\geq 1$, is used in  exactly the same way as in Section~\ref{sec:arclikeexistence}. That is,
   $k_*=k-1$ for the purpose of the proof of Theorem~\ref{thm:boundedexistenceBlog}, but a different choice
    is used in the proof of Theorem~\ref{thm:boundedexistence}.

\begin{figure}
 \begin{center}
   \def\svgwidth{.95\textwidth}
   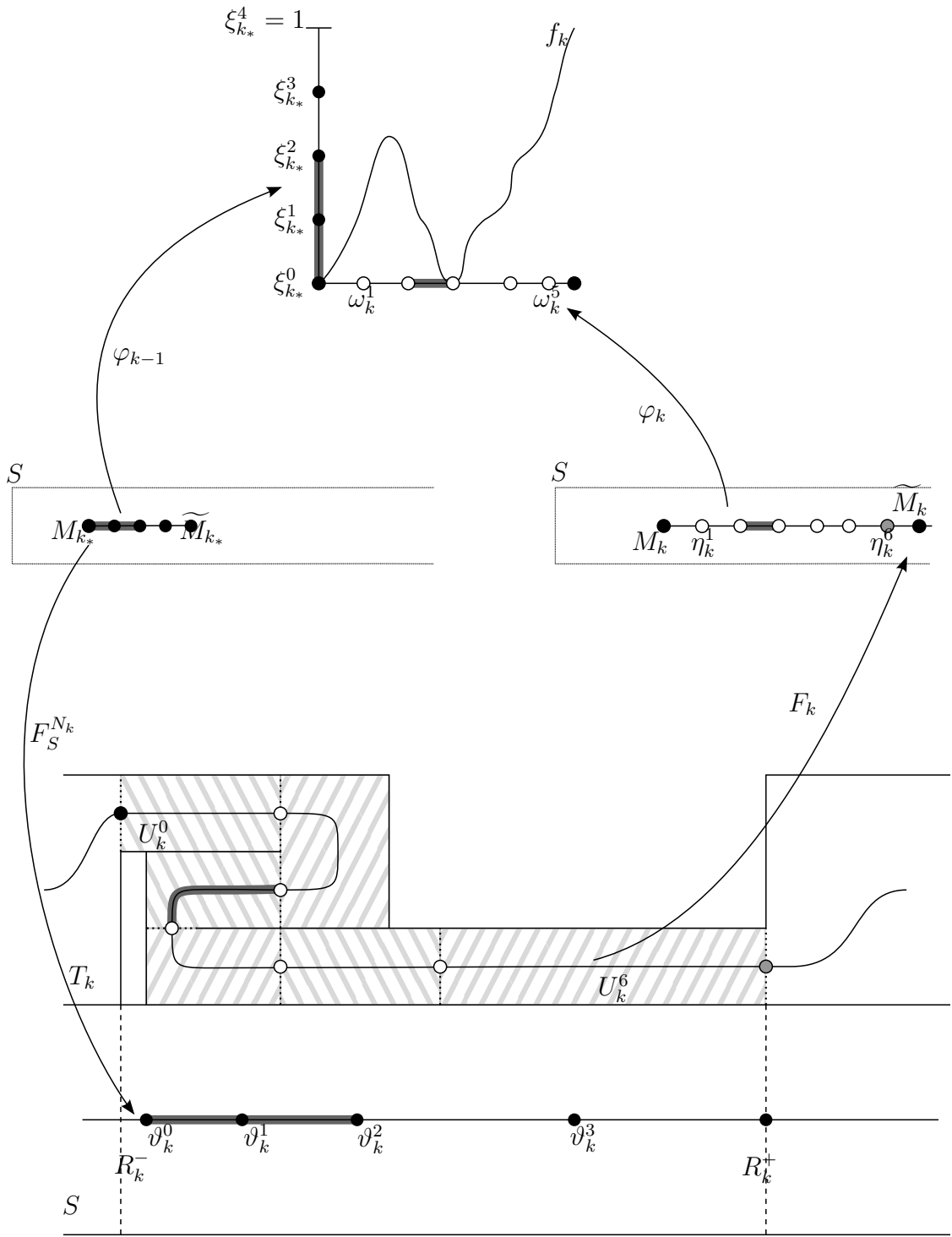
 \end{center}
 \caption{Construction of the tracts for Theorem \ref{thm:boundedexistenceBlog}.\label{fig:bounded_tract_construction}}
\end{figure}

    \textbf{I1.}  Step I1 is the  same as in Section~\ref{sec:arclikeexistence}. 
   
   \textbf{I2.} In step I2, we first choose $N_k$ sufficiently large,
      and then fix $R_k^-$ and $R_k^+$.

     To choose $N_k$, we define $\alpha_{k_*}$ as in~\eqref{eqn:alpha}
      and require~\eqref{eqn:Nk} to hold. Condition~\eqref{eqn:Nkright} becomes
     \begin{equation} \label{eqn:Nkrightbounded}
       F_S^{N_k}(M_{k_*}) > R_{k-1}^+ + 1.
     \end{equation}
    (Recall that $M_{k_*}=10$.) 
    We also require $N_k$ to be so large that
    \begin{equation} F_S^{N_k}(M_{k_*}) > R^-(T_{k-1}, 1/2, \widetilde{M}_{k-1} ).  
    \end{equation}

  We then define 
         \begin{equation} \label{eqn:boundedNk}
               R_k^-\defeq F_S^{N_k}(M_{k_*})\qquad\text{and}\qquad
         R_k^+ \defeq  F_S^{N_k}(\widetilde{M}_{k_*}). \end{equation}
    
  \textbf{I3.} Similarly as in~\eqref{eqn:thetadefn}, set 
     \[\theta_k^{\ell}\defeq F_S^{N_k}( \phi_{k_*}^{-1}(\xi_{k_*}^{\ell})) \]
     for $0\leq \ell \leq \gamma_{k_*}$. Note that $\theta_k^0 = R_k^-$ and $\theta_k^{\gamma_{k_*}} = R_k^+$. 
     
  The Jordan domains $U_k^j\subset \{x+iy\colon R_k^-<x<R_k^+\text{ and } \pi/2<y<3\pi/2\}$, for $1\leq j\leq m_k$,  and 
     arcs $(C_k^j)_{j=0}^{m_k}$ are chosen similarly as in
     Section~\ref{sec:arclikeexistence}; see Figure~\ref{fig:bounded_tract_construction}. Condition~\ref{I3item:Cj} remains unchanged,
      while~\ref{I3item:Cm} becomes
   \begin{equation}
  C_k^-\defeq  C_k^0 \subset \{x+ iy \colon  x = R_k^{-}\}\qquad\text{and}\qquad
                 C_k^+\defeq  C_k^{m_k} \subset \{x+ iy \colon  x = R_k^+\}.\end{equation}
   Then $U_k$ is defined as in~\eqref{eqn:Ukdefn}.
   
   The intervals $\tilde{I}_k^{\ell}$ are defined as before, with the exception that $\tilde{I}_{k}^1$ is extended to
     $4$ on the left, and $\tilde{I}_k^{m_k}$ is extended to $\infty$ on the right.
     (Note that the left end-point of $I_k^1$ is $0$ and the right end-point of
     $I_k^{m_k}$ is $1$, since $g_k$ fixed $0$ and $1$.)  Conditions~\ref{I3item:sizeofUk} to~\ref{I3item:Ukrealparts} are unchanged. To ensure that the tract $T$ has bounded decorations, the domains $U_k^j$ should also satisfy
     the following condition: Every point of $U_k^j$ can be connected to both components of $\partial U_k^j\setminus (C_k^{j-1}\cup C_k^j)$ by a curve in
     $U_k^j$ of Euclidean length at most $\pi+1$.

     The construction of domains $U_k^j$ and arcs $C_k^j$ with these
     properties is easily adapted from that given in Section~\ref{sec:arclikeexistence}; we omit the details. 

   This completes the description of the partial  tract $T_k$ and  the  map  $F_k$.

 \textbf{I4.} There  is no analogue of step I4. Observe that  the curve
     $F_k^{-1}([4,\infty))$ is guaranteed to run through the entire domain $U_k$ by construction. 

 \textbf{I5.} Define 
    \begin{equation}
       \eta_k^j\defeq  \max\{x\colon F_k^{-1}(x) \in C_k^j\} \qquad (j=1,\dots,m_k-1).\end{equation}
    Then choose $\widetilde{M}_k$ sufficiently large to ensure that
       \begin{align}
           &\re F_k^{-1}(\widetilde{M}_k) > R_k^+ + 3\qquad\text{and} \\
           &
         \re F_S^{-N_{1}}(F_ k^{-1}(F_S^{-N_{2}}(F_k^{-1}( \dots  (F_S^{-N_k}(F_k^{-1}(\widetilde{M}_k)))\dots )))) \geq k.\label{eqn:boundedMtilde}
       \end{align}    

    Finally, $\phi_k$ is defined as in~\eqref{eqn:phikdefn},  and $\gamma_k$ is also chosen as in~\ref{I5item:expansion} and~\ref{I5item:smallgaps}. 
    This completes the inductive construction. 

  The analysis of the construction proceeds as in the proof of  Theorem~\ref{thm:arclikeexistenceBlog}. 
    We first establish the pseudo-conjugacy relation.
     \begin{lem}[Pseudo-conjugacy relation]\label{lem:pullingback2}
       Let $k\geq 1$ and $w\in \overline{S}$. Then
   \begin{equation}\label{eqn:toproveforpseudoconjugacy2} |\psi_{k_*}( f_k(w)) - g_{k}(\psi_{k}(w))| < 3/\gamma_{k_*}. \end{equation}
     \end{lem}
  \begin{proof}
    The proof proceeds analogously to that of Lemma~\ref{lem:pullingback}. In the Claim stated and proved there,
    the point $\zeta$ should now be defined by 
     $\zeta \defeq F_T^{-1}(w)$; the Claim is then proved in a similar manner. The lemma follows from the claim as before. 
  \end{proof}   
     
    The definition of $d_{X_k}$ remains the same. In Proposition~\ref{prop:pseudoconjugacyconditions}, $\hat{T}$ should be replaced by $\hat{S}$; 
      the statement is otherwise unchanged, and the proof is completely analogous. 

  Finally, expansion of our system  $(X_k,f_{k+1})_{k=0}^{\infty}$ follows similarly as in Proposition~\ref{prop:fkexpanding}. Indeed,
    the proof is slightly simpler, since the set $A_k$ has uniformly bounded diameter by choice of $M_k=10$. 
 \begin{proof}[Proof of Theorem~\ref{thm:boundedexistenceBlog}]
    The claimed properties of the Julia continuum $J_{\s}(F)$ follow 
     from the construction and Lemma~\ref{lem:pullingback2}, just as 
     in Section~\ref{sec:arclikeexistence}. Recall that
    $y_0$ corresponds to the point $(0\mapsfrom 0\mapsfrom\dots)$ in $\uarrow{Y}$, and hence to 
    a point $z_0\in\Jsh(F)$ with $\re F^{n_k}(z_0)\leq M_k + 1 = 11$ for all $k$. 
    By Proposition~\ref{prop:boundedorbits}, $\Jsh(F)$ contains a unique non-escaping point, which has bounded orbit, so $z_0$ must have bounded orbit. (Alternatively, 
    Observation~\ref{obs:expansion2} implies directly that $\re F^n(z_0) < 9$ for all
    $n\geq 0$.) 

    Furthermore, the tract $T$ has bounded
    decorations: this  
    follows from our additional assumption on the domains $U_k^j$ formulated 
    at the end of step I3, together with 
    Proposition~\ref{prop:decorations}. The tract $S$ also has bounded decorations (which follows from Proposition~\ref{prop:decorations} or the 
    explicit formula for $F_S$), and both tracts clearly have bounded slope. This completes the proof. 
 \end{proof}

 \begin{proof}[Proof of Theorem~\ref{thm:boundedexistence}]
    In precisely the same manner as in Theorem~\ref{thm:allinone}, we can realise all the relevant continua as bounded-address Julia continua of the same function,
    using the second part of Proposition~\ref{prop:countable}. 
     In the construction, we should now replace~\eqref{eqn:boundedMtilde} by
\[          F_S^{-N_{k_1}}(F_ k^{-1}(F_S^{-N_{k_2}}(F_k^{-1}( \dots  (F_S^{-N_k}(F_k^{-1}(\widetilde{M}_k)))\dots )))) \geq k, \]
  where $k_1 < k_2 <  \dots < k_{\ell}= k$ are chosen such that
        $k_{i}= (k_{i+1})_*$ and $(k_1)_*=0$. 
 \end{proof}

\begin{rmk}[Alternative proof of Theorem~\ref{thm:mainarclike}]\label{rmk:alternativeproof}
  Using Theorem~\ref{thm:boundedexistence}, we can  
    give an alternative proof of (the second half of) Theorem~\ref{thm:mainarclike}. 

   Indeed,
     let $X$ be  any arc-like continuum having a terminal point $x_1$. Then there  is an  arc-like continuum 
     $Y\supset X$ having a terminal point $x_0$ such that $Y$ is irreducible between $x_0$ and  $x_1$. 
     (Indeed, $X$ can be  written as the inverse limit  of a sequence of continuous surjective functions
        $g_k\colon [0,1]\to[0,1]$ with $g_k(1)=1$. We can extend each $\tilde{g}_k$ to a map
      $[-1,1]\to[-1,1]$ fixing $-1$, and the inverse limit  of these extensions is the desired continuum $Y$. 

   Since the function $f$ from Theorem~\ref{thm:boundedexistence} has bounded  slope,
     it  follows from Theorem~\ref{thm:subsetsasJuliacontinua} that every arc-like continuum having a
     terminal point is also realised as a Julia continuum of $f$. 

   A key difference between the two proofs is that the Julia continua from Theorem~\ref{thm:subsetsasJuliacontinua}
     escape uniformly  to infinity, so the examples in Section~\ref{sec:onepointuniform} cannot be obtained 
     in this manner. Furthermore, the function constructed in our original proof has only  one  tract, but
     does not have  bounded decorations, while the construction of Theorem~\ref{thm:boundedexistence} has
    bounded decorations,  but also requires two tracts. Recall  from Corollary~\ref{cor:unboundedrequired}
     that it is impossible to achieve both at the same time if we wish to realise all possible arc-like continua with
     terminal points as Julia continue of the same function.
\end{rmk}

\section{Periodic Julia continua}\label{sec:periodicexistence}
  We now construct periodic Julia continua, proving Theorem \ref{thm:periodicexistence}. 

  \begin{thm}[Constructing invariant Julia continua]\label{thm:periodicexistenceprecise}
    Let $\hat{Y}$ be a continuum, and suppose that there are 
     $y_0,y_1\in \hat{Y}$ such that $\hat{Y}$ is a Rogers continuum from
     $y_0$ to $y_1$. 
     There exist a logarithmic tract $T$ with $\overline{T}\subset\HH$ and a 
      conformal  isomorphism $F\colon T\to\HH$ with $F(\infty)=\infty$ such that 
     \[ \hat{C} \defeq  \{z\in T\colon  F^n(z)\in T\text{ for all $n\geq 0$} \} \cup \{\infty\}\]
    is homeomorphic to  $\hat{Y}$, with $\infty$ corresponding to $y_1$, and the
    unique finite fixed point of $F$ corresponding to $y_0$. 
  
  Moreover,  $T$ can be chosen to have bounded slope and bounded decorations. 
  \end{thm}

 Let $g\colon [0,1]\to[0,1]$ be a map below the
    identity;  that is, $g$ fixes $0$ and  $1$, $g(x)<x$ for $0<x<1$. 
  Let $(\tau_k)_{k= 0}^{\infty}$  be a strictly increasing sequence in $(0,1)$ such that
    \begin{enumequ}
     \item $g(x)<g(\tau_k)$ for all $x<\tau_k$ and\label{item:peak}
     \item $g(\tau_{k+3}) = \tau_k$  \label{item:cascade}
   \end{enumequ}
   for all $k\geq 0$. For example, such a sequence can be defined by  
     \begin{align*}
     \tau(a) &\defeq \min\{\tau \in [0,1]\colon g(\tau)=a\} \quad (a\in [0,1]), \\
    \tau_0 \defeq  \tau(1/2), \quad \tau_1 &\defeq 
          \tau\left(\frac{1+\tau_0}{3}\right), \quad \tau_2\defeq \tau\left(\frac{1+4\tau_0}{6}\right)\quad\text{and}\quad
          \tau_{k+3}\defeq \tau(\tau_k) \quad (k\geq 0). \end{align*}
      This sequence has the desired properties since $\tau(a)>a$ for all $a\in(0,1)$ and since $\tau(a)$ is 
       strictly increasing in $a$. 
       Observe that $\tau_k\to 1$, as $1$ is the only positive fixed point of $g$. 
      For convenience, we set $\tau_{k-3}\defeq g(\tau_k)$ for $k=1,2$. 
 
  Our goal is to construct $T$ and $F$ as in Theorem~\ref{thm:periodicexistence}, with  
    $\hat{Y} \defeq \invlim([0,1],g)$ being the Rogers continuum defined by $g$.  The fact that 
     \begin{equation} 
         g(x)\leq g(\tau_k)= \tau_{k-3} < \tau_{k-2}, \qquad \tau_{k-1}\leq x\leq \tau_k,\qquad k \geq 1, \end{equation}
     allows us to construct the tract
     $T$ in stages, with the $k$-th piece, for $k\geq 1$, representing 
     \begin{equation}
    g\colon A_k\to B_{k-3}, \qquad\text{where}\quad A_k \defeq  [\tau_{k-1},\tau_{k}], \quad B_{k-3} \defeq [0, \tau_{k-3} ]. 
    \end{equation}
  The same fact also allows us to express $g|_{[0,1)}$ as an expanding system, to  which we  will be  able
   to apply  Proposition~\ref{prop:conjugacyprinciple}. Indeed, let $(\gamma_k)_{k=1}^{\infty}$ be  a sequence
    of numbers $\gamma_k\geq 1$.  We will consider metrics $d_Y$ on $Y\defeq [0,1)$ with the following properties.
    \begin{enumequ}
       \item  When restricted  to  $B_0$, $d_Y$ is some metric $d_0$ equivalent to the Euclidean metric, and
            the $d_Y$-length of  $B_0$ is $\ell_0 \defeq \ell_{d_Y}(B_0)\geq 1$.\label{item:B0metric}
     \item  For every $k\geq 1$, the metric $d_Y$ is a constant multiple of the
          Euclidean metric on $A_k$, and $\ell_{d_Y}(A_k)= \gamma_k$.\label{item:Akmetric}
    \end{enumequ}

 \begin{obs}[Expanding metric for  $g$]\label{obs:periodicexpansion}
    Let $g\colon [0,1]\to[0,1]$ be a map below the identity with $(\tau_k)_{k=-2}^{\infty}$ as above. Let 
     $d_0$ be a metric on $B_0$ as in~\ref{item:B0metric}. 
    
    Then there are functions $\Gamma_k\colon [1,\infty)\to [1,\infty)$, for $k\geq 1$, such that
      $\Gamma_k(\gamma)>\gamma$ for all $\gamma\in [1,\infty)$ and such that the following property holds. 
       Suppose that $(\gamma_k)_{k=1}^{\infty}$ is a sequence satisfying 
         $\gamma_k\geq \Gamma_k(\gamma_{k-1})$ for all $k\geq 1$, and that $d_Y$ is the metric on $Y=[0,1)$ that agrees with $d_0$ on $B_0$ and
        satisfies~\ref{item:Akmetric}. Then 
        the inverse system $g\colon Y\to Y$ is expanding with respect to the metric $d_Y$, with constants $\lambda=2$ and $K=8\ell_0$. 
 \end{obs}
 \begin{proof}
%
       We may define $\Gamma_k(\gamma) >\gamma$ such that the following holds whenever
       $\gamma_1,\dots,\gamma_{k-1}\leq \gamma$ and $\gamma_k \geq \Gamma_k(\gamma)$:  if 
       $x,y\in A_k$ with $d_Y(x,y)\leq 1$, then  
         \begin{equation}\label{eqn:iteratedshrinking} d_Y(g^n(x),g^n(y))\leq \frac{1}{\max(k,8)} \end{equation} 
       for all $n\geq 1$. 
       It is enough to prove this when $\gamma_1=\dots=\gamma_{k-1}=\gamma$ (decreasing one or several of the $\gamma_i$ for $i<k$ does not affect 
       $d_Y(x,y)$, and 
          $d_Y(g^n(x),g^n(y))$ will decrease or remain the same). Since $g^n\to 0$ uniformly on $A_k$, there exists $n_0$ such that 
          $\diam_Y( g^n(A_k)) \leq 1/\max(k,8)$ for $n\geq n_0$. So~\eqref{eqn:iteratedshrinking} holds automatically
          for such $n$. On the other hand, for $1\leq n < n_0$, by uniform continuity of
          $g^n|_{A_k}$ there is some $\delta_n$ such that~\eqref{eqn:iteratedshrinking}
          holds when $\lvert x - y\rvert < \delta_n$. Choosing $\Gamma_k(\gamma)\geq \max_{n < n_0} 1/\delta_n$ ensures that the latter holds when $d_Y(x,y)\leq 1$. 
          
   Now suppose that $(\gamma_k)_{k=0}^{\infty}$ is a sequence with $\gamma_k\geq \Gamma_k(\gamma_{k-1})$ for $k\geq 1$. 
     We first show that condition~\ref{item:expandingbackwardsshrinking} of Definition~\ref{defn:expandingsystem} holds for
     $d_Y$ and $g$. We first prove this when $\Delta=1$, and $x$ and $y$ both belong either to $B_0$, or to a common $A_k$. Let $\eps>0$. If
     $x,y\in A_k$ with $k\geq k_0\defeq \bigl\lceil 1/\eps\bigr\rceil$, 
         then~\eqref{eqn:iteratedshrinking} implies that $d_Y(g^n(x),g^n(y))\leq \eps$ for all $n\geq 1$. On the other hand,
     $g^n|_{B_{k_0}}\to 0$ uniformly. So there exists $n_0\geq 1$, independent of $x$ and $y$,
     such that $d_Y(g^n(x),g^n(y))\leq \eps$ for $n\geq n_0$. This proves the claim
     under our assumption on $x$, $y$ and $\Delta$. The general case follows since every interval of $d_Y$-diameter $\leq \Delta$ can be written as the union of
     at most $2\Delta$ adjacent intervals of the form we just treated.
     
    Now let 
      $x,y\in [0,1)$; we  must  establish that condition~\eqref{eqn:expandinginversesystem} of Definition~\ref{defn:expandingsystem} holds. 
      We may suppose that $x<y$; write $I\defeq [x,y]$. Then, as in the preceding paragraph, we may
      write $I$ as the union of at most $2\cdot \lceil \ell_{d_Y}(I)\rceil$ adjacent intervals with disjoint interiors, where each interval is either 
      equal to $I\cap B_0$, or is contained in one of the $A_k$ and has $d_Y$-length at most $1$. We have $\ell_{d_Y}(g(I\cap B_0))\leq \ell_0$.
      On the other hand, if $J$ is an interval of the second type, then $\ell_{d_Y}(g(J)) \leq 1/8$ by~\eqref{eqn:iteratedshrinking}. Hence 
\[      d_{Y}(g(x),g(y)) \leq \ell_0 + \frac{2\lceil \ell_{d_Y}(I)\rceil}{8} < 2\ell_0 + 
               \frac{1}{4} d_Y(x,y) \leq 
               \frac{\max(8\ell_0,d_Y(x,y))}{2}. \qedhere\]   
 \end{proof}

\begin{figure}
 \begin{center}
   \def\svgwidth{\textwidth}
   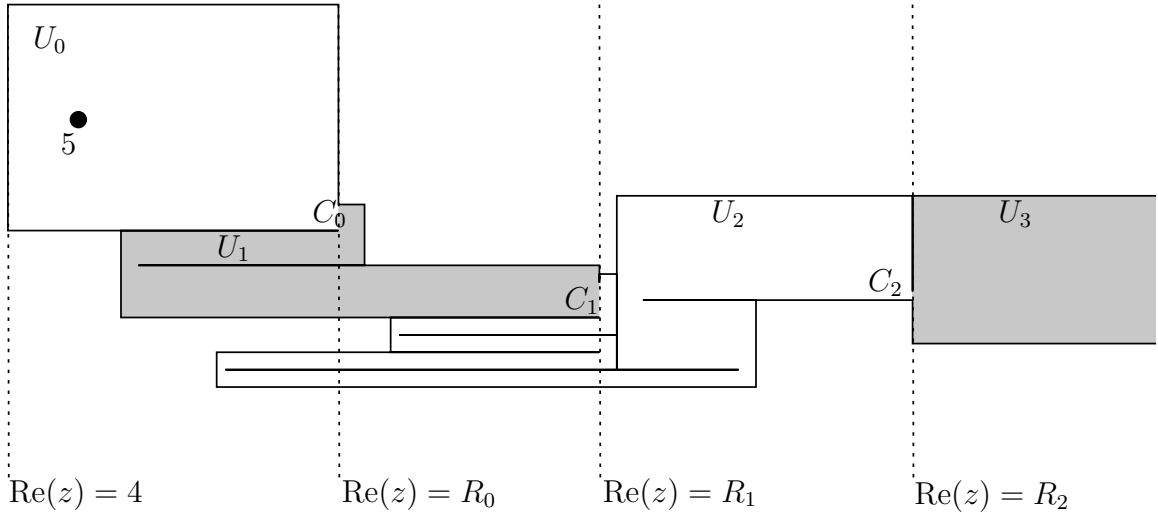
 \end{center}
 \caption{The shape of the tract $T$ used to construct a prescribed invariant Julia continuum.\label{fig:tract_definition-periodic}}
\end{figure}

 \subsection*{Description of $F$ and $T$}
      Our tract  $T$ is given as a disjoint union of the form
    \begin{equation}\label{eqn:inv-tract} T = \bigcup_{j\geq 0} U_j \cup \interior(C_{j}) \subset \{ x + iy\colon  x > 4, |y|<\pi \}. \end{equation}
    (See Figure~\ref{fig:tract_definition-periodic}.) Here $(C_j)_{j\geq 0}$ is a sequence of cross-cuts of $T$, with $C_{j-1}$ and $C_{j}$ bounding the subdomain $U_j$ in $T$ for $j\geq 1$. 
    Furthermore, 
    \begin{align} \notag U_0 &= \{ x + iy\colon   4 < x < R_0, |y|<\pi/2\},  \\
         U_j &\subset \left\{x + iy\colon   4 < x < R_{j}, -\pi/2\cdot \left(1 + \frac{j}{j+1}\right) < y< \pi/2\right\}\quad(j\geq 1),   
                  \quad\text{and} \\
           \notag C_{j} &\subset \{z\colon \re z = R_j\} \cap \partial U_j \cap  \partial U_{j+1}, 
   \end{align}
    where $(R_j)_{j=0}^{\infty}$ is a strictly (and rapidly) increasing sequence of numbers $R_j\geq 10$, and the 
     domains $U_j$ converge uniformly to $\infty$ as $j\to\infty$.

   The conformal isomorphism $F\colon T\to\HH$ is uniquely determined by requiring $F(5)=5$ and $F(\infty)=\infty$. 
   For $k\geq 0$, we also define \emph{partial tracts}
   \begin{equation} \label{eqn:inv-partial-tract} T_k \defeq  \bigcup_{j=0}^k U_j \cup \bigcup_{j=0}^{k-1} \interior(C_{j}) \end{equation}
    (observe that  $T_k$ is bounded). The associated conformal isomorphisms 
    $F_k\colon T_k\to\HH$ are determined by requiring $F_k(5)=5$ and  $F_k(\zeta_k)=\infty$, 
    where $\zeta_k\in C_k$ is defined
    as part of the construction. 
    
    In analogy with Proposition~\ref{prop:caratheodoryconvergence2}, 
     we require a result about the approximation of 
     $F$ by the functions $F_k$. To state it, let us introduce 
     the following terminology. Suppose that $K\geq 0$ and that
     $T_{K}$ is a partial tract as above, with associated function $F_{K}$. 
    Then we denote by $\F_{>K}(F_K)$ the set that contains
     all functions 
     $F$ for which $T_K$ is the $K$-th partial tract
     and for which $\zeta_K\in C_K$, as well as their $k$-th 
     approximations $F_k$ for $k>K$.  
     
     That is, if $F\in \F_{>K}(F_K)$, 
     then $F$ is a conformal isomorphism $F\colon T\to \HH$, where
     $T$ is either of the form~\eqref{eqn:inv-tract} or of the form~\eqref{eqn:inv-partial-tract} for some $k>K$, and where (in either case)   
     the choices of $(R_j)_{j=0}^{K}$, $(U_j)_{j=0}^{K}$ and
     $(C_j)_{j=0}^{K-1}$ are the same as for $T_K$. 
     Furthermore, $\zeta_k \in C_{K}$, $F(5)=5$ and
     $F$
     extends continuously to a point $\zeta$ with $F(\zeta)=\infty$, where
     $\zeta=\infty$ if $T$ is of the form~\eqref{eqn:inv-tract}, or
     $\zeta\in \partial T\cap \{z\colon \re z = R_k\}$ if 
     it is of the form~\eqref{eqn:inv-partial-tract}.

 \begin{prop}[Approximation by partial tracts]\label{prop:periodiccloseness}
   Let $K\geq 0$, and let $F_K$, $T_K$ and $\zeta_K$ be as above. 
   Let $R>0$. 
   
    Then there is an arc $\tilde{C}_K\subset \partial T_K\cap \{z\colon \re z = R_K\}$ 
      with $\zeta_K\in \tilde{C}_K$ such that 
        \[ \dist_{T_K}( F_K^{-1}(z) , F^{-1}(z) )  \leq \frac{1}{4} \]
     for all $F\in \F_{>K}(F_K)$ with $C_K = C_K(F) \subset \tilde{C}_K$ 
     and all $z\in\C$ with $4\leq \re z \leq R$ and 
     $\lvert \im z\rvert \leq \pi$. 
 \end{prop}
 \begin{proof}[Sketch of proof] 
 The proof is similar to that of
  Proposition~\ref{prop:caratheodoryconvergence2}. By definition,
  the tracts $T(F)$ of $F\in \F(F_K)$ converge to $T_K$ with respect
  to Carath\'eodory kernel convergence (with respect to the base point $5$)
  as the diameter of $C_K = C_K(F)$
  tends to zero. So again $ F^{-1}\circ \alpha_F \to F_K^{-1}$, 
  where $\alpha_F$ is a suitable
  M\"obius transformation fixing $5$. Let
  $Q\subset T_K$ 
   be a quadrilateral of large modulus one of whose ``vertical'' sides contains
  $\zeta$ in its interior, whose ``horizontal'' sides are on the line segment
  at real part $R_K$, and such that $Q$ separates $5$ from $\zeta$ in $T_K$.

 Let $R>0$ be so large that $Q$ separates $F_K^{-1}(R)$ from
  $5$. When $F\in\F(F_K)$ with $\diam(C_K(F))$ sufficiently small, 
  $Q$ also separates $F^{-1}(\alpha_F(R))$ from $5$. As in the proof
  of Proposition~\ref{prop:caratheodoryconvergence2}, it follows that 
  $\lvert \alpha_F(R)\rvert$ is large, with the lower bound depending only
  on the modulus of $Q$. We conclude that 
  $\alpha_F\to\id$ uniformly as $\diam C_K(F)\to 0$, as desired. 
\end{proof}
    
\subsection*{Recursive construction}
  The definition of the domains $U_k$ and  arcs $C_k$ is carried out in a similar spirit to those in  Sections~\ref{sec:arclikeexistence} 
    and~\ref{sec:boundedexistence}. As part of the construction, we obtain a homeomorphism
    \[ \phi\colon [4,\infty) \to [0,1), \]
    where 
     \[ \phi|_{[R_{k+2},R_{k+3}]}\colon [R_{k+2},R_{k+3}] \to A_k \]
    is defined in the $(k+2)$-th step of the construction. This homeomorphism will satisfy 
     $\phi(R_k)=\tau_{k-3}$ for all $k\geq 1$.  
   The construction ensures that 
    \begin{equation}
      \psi\colon \overline{T} \to [0,1); \qquad \psi(z) \defeq \phi(\re z) \end{equation}
    is a pseudo-conjugacy between the systems
     $(\overline{T}, F^{-1})$, with the Euclidean metric, and $([0,1),g)$,  with a metric 
     $d_Y$ as defined above, where $(\gamma_k)_{k=0}^{\infty}$ chosen recursively during the construction.

  The $\gamma_k$ are chosen to be a natural numbers. The metric $d_0$ on $B_0$
    is chosen to be a linear multiple of the Euclidean metric on each of the
     intervals $B_{-2}=[0,\tau_{-2}]$,  $[\tau_{-2},\tau_{-1}]$ and  $[\tau_{-1},\tau_0]$, with  each
     interval having integer length. 
    The choice of $d_Y$ on $B_k$ (for $k\geq -2$) determines 
    a finite set 
      \[ \Xi_k\defeq \{\xi\in B_k\colon d_Y(0,\xi)\in\Z\} \supset \{0,\tau_{-2},\dots , \tau_k\} \] 
    that partitions $B_{k}$ into intervals of
    unit $d_Y$-length. 

 To anchor the recursion, set
     \[R_{-1} \defeq 10,\quad R_{0}\defeq 14 \quad\text{and} \quad R_1\defeq 15.  \]
   Define $\phi\colon [4,R_1] \to [0, \tau_{-2}]$ to be 
     linear with $\phi(4)=0$ and  $\phi(R_1)=\tau_{-2}$, and define the metric  $d_Y$ on
      $[0,\tau_{-2}]$ by 
       \begin{equation}
           d_Y(x,y)\defeq 5|\phi^{-1}(x) - \phi^{-1}(y)|.
        \end{equation}
    (So $[0,\tau_{-2}]$ has $d_Y$-length $55$.) 
     Also set
    $\zeta_0 \defeq R_0$. 
%
%

 Now 
   the inductive step proceeds as follows. Let $k\geq  1$. 
   Assume that the domains $U_i$ are chosen for $i<k$, and the arcs $C_i$ for $i<k-1$. 
   The point  $\zeta_{k-1}$ will also have  been chosen, so $F_{k-1}$ is defined. 
   Moreover, $R_i$ is defined for $i\leq k$, an order-preserving homeomorphism 
     $\phi\colon [4,R_k] \to [0,\tau_{k-3}]$ has been constructed,
    and the metric $d_Y$ has been defined  on $B_{k-3}$.
    A finite set
    \[  \Omega_{i} = 
        \{ \tau_{i-1} = \omega_i^0 < \omega_i^1 <  \dots < \omega_i^{m_i}= \tau_{i}\} \]
       will also have been chosen for  $1\leq i<k$. 
       To avoid consideration of special cases in the case of $k=1,2$, we also
       define $\Omega_{-1}$ and $\Omega_0$ by setting
       $m_{-1}\dots m_0\defeq 1$, 
       $\omega_{-1}^0\defeq \tau_{-2}$, $\omega_{-1}^1 \defeq \omega_0^0 \defeq
       \tau_{-1}$ and $\omega_0^1\defeq \tau_0$. 

       As 
       in~\eqref{eqn:Ukdefn},
     each domain $U_i$ (for $1\leq i < k$) is given by a finite union
      \begin{equation}\label{eqn:Ui_periodic}
       U_i = \bigcup_{j=1}^{m_i} U_i^j \cup \bigcup_{j=1}^{m_i-1} C_i^j \end{equation}
    of Jordan domains $U_i^j$ and arcs $C_i^j$ connecting them, together with
    arcs $C_i^0=C_{i-1}$ and $C_i^{m_i}\supset C_i$. 
     For $i=-1,0$, these are given by 
    \begin{align*}
        U_{-1}^1 &\defeq \{x+iy\colon 5<x<R_{-1}, |y|<\pi/2\}, \\
                U_0^1&\defeq \{x+iy\colon R_{-1}<x<R_0, |y|<\pi/2\},  \\
         C_{-1}^0 &\defeq \{ 5 + iy\colon |y|\leq  \pi/2\}, \quad
         C_{-1}^1\defeq C_0^0 \defeq\{ R_{-1} + iy\colon |y|\leq\pi/2\}, \quad\text{and}\quad \\
         C_0^1 &\defeq \{ R_0 + iy\colon |y|\leq\pi/2\}. 
     \end{align*}

 It turns out to be convenient to slightly change the order of steps compared with those in  
  Section~\ref{sec:arclikeexistence};
  more precisely, we begin the inductive step with the analogue of step I5. To  emphasise the 
  similarity with previous sections, we shall retain the same numbering. 

\textbf{I5.} For $j=0,\dots, m_{k-2}$, define 
       \begin{equation}\label{eqn:periodiceta}
    \eta_{k-2}^j\defeq \min\{ x\in [5,\infty)\colon F_{k-1}^{-1}(x)\in C_{k-2}^j \}.  
       \end{equation} 
      We set $R_{k+1}\defeq \eta_{k-2}^{m_{k-2}}$.
       We now extend $\phi$ to $[R_{k},R_{k+1}]$ in a piecewise linear manner such that 
       \[ \phi(\eta_{k-2}^j) = \omega_{k-2}^j \] 
       for $j=1,\dots,m_{k-2}$. 
       We shall see in the proof of Theorem~\ref{thm:periodicexistenceprecise} below that 
           \begin{equation}\label{eqn:eta1k}
                 \eta^1_{k-2}>R_{k}.
           \end{equation}
  Hence this does indeed define a homeomorphism $[R_{k},R_{k+1}]\to [\tau_{k-3},\tau_{k-2}]$. 

   We now define the metric $d_Y$ on $I\defeq [\tau_{k-3},\tau_{k-2}]$ to be a constant multiple
     of the Euclidean metric, subject to the following requirements.  Note that $I= A_{k-2}$ for $k\geq 3$.
    \begin{enumequ}
       \item $\ell_{d_Y}(I)$ is a positive integer. 
      \item  If  $\tilde{I}\subset I$ has $\ell_{d_Y}(\tilde{I})\leq 1$,  then 
                 $\diam(\phi^{-1}(\tilde{I})) <1/4$. 
       \item If $k\geq 3$ (so that the restriction $d_0$ of $d_Y$ to $B_0$ has been chosen), then 
         $\gamma_{k-2} \defeq \ell_{d_Y}( I ) \geq \Gamma_{k-2}(\gamma_{k-3})$,
            where $\Gamma_{k-2}$ is the function from 
      Observation~\ref{obs:periodicexpansion}.\label{item:periodicgamma}
    \end{enumequ}

 \textbf{I1.} As in~\ref{I1item:Omegaimage}, we choose $m_k$ and $\Omega_k$
  such that, for $1\leq j\leq m_k$, the image 
     $g([\omega_k^{j-1},\omega_k^j])$ contains at most one element of $\Xi_{k-3}\setminus\{0\}$. 

 \textbf{I2.} 
   Choose an arc $\tilde{C}_{k-1}\subset C_{k-1}^{m_{k-1}}$ with $\zeta_{k-1}\in \interior(\tilde{C}_{k-1})$ by applying 
    Proposition~\ref{prop:periodiccloseness} to $T_{k-1}$. Here we use
    $R=R_{k+1}$. 

  \textbf{I3.} We now define the domain  $U_k$, of the form~\eqref{eqn:Ui_periodic}, very  similarly to step I3 in  Section~\ref{sec:arclikeexistence}.
     For $j\in\{1,\dots,m_k\}$, let $I_k^j$ be the smallest closed interval bounded 
      by two points of $\Xi_{k-3}$ 
      whose relative interior in $[0,\tau_{k-3}]$ contains $g([\omega_k^{j-1},\omega_k^j])$. 
       We also define $\tilde{I}_k^j \defeq \phi^{-1}(I_k^j)$.
       (In contrast to  Section~\ref{sec:arclikeexistence},
       there is no need for a different definition when $j=m_k$.) Recall that $g(\omega_k^{m_k}) = g(\tau_k) = \tau_{k-3}$, so 
       the right end-point of $\tilde{I}_k^{m_k}$ is $\phi^{-1}(\tau_{k-3})=R_k$. 
     
    Requirement~\ref{I3item:Cj} remains unchanged from Section~\ref{sec:arclikeexistence},~as does \ref{I3item:sizeofUk}. 
     Condition~\ref{I3item:Cm} is replaced by
    \begin{enumequ}
     \item
      $C_{k-1} \defeq C_k^0\subset \tilde{C}_{k-1}$ and 
      $C_k^{m_k} \subset \{ R_k + iy\colon |y|\leq \pi  \}$. 
    \end{enumequ}
      Requirements~\ref{I3item:Ukhyperbolicdistance} and~\ref{I3item:Ukrealparts} are slightly changed as follows, for each $1\leq j \leq m_k$. 
     \begin{enumequ}
        \item $\dist_{T_k}(C_k^{j-1},C_k^j) \geq 1$, and the same holds for the tract of any $F\in \F(F_k)$ with $C_k(F)\subset C_k^{m_k}$.%
            \label{item:periodichyperbolicdistance}
        \item Suppose that $\alpha$ is a geodesic segment of $T_k$ 
          that connects $C_k^{j-1}$ and $C_k^j$. Then $\dist_{T_k}(z, U_k^j) \leq 1/4$ 
          for all $z\in\alpha$. (I.e., $\alpha$ cannot protrude too far from the cross-cuts $C_k^{j-1}$ and $C_k^j$.)
          The same holds for the tract of any $F\in \F(F_k)$ with $C_k(F)\subset C_k^{m_k}$. \label{item:periodicgeodesics}
     \end{enumequ}
     (Both of these can be achieved by choosing the arcs sufficiently small.) 
     As in step I3 of Section~\ref{sec:boundedexistence}, we also require the following condition to ensure the bounded decorations condition. 
    \begin{enumequ}
      \item 
          Each $U_k^j$ is a Jordan domain, and every point of $U_k^j$ 
         can be connected to both
         components of $\partial U_k^j \setminus (C_k^{j-1} \cup C_k^j)$
         by a curve in $U_k^j$ of diameter at most $2\pi$.  
\label{item:periodicI3boundeddecorations}
    \end{enumequ}  
   We also choose $\zeta_k\in \interior(C_k^{m_k})$, completing the description of $T_k$ and $F_k$. 

 \textbf{I4.} There is no analogue of step I4, and the inductive construction is complete.

\begin{proof}[Proof of Theorem \ref{thm:periodicexistenceprecise}]
  We first verify that the construction indeed guarantees~\eqref{eqn:eta1k}. 
    For $k=1$, this follows from the standard estimate~\eqref{eqn:standardestimate}. So suppose that
     $k\geq 2$, and let 
     $4\leq x \leq R_{k}$. Then $F_{k-2}^{-1}(x)\in  T_{k-3}$ by definition  of $R_k$. Furthermore, 
      $\dist_{T_{k-2}}(F_{k-1}^{-1}(x),F_{k-2}^{-1}(x))\leq 1/4$ by step I2.
     Hence, by~\ref{item:periodichyperbolicdistance}, 
      $F_{k-1}^{-1}(x)\not\in C_{k-2}^1$ for these values of $x$.  In particular, $\eta^1_{k-2}>R_{k}$, 
      as required. 

 Now let $T$ be the tract we have just defined. Then $T$ clearly has bounded slope. The tract also has bounded
   decorations
   by Proposition \ref{prop:decorations} and~\ref{item:periodicI3boundeddecorations}.

  The (autonomous) inverse system generated by $(\overline{T},F^{-1})$ is 
     expanding when equipped with
     the Euclidean metric (using an analogue of Observations~\ref{obs:expansion} and~\ref{obs:expansion2}). The system $([0,1),g)$ with 
     the metric $d_Y$ is expanding
    by~\ref{item:periodicgamma}.

  We claim that $\psi\colon \overline{T} \to [0,1)$ is a pseudo-conjugacy between the two systems.
    To establish the pseudo-conjugacy relation, let $w\in\overline{T}$; we claim that 
     \begin{equation} \dist_Y( g(\psi(w)), \psi(F^{-1}(w))) \leq \max(4,\ell_0).\label{eqn:periodicpc}\end{equation}
     Set $\omega \defeq \psi(w)$, $x\defeq \re  w$ and  $z\defeq F^{-1}(w)$.  
       Set $k=0$ if $\omega\in  B_0$, and otherwise
       let $k$ be such that $\omega\in A_k$. Then $x\leq R_{k+3}$, and hence 
      \begin{equation}
        \dist_T(F_{k+1}^{-1}(x), F^{-1}(x)) \leq 
         \dist_{T_{k+1}}(F_{k+1}^{-1}(x), F^{-1}(x)) \leq 1/4\label{eqn:pcproof1}
      \end{equation}
       by I2.  Since $F\colon T\to\HH$ is a conformal isomorphism, also 
      \begin{equation}
        \dist_{T}(z,F^{-1}(x)) \leq \pi/x.\label{eqn:pcproof2}
      \end{equation}

     First suppose that $k=0$. Then 
      $g(\psi(w))\in B_0$ and $x \leq R_{3}$. Thus $F_1^{-1}(x)\in U_0$ by definition, and
      $F^{-1}(w)\in T_2$ by~\eqref{eqn:pcproof1} and~\eqref{eqn:pcproof2}. So 
      $\psi(z)\in B_0$, and~\eqref{eqn:periodicpc} follows.
  
     Now suppose that $k\geq 1$ and 
       $x > \eta_k^0$, as defined in~\eqref{eqn:periodiceta}.  Let $j\in\{1,\dots,m_{k-1}\}$ be such that
       $\eta_k^{j-1} < x \leq \eta_k^j$. Then $F_{k+1}^{-1}(x)$ belongs to a geodesic segment of $T_k$ connecting 
      $C_{k}^{j-1}$ and $C_{k}^{j}$, and hence by~\ref{item:periodicgeodesics},
      \[ \dist_{T}( F_{k+1}^{-1}(x), U_k^j) \leq 1/4. \]
       Note that $x\geq R_0\geq 10$, and hence \eqref{eqn:pcproof1} and~\eqref{eqn:pcproof2},
       together with~\ref{item:periodichyperbolicdistance}, imply that 
        $z$ either belongs to $U^j_k$, or to one of the quadrilaterals adjacent to $U^j_k$. 
       Recall that $g(\omega)\in \tilde{I}^j_k$, which has  $d_Y$-length at most $2$. It follows that
       $d_Y(g(\omega), \psi(z))\leq 3$, as desired. 

   Finally, suppose that  $R_{k+2} \leq x \leq \eta_k^0$. For $5\leq \zeta \leq \eta_k^0$, we have
      $F_{k+1}^{-1}(\zeta)\in T_{k-1}$ by definition of $\eta_k^0$. Moreover,
       $\dist_{T_{k}}(F_{k}^{-1}(R_{k+2}), F_{k+1}^{-1}(R_{k+2}))\leq 1/4$ by step I2, and hence
        $F_{k+1}^{-1}(R_{k+2})\in U_{k-1}^{m_{k-1}}$. 

       So if $\rho\in [5,\infty)$ is minimal with 
          $F_{k+1}^{-1}(\rho) \in  C_{k-1}^{m_{k-1}-1}$, then $\rho \leq R_{k+2}\leq x$. Thus $x$
          belongs to a geodesic of $T_{k+1}$ connecting the two  cross-cuts bounding $U_{k-1}^{m_{k-1}}$.
       By~\ref{item:periodicgeodesics}, we have
         $\dist_T(x,U_{k-1}^{m_{k-1}})\leq 1/4$. 
         Arguing as above, we conclude that $d_Y(g(\omega), \psi(z))\leq 4$, and the proof of~\eqref{eqn:periodicpc} is
         complete. 
      
    Surjectivity  of $\psi$ is obvious, and the remaining two conditions follow as in Proposition~\ref{prop:pseudoconjugacyconditions}.
     So the two inverse limits are homeomorphic, and the homeomorphism extends to their one-point
      compactifications, $\hat{C}$ and $\hat{Y}$, with $\infty$ corresponding to $y_1$.  Moreover,
      by Observation~\ref{obs:convergingtotheconjugacy}, the unique fixed point of $F$ in $C$ must
      correspond  to $y_0$. This completes the proof.
\end{proof}

\begin{proof}[Proof of Theorem \ref{thm:periodicexistence}]
 The fact that every invariant Julia continuum of a function with bounded-slope tracts is of the required form was already established in 
   Theorem \ref{thm:invariantarclike}. Conversely, by Theorem \ref{thm:realization}, we can realise the invariant Julia continuum
   from Theorem \ref{thm:periodicexistenceprecise} by a disjoint-type entire function of bounded slope. 
\end{proof}

\begin{proof}[Proof of Theorem \ref{thm:pseudoarcs}]
  By a classical result of Henderson \cite{hendersonpseudoarc}, the pseudo-arc can be written as an inverse limit of a single
   function $f\colon [0,1]\to[0,1]$ with $f(x)<x$ for all $x\in (0,1)$. 
   By Theorem \ref{thm:periodicexistenceprecise}, there is a 
   disjoint-type 
   function $F\in\BlogP$, having only one tract $T$, such that the invariant Julia continuum in $T$ is a pseudo-arc. 
   Furthermore, this tract has bounded slope and bounded decorations, and hence Corollary~\ref{cor:pseudoarcs}
    implies that every Julia continuum is a pseudo-arc.

  Finally, by Theorem \ref{thm:realization}, there is a disjoint-type function $f\in\B$ with the same property.
\end{proof}

\section{Constructing examples with a finite number of tracts and singular values}\label{sec:classS}

We now turn to Theorem \ref{thm:Stracts}. So far, we have shown that each of the 
  examples in question can be constructed in the class $\B$, with the desired number
  of tracts (one or two). With the exception of Theorem \ref{thm:pseudoarcs}, in which we require control over \emph{all} Julia continua,
  Theorem \ref{thm:realization} also shows that they can be constructed in the class $\classS$, but with a potentially infinite number of
  tracts.

In order to show that we can also realise our examples in the class $\classS$, without 
  having to introduce additional tracts, we use Bishop's results from \cite{bishopfolding},
  rather than those from~\cite{bishopclassSmodels}. 
  Let $G$ be an infinite locally bounded tree in the plane. Following Bishop, we say that
   $G$ has \emph{bounded geometry} if the following hold.
\begin{enumerate}[(i)]
   \item The edges of $G$ are $C^2$, with uniform bounds.\label{item:C2bounds}
   \item The angles between adjacent edges are bounded uniformly away from zero.\label{item:boundedangles}
   \item Adjacent edges have uniformly comparable lengths.\label{item:adjacentedges}
   \item For non-adjacent edges $e$ and $f$, we have
        $\diam(e) \leq C\cdot \dist(e,f)$ for some constant $C$ depending only on $G$.\label{item:nonadjacent}
\end{enumerate}
Bishop's theorem \cite[Theorem 1.1]{bishopfolding} is as follows:
\begin{thm}[Construction of entire functions with two singular values]\label{thm:bishop}
  Suppose that $G$ has bounded geometry. Suppose furthermore that 
    $\tau\colon \C\setminus G\to \HH$ is holomorphic with the following properties.
   \begin{enumerate}[(1)]
    \item For every component $\Omega_j$ of $\C\setminus G$, 
                $\tau\colon \Omega_j\to \HH$ is a conformal isomorphism whose inverse
                $\sigma_j$ extends continuously to the closure of $\HH$ with $\sigma_j(\infty)=\infty$.
    \item For every (open) edge $e\in G$, and each $j$, every component of
          $\sigma_j^{-1}(e)\subset \partial \HH$ has length at least $\pi$.\label{item:taulength}
   \end{enumerate}

  Then there is an entire function $f\in\classS$ and a quasiconformal map $\phi$ such that
      $f\circ \phi = \cosh\circ \tau$ on the complement of the set
      \[ G(r_0) \defeq  \bigcup_{e\in G} \{z\in\C\colon  \dist(z,e) < r_0\diam(e)\}. \]
     (Here $r_0$ is a universal constant, and the union is over all edges of $G$.)

   The only critical points of $f$ are $\pm 1$ and $f$ has no asymptotic values. If $d$ is such that $G$ has no
     vertices of valence greater than $d$, then $f$ has no critical points of degree greater than $4d$.
\end{thm}
We note that the maximal dilatation of $\phi$ is bounded by a constant depending only on the 
  bounded geometry bounds of $G$; we will not require this fact. 
\begin{rmk}[Smaller bounds on the degree]\label{rmk:bishopbound}
  This is the theorem as stated in \cite[Theorem~3.1]{bishopclassSmodels}. All trees considered in this
    section will have maximal degree at most $4$, leading to a bound of $16$ on the degree of critical points.
    As discussed in
    \cite{bishopclassSmodels}, the bound of $4d$ can be reduced to $\max(4,d)$ by modifying the 
      construction in
       \cite{bishopclassSmodels}, and replacing $\tau$ by an integer multiple. This 
       leads to a bound of $4$ on the critical points of our functions, as remarked in Section~\ref{sec:intro2}. 
\end{rmk}

 In order to use this result to obtain the desired examples, there are two steps that we need to take. 
  \begin{itemize}
    \item We will modify the construction of the domains $\T$ of our functions,
        to obtain a union $\tilde{\T}$ of ``tracts'' such that $\C\setminus \exp(\tilde{T})$ is 
         a bounded-geometry tree.
       
       Along with these tracts, we  construct $\tilde{F}\colon \tilde{\T}\to \HH$,
        conformal on each component of $\tilde{\T}$, such that
        the map $\tau(\exp(z))\defeq \tilde{F}(z)$ satisfies the hypotheses of Theorem~\ref{thm:bishop}. 

       When we consider the restriction $F$ of  $\tilde{F}$ to the preimage of a suitable right half-plane,
          we will obtain a disjoint-type function in $\BlogP$, still having the desired properties of the original  construction.

     Observe that $\overline{\T}$ will need to fill out the entire plane, which means 
        that we have to ensure that the individual tracts each fill out a complete
          horizontal strip of height  $2\pi$.
   \item Once this is achieved, let $f_0$ be the function from Theorem \ref{thm:bishop}, and set 
           $f \defeq  \lambda f_0$, where $\lambda$ is chosen sufficiently small to ensure that $f$ is of disjoint type. 
           The functions
           $f$ and $g \defeq  \cosh\circ\tau$ are no longer necessarily ``quasiconformally equivalent'' near infinity 
             in the sense 
           of~\cite{boettcher}, and hence we cannot conclude that they are quasiconformally equivalent near their 
            Julia sets.
          However, if the set $G(r_0)$ is disjoint from the orbit of the Julia continuum $\CH$ of $g$ under consideration, 
            then the
          arguments from \cite{boettcher} still apply to show that there is a corresponding Julia continuum of $f$ 
           homeomorphic to $\CH$. 

     Indeed, if we lift $\phi$ to a map $\Phi$ in logarithmic coordinates, then 
         $\Phi$ does not move a point $\zeta$ by more than a bounded hyperbolic distance,
          provided $\re \zeta$ is sufficiently large. 
          If $F_0$ is a logarithmic transform of $f$ and
          $F_0\circ \Phi = \tilde{F}$ on the orbit of a Julia continuum of $\tilde{F}$
          (which is guaranteed when the above
          condition on $G(r_0)$ is satisfied), then we easily obtain a homeomorphism between this continuum
           and a corresponding one for $F_0$, using the conjugacy principle from Section~\ref{sec:conjugacy}.
           Compare \cite[Lemma 2.6]{boettcher}.

       We remark that the set $G(r_0)$ can in fact be replaced by a smaller set 
         $V_{\mathcal{I}}$ (compare \cite[Lemma 1.2]{bishopfolding}), and 
         under the right conditions this set will be automatically disjoint from the \emph{bounded-address} continua 
         that we construct in Theorems
         \ref{thm:boundedexistence} and \ref{thm:periodicexistence}. However, some care is still required in the case 
          of Theorem
         \ref{thm:mainarclike}.
  \end{itemize}
 
 We shall now discuss in some more detail how to ensure these properties in the case of 
  Theorem~\ref{thm:arclikeexistenceBlog} and  Proposition~\ref{prop:arclikeexistenceprecise}. 
  The techniques and estimates are very similar to those used in Bishop's examples. Hence we shall focus on
  the ideas rather than give precise estimates.

\begin{figure}
 \begin{center}
   \def\svgwidth{\textwidth}
   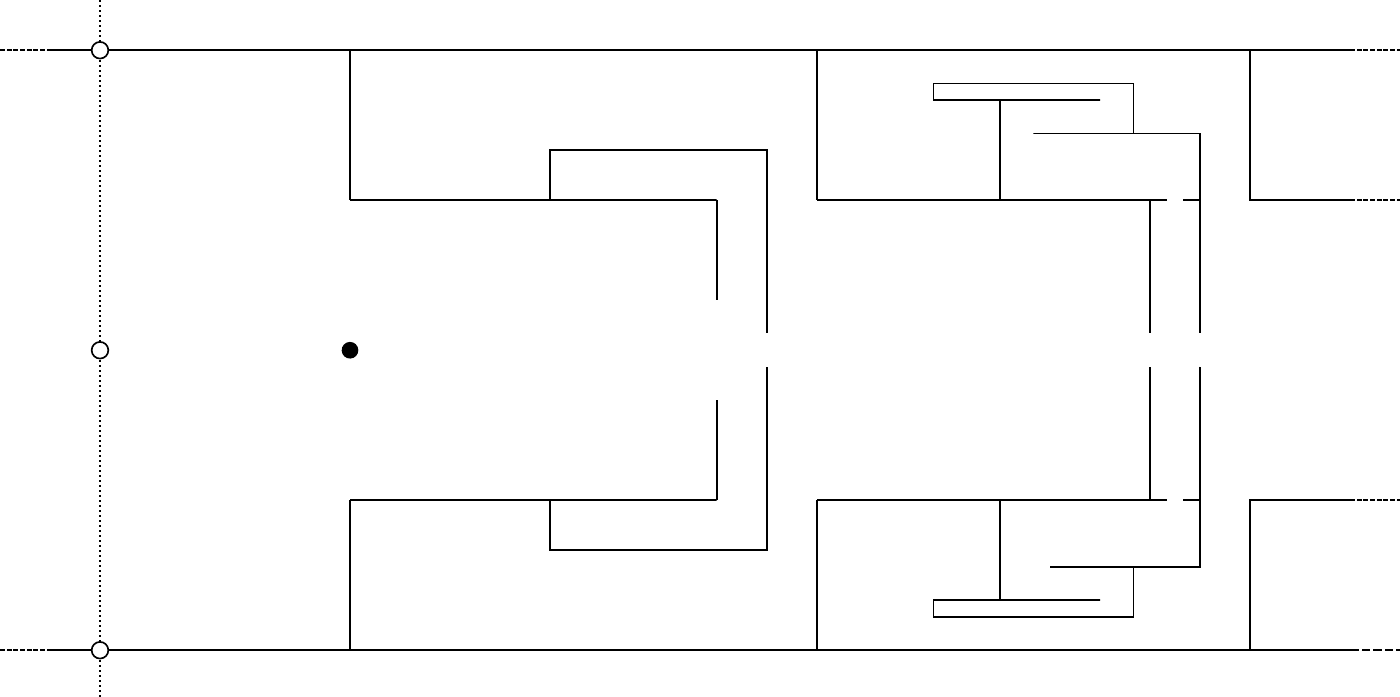
 \end{center}
 \caption{A modification $\tilde{T}$ of the tract $T$ from Figure \ref{fig:tract_definition}, which allows the function from Theorem
     \ref{thm:mainarclike} to be constructed in the class $\classS$, without additional tracts.\label{fig:tract_construction_S}}
\end{figure}

\begin{prop}
 Let the assumptions of Theorem \ref{thm:arclikeexistenceBlog} and Proposition~\ref{prop:arclikeexistenceprecise}
    be satisfied, with the
   additional requirement that $M_k \geq P$ for all $k$, where $P\geq 10$ is a certain universal constant.
 
 Then there is a 
   simply-connected domain
    $\tilde{T}\subset \{a+ib \colon  |b|\leq \pi\}$ and a conformal isomorphism
    $\tilde{F}\colon \tilde{T}\to \HH$ with the following properties.
\begin{enumerate}[(a)]
   \item The domain $T \defeq  \{ \tilde{F}^{-1}(z)-1\colon \re z > 1 \}$ and the map
              $F\colon T\to \HH; F(z) = \tilde{F}(z+1) - 1$ satisfy the conclusions of 
             Theorem~\ref{thm:arclikeexistenceBlog} and 
             Proposition~\ref{prop:arclikeexistenceprecise}.\label{item:disjointtyperestriction}
   \item $G\defeq \C\setminus \exp(\tilde{T})$  (with suitable vertices added) is a bounded-geometry tree, 
              and the function $\tau$ defined by
              $\tau(\exp(z)) \defeq  \tilde{F}(z)$ satisfies the hypotheses of Theorem \ref{thm:bishop}.\label{item:bishop}
   \item The hyperbolic distance (in $T$) between the orbit of the constructed Julia continuum and the set $\exp^{-1}(G(r_0))$ is
             bounded from below. \label{item:Gr0}
\end{enumerate}
\end{prop}
\begin{proof}[Sketch of proof]
  We slightly modify the construction of the tract $T$ from the proof of Theorem~\ref{thm:arclikeexistenceBlog} 
    as shown in Figure \ref{fig:tract_construction_S}. That is,
    the tract now also contains the half-strip $\{\re z < 0, |\im z|<\pi\}$, and furthermore
    additional ``chambers'' $V_k$ are added around the domains $U_k$, attached between real parts 
     $R_k+1$ and $R_k+2$. This is done to ensure that $\tilde{T}$ is dense in the 
    strip $\{|\im z | < \pi\}$ and the 
    complement of $\exp(\tilde{T})$ is an infinite tree in the plane. We also add a second additional ``iris'' 
    at real part $R_k+1$, closing the tract up to a small opening size $\tilde{\chi}_k$. This is chosen in step I4, after $\xi_k$ has 
    been chosen, to make sure that $F_k^{-1}(x + 2\pi is(k) )$ does not enter the chamber $V_k$ for any $x\geq 0$, and hence~\ref{I4item:wholegeodesic} still holds.

   The function
     $\tilde{F}$ is chosen such that $\tilde{F}(p)=p$ and $\tilde{F}'(p)>0$, where $p$ is a sufficiently large
     universal constant, and $P\geq p+4$. If $p$ is sufficiently large, this ensures that 
     $T\subset \{z\in\HH\colon \re z> 4\}$, and in particular 
     the function $F$ defined in~\ref{item:disjointtyperestriction} is of disjoint type. We can now
     carry out the construction precisely  as in Section~\ref{sec:arclikeexistence}, 
     so that~\ref{item:disjointtyperestriction} holds. 
     However, some additional properties will be required in the recursive construction to ensure that~\ref{item:bishop} and~\ref{item:Gr0} hold for an appropriate
    partition of $G$ into edges.

  We next discuss how to subdivide the edges in such a way as to obtain a bounded-geometry tree. 
    As the
    exponential map is conformal, we can carry  out this discussion in logarithmic coordinates, i.e.\ for the
    boundary of $\tilde{T}$.        In these, all edges will be straight lines, and 
     meet at angles that are multiples of $\pi/2$. 
Everything will be carried out in a manner symmetric to the real axis, so that 
    edges on the horizontal lines at imaginary parts $\pm \pi$ are compatible when mapped by  the exponential map. 
     The upper  and  lower boundaries of the part of $\tilde{T}$ in  the left half-plane will be taken  to  be 
     edges themselves (they map to the single edge $[-1,0]$ after projection by the exponential map). 
    With such choices, $\exp(\partial\tilde{T})$ will indeed be an infinite locally bounded planar tree 
     satisfying~\ref{item:C2bounds} and~\ref{item:boundedangles} in the definition of a bounded-geometry tree. It remains to explain how to choose
     the subdivision such that~\ref{item:adjacentedges} and~\ref{item:nonadjacent} also hold.

     Along the horizontal edges of the central strip and the upper and lower
    boundaries of the tract, we use edges whose length is uniformly bounded from below and above. 
     Where there is an  iris that closes off a gap up to  a small opening
    (at the real parts $R_k$, the arcs $C_k$, and the arcs $C_k^j$ constructed when defining $U_k$), 
    we  can use edges whose length decreases geometrically, so that the edges adjacent to  an opening
    are comparable in length to the size of the opening. For example, for the irises at real part $R_k$, the
    edges adjacent to the boundary of the central strip have length bounded from below, while the two  edges adjacent
    to the gap have length comparable to $\chi_k$. (More precisely, in order to be able to ensure~\ref{item:Gr0}
    in the following,  these edges should be shorter than, but  still comparable to, $\chi_k/(4r_0)$.) 

   Recall that, when constructing the domain $U_k$, we may choose the 
    domains $U_k^j$ as rectangles, sitting on top  of each other with $U_k^{m_k}$ adjacent to the
     central strip, and  $U_k^1$ closest to the upper boundary. The width of these rectangles can be chosen
     to be bounded from below. We may also choose the height of $U_k^{m_k}$ to be  bounded from below,
     while the height of $U_k^j$ decrease geometrically as $j$ changes from $m_k$ to  $1$, in such a way that the total height of $U_k$ is bounded by
     $\pi/4$, independently of the number $m_k$ of domains $U_k^j$. Furthermore, we can choose a subdivision such that the
    length of edges used in $\partial U_k^j$ decrease geometrically in $j$, and such that
    edges are only ever adjacent to edges that are comparable in size.  
    Since the imaginary parts of $U_k$ are bounded by $3\pi/4$, the size of the gap between the upper boundary of
     $\tilde{T}$ and the final domain  $U_k$ is bounded from below; ensures
     that condition~\ref{item:nonadjacent}
     from the definition of bounded-geometry trees is satisfied.
  
   In this way, we have shown how to construct $\tilde{T}$ such that $\C\setminus\exp(\tilde{T})$ can be given the structure of a bounded-geometry
     tree $\tilde{G}$. The map $\tau$ in Theorem~\ref{thm:bishop} is given precisely by the  map $z\mapsto \tilde{F}(\Log z)$. 
     We must show that the choices can additionally be made so that the $\tau$-length condition~\ref{item:taulength} in Theorem~\ref{thm:bishop}
     is also satisfied, as well as~\ref{item:Gr0}. As we will see, the first requires us to choose $\chi_k$ sufficiently small in step I4, 
      and the second to modify the partition $\tilde{G}$ of $\partial \tilde{T}$ on the domain
     $U_k^1$, depending on this choice of $\chi_k$. 
     
   First, it is straightforward to check that all edges of $\tilde{G}$ on the boundary of the central strip have
     $\tau$-length at least $\pi$, provided that $p$ was chosen large enough. Furthermore,
     it follows for the vertical edges leading to an iris in  the main part of the tract, using the properties of our partition and 
      simple estimates on harmonic measure. 
  
    By choosing $\chi_k$ sufficiently small, we can also ensure that each edge in the domains $U_j$ and $V_j$ is mapped to a sufficiently 
      large interval under $\tilde{F}$, and hence~\ref{item:taulength} holds. Indeed, the subdivision
      into edges that we discussed is independent of the choice of $\chi_k$, and as $\chi_k$ becomes small, the length of the
      image of each edge becomes large. (This is the same argument as in step I4 of the inductive construction in
      Section~\ref{sec:arclikeexistence}, considering the harmonic measure of the edge as viewed from $R_k+1$.) 
      Hence we can ensure~\ref{item:taulength} for the partition of $\tilde{G}$ by choosing $\chi_k$ sufficiently small. 

   It remains to address~\ref{item:Gr0}. Let $\CH=C\cup\{\infty\}$ be the constructed Julia continuum. 
    Then $F^n(C)$ lies within a bounded hyperbolic distance (in $T$ resp.\ $T+2\pi is(k)$) of
     \begin{itemize}
       \item the geodesic $[p,\infty)$ when
           $n\neq n_k,n_k-1$, 
       \item the geodesic $[p,\infty)+2\pi i s(k)$ when $n=n_k$, and
        \item the geodesic $F^{-1}([M_k,\infty)+2\pi i s(k))$when $n=n_k-1$. 
    \end{itemize}
   Hence, by construction,~\ref{item:Gr0} holds automatically for $F^n(C)$ when $n\neq n_k-1$ (provided the 
     initial lengths of edges were chosen sufficiently small). Furthermore, it is not difficult to see that the
     geodesic in the remaining case stays away sufficiently far from the boundary to ensure that~\ref{item:Gr0} holds (again assuming the edges were chosen sufficiently small),
     except possibly for the part of $F^{n_k-1}(C)$ contained in the domain $U_k^1$.

   To address the situation in $U_k^1$, we need to modify the partition $\tilde{G}$. Indeed, as $\chi_k\to 0$, the point
     $z_k\defeq F^{-1}(M_k+2\pi i  s(k))$ will tend closer and close to the boundary of $\tilde{T}$. Hence we can choose the final partition of $\partial U_k^1$ only once
     the value of $\chi_k$ has been fixed. This is done by subdividing the edges of $\tilde{G}$ in $U_k^1$ in such a way that they shrink as they approach
     $F_k^{-1}(2\pi i s(k))$, up to a length that agrees with $\dist(z_k,\partial U_k^1)$ up to a small universal multiplicative constant 
     (depending on $r_0$). This ensures that~\ref{item:Gr0} also holds for the final part of the geodesic. If $p$, and hence $M_k$, was chosen large enough, 
     then the $\tau$-length of the resulting edges (whose harmonic measure seen from $z_k$ is bounded from below by a universal constant) is also 
     at least $\pi$. The tract $\tilde{T}$ with the resulting partition of $\partial\tilde{T}$ into a graph $G$, then satisfies all of the conditions stated in the proposition. 
\end{proof}
  
As the construction for Theorem~\ref{thm:arclikeexistenceBlog} is used to construct the examples
    in Theorems~\ref{thm:nonuniform},~\ref{thm:mainarclike},~\ref{thm:nonescapingaccessible} and~\ref{thm:onepointuniform},
    all of these can be constructed in class $\classS$ with a single tract, two critical values, no finite
    asymptotic values, and no critical points of degree greater than $16$
    (resp.\ 4 after taking account of the modifications in Remark~\ref{rmk:bishopbound}).
   
The construction for Theorem~\ref{thm:boundedexistence} is similar,
   modifying the proof in Section~\ref{sec:boundedexistence}.
   Here we should choose a partition of the boundary of the strip 
   $S$ into edges first, such that these edges have length at least $\pi$ when 
   mapped forward under the conformal map $F_S$, and 
   $R_k^-$ is chosen sufficiently large so that the boundary of the domain $U_k$ 
   can be subdivided in such a way to ensure bounded geometry.
   This is clearly possible because the set $\Omega_k$, and hence the number of rectangles in $U_k$, 
    is known before the value
     $R_k^-$ is chosen. We should also make sure that the domain $U_k$ fills out the piece of the strip between 
    real parts
    $R_k^-$ and $R_k^+$ (unlike in Figure \ref{fig:tract_definition-bounded}), and, after $U_k$ is chosen, 
    reduce the size
    of the cross-cut $C_k^-$ in order to ensure $\tau$-length at least $\pi$ for each of the edges.

  Now we turn to Theorem~\ref{thm:periodicexistence}, as  
   established in  Section~\ref{sec:periodicexistence}. The construction here 
    is a little more delicate, because the domain $U_k$, for large $k$, can 
   potentially reach very far back to the left. More precisely, for fixed $k$, there may be some large values of $\tilde{k}$ such that
   $g(A_{\tilde{k}})$ contains points of $A_k$. 
    If the set $\Omega_{\tilde{k}}$ (which is not known at the time that $R_{k-3}$ is chosen) is very large, then potentially there may be 
   many pieces intersecting the line $\{\re z = R_{k-3}\}$, and hence it may be difficult to control the size of the corresponding edges of the tree here. 

  To resolve this problem, we observe that~--- similarly to Proposition~\ref{prop:countable}~--- we may assume that the function $g$ is piecewise linear on the interval $[0,1)$, with countably many
   points of nonlinearity, which accumulate only at $1$, and everywhere locally non-constant. This means that we can 
   choose the domain
   $U_k$ to consist (essentially, and apart from a number of ``irises'' opening at the arcs  $C_k^j$) 
    of finitely many rectangles, one for each interval of monotonicity of $g$ in the 
   interval $[\theta_{k+1}, \theta_{k+2}]$. This means that the shape of the tract depends only on the function $g$; only the way that the
   tract is stretched along the real axis varies inductively. 
   This observation allows us to carry out the desired construction.
   Again, once the piece $U_k$ is chosen, we should shrink the opening size of the cross-cut $C_{k-1}$ so that
   edges have $\tau$-length at least $\pi$. This also allows us to ensure that the sequence $(R_k)$ grows sufficiently rapidly. 

 It remains to consider Theorem~\ref{thm:pseudoarcs}. The preceding discussion leads to a disjoint-type entire 
    function of the
    desired form (having one tract,  two critical  values, no asymptotic values, and no critical points of
     multiplicity five or  greater) having a pseudo-arc Julia continuum. 
    It is easy to see that the entire function obtained in the 
    preceding discussion still satisfies the bounded decorations condition. 
   By Corollary~\ref{cor:pseudoarcs} all Julia continua are pseudo-arcs, as desired.

\bibliographystyle{amsalpha}
\bibliography{../../Biblio/biblio}

\end{document}